\documentclass[A4, 12pt]{amsart}
\usepackage{amsmath}
\usepackage{amsthm}
\usepackage{amssymb}
\usepackage{amscd}
\usepackage{bm}
\usepackage[all]{xy}
\usepackage{mathrsfs}
\usepackage[hidelinks]{hyperref}

\usepackage{upref}
\usepackage{graphics}
\usepackage[bbgreekl]{mathbbol}
\usepackage {color}

\makeatletter
  
  \@addtoreset{equation}{section}
\makeatother

\makeatletter
\renewcommand{\@secnumfont}{\bfseries}
\makeatother

\theoremstyle{plain} 
 \newtheorem{theorem}[equation]{Theorem}
 \newtheorem{proposition}[equation]{Proposition}
 
 \newtheorem{lemma}[equation]{Lemma}
 \newtheorem{sublemma}[equation]{Sublemma}
 \newtheorem{corollary}[equation]{Corollary}

 \theoremstyle{definition}
 \newtheorem{definition}[equation]{Definition}
 \newtheorem{construction}[equation]{Construction}

\theoremstyle{remark}
 \newtheorem{remark}[equation]{Remark}
 \newtheorem{example}[equation]{Example}
 \newtheorem{condition-list}[equation]{}
 
\def \N {\mathbb{N}}
\def \Z {\mathbb{Z}}

\def \Spf {\text{\rm Spf}}
\def \Spec {\text{\rm Spec}}
\def \BKF {\text{\rm BKF}}

\def \Hom {\text{\rm Hom}}

\def \HIG {\text{\rm HIG}}
\def \Ker {\text{\rm Ker}}

\def \id {\text{\rm id}} 
\def \fpqc {\text{\rm fpqc}}
\def \Zar {\text{\rm Zar}}
\def \Der {\text{\rm Der}}
\def \MIC {\text{\rm MIC}}

\def \Crystal{\text{\rm CR}}

\def \qnilp {\text{\rm q-nilp}}
\def \ZAR {\text{\rm ZAR}}
\def \Mod {\text{\rm\bf Mod}}
\def \Cone {\text{\rm Cone}}

\def \Cov {\text{\rm Cov}}
\def \cont{\text{\rm cont}}
\def \Set {\text{\rm\bf Set}}
\def \Map {\text{\rm Map}}
\def \Spa {\text{\rm Spa}}
\def \CMap {\text{\rm Map}}

\def \fSet {\text{\rm\bf fSet}}
\def \hCrystal {\widehat{\text{\rm CR}}}
\def \Ob {\text{\rm Ob}\,}
\def \disc {\text{\rm disc}}
\def \ev {\text{\rm ev}}
\def \proet {\text{\rm pro\'et}}

\def \et {\text{\'et}}
\def \fproj {\text{\rm fproj}}
\def \BKF {\text{\rm BKF}}
\def \AffZar {\text{\rm AffZar}}
\def \Aff {\text{\rm Aff}}

\def \Qis {\text{\rm Qis}}
\def \Mor {\text{\rm Mor\,}}

\def \Omegab {\Omega^{\bullet}}

\def \hy {\text{-}}
\def \bis {$^{\text{bis}}$\,}
\def \projl {\underset{\leftarrow}\ell}

\def \vf {v_{\text{\rm f}}}
\def \vS {v_{\mathbf{S}}}
\def \pq {[p]_q}

\def \varprojliml {\varprojlim\limits}

\def \mfi {\mathfrak{i}}
\def \ff  {\mathfrak{f}}
\def \fp {\mathfrak{p}}
\def \fg {\mathfrak{g}}
\def \fX {\mathfrak{X}}
\def \fY {\mathfrak{Y}}

\def \fU {\mathfrak{U}}
\def \fD {\mathfrak{D}}

\def \fB {\mathfrak{B}}

\def \fV {\mathfrak{V}}

\def \fh {\mathfrak{h}}

\def \off  {\overline{\mathfrak{f}}}

\def \mtg {\mathtt{g}}
\def \mti {\mathtt{i}}
\def \mtp {\mathtt{p}}

\def \CO {\mathcal {O}}
\def \CB {\mathcal{B}}
\def \CC {\mathcal{C}}

\def \CH {\mathcal{H}}

\def \CF {\mathcal{F}}
\def \CJ {\mathcal{J}}
\def \CI {\mathcal{I}}
\def \CK {\mathcal{K}}
\def \CA {\mathcal{A}}
\def \CS {\mathcal{S}}
\def \CT {\mathcal{T}}
\def \CM {\mathcal{M}}
\def \CN {\mathcal {N}}
\def \CG {\mathcal{G}}
\def \CP {\mathcal{P}}
\def \CL {\mathcal{L}}

\def \uCM {\underline{\mathcal{M}}}

\def \CKb {\mathcal{K}^{\bullet}}
\def \oCKb {\overline{\mathcal{K}}^{\bullet}}

\def \hCO {\widehat{\mathcal{O}}}
\def \hbZ {\widehat{\mathbb{Z}}}

\def \bM {\mathbb{M}}
\def \bA {\mathbb{A}}
\def \bU {\mathbb{U}}
\def \bV {\mathbb{V}}
\def \bL {\mathbb{L}}

\def \ubM {\underline{\mathbb{M}}}
\def \ubA {\underline{\mathbb{A}}}

\def \scrS {\mathscr{S}}

\def \hA {\widehat{A}}
\def \hR {\widehat{R}}

\def \hj {\widehat{j}}

\def \hpartial{\widehat{\partial}}
\def \halpha {\widehat{\alpha}}
\def \hbeta {\widehat{\beta}}
\def \hvarphi {\widehat{\varphi}}
\def \hgamma{\widehat{\gamma}}

\def \hotimes {\widehat{\otimes}}

\def \tf {\widetilde{f}}

\def \tD {\widetilde{D}}

\def \tI {\widetilde{I}}

\def \ts {\widetilde{s}}

\def \tv {\widetilde{v}}

\def \tg {\widetilde{g}}
\def \tp {\widetilde{p}}

\def \txi {\widetilde{\xi}}

\def \tLambda {\widetilde{\Lambda}}

\def \tchi {\widetilde{\chi}}

\def \tXi {\widetilde{\Xi}}

\def \tnu {\widetilde{\nu}}

\def \tCC {\widetilde{\CC}}
\def \tCA {\widetilde{\CA}}
\def \tCM {\widetilde{\CM}}

\def \oJ {\overline{J}}

\def \oR {\overline{R}}
\def \oA {\overline{A}}
\def \oD {\overline{D}}

\def \oY {\overline{Y}}
\def \oX {\overline{X}}

\def \of {\overline{f}}
\def \otau {\overline{\tau}}

\def \otheta{\overline{\theta}}

\def \ootau {\overline{\overline{\tau}}}

\def \oof {\overline{\overline{f}}}

\def \ofD {\overline{\mathfrak{D}}}

\def \uM {\underline{M}}
\def \uT {\underline{T}}
\def \uA {\underline{A}}

\def \uc {\underline{c}}
\def \ug {\underline{g}}

\def \ut  {\underline{t}}

\def \ui {\underline{i}}

\def \ur {\underline{r}}
\def \uvarphi{\underline{\varphi}}
\def \ualpha {\underline{\alpha}}
\def \ubeta {\underline{\beta}}

\def \uvarepsilon {\underline{\varepsilon}}

\def \utheta {\underline{\theta}}
\def \ugamma {\underline{\gamma}}
\def \uGamma {\underline{\Gamma}}
\def \uzero {\underline{0}}
\def \uone {\underline{1}}
\def \upartial {\underline{\partial}}
\def \unabla{\underline{\nabla}}

\def \bmI{\bm{I}}
\def \bmg {\bm{g}}

\def \trmf {\text{\rm f}}

\DeclareSymbolFontAlphabet{\mathbbm}{bbold}
\DeclareSymbolFontAlphabet{\mathbb}{AMSb}%

\newcommand{\prism} {\scalebox{0.8}{$\mathbbm \Delta$}}

\newcommand{\pone} {\scalebox{0.5}{$(1)$}}
\newcommand{\bcdot} {\scalebox{0.85}{\hspace{0.05em}$\boldsymbol{\cdot}$\hspace{0.05em}}}

\begin{document}
\title{Prismatic cohomology and $A_{\inf}$-cohomology with coefficients.}


\author{Takeshi Tsuji}
\address{Graduate School of Mathematical Sciences, The University of Tokyo, 
3-8-1 Komaba, Meguro-ku, Tokyo, 153-8914, Japan}
\email{t-tsuji@ms.u-tokyo.ac.jp}

\begin{abstract}
For a smooth $p$-adic formal scheme over the ring of integers of a perfectoid field of 
mixed characteristic $(0,p)$ containing all $p$-power roots of unity, 
we prove that the prismatic cohomology of a locally finite free prismatic crystal
is isomorphic to the $A_{\inf}$-cohomology of the corresponding relative Breuil-Kisin-Fargues
module, which is a certain type of locally finite free $\bA_{\inf}$-module,
on the pro\'etale site of the generic fiber. We use a global description
of the former in terms of $q$-Higgs modules via cohomological descent.
We also discuss its compatibility with inverse image functors, scalar extensions under 
Frobenius, and tensor products.
\end{abstract}

\maketitle

\tableofcontents

\section*{Introduction}
Inspired by generalized representations introduced by G.~Faltings in his
theory of $p$-adic Simpson correspondence \cite{FaltingsSimpson}, the notion of
relative Breuil-Kisin-Fargues modules, a certain analogue over $A_{\inf}$, 
was introduced and studied in \cite{MT} as a theory of coefficients for 
$A_{\inf}$-cohomology \cite{BMS}. In particular, we have a fully faithful
functor from the category of locally finite free prismatic crystals to that 
of relative Breuil-Kisin-Fargues modules, which induces an equivalence between objects
with Frobenius structure. In this paper,  we prove that the functor 
is compatible with the two integral cohomology theories:
prismatic cohomology and $A_{\inf}$-cohomology, i.e.,
the two coefficient systems give the same integral cohomology.

Let $C$ be a perfectoid field of mixed characteristic $(0,p)$ containing
all $p$-power roots of unity, let $\CO$ be the ring of 
integers of $C$, let $\CO^{\flat}$ be the tilt of $\CO$:
$\varprojlim_{\N, \text{Frob}}\CO/p\CO
=\varprojlim_{\N, x\mapsto x^p}\CO$, and let $A_{\inf}$ 
denote $W(\CO^{\flat})$, which is equipped with a lifting of Frobenius 
$\varphi$ and the map of Fontaine $\theta\colon A_{\inf}\to \CO$. 
We choose and fix a compatible system of primitive $p^n$th
roots of unity $\varepsilon=(\zeta_n)_{n\in \N}$, $\zeta_{n+1}^p=\zeta_n$
$(n\in \N)$ in $\CO$, and define  elements of $A_{\inf}$ by 
$\mu=[\varepsilon]-1$ and $\txi=\varphi(\mu)\mu^{-1}$. 
We regard $\CO$ as an $A_{\inf}$-algebra by the map
$\theta\circ\varphi^{-1}\colon A_{\inf}\to \CO$, which induces
an isomorphism $A_{\inf}/\txi A_{\inf}\xrightarrow{\cong}\CO$.
With $A_{\inf}$ being equipped with the $\delta$-structure
corresponding to $\varphi$, the pair $(A_{\inf}, (\txi))$  becomes
a bounded prism.

Let $\fX$ be a quasi-compact, separated, smooth, $p$-adic formal scheme 
over $\CO$, let $X$ be its adic generic fiber, and let 
$\nu_{\fX}\colon X_{\proet}^{\sim}\to \fX_{\Zar}^{\sim}$ be the 
projection morphism of topos. 
Then a relative Breuil-Kisin-Fargues module $\bM$ over $\fX$ is 
defined to be a locally finite free $\bA_{\inf,X}$-module 
on $X_{\proet}$ ``trivial modulo $<\mu$ Zariski locally on $\fX$". Following \cite{BMS},
its $A_{\inf}$-cohomology $R\Gamma_{A_{\inf}}(\fX,\bM)$ is defined to be the cohomology of 
the complex of $A_{\inf}$-modules $A\Omega_{\fX}(\bM)
:=L\eta_{\mu}(\widehat{R\nu_{\fX*}\bM})$ on $\fX_{\Zar}$, 
where the hat denotes the derived $p$-adic completion \cite[\S6]{MT}. 
For the prismatic side, we consider a locally finite free crystal
$\CF$ of $\CO_{\fX/A_{\inf}}$-modules on the prismatic 
site $(\fX/(A_{\inf},(\txi)))_{\prism}$ of $\fX$ over 
the bounded prism $(A_{\inf},(\txi))$ with its structure sheaf of rings
denoted by $\CO_{\fX/A_{\inf}}$.
Let $u_{\fX/A_{\inf}}$ denote the projection morphism of topos
$(\fX/(A_{\inf},(\txi)))_{\prism}^{\sim}\to \fX_{\Zar}^{\sim}$. Then the cohomology
of $\CF$ is given by the cohomology of the complex of $A_{\inf}$-modules
$Ru_{\fX/A_{\inf}*}\CF$ on $\fX_{\Zar}$. As mentioned in the first paragraph,
we have a fully faithful functor $\bM_{\BKF,\fX}$ from the category of locally finite
free crystals on $(\fX/(A_{\inf},(\txi)))_{\prism}$ to that of relative Breuil-Kisin-Fargues modules
on $\fX$. Our theorem is stated as a comparison isomorphism
between the two complexes of $A_{\inf}$-modules on $\fX_{\Zar}$ 
as follows.

\begin{theorem}
[Theorem \ref{thm:PrismCohAinfCohGlobComp}]\label{thm:IntroMainThm}
Let $\fX$ be a quasi-compact, separated, smooth, $p$-adic formal scheme over $\CO$.
Let $\CF$ be a locally finite free crystal of $\CO_{\fX/A_{\inf}}$-modules
on $(\fX/(A_{\inf},(\txi)))_{\prism}$ 
and let $\bM$ be the associated relative Breuil-Kisin-Fargues module $\bM_{\BKF,\fX}(\CF)$  on $\fX$.
Then we have the following canonical isomorphism 
in $D^+(\fX_{\Zar},A_{\inf})$ functorial in $\CF$.
\begin{equation}\label{eq:IntroMainCompIsom}
Ru_{\fX/A_{\inf}*}\CF\xrightarrow{\;\cong\;} A\Omega_{\fX}(\bM)
\end{equation}
\end{theorem}

When $\fX$ is framed small affine (i.e. affine with invertible coordinates given), 
then both sides of \eqref{eq:IntroMainCompIsom} have the same description in terms of the $q$-Higgs 
complex with coefficients in an integrable $q$-Higgs module on 
a smooth lifting of $\fX$ over $A_{\inf}$ equipped with liftings of the coordinates
 \cite[\S6]{MT}, \cite[\S13]{TsujiPrismQHiggs}.  More generally, for 
the prismatic side, we have a similar description in terms of a
$q$-Higgs complex on the bounded prismatic envelope of an
embedding of $\fX$ into a smooth affine formal scheme over 
$A_{\inf}$ equipped with coordinates, and it allows us to give a 
description of $Ru_{\fX/A_{\inf}*}\CF$ for a general $\fX$ in terms of $q$-Higgs
complexes via cohomological descent by taking a Zariski hypercovering of $\fX$ 
and its embedding into an affine simplicial smooth formal scheme with coordinates 
over $A_{\inf}$ \cite[\S15]{TsujiPrismQHiggs}. We construct a comparison morphism from $Ru_{\fX/A_{\inf}*}\CF$
to $A\Omega_{\fX}(\bM)$ via this description, and then prove that it is an isomorphism
by reducing it to the case where $\fX$ is a framed small affine formal scheme
mentioned above.

We also show the compatibility of \eqref{eq:IntroMainCompIsom} with inverse image functors 
(Proposition \ref{eq:BKFPrismLocCohCompSimpFunct}), 
scalar extensions under Frobenius 
(Proposition \ref{prop:PrismAinfCohIsomFrobCompGlobal}),
and tensor products
(Proposition \ref{prop:PrismAinfCohIsomProdCompGlobal})
by using the corresponding compatibility, verified in \cite[\S15]{TsujiPrismQHiggs},
for the description of $Ru_{\fX/A_{\inf}*}\CF$ 
in terms of $q$-Higgs complexes via cohomological descent. 
As observed in loc.~cit., the compatibility with
tensor products is a little involved because products of $q$-Higgs
complexes are not compatible with pullback by non smooth morphisms
such as diagonal immersions; to make it compatible,  we
take products after pulling back to the envelope
of the product of two copies of a given embedding.

In a subsequent paper in preparation with Abhinandan, we
study syntomic complexes with coefficients in $\CF$
and the complex of nearby cycles of the $p$-adic
\'etale local system associated to $\bM$ as an application
of the results of \cite{MT}, \cite{TsujiPrismQHiggs}, and this paper.

\begin{remark} 
(1) For the constant case $\CF=\CO_{\fX/A_{\inf}}$, the isomorphism
\eqref{eq:IntroMainCompIsom} is given in \cite[\S17]{BS} by B.~Bhatt and P.~Scholze via
$q$-de Rham cohomology.\par

(2) If $\CF$ is equipped with a Frobenius structure $\varphi_{\CF}$, which induces a Frobenius
structure $\varphi_{\bM}$ on $\bM$, we have the associated 
locally finite free lisse $\Z_p$-sheaf $\bL=(\bM\otimes_{\bA_{\inf,X}}W(\hCO_{X^{\flat}}))^{\varphi_{\bM}=1}$
on  $X_{\proet}$  and an isomorphism of $\bA_{\inf,X}[\frac{1}{\mu}]$-modules
$\bL\otimes_{\hbZ_p}\bA_{\inf,X}[\frac{1}{\mu}]\cong \bM[\frac{1}{\mu}]$
\cite[Proposition 6.15]{MT}.
Therefore \eqref{eq:IntroMainCompIsom} induces an isomorphism
\begin{equation}\label{eq:LocPrismEtaleComp}
(Ru_{\fX/A_{\inf}*}\CF)[\tfrac{1}{\mu}]\xrightarrow{\;\cong\;} 
R\nu_{\fX*}(\bL\otimes_{\hbZ_p}\widehat{\bA_{\inf,X}})[\tfrac{1}{\mu}],\\
\end{equation}
where $[\frac{1}{\mu}]$ means $-\otimes^L_{A_{\inf}}A_{\inf}[\frac{1}{\mu}]$
and the hat denotes the derived $p$-adic completion. 
We have an exact sequence $0\to \bL\to\mu^{-r}\bM
\xrightarrow{\varphi_{\bM}-1} \mu^{-r}\bM\to 0$ for
any integer $r$ satisfying $\txi^{r}\bM\subset\varphi_{\bM}(\bM)$
\cite[Proposition 6.15]{MT}. Therefore \eqref{eq:IntroMainCompIsom} yields a distinguished triangle
\begin{equation}\label{eq:LocEtalePrismDescrip}
R\nu_{\fX*}\bL\longrightarrow (Ru_{\fX/A_{\inf}*}\CF)[\tfrac{1}{\mu}]
\xrightarrow{\;\varphi-1\;} (Ru_{\fX/A_{\inf}*}\CF)[\tfrac{1}{\mu}]
\longrightarrow R\nu_{\fX*}\bL[1].
\end{equation}
When $\fX$ is proper
over $\CO$ and $C$ is algebraically closed, we obtain the following 
isomorphism by taking $R\Gamma(\fX_{\Zar},-)$ of \eqref{eq:IntroMainCompIsom} and
combining it with \cite[Corollary 6.3]{MT} which relies on the primitive comparison 
theorem of Scholze \cite[Theorem 5.1]{ScholzepHTRigid}.
\begin{equation}\label{eq:GlobPrismEtaleComp}
R\Gamma((\fX/(A_{\inf},(\txi)))_{\prism},\CF)[\tfrac{1}{\mu}]
\cong R\Gamma(X_{\proet},\bL)\otimes^L_{\Z_p}A_{\inf}[\tfrac{1}{\mu}]
\end{equation}

(3) I.~Gaisin and T.~Koshikawa proved relative analogues of the isomorphisms
\eqref{eq:IntroMainCompIsom}  and \eqref{eq:GlobPrismEtaleComp} for the constant coefficients
for a smooth morphism $f\colon \fX\to \fY$ of $p$-adic formal schemes 
flat locally of finite type over $\CO$ \cite[Theorems 6.28 and 6.36]{GaisinKoshikawa} when 
$C$ is algebraically closed; the latter relies on the primitive comparison theorem of Scholze \cite[Theorem 5.1,
Corollary 5.12]{ScholzepHTRigid}. It is a natural and interesting question to ask if one can
prove their analogues with coefficients by using the description of the
prismatic cohomology of $\fX\times_{\fY}\Spf(\hCO^+_Y(W))$ over 
$(\bA_{\inf,Y}(W),(\txi))$ in terms of $q$-Higgs modules via cohomological 
descent (\cite[Theorem 15.9]{TsujiPrismQHiggs}, \eqref{eq:PrismCohQHiggsCpxSimp2}) 
for each affinoid perfectoid $W$ in $Y_{\proet}$ 
lying over the generic fiber of an open affine of $\fY$,
and by generalizing the argument in \S\ref{sec:compAinfcohLocal} to this setting.

(4) 
For Laurent $F$-crystals, i.e., $F$-crystals over 
the ring obtained from the structure ring by inverting the structure 
ideal of the base prism and taking the $p$-adic completion, 
we have results analogous to \eqref{eq:LocEtalePrismDescrip} by Y.~Min and Y.~Wang
\cite[Theorem 4.1]{MinWangPrismPhiGamma},
and by H.~Guo and E.~Reinecke \cite[Theorem 6.1]{GuoReineckeCrysLocSys}
under much more general settings. For the latter, see also the preceding result 
\cite[Theorem 9.1]{BS} by Bhatt and Scholze in the case of constant coefficients.\par

For a complete discrete valuation ring $\CO_K$ of mixed characteristic $(0,p)$ with
perfect residue field and a proper smooth morphism of smooth $p$-adic formal schemes
$f\colon \fX\to \fY$ over $\CO_K$, Guo and Reinecke further derived from
the analogue for Laurent $F$-crystals a comparison theorem between
the relative prismatic cohomology of a prismatic $F$-crystal in perfect complexes
$(\CF,\varphi_{\CF})$ on $\fX$ and the relative pro-\'etale cohomology of the
\'etale realisation $T(\CF,\varphi_{\CF})$, a relative analogue
of \eqref{eq:GlobPrismEtaleComp} with more general coefficients \cite[Theorem 9.1]{GuoReineckeCrysLocSys}; 
``$p$-adic completion" is removed thanks to some global natures of the relative
cohomologies. (Prototypes of this argument can be found in 
\cite[Lemma 4.26]{BMS} and \cite[Proposition 6.15]{MT}.)\par

For logarithmic formal schemes, T.~Koshikawa and Z.~Yao proved
a generalization of the result of Bhatt and Scholze mentioned above
\cite[Theorems 6.1 and 7.30]{KoshikawaYaoPrism}, from which
they derived an analogue of \eqref{eq:GlobPrismEtaleComp} 
for constant coefficients \cite[Proposition 8.3, Remark 8.7]{KoshikawaYaoPrism}.

(5) In a recent preprint \cite{YTianPrismEtSt}, Y.~Tian showed an analogue
of the isomorphisms \eqref{eq:LocPrismEtaleComp} and 
\eqref{eq:GlobPrismEtaleComp} in the semi-stable reduction case for a complete prismatic
$F$-crystal locally finite free on the ``analytic locus"
\cite[Theorem 5.6]{YTianPrismEtSt}; in the course of
its proof, he proved an analogue of \eqref{eq:IntroMainCompIsom} in the semi-stable case
for an affine $\fX$ with ``semi-stable coordinates" by using a $q$-Higgs
module with log poles on a log smooth lifting of $\fX$ over $A_{\inf}$
\cite[Theorem 4.6 (2)]{YTianPrismEtSt}. 
\par
\end{remark}

This paper is organized as follows. After reviewing basic facts on $\delta$-structures, prisms, prismatic sites,
and prismatic crystals in \S\ref{sec:prisms}, we summarize results obtained
in \cite {TsujiPrismQHiggs} in the following three sections \S\ref{sec:TwDeriv}, 
\S\ref{sec:connection}, and \S\ref{sec:PrismCrysQHiggs}
for the convenience of the readers; we refer to these sections
instead of their originals in later sections. In \S\ref{sec:Gsheaves}, we state and prove some facts
on sheaves with action of a profinite group $G$, which is used in 
the construction of the comparison map. In particular, we give 
a description of the cohomology in terms
of a Koszul complex when $G$ is a finite free $\Z_p$-module and 
study its behavior under morphisms between $G$'s and tensor products.
Every result in \S\ref{sec:Gsheaves} should be more or less known. In \S\ref{sec:RelativeBKF}, 
we recall the construction
of the relative Breuil-Kisin-Fargues module associated to a locally finite free
prismatic crystal, summarize its local cohomological properties used in the proof of the
comparison isomorphism in \S\ref{sec:compAinfcohLocal}, and then 
review the definition of 
$A\Omega_{\fX}(\bM)$ together with its product structure and
its functoriality in $\fX$. In the short section \S\ref{sec:ProetGammaZarShv}, associated to an embedding 
of a framed small affine formal scheme $\fX$ over $\CO$ into a smooth affine formal scheme 
over $A_{\inf}$ with invertible $d$ coordinates, we give a 
construction of a morphism from the pro\'etale topos 
to the topos of $\Z_p^d$-sheaves on the Zariski site simply by evaluating
pro\'etale sheaves on the finite \'etale covers of the adic generic fiber of each affine open of $\fX$
obtained by adjoining the $p$-power roots of the images of the given $d$ coordinates.
If we apply this construction to our simplicial settings, we obtain 
a direct image functor between simplicial topos which is {\it not} cartesian, 
i.e., not compatible with the direct image functors among the 
components. Therefore we study the derived functor of such a functor 
between families of topos over a category in the last section \S\ref{sec:DTopos}.
With these preliminaries, we construct a comparison morphism 
via $q$-Higgs complexes and prove the main theorem in 
\S\ref{sec:AinfcohCompMap}, 
\S\ref{sec:compAinfcohLocal}, and \S\ref{sec:compAinfcohGlobal} as mentioned after Theorem \ref{thm:IntroMainThm}.

Throughout this paper, we fix universes $\bV$ and $\bU$ such that $\bV\in \bU$.
A site is a $\bU$-small site and its topos means the associated $\bU$-topos; by a
sheaf (resp.~presheaf), we mean a $\bU$-sheaf (resp.~$\bU$-presheaf). 
Following \cite{SGA4}, we write $C^{\sim}$ (resp.~$C^{\wedge}$) for 
the category of sheaves (resp.~presheaves) of sets on a site $C$.
When we consider a site whose object is a certain type of 
data consisting of sets and maps, we always consider the $\bU$-small site
consisting of objects with their data belonging to $\bV$.
We always work with derived categories in the classical sense, i.e.,
they are triangulated categories, and not stable $\infty$-categories.

\section{$\delta$-structures, prisms, prismatic sites, and prismatic crystals}
\label{sec:prisms}
We review prismatic sites and prismatic crystals starting by recalling 
$\delta$-structures, prisms and their fundamental properties.

Let $A$ be a ring. 
A {\it $\delta$-structure on} $A$ is a map $\delta\colon A\to A$
satisfying $\delta(0)=\delta(1)=0$, 
$\delta(x+y)=\delta(x)+\delta(y)-\sum_{\nu=1}^{p-1}p^{-1}\binom p\nu x^\nu y^{p-\nu}$,
and $\delta(xy)=\delta(x)y^p+x^p\delta(y)+p\delta(x)\delta(y)$.
By a {\it lifting of Frobenius of }$A$, we mean a ring 
endomorphism $\varphi\colon A\to A$ satisfying $\varphi(x)\equiv x^p\mod pA$.
For a $\delta$-structure $\delta\colon A\to A$ on $A$, 
the map $\varphi\colon A\to A$ defined by $\varphi(x)=x^p+p\delta(x)$
is a lifting of Frobenius of $A$. When $A$ is $p$-torsion free, this gives a 
bijection from the set of $\delta$-structures on $A$ to that 
of liftings of Frobenius of $A$. We call a pair of a ring and a $\delta$-structure
on it a {\it $\delta$-ring}. For $\delta$-rings $A$ and $A'$, a {\it homomorphism
of $\delta$-rings} (or a {\it $\delta$-homomorphism}) $A\to A'$ 
is a ring homomorphism $f\colon A\to A'$
satisfying $\delta\circ f=f\circ\delta$. For a $\delta$-ring $R$, 
we call a $\delta$-ring $A$ with a $\delta$-homomorphism
$R\to A$ a {\it $\delta$-$R$-algebra}.

Let $I$ be an ideal of a ring $A$ containing $p$ and put $\hA=\varprojlim_n A/I^n$. 
For a $\delta$-structure $\delta\colon A\to A$, we see 
$\delta(I^{n+1})\subset I^n$ $(n\in \N)$ by induction on $n$. 
This implies $\delta(a+I^{n+1})\subset \delta(a)+I^n$ for every $a\in A$ and $n\in \N$.
Therefore $\delta\colon A\to A$ is continuous with respect to the $I$-adic topology 
of $A$, and induces a map $\widehat{\delta}\colon \hA\to \hA$,
which  is a $\delta$-structure on $\hA$. We call $\widehat{\delta}$ the
{\it $I$-adic completion of $\delta$}.

A map $\delta\colon A\to A$ is a $\delta$-structure on a ring $A$ if and only 
if the map $A\to W_2(A)$ defined by $x\mapsto (x,\delta(x))$ is a 
ring homomorphism. This allows us to show the following. For a polynomial 
ring $A=R[T_1,\ldots, T_d]$ in $d$ variables over a $\delta$-ring $R$ and 
$d$ elements $a_1,\ldots, a_d$ of $A$, there exists a unique $\delta$-$R$-algebra
structure on $A$ satisfying $\delta(T_i)=a_i$ for every $i\in \N\cap [1,d]$.
For $\delta$-homomorphisms $A\to A_i$ $(i=1,2)$ of $\delta$-rings,
there exists a unique $\delta$-structure on $A_1\otimes_AA_2$
such that the homomorphisms $A_i\to A_1\otimes_AA_2$
$(i=1,2)$ are $\delta$-homomorphisms. This represents the cofiber
product of $A_1\leftarrow A\to A_2$ in the category of $\delta$-rings.

\begin{definition}[{\cite[1.9]{TsujiPrismQHiggs}}]\label{def:IadicProperties}
Let $A$ be a ring and let $I$ be an ideal of $A$.\par
(1) We say that a homomorphism of rings $f\colon A\to A'$ is {\it $I$-adically
\'etale} (resp.~{\it smooth}, resp.~{\it flat}) if  the reduction mod $I^n$ is 
\'etale (resp.~smooth, resp.~flat) for every integer $n\geq 1$.\par
(2) Let $f\colon A\to A'$ be an $I$-adically smooth homomorphism,
and let $f_n\colon A_n\to A'_n$ denote its reduction modulo $I^n$ for each integer $n\geq 1$.
We say that elements $t_1,\ldots, t_d$ of $A'$ are {\it $I$-adic coordinates of $f$}
if $d(t_i\bmod I^n)$ $(1\leq i\leq d)$ form a basis of the differential module 
$\Omega^1_{A'_n/A_n}$ of $f_n$ for every integer $n\geq 1$.
(The condition is equivalent to saying that the 
$A_n$-homomorphism from the polynomial algebra $A_n[T_1,\ldots, T_d]$ to $A'_n$
defined by $T_i\mapsto t_i$ $(1\leq i\leq d)$ is \'etale.)
\end{definition}

\begin{remark}[{\cite[1.10]{TsujiPrismQHiggs}}]\label{rmk:IadicFlatRegSeq}
Let $A$ be a ring, let $I$ be an ideal of $A$, and let $A'$ be an $A$-algebra.
If $I$ is generated by a regular sequence which is also $A'$-regular,
and $A'/IA'$ is flat over $A/I$, then we see that 
$A'$ is $I$-adically flat over $A$ by the local criteria of flatness.
\end{remark}

\begin{proposition}[{\cite[Lemma 2.18]{BS}, \cite[1.11, 1.12]{TsujiPrismQHiggs}}]
\label{prop:DeltaStruEtaleMap}
Let $A$ be a $\delta$-ring, and let $I$ be an ideal of $A$ containing $p$.\par
(1) Let $A'$ be an $I$-adically \'etale $A$-algebra (Definition \ref{def:IadicProperties} (1)) 
$I$-adically complete and separated.
Then $A'$ admits a unique $\delta$-$A$-algebra structure.\par
(2) Let $A'$ and $B$ be $\delta$-$A$-algebras, and assume that
$A'$ is $I$-adically \'etale over $A$ and that $B$ is $I$-adically separated.
Then any homomorphism $f\colon A'\to B$ of $A$-algebras is a 
$\delta$-homomorphism.
\end{proposition}

\begin{definition}[{\cite[Definition 3.2, Lemma 3.7 (1)]{BS}, \cite[4.1]{TsujiPrismQHiggs}}]
A {\it $\delta$-pair} $(A,I)$ is a pair of a $\delta$-ring $A$ and its ideal $I$.
We say that a $\delta$-pair $(A,I)$ is a {\it bounded prism} if
(i) $I$ is an invertible ideal, (ii) $A$ is $(pA+I)$-adically complete and separated,
(iii) $p\in I+\varphi(I)A$, and (iv) $A/I[p^{\infty}]=A/I[p^N]$ for some integer $N\geq 1$.
(Under the conditions (i) and (ii), the condition (iii) is known to be equivalent to 
$\delta(d)\in A^{\times}$ when  $I=dA$ \cite[Lemma 2.25]{BS}.)
A {\it homomorphism of $\delta$-pairs} (resp.~{\it bounded prisms}) $(A,I)\to (A',I')$
is a homomorphism of $\delta$-rings $f\colon A\to A'$ satisfying $f(I)\subset I'$.
\end{definition}

\begin{proposition}[{\cite[4.5]{TsujiPrismQHiggs}}]\label{prop:PrismModIComp}
For a bounded prism $(A,I)$,  $A/I^n$ is $p$-adically complete and separated
for every integer $n\geq 1$. 
\end{proposition}

\begin{proposition}[{\cite[Lemma 3.5]{BS}}]
For any homomorphism of bounded prisms $f\colon (A,I)\to (A',I')$, we have
$I'=f(I)A$.
\end{proposition}

\begin{proposition}[{\cite[4.6]{TsujiPrismQHiggs}}]
\label{prop:bddPrismFlatAlg}
Let $(A,I)$ be a bounded prism, and let $A'$ be a $(pA+I)$-adically flat
$\delta$-$A$-algebra $(pA+I)$-adically complete and separated. Then
the pair $(A', IA')$ is a bounded prism.
\end{proposition}

\begin{definition}[{\cite[Proposition 3.13]{BS}, \cite[4.8]{TsujiPrismQHiggs}}]
\label{def:bddPrismEnv}
Let $(R,I)$ be a bounded prism, and let $(A,J)$ be a $\delta$-pair over 
$(R,I)$. We say that a homomorphism $f\colon (A,J)\to (D,ID)$ of $\delta$-pairs
over $(R,I)$ is a {\it bounded prismatic envelope of $(A,J)$ over $(R,I)$}
if $(D,ID)$ is a bounded prism, and any homomorphism 
$g\colon (A,J)\to (B,IB)$ of $\delta$-pairs over $(R,I)$ with $(B,IB)$
a bounded prism uniquely factors through a homomorphism of 
bounded prisms $h\colon (D,ID)\to (B,IB)$ over $(R,I)$; $g=h\circ f$.
\end{definition}

\begin{proposition}[{\cite[Proposition 3.13]{BS},\cite[4.12]{TsujiPrismQHiggs}}]
\label{prop:SuffCondExBdd}
Let $(R,\xi R)$ be a bounded prism,  and let $(A,J)$ be a $\delta$-pair over $(R,\xi R)$. 
Put $\tI=pR+\xi R$, $\oR=R/\tI$, $\oA'=A_{1+J}/\tI A_{1+J}$, and $\oJ'=J\oA'$.
Assume that $A_{1+J}$ is $\tI$-adically flat over $R$, $A/J$ is $p$-adically flat over $R/\xi R$,
and that there exists a regular sequence $T_1,\ldots, T_d\in \oJ'$ generating $\oJ'$
with quotients $\oA'/\sum_{i=1}^rT_i\oA'$ $(r\in \N\cap [0,d])$  flat over $\oR$.
Then $(A,J)$ has a bounded prismatic envelope $(A,J)\to (D,\xi D)$ over $(R,\xi R)$,
and $D$ is $\tI$-adically flat over $R$. 
\end{proposition}

\begin{proof}
Put $A'=\varprojlim_n A_{1+J}/\tI^n A_{1+J}$ and 
$J'=\varprojlim_n (J\cdot (A_{1+J}/\tI^nA_{1+J}))$. Then we have
$A'/\tI^n A'\cong A_{1+J}/\tI^nA_{1+J}$ (\cite[4.9]{TsujiPrismQHiggs}),
$(A'/J')/p^n(A'/J')\cong A_{1+J}/(JA_{1+J}+p^nA_{1+J})\cong(A/J)/p^n(A/J)$, 
$A'$ admits a unique $\delta$-$A$-algebra structure,
and a bounded prismatic envelope of $(A,J)$ over $(R,\xi R)$ is
the same as that of $(A',J')$ over $(R,\xi R)$. As $A'/\tI A'=\oA'$
and $1+J\oA'\in (\oA')^{\times}$, we are reduced to the case
$\oA'=A/\tI A$ and $A$ is $\tI$-adically flat over $R$, 
i.e., \cite[4.12]{TsujiPrismQHiggs},  by replacing $(A,J)$ with $(A',J')$.
\end{proof}

\begin{proposition}[{\cite[4.14]{TsujiPrismQHiggs}}]
\label{prop:SuffCondExBdd2}
Let $(R,\xi R)$ be a bounded prism, put $\tI=pR+\xi R$, and let $f\colon (A,J)\to (A',J')$ be 
a homomorphism of $\delta$-pairs over $(R,\xi R)$. Assume that $f\colon A\to A'$ is 
$\tI$-adically smooth, has $\tI$-adic coordinates,
and induces an isomorphism $A/J\xrightarrow{\cong} A'/J'$. 
Then $(A',J')$ has a bounded prismatic envelope $D'$ over $(R,\xi R)$
if $(A,J)$ has a bounded prismatic envelope $D$ over $(R,\xi R)$.
Moreover $D'$ is $\tI$-adically flat over $D$.
\end{proposition}

\begin{definition}[{\cite[Definition 4.1]{BS}, \cite[11.1]{TsujiPrismQHiggs}}]
\label{def:PrismaticSiteCrystal}
Let $(R,I)$ be a bounded prism, and let $\fX$ be a $p$-adic formal scheme
over $\Spf(R/I)$, where $R/I$ is equipped with the $p$-adic topology, for which
$R/I$ is complete and separated (Proposition \ref{prop:PrismModIComp}).\par
(1) We define the {\it prismatic site} $(\fX/(R,I))_{\prism}$ (or $(\fX/R)_{\prism}$)
of $\fX$ over $(R,I)$ as follows: An object of the underlying category
is a pair of a bounded prism $(P,IP)$ over $(R,I)$ and a morphism 
$v\colon\Spf(P/IP)\to\fX$ over $\Spf(R/I)$. A morphism 
$u\colon ((P',IP'),v')\to ((P,IP),v)$ in the underlying category is
a homomorphism of bounded prisms $u\colon (P,IP)\to (P',IP')$ over 
$(R,I)$ compatible with $v$ and $v'$: $v'=v\circ \Spf(u\otimes R/I)$.
We abbreviate $((P,IP),v)$ to $P$ if there is no risk of confusion.
For morphisms $u_{\nu}\colon ((P_{\nu},IP_{\nu}),v_{\nu})\to ((P,IP),v)$ $(\nu=1,2)$ with one of $u_{\nu}$
is $(pR+I)$-adically flat, the fiber product is represented by 
the $(pR+I)$-adic completion of $P_1\otimes_PP_2$ thanks to 
Proposition \ref{prop:bddPrismFlatAlg} (and \cite[4.9]{TsujiPrismQHiggs}). 
For each object $((P,IP),v)$, we define $\Cov_{\fpqc}((P,IP),v)$ to be
the set of finite families of morphisms 
$(u_{\lambda}\colon ((P_{\lambda},IP_{\lambda}),v_{\lambda})\to ((P,IP),v))$
such that $u_{\lambda}$ is $(pR+I)$-adically flat (Definition \ref{def:IadicProperties} (1))
for every $\lambda$
and the union of the images of $\Spf(u_{\lambda})\colon 
\Spf(P_{\lambda})\to \Spf(P)$ is $\Spf(P)$. By the remark
above on fiber products, this defines a pretopology. We
equip $(\fX/(R,I))_{\prism}$ with the associated topology.\par
We define a sheaf of rings $\CO_{\fX/R}$ (resp.~$\CO_{\fX/R,n}$ for $n\in \N$) by 
setting $\CO_{\fX/R}((P,IP),v)=P$ (resp.~$\CO_{\fX/R,n}((P,IP),v)=P_n:=P/(pP+IP)^{n+1}$).
The lifting of Frobenius $\varphi_P$ of $P$ and its mod $(pR+I)^{n+1}$ reduction $\varphi_{P_n}$
define $\varphi\colon \CO_{\fX/R}\to \CO_{\fX/R}$ and
$\varphi_n\colon \CO_{\fX/R,n}\to \CO_{\fX/R,n}$.\par

(2) A {\it crystal of $\CO_{\fX/R}$-modules} (resp.~{\it $\CO_{\fX/R,n}$-modules}
for $n\in \N$) is a presheaf of $\CO_{\fX/R}$-modules (resp.~$\CO_{\fX/R,n}$-modules)
$\CF$ such that the $P'$-linear homomorphism $\CF((P,IP),v)\otimes_{P}P'\to
\CF((P',IP'),v')$ is an isomorphism for every morphism 
$((P',IP'),v')\to ((P,IP),v)$ in $(\fX/R)_{\prism}$. A crystal of $\CO_{\fX/R,n}$-modules is always a sheaf.
Let $\Crystal_{\prism}(\CO_{\fX/R,n})$ denote the category of crystals of $\CO_{\fX/R,n}$-modules.
We write $\Crystal^{\fproj}_{\prism}(\CO_{\fX/R})$ for the category of crystals of 
$\CO_{\fX/R}$-modules $\CF$ such that $\CF((P,IP),v)$ is a finite projective 
$P$-module for every object $((P,IP),v)$ of $(\fX/R)_{\prism}$. 
Every object of $\Crystal^{\fproj}_{\prism}(\CO_{\fX/R})$ is a sheaf on $(\fX/R)_{\prism}$. 
A {\it complete crystal of $\CO_{\fX/R}$-modules} is a presheaf of 
$\CO_{\fX/R}$-modules such that $\CF((P,IP),v)$ is $(pR+I)$-adically
complete and separated and the presheaf of $\CO_{\fX/R,n}$-modules
$((P,IP),v)\mapsto \CF((P,IP),v)/(pR+I)^{n+1}$ is a crystal for every $n\in \N$.
We write $\hCrystal_{\prism}(\CO_{\fX/R})$ for the category of complete
crystals of $\CO_{\fX/R}$-modules.
Every complete crystal of $\CO_{\fX/R}$-modules is a sheaf, 
and every object of $\Crystal^{\fproj}_{\prism}(\CO_{\fX/R})$ is a complete crystal
of $\CO_{\fX/R}$-modules.

(3)  Let $g\colon (R,I)\to (R',I')$ be a homomorphism of bounded prisms,
let $\fX'$ be a $p$-adic formal scheme over $\Spf(R'/I')$,
and let $f\colon \fX'\to \fX$ be a morphism over $\Spf(g)$.
Then the functor $f_{\prism}\colon(\fX'/R')_{\prism}\to (\fX/R)_{\prism}$
defined by $((P',I'P'),v')\mapsto ((P',IP'), f\circ v')$ is a cocontinuous
functor and therefore defines a morphism of topos 
$f_{\prism}\colon (\fX'/R')_{\prism}^{\sim}\to (\fX/R)^{\sim}_{\prism}$. 
We have $f_{\prism}^{-1}(\CF)((P',I'P'),v')=\CF((P',IP'),f\circ v')$.
Let $f_{\prism}^{-1}$ also denote the functor between the categories
of presheaves defined by the same formula. 
Note that the formula implies $f_{\prism}^{-1}(\CO_{\fX/R})=\CO_{\fX'/R'}$,
$f_{\prism}^{-1}(\CO_{\fX/R,n})=\CO_{\fX'/R',n}$ $(n\in \N)$,
and the image of a crystal of $\CO_{\fX/R}$-modules
(resp.~a crystal of $\CO_{\fX/R,n}$-modules $(n\in \N)$, resp.~a
complete crystal of $\CO_{\fX/R}$-modules) under $f_{\prism}^{-1}$
is a crystal of $\CO_{\fX'/R'}$-modules
(resp.~a crystal of $\CO_{\fX'/R',n}$-modules, resp.~a complete crystal
of $\CO_{\fX'/R'}$-modules). 
If $\CF\in \Ob \Crystal^{\fproj}_{\prism}(\CO_{\fX/R})$, we have
$f_{\prism}^{-1}(\CF)\in \Ob \Crystal^{\fproj}_{\prism}(\CO_{\fX'/R'})$. 
\end{definition}

\begin{remark}[{\cite[11.2, 11.4]{TsujiPrismQHiggs}}]
\label{rmk:PrismCrystalFrobTensor}
Let $(R,I)$ be a bounded prism, and let $\fX$ be a $p$-adic formal
scheme over $\Spf(R/I)$. \par
(1) The presheaf scalar extension
of a crystal (or a complete crystal) of $\CO_{\fX/R}$-modules under
$\CO_{\fX/R}\to \CO_{\fX/R,n}$ is a crystal of $\CO_{\fX/R,n}$-modules. 
A similar claim holds for the scalar extension of a crystal of 
$\CO_{\fX/R,n}$-modules under 
$\CO_{\fX/R,n}\to \CO_{\fX/R,m}$ for integers $n>m\geq 0$.\par
(2) For a crystal of $\CO_{\fX/R}$-modules $\CF$,
its presheaf scalar extension under 
$\varphi\colon \CO_{\fX/R}\to\CO_{\fX/R}$,
which is denoted by $\varphi^*\CF$, is 
a crystal of $\CO_{\fX/R}$-modules. 
For $\CF\in\Ob\Crystal^{\fproj}_{\prism}(\CO_{\fX/R})$, we have 
$\varphi^*\CF\in \Ob\Crystal^{\fproj}_{\prism}(\CO_{\fX/R})$.
A similar claim holds for a crystal of $\CO_{\fX/R,n}$-modules
$\CF$ and the scalar extension under
$\varphi_n\colon \CO_{\fX/R,n}\to \CO_{\fX/R,n}$.
For a complete crystal of $\CO_{\fX/R}$-modules $\CF$,
the inverse limit $\varprojlim_n\varphi_n^*(\CF\otimes_{\CO_{\fX/R}}\CO_{\fX/R,n})$,
which is denoted by $\hvarphi^*\CF$, is a complete crystal
of $\CO_{\fX/R}$-modules and satisfies
$(\hvarphi^*\CF)\otimes_{\CO_{\fX/R}}\CO_{\fX/R,n}\cong \varphi_n^*(\CF\otimes_{\CO_{\fX/R}}
\CO_{\fX/R,n})$. 
For $\CF\in \Ob\Crystal^{\fproj}_{\prism}(\CO_{\fX/R})$,
we have $\hvarphi^*\CF\cong \varphi^*\CF$ because 
$\varphi^*\CF\in \Ob\Crystal^{\fproj}_{\prism}(\CO_{\fX/R})$ is a 
complete crystal of $\CO_{\fX/R}$-modules and 
$\varphi^*\CF\otimes_{\CO_{\fX/R}}\CO_{\fX/R,n}
\cong\varphi_n^*(\CF\otimes_{\CO_{\fX/R}}\CO_{\fX/R,n})$
for $n\in \N$. The functors $\varphi^*$, $\varphi_n^*$, and 
$\hvarphi^*$ above are obviously compatible with 
$f_{\prism}^{-1}$ in Definition \ref{def:PrismaticSiteCrystal} (3).
\par
(3) For crystals of $\CO_{\fX/R}$-modules $\CF$ and $\CG$,
the presheaf tensor product of $\CF$ and $\CG$ as
$\CO_{\fX/R}$-modules is a crystal of $\CO_{\fX/R}$-modules.
If $\CF, \CG\in \Ob\Crystal_{\prism}^{\fproj}(\CO_{\fX/R})$, then
it belongs to $\Crystal_{\prism}^{\fproj}(\CO_{\fX/R})$.
A similar claim holds for crystals of $\CO_{\fX/R,n}$-modules.
For complete crystals of $\CO_{\fX/R}$-modules $\CF$ and $\CG$,
and their scalar extensions $\CF_n$ and $\CG_n$ under
$\CO_{\fX/R}\to\CO_{\fX/R,n}$ $(n\in \N)$, 
the inverse limit $\varprojlim_n(\CF_n\otimes_{\CO_{\fX/R,n}}\CG_n)$,
which is denoted by $\CF\hotimes_{\CO_{\fX/R}}\CG$, is 
a complete crystal of $\CO_{\fX/R}$-modules, and we have
$(\CF\hotimes_{\CO_{\fX/R}}\CG)\otimes_{\CO_{\fX/R}}\CO_{\fX/R,n}
\cong \CF_n\otimes_{\CO_{\fX/R,n}}\CG_n$.
If $\CF, \CG\in \Ob\Crystal^{\fproj}_{\prism}(\fX/R)$, 
we have $\CF\hotimes_{\CO_{\fX/R}}\CG\cong \CF\otimes_{\CO_{\fX/R}}\CG$
because $\CF\otimes_{\CO_{\fX/R}}\CG\in\Ob\Crystal^{\fproj}_{\prism}(\fX/R)$
is a complete crystal of $\CO_{\fX/R}$-modules and 
$(\CF\otimes_{\CO_{\fX/R}}\CG)\otimes_{\CO_{\fX/R}}\CO_{\fX/R,n}
\cong \CF_n\otimes_{\CO_{\fX/R,n}}\CG_n$ for $n\in \N$.
The tensor products and the completed one $\hotimes$ considered above are
compatible with $f_{\prism}^{-1}$ in Definition \ref{def:PrismaticSiteCrystal} (3).
\end{remark}

Let $(R,I)$ and $\fX$ be as in Definition \ref{def:PrismaticSiteCrystal}, let $\fX_{\ZAR}$
(resp.~$\fX_{\Zar}$) be the category of $p$-adic formal schemes
over $\fX$ (resp.~open formal subschemes of $\fX)$ equipped 
with the Zariski topology. Then the functor $(\fX/(R,I))_{\prism}\to \fX_{\ZAR}$
sending $((P,IP),v)$ to $(\Spf(P/IP),v)$ is cocontinuous and defines a morphism of topos 
$U_{\fX/(R,I)}\colon (\fX/(R,I))_{\prism}^{\sim}\to \fX_{\ZAR}^{\sim}$.
The inclusion  $\fX_{\Zar}\hookrightarrow \fX_{\ZAR}$ is a
continuous functor preserving finite inverse limits, and therefore defines
a morphism of topos $\varepsilon_{\fX,\Zar}\colon 
\fX_{\ZAR}^{\sim}\to \fX_{\Zar}^{\sim}$. We define a morphism
of topos \begin{equation}\label{eq:PrismToposZarProj}
u_{\fX/(R,I)}\colon (\fX/(R,I))_{\prism}^{\sim}\longrightarrow \fX_{\Zar}^{\sim}
\end{equation}
to be the composition $\varepsilon_{\fX,\Zar}\circ U_{\fX/(R,I)}$
\cite[\S12]{TsujiPrismQHiggs}.
Under the notation in Definition \ref{def:PrismaticSiteCrystal} (3), 
we have a canonical isomorphism
of morphisms of topos
\begin{equation}\label{eq:PrismToposZarProjFunct}
f_{\Zar}\circ u_{\fX'/(R',I')}\cong u_{\fX/(R,I)}\circ f_{\prism}.
\end{equation}

\section{Twisted derivations}\label{sec:TwDeriv}
We review the definition of twisted derivations and their fundamental properties.

Let $R$ be a ring, let $A$ be an $R$-algebra, let $\alpha$ be an element
of $A$, and let $\gamma$ be an $R$-algebra endomorphism of $A$
satisfying $\gamma(x)\equiv x \mod \alpha A$ for every $x\in A$.
If $\alpha$ is $A$-regular, then one can define an $R$-linear map 
$\partial\colon A\to A$ by setting $\partial(x)=\alpha^{-1}(\gamma(x)-x)$
$(x\in A)$, and it is straightforward to see that $\partial$ satisfies the 
following.
\begin{align}
&\partial(1)=0\label{eq:TwDerivCond1}\\
&\partial(xy)=\partial(x)y+x\partial(y)+\alpha\partial(x)\partial(y)
\qquad (x,y\in A)
\label{eq:TwDerivCond2}
\end{align}
If $R$ is a $\delta$-ring and $A$ is a $\delta$-$R$-algebra,
it is natural to ask if $\gamma$ is a $\delta$-homomorphism.
For $x\in A$, we have 
\begin{align*}
\gamma(\delta(x))&=\delta(x)+\alpha\partial(\delta(x)),\\
\delta(\gamma(x))&=\delta(x)
+(\alpha^p+p\delta(\alpha))\delta(\partial(x))
+\delta(\alpha)\partial(x)^p
-\sum_{\nu=1}^{p-1}p^{-1}\binom p{\nu} x^{p-\nu}\alpha^{\nu}\partial(x)^{\nu}.
\end{align*}

Based on these observations, we define an $\alpha$-derivation 
$A$ over $R$ and its compatibility with a $\delta$-structure
without assuming $\alpha$ is $A$-regular as follows.

\begin{definition}[{\cite[5.5, 6.1]{TsujiPrismQHiggs}}]
\label{def:TwDeriv}
Let $R$ be a ring, let $A$ be an $R$-algebra, and let $\alpha\in A$.\par
(1) We define an {\it $\alpha$-derivation of $A$ over $R$}
to be an $R$-linear map $\partial\colon A\to A$ satisfying 
\eqref{eq:TwDerivCond1} and \eqref{eq:TwDerivCond2}.
Let $\Der^{\alpha}_R(A)$ denote the set of $\alpha$-derivations
of $A$ over $R$.
\par
(2) Let $f\colon A\to A'$ be a ring homomorphism, put $\alpha'=f(\alpha)$,
and let $\partial\in \Der^{\alpha}_R(A)$ and $\partial'\in \Der^{\alpha'}_R(A')$.
We say that $\partial'$ is an {\it extension of} $\partial$ ({\it along} $f$)
(or $f$ {\it is compatible with $\partial$ and $\partial'$}) if the equality
$\partial'\circ f=f\circ \partial$ holds. \par
(3) Suppose that $R$ is a $\delta$-ring, $A$ is a $\delta$-$R$-algebra,
$\delta(\alpha)\in \alpha A$, and we are given an element $\beta\in A$ 
satisfying $\delta(\alpha)=\alpha\beta$. Let 
$\partial$ be an $\alpha$-derivation of $A$ over $R$.
For $x\in A$, we say that {\it $x$ is $\delta$-compatible with respect to 
$\partial$ and $\beta$} if it satisfies
\begin{equation}
\partial(\delta(x))=(\alpha^{p-1}+p\beta)\delta(\partial(x))+\beta\partial(x)^p
-\sum_{\nu=1}^{p-1}p^{-1}\binom p{\nu} x^{p-\nu}\alpha^{\nu-1}\partial(x)^{\nu}.
\end{equation}
We say that $\partial$ is {\it $\delta$-compatible with respect to $\beta$}
if every element of $A$ is $\delta$-compatible with respect to $\partial$ and $\beta$.
Let $\Der^{\alpha,\beta}_{R,\delta}(A)$ denote the set of 
$\alpha$-derivations of $A$ over $R$ $\delta$-compatible with respect
to $\beta$.
\end{definition}

It is straightforward to verify the following. 

\begin{lemma}[{\cite[5.7, 6.4]{TsujiPrismQHiggs}}]
\label{lem:TwDerivEndomDeltaComp}
Let $R$ be a ring, let $A$ be an $R$-algebra, and let $\alpha\in A$.
In (2) and (3), we assume that $R$ is a $\delta$-ring,  $A$ is a $\delta$-$R$-algebra,
$\delta(\alpha)\in \alpha A$, and we are given $\beta\in A$
satisfying $\delta(\alpha)=\alpha\beta$.\par
(1) For $\partial\in \Der^{\alpha}_R(A)$, $\gamma=\id_A+\alpha\partial$
is an endomorphism of the $R$-algebra $A$ and satisfies $\partial(xy)=\partial(x)y+\gamma(x)\partial(y)$
for $x,y\in A$. If $\alpha$ is $A$-regular, this gives
a bijection between $\Der^{\alpha}_R(A)$ and the set
of endomorphisms of the $R$-algebra $A$ congruent to $\id_A$ modulo $\alpha A$.\par
(2) For $\partial\in \Der^{\alpha}_R(A)$, the endomorphism $\gamma$ in (1)
is a $\delta$-homomorphism if $\partial$ is $\delta$-compatible with respect to $\beta$.
The converse is also true if $\alpha$ is $A$-regular.\par
(3) For $\partial\in \Der^{\alpha}_R(A)$, we have $\partial\circ\varphi=(\alpha^{p-1}+p\beta)
\varphi\circ\partial$ if $\partial$ is $\delta$-compatible with respect to $\beta$.
The converse is also true if $A$ is $p$-torsion free.
\end{lemma}

We can interpret $\alpha$-derivations and their $\delta$-compatibility
in terms of ring homomorphisms and their $\delta$-compatibility as follows.

For a ring $A$ and $\alpha\in A$, we define an $A$-algebra $E^{\alpha}(A)$
to be the quotient $A[T]/(T(T-\alpha))$ of the polynomial algebra $A[T]$
over $A$ in one variable. We write $\pi_{A,0}$ (resp.~$\pi_{A,1}$)
for the $A$-algebra homomorphism $E^{\alpha}(A)\to A$
sending $T$ to $0$ (resp.~$\alpha$). 
The construction of $E^{\alpha}(A)$ with $\pi_{A,\nu}$ $(\nu=0,1)$ 
is functorial in $A$ and $\alpha$ in the obvious sense.
If $A$ is a $\delta$-ring, $\delta(\alpha)\in \alpha A$,
and we are given $\beta\in A$ satisfying $\delta(\alpha)=\alpha\beta$,
then there exists a unique $\delta$-$A$-algebra structure 
on $E^{\alpha}(A)$ satisfying $\delta(T)=\beta T$ \cite[6.6]{TsujiPrismQHiggs}; the uniqueness holds as
$E^{\alpha}(A)$ is generated by $T$ as an $A$-algebra, and the existence
is verified by showing that the extension $A[T]\to W_2(E^{\alpha}(A));T\mapsto
(T,\beta T)$ of the homomorphism $A\to W_2(A);a\mapsto (a,\delta(a))$
factors through $E^{\alpha}(A)$. This $\delta$-structure is explicitly 
given by the following formula for $x,y\in A$ \cite[6.7]{TsujiPrismQHiggs}.
\begin{equation}
\delta(x+yT)=\delta(x)+\left\{(\alpha^{p-1}+p\beta)\delta(y)+\beta y^p
-\sum_{\nu=1}^{p-1}p^{-1}\binom p\nu x^{p-\nu}\alpha^{\nu-1}y^{\nu}\right\}T
\end{equation}
We write $E^{\alpha,\beta}_{\delta}(A)$
for $E^{\alpha}(A)$ with the $\delta$-structure above.
We see that the homomorphisms $\pi_{A,0}$ and $\pi_{A,1}$ 
are $\delta$-homomorphisms. The construction 
of $E^{\alpha,\beta}_{\delta}(A)$ is functorial in a $\delta$-algebra
$A$, $\alpha$, and $\beta$ in the obvious sense.

\begin{proposition}[{\cite[5.11, 6.9]{TsujiPrismQHiggs}}]
\label{prop:TwDerivTwExtSec}
Let $A$ be an $R$-algebra and let $\alpha\in A$. 
In (2), we assume that $R$ is a $\delta$-ring, $A$ is a $\delta$-$R$-algebra,
$\delta(\alpha)\in \alpha A$, and we are given $\beta\in A$
satisfying $\delta(\alpha)=\alpha\beta$.\par
(1) For $\partial\in \Der^{\alpha}_R(A)$, the  map 
$s_{\partial}\colon A\to E^{\alpha}(A);x\mapsto x+\partial(x)T$ is a homomorphism of $R$-algebras.
This gives a bijection between $\Der^{\alpha}_R(A)$ and the set
of $R$-sections of $\pi_{A,0}\colon E^{\alpha}(A)\to A$.\par
(2) Let  $\partial$ and $s_{\partial}$ be as in (1), and let $E^{\alpha}(A)$
be equipped with the $\delta$-structure defined by $\beta$. For $x\in A$,
we have $\delta\circ s_{\partial}(x)=s_{\partial}\circ \delta(x)$
if and only if $x$ is $\delta$-compatible with respect to $\partial$
and $\beta$. In particular, $s_{\partial}$ is a $\delta$-homomorphism
if and only if $\partial$ is $\delta$-compatible with respect to $\beta$.
\end{proposition}

\begin{remark}[{\cite[5.6 (1), (2), 6.5 (1), 6.11]{TsujiPrismQHiggs}}]
\label{rmk:TwDerivScalExtComp} Let $R$ be a ring
and let $A$ be an $R$-algebra.\par
(1) Let $R'$ be an $R$-algebra,  and put $A'=A\otimes_RR'$
and $\alpha'=\alpha\otimes 1\in A'$.
For $\partial\in \Der_R^{\alpha}(A)$, the $R'$-linear extension 
$\partial'=\partial\otimes \id_{R'}\colon A'\to A'$ of $\partial$ is an $\alpha'$-derivation
of $A'$ over $R'$. Suppose that $R$ is a $\delta$-ring, $A$ and $R'$ are
$\delta$-$R$-algebras, $A'$ is equipped with the $\delta$-structure
induced by those of $R'$ and $A$, and $\delta(\alpha)\in \alpha A$.
Choose $\beta\in A$ satisfying $\delta(\alpha)=\alpha\beta$.
Then, by using Proposition \ref{prop:TwDerivTwExtSec} (2), we see that $\partial'$ is $\delta$-compatible
with respect to $\beta\otimes 1\in A'$ 
if $\partial$ is $\delta$-compatible with respect to $\beta$.\par
(2)  Let $I$ be an ideal of $R$, put $\hR=\varprojlim_nR/I^n$, 
$\hA=\varprojlim_n A/I^n A$, and 
let $\halpha$ be the image of $\alpha$ in $\hA$. 
Then, for $\partial\in\Der^{\alpha}_R(A)$, the inverse limit
$\hpartial=\varprojlim_n \partial\otimes \id_{A/I^nA}\colon \hA\to \hA$
is an $\halpha$-derivation of $\hA$ over $\hR$. 
Suppose that $I$ contains $p$, $R$ is a $\delta$-ring, $A$ is a $\delta$-$R$-algebra,
$\hR$ and $\hA$ are equipped with the $I$-adic completion of the $\delta$-structures of $R$ and $A$,
and $\delta(\alpha)\in \alpha A$. Choose $\beta\in A$ satisfying $\delta(\alpha)=\alpha\beta$,
and let $\hbeta$ denote its image in $\hA$.
Then $\hpartial$ is $\delta$-compatible with respect to $\hbeta$ if
$\partial$ is $\delta$-compatible with respect to $\beta$.
\end{remark}

\begin{example}\label{eq:qHiggsDerivPolRing}
($q$-Higgs derivations)
Let $\Z[q]$ be the polynomial ring over $\Z$ in one variable $q$,
and put $\mu=q-1$. We define $[n]_q\in \Z[q]$ for $n\in \N$
to be $\frac{q^n-1}{q-1}=\sum_{r=0}^{n-1}q^r$. 
Let $\Lambda$ be a finite set, and let $S$ be the polynomial
ring $\Z[q][t_i\;(i\in\Lambda)]$ over $\Z[q]$ in variables
$t_i$ $(i\in \Lambda)$ indexed by $\Lambda$. 
Then one can define endomorphisms
$\gamma_{S,i}$ $(i\in\Lambda)$ of the $\Z[q]$-algebra
$S$, commuting with each other, 
by $\gamma_{S,i}(t_i)=q^pt_i$ and $\gamma_{S,i}(t_j)=t_j$ $(j\neq i)$.
Since $\gamma_{S,i}(x)\equiv x \mod t_i\mu S$ and $t_i\mu$ is
$S$-regular, we can define a $t_i\mu$-derivation $\theta_{S,i}$ 
of $S$ over $\Z[q]$ by $\theta_{S,i}(x)=(t_i\mu)^{-1}(\gamma_{S,i}(x)-x)$
$(x\in S)$ (Lemma \ref{lem:TwDerivEndomDeltaComp} (1)). For $(n_j)_{j\in\Lambda}
\in \N^{\Lambda}$, 
we have $\theta_{S,i}(\prod_{j\in\Lambda}t_j^{n_j})
=[pn_i]_qt_i^{-1}\prod_{j\in\Lambda}t_j^{n_j}$ if $n_i>0$ and $0$ if $n_i=0$.
This implies $\theta_{S,i}(x)\equiv p \frac{\partial x}{\partial t_i}$
mod $\mu S$ for $x\in S$. 
The endomorphisms $\theta_{S,i}$ commute with each other since
$\gamma_{S,i}$ commute with each other and $\gamma_{S,i}(t_j)=t_j$ $(j\neq i)$.

We equip $\Z[q]$ and $S$ with the $\delta$-structures corresponding
to the liftings of Frobenius $\varphi_{\Z[q]}$ and $\varphi_S$ 
defined by $q\mapsto q^p$ and $t_i\mapsto t_i^p$. 
We have $\delta(\mu)=p^{-1}(((1+\mu)^p-1)-\mu^p)\in\mu\Z[q]$.
Put $\eta=\delta(\mu)\mu^{-1}=\sum_{\nu=1}^{p-1}p^{-1}\binom p\nu\mu^{\nu-1}$.
We have $\varphi_S\circ\gamma_{S,i}=\gamma_{S,i}\circ\varphi_S$, which implies
$\delta\circ\gamma_{S,i}=\gamma_{S,i}\circ\delta$ since $\gamma_{S,i}$
is a ring endomorphism and $S$ is $p$-torsion free. 
Therefore Lemma \ref{lem:TwDerivEndomDeltaComp} (2) 
shows that $\theta_{S,i}$ is $\delta$-compatible
with respect to $\delta(t_i\mu)(t_i\mu)^{-1}=t_i^p\delta(\mu)(t_i\mu)^{-1}=t_i^{p-1}\eta$.

Let $R$ be a $\delta$-$\Z[q]$-algebra, and let $A$ be
$R[t_i\; (i\in\Lambda)]\cong S\otimes_{\Z[q]}R$ equipped with
the $\delta$-structure induced by those on $S$ and $R$.
Then, by Remark \ref{rmk:TwDerivScalExtComp} (1), the $R$-linear extension 
$\theta_{A,i}=\theta_{S,i}\otimes \id_R\colon A\to A$
of $\theta_{S,i}$ is a $t_i\mu$-derivation of $A$ over $R$
$\delta$-compatible with respect to $t_i^{p-1}\eta$ for each $i\in \Lambda$. 
We have $\theta_{A,i}\circ\theta_{A,j}=\theta_{A,j}\circ\theta_{A,i}$
and $\theta_{A,i}\equiv p\frac{\partial\;\;}{\partial t_i}$ mod $\mu A$.
\end{example}

Next we review an interpretation of the commutativity of twisted derivations
in terms of ring homomorphisms. Let $R$ be a ring, 
let $A$ be an $R$-algebra, let $\Lambda$ be a finite set, and
let $\ualpha=(\alpha_i)_{i\in\Lambda}\in A^{\Lambda}$. Similarly to 
$E^{\alpha}(A)$, we define an $A$-algebra $E^{\ualpha}(A)$ to be
the quotient $A[T_i \,(i\in\Lambda)]/(T_i(T_i-\alpha_i),\, i\in\Lambda)$ of
the polynomial $A$-algebra $A[T_i\, (i\in\Lambda)]$ over $A$
in variables $T_i$ $(i\in\Lambda)$.
When $\Lambda=\{1,2\}$, we have the following canonical isomorphisms
for $(i,j)=(1,2), (2,1)$.
\begin{equation}\label{eq:MultTwExtAlg}
E^{\alpha_i}(E^{\alpha_j}(A))
\cong (A[T_j]/(T_j(T_j-\alpha_j)))[T_i]/(T_i(T_i-\alpha_i))\cong E^{\ualpha}(A)
\end{equation}
If $R$ is a $\delta$-ring, $A$ is a $\delta$-$R$-algebra,
$\delta(\alpha_i)\in \alpha_i A$ $(i\in \Lambda)$, and
we are given $\ubeta=(\beta_i)_{i\in\Lambda}\in A^{\Lambda}$ satisfying
$\delta(\alpha_i)=\alpha_i\beta_i$ $(i\in \Lambda)$,
then there exists a unique $\delta$-$A$-algebra structure 
on $E^{\ualpha}(A)$ satisfying $\delta(T_i)=\beta_iT_i$
for every $i\in \Lambda$. We write 
$E^{\ualpha,\ubeta}_{\delta}(A)$ for the $\delta$-$A$-algebra
$E^{\ualpha}(A)$ equipped with this $\delta$-structure. 
When $\Lambda=\{1,2\}$, the composition \eqref{eq:MultTwExtAlg}
is a $\delta$-homomorphism if $E^{\alpha_i}(E^{\alpha_j}(A))$
is equipped with the $\delta$-structure defined by that of 
$E^{\alpha_j,\beta_j}_{\delta}(A)$ and $\beta_i$. 

\begin{lemma}[{\cite[5.20]{TsujiPrismQHiggs}}]
\label{lem:TwDerivComm}
Let $R$ be a ring, let $A$ be an $R$-algebra, and let $\alpha_1,\alpha_2\in A$.
Let $\partial_i\in \Der^{\alpha_i}_R(A)$ $(i=1,2)$, assume
$\partial_1(\alpha_2)=\partial_2(\alpha_1)=0$, and let $s_i$ be 
the $R$-algebra homomorphism $A\to E^{\alpha_i}(A)$ corresponding to 
$\partial_i$ by Proposition \ref{prop:TwDerivTwExtSec} (1).
Then, for $x\in A$, we have 
$\partial_1\circ\partial_2(x)=\partial_2\circ\partial_1(x)$ if and only if 
the images of $x$ under the compositions below for $(i,j)=(1,2), (2,1)$
coincide. In particular, $\partial_1$ and $\partial_2$ are commutative, i.e.,
$\partial_1\circ\partial_2=\partial_2\circ\partial_1$ if and only if 
$E^{\alpha_1}(s_2)\circ s_1=E^{\alpha_2}(s_1)\circ s_2$.
\begin{equation}
A\xrightarrow{\;s_i\;}E^{\alpha_i}(A)
\xrightarrow{E^{\alpha_i}(s_j)}
E^{s_j(\alpha_i)}(E^{\alpha_j}(A))=E^{\alpha_i}(E^{\alpha_j}(A))\cong E^{(\alpha_1,\alpha_2)}(A).
\end{equation}
\end{lemma}

\begin{remark}\label{rmk:TwDerivComm}
 We follow the notation in Lemma \ref{lem:TwDerivComm}.\par
(1) Let $I$ be an ideal of $R$ and let $\scrS$ be a subset of $A$
 such that $A$ is $I$-adically separated and $R[\scrS]\subset A$
 is $I$-adically dense. Then Lemma \ref{lem:TwDerivComm} implies
 that $\partial_1$ and $\partial_2$ are commutative if and only 
 if $\partial_1\circ\partial_2(s)=\partial_2\circ\partial_1(s)$ for 
 every $s\in \scrS$.\par
(2) Suppose that $R$ is a $\delta$-ring,
$A$ is a $\delta$-$R$-algebra, $\delta(\alpha_i)\in \alpha_iA$ $(i=1,2)$,
and we are given $\beta_i\in A$ $(i=1,2)$ satisfying $\delta(\alpha_i)=\alpha_i\beta_i$
$(i=1,2)$. If $\partial_i\in \Der^{\alpha_i,\beta_i}_{R,\delta}(A)$ and 
$\partial_1(\beta_2)=\partial_2(\beta_1)=0$, then the two compositions
considered in Lemma \ref{lem:TwDerivComm} define $\delta$-homomorphisms to
the $\delta$-ring $E^{(\alpha_1,\alpha_2),(\beta_1,\beta_2)}_{\delta}(A)$.
\end{remark}

By using Proposition \ref{prop:TwDerivTwExtSec}, Lemma \ref{lem:TwDerivComm}, 
and Remark \ref{rmk:TwDerivComm}, we obtain
the following properties of twisted derivations.

\begin{proposition}\label{prop:TwDerivEtaleExt}
Let $R$ be a ring and let $A$ be an $R$-algebra.
Let $I$ be an ideal of $R$, and let $A'$ be an $I$-adically 
\'etale $A$-algebra (Definition \ref{def:IadicProperties} (1)) 
$I$-adically complete and separated.\par
(1) {\rm(\cite[5.14]{TsujiPrismQHiggs})}
For $\alpha\in IA$, any $\partial\in \Der^{\alpha}_{R}(A)$
has a unique extension $\partial'\in \Der^{\alpha}_R(A')$.\par
(2) {\rm(\cite[6.12]{TsujiPrismQHiggs})} 
Let $\alpha$, $\partial$, and $\partial'$ be as in (1). 
Suppose that $I$ contains $p$, $R$ is a $\delta$-ring, $A$ is a $\delta$-$R$-algebra, 
$\delta(\alpha)\in \alpha A$, and we are given $\beta\in A$
satisfying $\delta(\alpha)=\alpha\beta$. Then $\partial\in \Der^{\alpha,\beta}_{R,\delta}(A)$
implies $\partial'\in \Der^{\alpha,\beta}_{R,\delta}(A')$
for the unique $\delta$-$A$-algebra structure on $A'$
(Proposition \ref{prop:DeltaStruEtaleMap} (1)).\par
(3) {\rm(\cite[5.17]{TsujiPrismQHiggs})} Let $\alpha$, $\partial$, and $\partial'$ be as in (1). 
Then, for an $A'$-algebra $A''$ $I$-adically separated,
an $\alpha$-derivation $\partial''$ of $A''$ over $R$
is an extension of $\partial'$ if and only if it is an extension of $\partial$.\par
(4) {\rm(\cite[5.21]{TsujiPrismQHiggs})} Let $\alpha_i\in IA$ $(i=1,2)$, let $\partial_i\in \Der_R^{\alpha_i}(A)$
$(i=1,2)$, and let $\partial'_i\in \Der_R^{\alpha_i}(A')$ be 
the unique extension of $\partial_i$ (see (1)). Suppose
$\partial_1(\alpha_2)=\partial_2(\alpha_1)=0$. Then 
we have $\partial'_1\circ\partial_2'=\partial'_2\circ\partial'_1$
if $\partial_1\circ\partial_2=\partial_2\circ\partial_1$.
\end{proposition}

\begin{proposition}\label{prop:TwDerivPrismEnvExt}
Let $(R,I)$ be a bounded prism,  let $(A,J)$ be a $\delta$-pair
over $(R,I)$, and suppose that there exists a bounded prismatic
envelope $(A,J)\to (D,ID)$ of $(A,J)$ over $(R,I)$
(Definition \ref{def:bddPrismEnv}).\par
(1) {\rm (\cite[6.17]{TsujiPrismQHiggs})} For $\alpha,\beta\in A$ with $\delta(\alpha)=\alpha\beta$, 
any $\partial\in \Der_{R,\delta}^{\alpha,\beta}(A)$ satisfying
$\partial(J)\subset J$ has a unique extension 
$\partial'\in \Der_{R,\delta}^{\alpha,\beta}(D)$.\par
(2) {\rm(\cite[6.20]{TsujiPrismQHiggs})} Let $\alpha$, $\beta$, $\partial$, and $\partial'$ be the same
as in (1), let $(D,ID)\to (B,IB)$ be a homomorphism of bounded prisms
over $(R,I)$, and let $\partial''\in \Der^{\alpha,\beta}_{R,\delta}(B)$.
Then $\partial''$ is an extension of $\partial'$ if and only if
$\partial''$ is an extension of $\partial$.\par
(3) {\rm(\cite[6.23]{TsujiPrismQHiggs})} Let $\alpha_i,\beta_i\in A$ $(i=1,2)$ with
$\delta(\alpha_i)=\alpha_i\beta_i$,  let
$\partial_i\in\Der_{R,\delta}^{\alpha_i,\beta_i}(A)$ $(i=1,2)$
satisfying $\partial_i(J)\subset J$, and let 
$\partial_i'\in \Der_{R,\delta}^{\alpha_i,\beta_i}(D)$ be the
unique extension of $\partial_i$ for $i=1,2$ (see (1)).
Assume $\partial_i(\alpha_j)=\partial_i(\beta_j)=0$ for $(i,j)=(1,2), (2,1)$.
Then we have $\partial_1'\circ\partial_2'=\partial_2'\circ\partial_1'$
if $\partial_1\circ\partial_2=\partial_2\circ\partial_1$.
\end{proposition}

\section{Connections over twisted derivations}\label{sec:connection}
In this section, we recall a connection over twisted derivations,
its de Rham complex, and their behavior under a scalar extension.

\begin{definition}[{\cite[8.1]{TsujiPrismQHiggs}}]
\label{def:TwConnectionOneVar}
Let $A$ be a ring, let $\alpha\in A$, and let $\partial$ be an $\alpha$-derivation of $A$
over $\Z$. Let $\gamma$ be the ring endomorphism 
$\id_A+\alpha\partial$ of $A$ (Lemma \ref{lem:TwDerivEndomDeltaComp} (1)). 
An {\it $(\alpha,\partial)$-connection}
on an $A$-module $M$ is an additive map
$\nabla_M\colon M\to M$ satisfying $\nabla_M(ax)=
\gamma(a)\nabla_M(x)+\partial(a)x$ for $a\in A$ and $x\in M$.
A {\it homomorphism $f\colon (M,\nabla_M)\to (M',\nabla_{M'})$ of 
$A$-modules with $(\alpha,\partial)$-connection} is an 
$A$-linear map $f\colon M\to M'$ satisfying $\nabla_{M'}\circ f=f\circ\nabla_M$.
\end{definition}

\begin{lemma}[{\cite[8.2]{TsujiPrismQHiggs}}]
\label{lem:TwConnSemilinEnd}
Let $A$, $\alpha$, $\partial$, and $\gamma$ be as 
in Definition \ref{def:TwConnectionOneVar}, and
let $M$ be an $A$-module.
If $\nabla_M$ is an $(\alpha,\partial)$-connection on $M$,
then the additive endomorphism $\gamma_M=\id_M+\alpha\nabla_M$
of $M$ is $\gamma$-semilinear over $A$, i.e., $\gamma_M(ax)=
\gamma(a)\gamma_M(x)$ $(a\in A,x\in M)$. If $\alpha$ is $M$-regular,
then this gives a bijection between the set of $(\alpha,\partial)$-connections
on $M$ and the set of $\gamma$-semilinear endomorphisms of $M$  over $A$
congruent to $\id_M$ modulo $\alpha M$. Note $\gamma(\alpha A)\subset \alpha A$.
\end{lemma}

\begin{remark}\label{rmk:TwConnectionRlin}
Let $A$, $\alpha$, $\partial$, and $\gamma$ be as 
in Definition \ref{def:TwConnectionOneVar}, suppose that $A$ is an algebra over a ring $R$,
and $\partial$ is an $\alpha$-derivation over $R$, i.e., $R$-linear. Then 
every $(\alpha,\partial)$-connection and its associated $\gamma$-semilinear
endomorphism are $R$-linear since $\partial(r\cdot 1_A)=0$ for every $r\in R$.
\end{remark}

Let $A$, $\alpha$, $\partial$, and $\gamma$ be as 
in Definition \ref{def:TwConnectionOneVar}.
We keep the notation in Definition \ref{def:TwConnectionOneVar}.
Let $(M,\nabla_M)$ and $(M',\nabla_{M'})$ be $A$-modules
with $(\alpha,\partial)$-connection.  Then
we see that the map $M\times M'\to M\otimes_AM'$
sending $(x,x')$ to 
\begin{align}\label{eq:TensorSingleConnection}
\nabla_{M}(x)\otimes \gamma_{M'}(x')+
x\otimes\nabla_{M'}(x')
=\nabla_{M}(x)\otimes x'+x\otimes\nabla_{M'}(x')+\alpha\nabla_{M}(x)\otimes\nabla_{M'}(x')
\end{align}
is $A$-bilinear and defines an $(\alpha,\partial)$-connection 
$\nabla_{M\otimes M'}\colon M\otimes_AM'\to M\otimes_AM'$ on $M\otimes_AM'$.
We define the {\it tensor product
of $(M,\nabla_M)$ and $(M',\nabla_{M'})$} (resp.~$\nabla_M$ {\it and} $\nabla_{M'}$) to be
$(M\otimes_AM',\nabla_{M\otimes M'})$ (resp.~$\nabla_{M\otimes M'}$). 
It is straightforward to verify that $\gamma_{M\otimes M'}=\id_{M\otimes M'}+
\alpha\nabla_{M\otimes M'}$ coincides with $\gamma_M\otimes\gamma_{M'}$.
This implies that, for another $A$-module with $(\alpha,\partial)$-connection
$(M'',\nabla_{M''})$, the isomorphism 
$(M\otimes_AM')\otimes_AM''\cong M\otimes_A(M'\otimes_AM'')$
is compatible with $(\alpha,\partial)$-connections.

Let $A$ be a ring and let $\Lambda$ be
a finite set. In the following, we assume that we are given $(\alpha_i)_{i\in\Lambda}
\in A^{\Lambda}$ and $\partial_i\in \Der^{\alpha_i}_{\Z}(A)$ for each
$i\in \Lambda$ satisfying the the following two conditions.
\begin{align}
&\partial_i(\alpha_j)=0\quad(i,j\in \Lambda, i\neq j)\label{eq:TwDerivConnCond1}\\
&\partial_i\circ\partial_j=\partial_j\circ\partial_i \quad(i,j\in\Lambda)\label{eq:TwDerivConnCond2}
\end{align}
For $i\in \Lambda$, we define the ring endomorphism $\gamma_i$ 
of $A$ to be $\id_A+\alpha_i\partial_i$.
We have $\gamma_i\circ\gamma_j=\gamma_j\circ\gamma_i$
and $\gamma_i\circ\partial_j=\partial_j\circ\gamma_i$ for $i\neq j$.
Put $\ualpha=(\alpha_i)_{i\in\Lambda}$,
$\upartial=(\partial_i)_{i\in\Lambda}$, and 
$\ugamma=(\gamma_i)_{i\in\Lambda}$.

\begin{example}
The $t_i\mu$-derivations $\theta_{A,i}$ $(i\in \Lambda)$ of $A$ over $R$
constructed in the last paragraph in Example \ref{eq:qHiggsDerivPolRing} 
satisfy the conditions \eqref{eq:TwDerivConnCond1}
and \eqref{eq:TwDerivConnCond2}. (We put 
$\alpha_i=t_i\mu$ and $\partial_i=\theta_{A,i}$ for $i\in\Lambda$.)
\end{example}

\begin{definition}[{\cite[8.8, 8.9, 8.11]{TsujiPrismQHiggs}}]
\label{def:TwConnectionDRcpx}
(1) An {\it $(\ualpha,\upartial)$-connection} on an $A$-module
is a family $\unabla_{M}=(\nabla_{M,i})_{i\in\Lambda}$ consisting
of $(\alpha_i,\partial_i)$-connections $\nabla_{M,i}$ on $M$
(Definition \ref{def:TwConnectionOneVar}). 
We write $\gamma_{M,i}$ for the $\gamma_i$-semilinear endomorphism
$\id_M+\alpha_i\nabla_{M,i}$ of $M$ associated to $\nabla_{M,i}$
(Lemma \ref{lem:TwConnSemilinEnd}),
and let $\ugamma_{M}$ denote the family $(\gamma_{M,i})_{i\in\Lambda}$.
A {\it homomorphism of $A$-modules with $(\ualpha,\upartial)$-connection}
$f\colon (M,\unabla_M)\to (M',\unabla_{M'})$ is an $A$-linear map
$f\colon M\to M'$ satisfying $\nabla_{M',i}\circ f=f\circ\nabla_{M,i}$ for
every $i\in \Lambda$. 
We say that $\unabla_M$ is {\it integrable} if 
$\nabla_{M,i}\circ\nabla_{M,j}=\nabla_{M,j}\circ\nabla_{M,i}$
for every $i,j\in\Lambda$. 
If $\unabla_M$ is integrable, the endomorphisms $\gamma_{M,i}$ $(i\in\Lambda)$
commute with each other
by \eqref{eq:TwDerivConnCond1}.
We write $\MIC(A,(\ualpha,\upartial))$ for the category of 
$A$-modules with integrable $(\ualpha,\upartial)$-connection.
\par
(2)  For an integrable  $(\ualpha,\upartial)$-connection
$\unabla_M=(\nabla_{M,i})_{i\in\Lambda}$ on an $A$-module $M$,
we define the {\it de Rham complex} $(\Omega^r(M,\unabla_M),\nabla^r_M)_{r\in \N}$ 
of $(M,\unabla_M)$ to be the Koszul complex associated to the commutative
operators $(\nabla_{M,i})_{i\in\Lambda}$ on $M$ as follows:
$\Omega^r(M,\unabla_M)=M\otimes_{\Z}\wedge^r(\oplus_{i\in\Lambda}\Z\omega_i)$
$(r\in \N)$ and $\nabla^r_M(x\otimes\omega_{\bmI})=
\sum_{i\in\Lambda}\nabla_{M,i}(x)\otimes(\omega_i\wedge\omega_{\bmI})$
$(r\in \N, x\in M, \bmI\in \Lambda^r)$. Here $\omega_{\bmI}$
denotes $\omega_{i_1}\wedge\cdots\wedge \omega_{i_r}$ for
$\bmI=(i_1,\ldots, i_r)\in\Lambda^r$.\par
(3) For two $A$-modules with $(\ualpha,\upartial)$-connection
$(M,\unabla_M)$ and $(M',\unabla_{M'})$,
we define the {\it tensor product
$(M\otimes_AM',\unabla_{M\otimes M'})$} (resp.~$\unabla_{M\otimes M'}$)
{\it of $(M,\unabla_M)$ and $(M',\unabla_{M'})$} (resp.~$\unabla_M$ {\it and} $\unabla_{M'}$)
by $\unabla_{M\otimes M'}=(\nabla_{M\otimes M',i})_{i\in\Lambda}$,
where $\nabla_{M\otimes M',i}$ is the tensor product of 
the $(\alpha_i,\partial_i)$-connections $\nabla_{M,i}$ and $\nabla_{M',i}$. 
If $\unabla_M$ and $\unabla_{M'}$ are integrable, one can verify,
by an explicit computation, that $\unabla_{M\otimes M'}$ is also integrable.
\end{definition}

For $A$-modules with integrable $(\ualpha,\upartial)$-connection 
$(M,\unabla_M)$ and $(M',\unabla_{M'})$, we can define a morphism of complexes
\begin{equation}\label{eq:DRCpxProdMap}
\Omegab(M,\unabla_M)\otimes_{\Z}\Omegab(M',\unabla_{M'})\longrightarrow
\Omegab(M\otimes_AM',\unabla_{M\otimes M'})
\end{equation}
by sending $(x\otimes\omega_{\bmI})\otimes(x'\otimes\omega_{\bmI'})$
to $x\otimes\gamma_{M',\bmI}(x')\otimes\omega_{\bmI}\wedge\omega_{\bmI'}$
\cite[8.13]{TsujiPrismQHiggs}.
Here $\gamma_{M',\bmI}$ denotes $\gamma_{M',i_1}\circ\dots \circ \gamma_{M',i_r}$
for $\bmI=(i_1,\ldots, i_r)\in \Lambda^r$. We write $z\wedge_{\ugamma_{M'}}z'$
for the image of $z\otimes z'$ under \eqref{eq:DRCpxProdMap}. Then, for another $A$-module with
integrable $(\ualpha,\upartial)$-connection $(M'',\unabla_{M''})$, 
we have $(z\wedge_{\ugamma_{M'}}z')\wedge_{\ugamma_{M''}}z''
=z\wedge_{\ugamma_{M'\otimes M''}}(z'\wedge_{\ugamma_{M''}}z'')$.

\begin{remark}\label{rmk:dRComplexRlin}
If $A$ is an algebra over a ring $R$, and $\partial_i$ is
an $\alpha_i$-derivation of $A$ over $R$ for every $i\in \Lambda$, 
then the differential maps 
of the de Rham complex of an $A$-module with $(\ualpha,\upartial)$-connection
are $R$-linear and the product morphism \eqref{eq:DRCpxProdMap} factors through
the tensor product over $R$.
\end{remark}

To study the behavior of $(\ualpha,\upartial)$-connections under scalar extensions, 
we give an interpretation of an integrable $(\ualpha,\upartial)$-connection in terms of 
the $A$-algebra $E^{\ualpha}(A)=A[T_i\,(i\in \Lambda)]/(T_i(T_i-\alpha_i),i\in\Lambda)$
as follows. (See the paragraph before Lemma \ref{lem:TwDerivComm} for 
$E^{\ualpha}(A)$.)  For $\ui\subset \Lambda$, we define $\partial_{\ui}$ 
(resp.~$T_{\ui}\in E^{\ualpha}(A)$)
to be the composition of $\partial_i$ $(i\in \ui)$ which does not depend
on the choice of  order of composition by \eqref{eq:TwDerivConnCond2} (resp.~$\prod_{i\in \ui}T_i$). 
Then, by using \eqref{eq:TwDerivConnCond1}, we see that
the additive map $s_{\upartial}\colon A\to E^{\ualpha}(A)$ defined by 
$s_{\upartial}(a)=\sum_{\ui\subset \Lambda}\partial_{\ui}(a)T_{\ui}$ is 
a ring homomorphism \cite[8.14, 8.15]{TsujiPrismQHiggs}. We define $E^{\ualpha,\upartial}(A)$ 
to be $E^{\ualpha}(A)$ regarded as an $A$-bialgebra by viewing 
the canonical $A$-algebra structure (resp.~$s_{\upartial}$) as 
a left (resp.~right) $A$-algebra structure.
The tensor product $E^{\ualpha,\upartial}(A)\otimes_A-$
(resp.~$-\otimes_AE^{\ualpha,\upartial}(A)$) is taken with respect to the
right (resp.~left) $A$-algebra structure.
We regard
$E^{\ualpha,\upartial}(A)\otimes_AE^{\ualpha,\upartial}(A)$ 
as an $A$-bialgebra by giving the left (resp.~right) $A$-algebra structure via
that of the left (resp.~right) $E^{\ualpha,\upartial}(A)$, and we define
$\overline{E^{\ualpha,\upartial}(A)\otimes_AE^{\ualpha,\upartial}(A)}$ to be
the quotient of the $A$-bialgebra $E^{\ualpha,\upartial}(A)\otimes_AE^{\ualpha,\upartial}(A)$
by the ideal generated by $T_i\otimes T_i$ $(i\in \Lambda)$. 
Then $\overline{E^{\ualpha,\upartial}(A)\otimes_AE^{\ualpha,\upartial}(A)}$ is 
a left free $A$-module with basis $T_{\ui}\otimes T_{\ui'}$
$(\ui,\ui'\subset\Lambda,\ui\cap\ui'=\emptyset)$, and 
there exists an $A$-bialgebra homomorphism 
$\delta^{\ualpha,\upartial}\colon E^{\ualpha,\upartial}(A)\to 
\overline{E^{\ualpha,\upartial}(A)\otimes_AE^{\ualpha,\upartial}(A)}$ sending
$T_i$ to $T_i\otimes 1+1\otimes T_i$ $(i\in\Lambda)$ \cite[8.23 (2)]{TsujiPrismQHiggs}.

Let $(M,\unabla_M)$ be an $A$-module with integrable $(\ualpha,\upartial)$-connection.
For $\ui\subset \Lambda$, we define $\nabla_{M,\ui}$ to be the composition of
$\nabla_{M,i}$ $(i\in\ui)$ which does not depend on the choice of order of composition.
We regard $M$ as a right $A$-module. Then by using $\nabla_{M,i}(xa)=\nabla_{M,i}(x)a+
x\partial_i(a)+\nabla_{M,i}(x)\partial_i(a)\alpha_i$ $(x\in M, a\in A)$ and \eqref{eq:TwDerivConnCond1}, 
we see that the additive map $s_{\unabla_M}\colon M\to M\otimes_AE^{\ualpha,\upartial}(A)$
defined by $s_{\unabla_M}(x)=\sum_{\ui\subset \Lambda}\nabla_{M,\ui}(x)\otimes T_{\ui}$
$(x\in M)$ is a right $A$-linear map \cite[8.24, 8.26]{TsujiPrismQHiggs}. 

\begin{proposition}[{\cite[8.29]{TsujiPrismQHiggs}}]
Under the notation above, the right $A$-linear map $s_{\unabla_M}$
satisfies the two properties \eqref{eq:TwStratCond1} and \eqref{eq:TwStratCond2} 
below for $s_M=s_{\unabla_M}$, where $\pi_0\colon E^{\ualpha,\upartial}(A)\to A$
denotes the $A$-bialgebra homomorphism defined by $T_i\mapsto 0$ $(i\in\Lambda)$.
Moreover this gives an equivalence of categories between $\MIC(A,(\ualpha,\upartial))$
and the category of pairs $(M,s_M)$ with $s_M$ a right $A$-linear map $s_M\colon M\to 
M\otimes_AE^{\ualpha,\upartial}(A)$ satisfying \eqref{eq:TwStratCond1} 
and \eqref{eq:TwStratCond2}.
\begin{align}
&\text{The composition }
M\xrightarrow{s_M} M\otimes_AE^{\ualpha,\upartial}(A)
\xrightarrow{\id_M\otimes\pi_0} M\text{ is the identity map.}\label{eq:TwStratCond1}\\
&\text{The following diagram is commutative.}\label{eq:TwStratCond2}\\
&\xymatrix@C=40pt{
M\ar[r]^(.4){s_M}\ar[d]_{s_M}&
M\otimes_AE^{\ualpha,\upartial}(A)\ar[r]^(.4){s_M\otimes\id}&
M\otimes_AE^{\ualpha,\upartial}(A)\otimes_AE^{\ualpha,\upartial}(A)\ar[d]\\
M\otimes_AE^{\ualpha,\upartial}(A)\ar[rr]^(.45){\id_M\otimes\delta^{\ualpha,\upartial}}&&
M\otimes_A\overline{E^{\ualpha,\upartial}(A)\otimes_AE^{\ualpha,\upartial}(A)}
}\notag
\end{align}
\end{proposition}

\begin{remark}[{\cite[8.30 (2)]{TsujiPrismQHiggs}}]
\label{eq:TwStratTensor}
Let $(M_{\nu},\unabla_{M_\nu})$ $(\nu=1,2)$ be $A$-modules with $(\ualpha,\upartial)$-connection,
and let $s_{M_{\nu}}\colon M_{\nu}\to M_{\nu}\otimes_AE^{\ualpha,\upartial}(A)$ $(\nu=1,2)$ be the
right $A$-linear maps corresponding to $\unabla_{M_{\nu}}$. Then the
tensor product  of $\unabla_{M_1}$ and $\unabla_{M_2}$ corresponds to the composition 
$$M_1\otimes_AM_2\xrightarrow{s_{M_1}\otimes\id_{M_2}}
M_1\otimes_AE^{\ualpha,\upartial}(A)\otimes_AM_2
\xrightarrow{\id_{M_1}\otimes \ts_{M_2}}M_1\otimes_AM_2\otimes_AE^{\ualpha,\upartial}(A),$$
where $\ts_{M_2}$ denotes the $E^{\ualpha,\upartial}(A)$-linear extension
$E^{\ualpha,\upartial}(A)\otimes_AM_2\to M_2\otimes_AE^{\ualpha,\upartial}(A)$ of $s_{M_2}$.
\end{remark}

Now let us discuss scalar extensions of $(\ualpha,\upartial)$-connections.
Let $A'$ be a ring, let $\Lambda'$ be 
a finite set, and suppose that we are given 
$\ualpha'=(\alpha'_{i'})_{i'\in\Lambda'}\in (A')^{\Lambda'}$
and $\upartial'=(\partial'_{i'})_{i'\in\Lambda'}\in \prod_{i'\in\Lambda'}\Der^{\alpha'_{i'}}_{\Z}(A')$
satisfying the same conditions as \eqref{eq:TwDerivConnCond1} and \eqref{eq:TwDerivConnCond2}.
We define a family $\ugamma'=(\gamma'_{i'})_{i'\in\Lambda'}$ of 
endomorphisms of $A'$ by $\gamma'_{i'}=\id_{A'}+\alpha'_{i'}\partial'_{i'}$.
Suppose that we are given
a ring homomorphism $f\colon A\to A'$,
 a map $\psi\colon \Lambda\to \Lambda'$, and $\uc=(c_i)_{i\in\Lambda}\in (A')^{\Lambda}$
 satisfying  $f(\alpha_i)=c_i\alpha'_{\psi(i)}$ and $\partial'_{i'}(c_i)=0$
for all $i\in\Lambda$ and $i'\in \Lambda'\backslash\{\psi(i)\}$.
The triplet $(f,\psi,\uc)$ induces a homomorphism of left algebras
$E^{\psi,\uc}(f)\colon E^{\ualpha}(A)\to E^{\ualpha'}(A');$
$T_i\mapsto c_iT_{\psi(i)}$  $(i\in\Lambda)$ lying over $f$. 
We abbreviate $E^{\psi,\uc}(f)$ to $E^{\psi}(f)$ if $c_i=1$ for all 
$i\in\Lambda$.
\begin{lemma}[{\cite[9.5]{TsujiPrismQHiggs}}]
\label{lem:TwConnScExtCond}
The homomorphism $E^{\psi,\uc}(f)$ defines a homomorphism of bialgebras
$E^{\ualpha,\upartial}(A)\to E^{\ualpha',\upartial'}(A')$ lying over $f\colon A\to A'$
if and only if the following holds for every $i'\in \Lambda'$ and $a\in A$.
\begin{equation}\label{eq:TwDerivScExt}
\partial'_{i'}(f(a))=\sum_{\emptyset\neq \ui\subset \psi^{-1}(i')}
f(\partial_{\ui}(a))\prod_{i\in \ui}c_i\cdot(\alpha'_{i'})^{\sharp \ui-1}
\end{equation}
This equality implies
\begin{equation}\label{eq:TwEndScExt}
\gamma'_{i'}(f(a))=f((\prod_{i\in \psi^{-1}(i')}\gamma_i)(a)),
\end{equation}
and the converse is also true if $\alpha'_{i'}$ is $A'$-regular.
\end{lemma}
We assume that $E^{\psi,\uc}(f)$ is a bialgebra homomorphism lying over $f$
in the following. Let $(M,\unabla_M)$ be an $A$-module with integrable
$(\ualpha,\upartial)$-connection, and let $s_M=s_{\unabla_M}$
be the right $A$-linear map $M\to M\otimes_AE^{\ualpha,\upartial}(A)$
corresponding to $\nabla_{\uM}$. Put $M'=M\otimes_AA'$. Then 
one can show that the right $A'$-linear extension 
$M'\to M'\otimes_{A'}E^{\ualpha',\upartial'}(A')$
of the composition of $s_M$ with $\id_M\otimes E^{\psi,\uc}(f)\colon 
M\otimes_AE^{\ualpha,\upartial}(A)\to 
M\otimes_AE^{\ualpha',\upartial'}(A')\cong
M'\otimes_{A'}E^{\ualpha',\upartial'}(A')$
satisfies the properties \eqref{eq:TwStratCond1} and \eqref{eq:TwStratCond2} 
\cite[9.9]{TsujiPrismQHiggs}, and therefore
defines an $(\ualpha',\upartial')$-connection $\unabla_{M'}$ on $M'$,
which we call the {\it  scalar extension of $\unabla_M$ under} 
$(f,\psi,\uc)$ (or simply $f$). By Remark \ref{eq:TwStratTensor}, 
we see that the scalar extension is compatible with tensor products \cite[9.12]{TsujiPrismQHiggs}.
This construction is obviously functorial
in $(M,\unabla_M)$ and defines a functor \cite[9.10]{TsujiPrismQHiggs}
\begin{equation}\label{eq:TwConnectionScExtFunct}
(f,\psi,\uc)^*\colon \MIC(A,(\ualpha,\upartial))\longrightarrow
\MIC(A',(\ualpha',\upartial')).
\end{equation}
We abbreviate $(f,\psi,\uc)^*$ to $(f,\psi)^*$ if $c_i=1$ for all $i\in \Lambda$.
This functor is compatible with compositions of $(f,\psi,\uc)$'s
\cite[9.23 (2)]{TsujiPrismQHiggs}. 

Let $(M,\unabla_{M})$ be an $A$-module with integrable $(\ualpha,\upartial)$-connection,
and let $(M',\unabla_{M'})$  be its  scalar extension under $(f,\psi,\uc)$. Then 
$\nabla_{M',i'}$ and $\gamma_{M',i'}$ $(i'\in\Lambda')$ 
are explicitly given by the following formulas for 
$x\in M$, similar to \eqref{eq:TwDerivScExt} and \eqref{eq:TwEndScExt} \cite[9.13]{TsujiPrismQHiggs}.
\begin{align}
\nabla_{M',i'}(x\otimes 1)&=\sum_{\emptyset\neq\ui\subset\psi^{-1}(i')}\nabla_{M,\ui}(x)
\otimes\prod_{i\in \ui}c_i\cdot(\alpha'_{i'})^{\sharp \ui-1}\label{eq:TwConnScExtFormula1}\\
\gamma_{M',i'}(x\otimes 1)&=(\prod_{i\in \psi^{-1}(i')}\gamma_{M,i})(x)\otimes 1
\label{eq:TwConnScExtFormula2}
\end{align}
Choose and fix a total order of the set $\Lambda$. Then one can 
define a morphism of complexes
\begin{equation}\label{eq:TwDrCpxScExt}
\Omegab_{f,\psi,\uc}(M,\unabla_{M})\colon \Omegab(M,\unabla_M)\longrightarrow \Omegab(M',\unabla_{M'})
\end{equation}
as follows \cite[9.19]{TsujiPrismQHiggs}. 
We abbreviate $\Omegab_{f,\psi,\uc}$ to $\Omegab_{f,\psi}$ if $c_i=1$ for all $i\in\Lambda$.
For $i\in \Lambda$, we define $\Lambda_{\psi,i}^{<}$ to be the subset of 
$\Lambda$ consisting of $j\in \Lambda$ satisfying $j<i$ and $\psi(j)=\psi(i)$,
and $\gamma_{M,\psi,i}^{<}$ to be the composition of $\gamma_{M,j}$ $(j\in \Lambda_{\psi,i}^{<})$.
For $\bmI=(i_n)_{1\leq n\leq r}\in \Lambda^r$, we define
$\gamma_{M,\psi,\bmI}^{<}$ (resp.~$c_{\bmI}\in A'$) to be the composition of $\gamma_{M,\psi,i_n}^{<}$
$(n\in \N\cap [1,r])$ (resp.~$\prod_{n=1}^rc_{i_n}$). 
Then the morphism \eqref{eq:TwDrCpxScExt} is defined by the
following formula
\begin{equation}
\Omega^r_{f,\psi,\uc}(M,\unabla_M)(x\otimes \omega_{\bmI})=(\gamma_{M,\psi,\bmI}^{<}(x)\otimes 
c_{\bmI})\otimes\omega_{\psi^r(\bmI)}\qquad (x\in M, r\in \N, \bmI\in \Lambda^r),
\end{equation}
where $\psi^r$ denotes the product $\Lambda^r\to(\Lambda')^r;(i_n)\mapsto (\psi(i_n))$ 
of $\psi$. The compatibility with the differential maps can be verified by an explicit computation
using the formulas \eqref{eq:TwConnScExtFormula1} and \eqref{eq:TwConnScExtFormula2}.
The morphism $\Omegab_{f,\psi,\uc}(M,\unabla_M)$ is obviously functorial in $(M,\unabla_M)$.
One can verify by an explicit computation 
that $\Omegab_{f,\psi,\uc}(M,\unabla_M)$ is also compatible with compositions of $(f,\psi,\uc)$'s with $\Lambda$'s totally ordered and $\psi$'s order preserving \cite[9.23 (2)]{TsujiPrismQHiggs}.
When $\psi$ is injective, we see that it is compatible with the product morphism 
\eqref{eq:DRCpxProdMap} \cite[9.21]{TsujiPrismQHiggs}.  When $\psi$ is not injective, we have some
weaker compatibility \cite[9.24]{TsujiPrismQHiggs}. 

Finally we discuss the scalar extension by a lifting of Frobenius under a certain 
setting, for which $A$ and $\theta_{A,i}$ in Example \ref{eq:qHiggsDerivPolRing} is a typical example.
Let $\Z[q]$, $\mu=q-1$,  $\eta=\delta(\mu)\mu^{-1}$, and $[n]_q$ $(n\in \N)$ be as in Example \ref{eq:qHiggsDerivPolRing}.
Let $A$ be a $\Z[q]$-algebra equipped with a lifting of Frobenius $\varphi_A$
compatible with that of $\Z[q]$, let $\Lambda$ be a finite totally ordered set,
let $\ut=(t_i)_{i\in\Lambda}$ be a family of elements of $A$ satisfying
$\varphi_A(t_i)=t_i^p$ for every $i\in \Lambda$, put $\alpha_i=t_i\mu$ 
$(i\in \Lambda)$,
and suppose that we are given an $\alpha_i$-derivation $\partial_i$ of $A$ over $\Z[q]$
for each $i\in \Lambda$ such that $\partial_i\circ\partial_j=\partial_j\circ\partial_i$
and $\partial_j(t_i)=0$ for $i,j\in \Lambda$, $i\neq j$. Then the ring $A$ with
$\ualpha:=(\alpha_i)_{i\in\Lambda}$ and $\upartial=(\partial_i)_{i\in\Lambda}$
satisfies the conditions \eqref{eq:TwDerivConnCond1} and \eqref{eq:TwDerivConnCond2}.
We further assume that the following equalities hold.
\begin{equation}\label{eq:FrobTwDerivCond}
 \partial_i\circ\varphi_A=t_i^{p-1}[p]_q\varphi_A\circ\partial_i\qquad (i\in\Lambda)
\end{equation}

\begin{remark}[{\cite[9.25 (2)]{TsujiPrismQHiggs}}]
\label{rmk:TwDerivFrobComp}
If $\varphi_A$ is induced by a $\delta$-structure on $A$ satisfying 
$\delta(t_i)=0$, and $\partial_i$ is $\delta$-compatible with respect to $t_i^{p-1}\eta$,
then \eqref{eq:FrobTwDerivCond} holds by 
Proposition \ref{lem:TwDerivEndomDeltaComp} (3) and $\alpha_i^{p-1}+pt_i^{p-1}\eta=
t_i^{p-1}(\mu^{p-1}+p\eta)=t_i^{p-1}\varphi(\mu)\mu^{-1}=t_i^{p-1}[p]_q$.
\end{remark}

We have $\varphi_A(\alpha_i)=t_i^p\varphi(\mu)=t_i^{p-1}[p]_q\alpha_i$.
Therefore, defining $\uc=(c_i)_{i\in\Lambda}\in A^{\Lambda}$
by $c_i=t_i^{p-1}[p]_q$, we see that the triplet $(\varphi_A,\id_{\Lambda},\uc)$
induces a homomorphism of bialgebras
$E^{\id_{\Lambda},\uc}(\varphi_A)\colon E^{\ualpha,\upartial}(A)\to E^{\ualpha,\upartial}(A)$
over $\varphi_A$ by using Lemma \ref{lem:TwConnScExtCond}. Therefore $(\varphi_A,\id_{\Lambda},\uc)$
defines a scalar extension functor \eqref{eq:TwConnectionScExtFunct}
\begin{equation}\label{eq:TwConnFrobPBFunct}
\varphi_A^*=(\varphi_A,\id_{\Lambda},\uc)^*\colon \MIC(A,(\ualpha,\upartial))\longrightarrow \MIC(A,(\ualpha,\upartial))
\end{equation}
preserving tensor products \cite[9.28]{TsujiPrismQHiggs}.
For an $A$-module with $(\ualpha,\partial)$-connection $(M,\unabla_M)$,
its scalar extension $(\varphi_A^*M,\unabla_{\varphi_A^*M})$ is given by 
the formula \eqref{eq:TwConnScExtFormula1} \cite[9.29]{TsujiPrismQHiggs}
\begin{equation}
\nabla_{\varphi_A^*M,i}(x\otimes 1)=\nabla_{M,i}(x)\otimes t_i^{p-1}[p]_q\quad (x\in M),
\end{equation}
and we have a morphism of complexes \eqref{eq:TwDrCpxScExt}
compatible with products  \eqref{eq:DRCpxProdMap} \cite[9.30]{TsujiPrismQHiggs}
\begin{equation}\label{eq:TwdRCpxFrobPB}
\varphi_{A,\Omega}^{\bullet}(M,\unabla_M)\colon \Omega^{\bullet}(M,\unabla_M)
\longrightarrow \Omega^{\bullet}(\varphi_A^*M,\unabla_{\varphi_A^*M})
\end{equation}
sending $x\otimes \omega_{\bmI}$ to $x\otimes [p]_q^rt_{\bmI}^{p-1}
\otimes\omega_{\bmI}$ for $x\in M$, $r\in \N$, and 
$\bmI=(i_n)_{1\leq n\le r}\in \Lambda^r$, where $t_{\bmI}=\prod_{n=1}^{r}t_{i_n}$.

Let $(A',\varphi_{A'}, \ut'=(t'_{i'})_{i'\in \Lambda'}, \upartial'=(\partial'_{i'})_{i'\in\Lambda'})$
be another set of data satisfying the same conditions as
$(A,\varphi_A,\ut, \upartial)$ above, and define $\ualpha'$ and $\uc'$ in the
same way as $\ualpha$ and $\uc$ by using $\ut'$. 
Suppose that we are given a homomorphism of $\Z[q]$-algebras
$f\colon A\to A'$ and an order preserving map $\psi\colon \Lambda\to\Lambda'$
such that $f\circ\varphi_A=\varphi_{A'}\circ f$,
$f(t_i)=t_{\psi(i)}'$, which implies $f(\alpha_i)=\alpha'_{\psi(i)}$,
and that the homomorphism $E^{\psi}(f)\colon E^{\ualpha,\upartial}(A)
\to E^{\ualpha',\upartial'}(A')$ induced by $f$ and $\psi$ is a bialgebra
homomorphism over $f$. (See Lemma \ref{lem:TwConnScExtCond}.)
Since $f\circ\varphi_A=\varphi_{A'}\circ f$ and 
$c'_{\psi(i)}=f(c_i)$, the composition of $(\varphi_A,\id_{\Lambda},\uc)$
and $(f,\psi,\uone_{\Lambda})$ coincides with 
that of $(f,\psi,\uone_{\Lambda})$ and $(\varphi_{A'},\id_{\Lambda'},\uc')$
\cite[9.31]{TsujiPrismQHiggs}.
Therefore the scalar extension by $(f,\psi)$ is compatible with
the scalar extensions by $(\varphi_A,\id_{\Lambda},\uc)$ and
by $(\varphi_{A'},\id_{\Lambda'},\uc')$ \cite[9.32]{TsujiPrismQHiggs}.
For an $A$-module with 
integrable $(\ualpha,\upartial)$-connection $(M,\unabla_M)$
and its scalar extension $(M',\unabla_{M'})$ under $(f,\psi)$,
the pullback morphism $\Omega^{\bullet}(M,\unabla_M)
\to \Omega^{\bullet}(M',\unabla_{M'})$ \eqref{eq:TwDrCpxScExt} is compatible
with the Frobenius pullbacks 
$\varphi^{\bullet}_{A,\Omega}(M,\unabla_M)$
and $\varphi^{\bullet}_{A',\Omega}(M',\unabla_{M'})$ \eqref{eq:TwdRCpxFrobPB}
\cite[9.33]{TsujiPrismQHiggs}.

\begin{example}\label{ex:qHiggsDerivFunct}
We follow the notation in Example \ref{eq:qHiggsDerivPolRing}.
Let $\Lambda'$ be another finite set, and define a $\delta$-$\Z[q]$-algebra 
$S'=\Z[q][t'_{i'}\,(i'\in\Lambda')]$ and a $t'_{i'}\mu$-derivation 
$\theta_{S',i'}$ of $S'$ over $\Z[q]$ for each $i'\in\Lambda'$
in the same way as $S$ and $\theta_{S,i}$ in Example \ref{eq:qHiggsDerivPolRing}.
Let $R'$ be a $\delta$-$\Z[q]$-algebra, and define
$A'\cong S'\otimes_{\Z[q]}R'$ and $\theta_{A',i'}$ $(i'\in \Lambda')$ in the
same way as $A$ and $\theta_{A,i}$ in Example \ref{eq:qHiggsDerivPolRing} by 
using $S'$, $\theta_{S',i'}$ and $R'$. 
Then, for any
$\delta$-homomorphism $g\colon R\to R'$ over $\Z[q]$
and a map $\psi\colon \Lambda\to \Lambda'$, 
the homomorphism $f\colon A\to A';t_i\mapsto t'_{\psi(i)} (i\in\Lambda)$ 
lying over $g$ is a $\delta$-homomorphism and the induced
homomorphism $E^{\psi}(f)\colon E^{(t_i\mu),(\theta_{A,i})}(A)
\to E^{(t'_{i'}\mu),(\theta_{A',i'})}(A')$ is a bialgebra homomorphism
over $f$. The former claim is obvious as $\delta(t_i)=0$ and $\delta(t'_{i'})=0$.
The latter one is reduced to the case $R=R'=\Z[q]$ and $g=\id$ by \eqref{eq:TwDerivScExt}.
Then $\alpha'_{i'}=t'_{i'}\mu$ is $A'$-regular, and the equality
\eqref{eq:TwEndScExt} holds by $\gamma_{i'}'(f(t_i))=q^pf(t_i)$ if $i'=\psi(i)$ and $=f(t_i)$ otherwise.

\end{example}

\section{Prismatic crystals and $q$-Higgs modules}\label{sec:PrismCrysQHiggs}
Let $\Z_p[[q-1]]$ be the ring of formal power series over $\Z_p$ in one variable $q-1$
equipped with the $\delta$-structure corresponding to the lifting
of Frobenius defined by $q\mapsto q^p$. We have $\delta(q)=0$.
We define $\mu=q-1$, $\eta=\delta(\mu)\mu^{-1}$,
and $[n]_q=\frac{q^n-1}{q-1}$ $(n\in \Z)$ as in Example \ref{eq:qHiggsDerivPolRing}.
The pair $(\Z_p[[\mu]],([p]_q))$ is a bounded prism. 

In this section, we review the description of prismatic crystals and
their cohomology given in \cite{TsujiPrismQHiggs} when a base bounded prism
lies over the bounded prism $(\Z_p[[\mu]],(\pq))$.

\begin{definition}[{\cite[10.1]{TsujiPrismQHiggs}}]\label{def:qprism}
(1) We call a bounded prism over the bounded prism
$(\Z_p[[\mu]],(\pq))$ a {\it $q$-prism}. A {\it morphism of $q$-prisms}
is a morphism of bounded prisms over $(\Z_p[[\mu]],(\pq))$. We call the
bounded prismatic envelope of a $\delta$-pair 
$(A,J)$ over $(\Z_p[[\mu]],(\pq))$ the {\it $q$-prismatic envelope}.
If $(A,J)$ is defined over a $q$-prism $(R,(\pq))$, it coincides
with the bounded prismatic envelope over $(R,(\pq))$.\par
(2) A {\it framed smooth $q$-prism} $(A/R, \ut)$ is a set of data
consisting of a $q$-prism $R$, a $(p,\pq)$-adically smooth $R$-algebra $A$
(Definition \ref{def:IadicProperties} (1)) 
$(p,\pq)$-adically complete and separated, and $(p,\pq)$-adic coordinates
$\ut=(t_i)_{i\in\Lambda}$ of $A$ over $R$ (Definition \ref{def:IadicProperties} (2)) 
indexed by  a finite totally ordered set $\Lambda$.
We equip $A$ with the unique $\delta$-$R$-algebra structure satisfying
$\delta(t_i)=0$ $(i\in\Lambda)$ (Proposition \ref{prop:DeltaStruEtaleMap} (1)). 
The pair $(A,(\pq))$ is a $q$-prism by 
Proposition \ref{prop:bddPrismFlatAlg}. When a $q$-prism $R$ is given, 
we also call $(A,\ut)$ a {\it framed smooth $\delta$-$R$-algebra}.
A {\it morphism of framed smooth $q$-prisms} 
$(A/R,(t_i)_{i\in\Lambda})\to (A'/R',(t'_{i'})_{i'\in\Lambda})$
is a triplet $(f/g,\psi)$ consisting of a morphism of $q$-prisms $g\colon R\to R'$,
a ring homomorphism $f\colon A\to A'$ lying over $g$, and a map 
$\psi\colon\Lambda\to\Lambda'$ preserving the orders such that
$f(t_i)=t'_{\psi(i)}$ $(i\in\Lambda)$, which implies that $f$ is a $\delta$-homomorphism. 
When $R=R'$ and $g=\id$, 
we also call it a {\it morphism of framed smooth $\delta$-$R$-algebras.}\par
(3) A {\it framed smooth $q$-pair} $((A,J)/R,\ut)$ is a framed
smooth $q$-prism $(A/R,\ut)$ equipped with an ideal $J$ containing $\pq$. 
When a $q$-prism $R$ is given, we also call $((A,J),\ut)$ a {\it framed smooth $\delta$-pair over $R$}.
We say that $((A,J)/R,\ut)$ (or $((A,J),\ut)$) is {\it admissible}
if $(A,J)$ has a $q$-prismatic envelope. 
A {\it morphism of framed smooth $q$-pairs} 
$((A,J)/R,\ut)\to ((A',J')/R',\ut')$ is a morphism $(f/g,\psi)$ 
of the  underlying framed smooth $q$-prisms satisfying
$f(J)\subset J'$. When $R=R'$ and $g=\id$, we also call it
a {\it morphism of framed smooth $\delta$-pairs over $R$.}
\end{definition}

\begin{remark}
(1) By Proposition \ref{prop:SuffCondExBdd}, we see that 
a framed smooth $q$-pair $((A,J)/R,\ut=(t_i)_{i\in\Lambda})$ is admissible
if $A/J$ is $p$-adically smooth over $R/\pq R$ and 
their exists a subset $\Lambda'$ of $\Lambda$ such that 
the images of $t_i$ $(i\in\Lambda')$ in $A/J$ form $p$-adic
coordinates of $A/J$ over $R/\pq R$ (\cite[4.13]{TsujiPrismQHiggs}).\par
(2) Let $R$ be a $q$-prism and let $(f,\psi)\colon ((A,J),\ut)\to ((A',J'),\ut')$
be a morphism of framed smooth $\delta$-pairs over $R$. Suppose that
$\psi$ is injective and $f$ induces an isomorphism $A/J\xrightarrow{\cong}A'/J'$.
Then, by Proposition \ref{prop:SuffCondExBdd2}, $((A',J'),\ut')$ is 
admissible if $((A,J),\ut)$ is admissible.
\end{remark}

We obtain the following by using Example \ref{eq:qHiggsDerivPolRing}, Proposition \ref{prop:TwDerivEtaleExt},
and Proposition \ref{prop:TwDerivPrismEnvExt}.
\begin{proposition}[{\cite[10.3]{TsujiPrismQHiggs}}]
\label{prop:qPrismEnvTheta}
Let $(A/R,\ut=(t_i)_{i\in \Lambda})$ be a framed smooth $q$-prism.
Then, for each $i\in \Lambda$, there exists a unique $t_i\mu$-derivation
$\theta_{A,i}$ of $A$ over $R$ $\delta$-compatible with $t_i^{p-1}\eta$ satisfying
$\theta_{A,i}(t_i)=\pq$ and $\theta_{A,i}(t_j)=0$ $(j\neq i)$. Moreover
we have $\theta_{A,i}\circ\theta_{A,j}=\theta_{A,j}\circ \theta_{A,i}$
$(i,j\in\Lambda)$ and $\theta_{A,i}(A)\subset \pq A$ $(i\in\Lambda)$.
Let $J$ be an ideal of $A$ containing $\pq$ making
$((A,J)/R,\ut)$ an admissible framed smooth $q$-pair,
and let $D$ be the $q$-prismatic envelope of $(A,J)$. Then
$\theta_{A,i}$ extends uniquely to a $t_i\mu$-derivation 
$\theta_{D,i}$ over $R$ $\delta$-compatible with $t_i^{p-1}\eta$
for each $i\in\Lambda$. Moreover we have 
$\theta_{D,i}\circ\theta_{D,j}=\theta_{D,j}\circ\theta_{D,i}$
$(i,j\in \Lambda)$.
\end{proposition}

\begin{definition}[{\cite[10.4, 10.5]{TsujiPrismQHiggs}}]
\label{def:qHiggsModCpx}
Let $((A,J)/R,\ut=(t_i)_{i\in\Lambda})$ be an admissible framed smooth
$q$-pair, and let $D$ be the $q$-prismatic envelope of $(A,J)$.
Then, with the notation in Proposition \ref{prop:qPrismEnvTheta}, 
$D$, $\ut\mu$, and $\utheta_D:=(\theta_{D,i})_{i\in\Lambda}$
satisfy the conditions \eqref{eq:TwDerivConnCond1} and \eqref{eq:TwDerivConnCond2}. We
call an integrable $(\ut\mu,\utheta_{D})$-connection $\utheta_M=(\theta_{M,i})_{i\in\Lambda}$ 
(Definition \ref{def:TwConnectionDRcpx} (1))
on a $D$-module $M$ a {\it $q$-Higgs field on $M$ over} $(\ut,\utheta_D)$,
and call a pair $(M,\utheta_M)$ a {\it $q$-Higgs module over} $(D,\ut,\utheta_D)$.
Note that $\theta_{M,i}$ is $R$-linear by Remark \ref{rmk:TwConnectionRlin}. 
We write $q\Omega^{\bullet}(M,\utheta_M)$ for the de Rham complex
$\Omega^{\bullet}(M,\utheta_M)$ (Definition \ref{def:TwConnectionDRcpx} (2))
and call it the {\it $q$-Higgs complex of}
$(M,\utheta_M)$. We write $q\HIG(D,\ut,\utheta_D)$ for the
category of $q$-Higgs modules over $(D,\ut,\utheta_D)$.
For $n\in \N$, we define $D_n$ to be the quotient $D/(p,\pq)^{n+1}D$,
and write $q\HIG(D_n,\ut,\utheta_D)$ for the full subcategory
of $q\HIG(D,\ut,\utheta_D)$ consisting of objects whose underlying
$D$-modules are annihilated by $(p,\pq)^{n+1}$.
\end{definition}

We follow the notation in Definition \ref{def:qHiggsModCpx}. Let $n$ be a
non-negative integer. Let $\varphi_{D_n}$ and $\theta_{D_n,i}$ $(i\in\Lambda)$
denote the reduction modulo $(p,\pq)^{n+1}$ of $\varphi_D$ and $\theta_{D,i}$,
respectively.
Then, by Remark \ref{rmk:TwDerivFrobComp}, 
we can apply the construction of \eqref{eq:TwConnFrobPBFunct} to
$D_n$, $\varphi_{D_n}$, $t_i$, and $\theta_{D_n,i}$ $(i\in\Lambda)$. 
We obtain a Frobenius pullback functor,
which preserves tensor products \cite[10.6]{TsujiPrismQHiggs},
\begin{equation}\label{eq:FrSmqHigFrobPB} 
\varphi_{D_n}^*\colon q\HIG(D_n,\ut,\utheta_D)
\longrightarrow q\HIG(D_n,\ut,\utheta_D).
\end{equation}

Next let us discuss scalar extensions of $q$-Higgs modules.

\begin{proposition}[{\cite[10.9]{TsujiPrismQHiggs}}]
\label{prop:FrSmqHiggDerivFunct}
Let $(f/g, \psi)\colon (A/R,\ut=(t_i)_{i\in\Lambda})\to 
(A'/R',\ut'=(t'_{i'})_{i'\in\Lambda'})$ be a morphism of framed smooth $q$-prisms.
The homomorphism $E^{\psi}(f)\colon E^{\ut\mu,\utheta_A}(A)\to E^{\ut'\mu,\utheta_{A'}}(A')$
is a bialgebra homomorphism over $f$. (See Lemma \ref{lem:TwConnScExtCond}.)
Let $J$ (resp.~$J'$) be an ideal of $A$ (resp.~$A'$) containing $\pq$
such that $(A,J)$ (resp.~$(A',J')$) has a $q$-prismatic envelope $D$ (resp.~$D'$). Assume 
$f(J)\subset J'$, and let $f_D$ denote the morphism of $q$-prisms $D\to D'$
induced by $f$. Then the homomorphism $E^{\psi}(f_D)\colon 
E^{\ut\mu,\utheta_D}(D)\to E^{\ut'\mu,\utheta_{D'}}(D')$ is a 
bialgebra homomorphism over $f_D$. 
\end{proposition}

One can deduce the claim for $(f/g,\psi)$ from Example \ref{ex:qHiggsDerivFunct}
by the unique lifting property for \'etale homomorphisms.
We derive the claim for $(f_D/g,\psi)$  from that of $(f/g,\psi)$
by using the universality of the $q$-prismatic envelope $D$
and the  $\delta$-structures of $E^{\ut\mu}(A)$ and 
$E^{\ut\mu}(D)$ (resp.~$E^{\ut'\mu}(A')$ and $E^{\ut'\mu}(D')$)
defined by $(t_i^{p-1}\eta)_{i\in\Lambda}$ 
(resp.~$((t'_{i'})^{p-1}\eta)_{i'\in\Lambda'}$)
as before Lemma \ref{lem:TwDerivComm}, with which the right structures
over $A$ and $D$ (resp.~$A'$ and $D'$) are $\delta$-homomorphisms.

By Proposition \ref{prop:FrSmqHiggDerivFunct}, a morphism of admissible
framed smooth $q$-pairs $(f/g,\psi)$
induces a scalar extension
functor \eqref{eq:TwConnectionScExtFunct}
\begin{equation}\label{eq:qHiggsCpxScExtFunct}
(f_D/g,\psi)^*\colon q\HIG(D,\ut,\utheta_{D})\longrightarrow q\HIG(D',\ut',\utheta_{D'})
\end{equation}
compatible with tensor products, and compositions of $(f/g,\psi)$'s \cite[10.11]{TsujiPrismQHiggs}. 

\begin{construction}\label{constr:qHiggsModCpxShv}
In our description of prismatic cohomology of a prismatic crystal, 
we use a $q$-Higgs complex defined as a complex of sheaves. 
We summarize its construction and its basic properties.\par
(1) (\cite[13.1]{TsujiPrismQHiggs}) Let $((A,J)/R,\ut=(t_i)_{i\in\Lambda})$ be an admissible framed smooth 
$q$-pair, and let $D$ be the $q$-prismatic envelope of $(A,J)$. 
Let $\fD$ and $\fD_n$ denote $\Spf(D)$ and $\Spec(D_n)$, respectively.
For each affine open $\Spf(\tD)\subset \fD$, $\theta_{D,i}$ $(i\in\Lambda)$
extend uniquely to $t_i\mu$-derivations $\theta_{\tD,i}$ of $\tD$ over $R$
$\delta$-compatible
with respect to $t_i^{p-1}\eta$ commuting with each other by Proposition \ref{prop:TwDerivEtaleExt}
(1), (2), and (4).
Varying $\tD$ and using Proposition \ref{prop:TwDerivEtaleExt} (3), 
we obtain a family of endomorphisms $\utheta_{\fD}=(\theta_{\fD,i})_{i\in\Lambda}$
of $\CO_{\fD}$. Let $n\in \N$, and let $(M,\utheta_M)$ be an object of $q\HIG(D_n,\ut,\utheta_D)$.
By taking the scalar extension of $(M,\utheta_M)$ under $D\to \tD$ and $\id_{\Lambda}$
\eqref{eq:TwConnectionScExtFunct} for each affine open $\Spf(\tD)\subset \fD$, 
we obtain a family of endomorphisms $\utheta_{\CM}=(\theta_{\CM,i})_{i\in\Lambda}$
of the quasi-coherent $\CO_{\fD_n}$-module $\CM$ on $\fD_{\Zar}=(\fD_n)_{\Zar}$ associated
to the $D_n$-module $M$. We call $(\CM,\utheta_{\CM})$ the {\it $q$-Higgs module over
$(\fD, \ut,\utheta_{\fD})$} {\it associated to} $(M,\utheta_M)$.
Similarly to Definition \ref{def:TwConnectionDRcpx} (2) and Definition \ref{def:qHiggsModCpx}, 
we define a complex $q\Omega^{\bullet}(\CM,\utheta_{\CM})$ to be the
Koszul complex associated to $\utheta_{\CM}$, which we call the {\it $q$-Higgs complex
of} $(\CM,\utheta_{\CM})$. The differential maps of $q\Omega^{\bullet}(\CM,\utheta_{\CM})$ are
$R_n$-linear by Remark \ref{rmk:dRComplexRlin}, where $R_n$ denotes $R/(p,\pq)^{n+1}R$ similarly to $D_n$.
\par
(2) (\cite[13.5 (2)]{TsujiPrismQHiggs}) Let $n\in \N$,  let $(M_{\nu},\utheta_{M_{\nu}})$ $(\nu=1,2)$ be objects of 
$q\HIG(D_n,\ut,\utheta_D)$, and let $(\CM_{\nu},\utheta_{\CM_{\nu}})$ $(\nu=1,2)$
and $(\CM_1\otimes\CM_2,\utheta_{\CM_1\otimes \CM_2})$ be the $q$-Higgs
modules over $(\fD,\ut,\utheta_{\fD})$ associated to $(M_{\nu},\utheta_{M_{\nu}})$
and their tensor product. The $\CO_{\fD_n}$-module 
$\CM_1\otimes\CM_2$ is the tensor product of the $\CO_{\fD_n}$-modules
$\CM_1$ and $\CM_2$.  Since the scalar extension functor \eqref{eq:TwConnectionScExtFunct} 
preserves tensor products, 
we obtain a morphism of complexes
\begin{equation}\label{eq:qHiggsCpxShvProd}
q\Omegab(\CM_1,\utheta_{\CM_1})\otimes_{R_n}q\Omegab(\CM_2,\utheta_{\CM_2})\longrightarrow
q\Omegab(\CM_1\otimes\CM_2,\utheta_{\CM_1\otimes\CM_2})
\end{equation}
by applying \eqref{eq:DRCpxProdMap} to the sections of the three $q$-Higgs complexes on each affine
open of $\fD$.\par
(3) (\cite[13.5 (1)]{TsujiPrismQHiggs}) Let $n\in \N$, let $(M,\utheta_{M})$ be an object of $q\HIG(D_n,\ut,\utheta)$,
and let $(\CM,\utheta_{\CM})$ be its associated $q$-Higgs module over $(\fD,\ut,\utheta_{\fD})$.
Let $\varphi_{\fD_n}$ denote the endomorphism of $\fD_n$ defined by $\varphi_{D_n}$,
and let $(\varphi_{\fD_n}^*\CM,\utheta_{\varphi_{\fD_n}^*\CM})$ denote the $q$-Higgs module
over $(\fD,\ut,\utheta_{\fD})$ associated to the Frobenius pullback 
$\varphi_{D_n}^*(M,\utheta_M)$ \eqref{eq:FrSmqHigFrobPB}. By the compatibility of the Frobenius
pullback functor \eqref{eq:TwConnFrobPBFunct} with scalar extensions
mentioned before Example \ref{ex:qHiggsDerivFunct}, we obtain a morphism of complexes
\begin{equation}\label{eq:qHiggCpxShvFrobPB}
\varphi^{\bullet}_{\fD_n,\Omega}(\CM,\utheta_{\CM})\colon 
q\Omega^{\bullet}(\CM,\utheta_{\CM})\longrightarrow
q\Omega^{\bullet}(\varphi_{\fD_n}^*\CM,\utheta_{\varphi^*_{\fD_n}\CM})
\end{equation}
functorial in $(M,\utheta_M)$, by applying \eqref{eq:TwdRCpxFrobPB} to the sections of the two $q$-Higgs complexes on 
each affine open of $\fD$. Since the Frobenius pullback morphism of de Rham complexes
\eqref{eq:TwdRCpxFrobPB} is compatible with the product \eqref{eq:DRCpxProdMap}, we see that the morphism
\eqref{eq:qHiggCpxShvFrobPB} is compatible with the product \eqref{eq:qHiggsCpxShvProd}.\par

(4) (\cite[14.13]{TsujiPrismQHiggs}) 
Let $(f/g,\psi)\colon ((A,J)/R,\ut=(t_i)_{i\in\Lambda})\to ((A',J')/R',\ut'=(t'_{i'})_{i'\in\Lambda'})$
be a morphism of admissible framed smooth $q$-pairs, let $D'$ be the $q$-prismatic envelope
of $(A',J')$,  and let $f_D$ be the morphism of $q$-prisms $D\to D'$ induced by $f$.
We define $D'_n$, $\fD'$, $\fD'_n$, $\utheta_{\fD'}=(\theta_{\fD',i'})_{i'\in\Lambda'}$
in the same way as $D_n$, $\fD$, etc.~by using $D'$ and $\utheta_{D'}$.
Let $\ff_D$ denote the morphism $\Spf(f_D)\colon \fD'\to \fD$.
Let $n\in \N$, let $(M,\utheta_M)$ be an object of $q\HIG(D_n,\ut,\utheta_D)$,
let $(M',\utheta_{M'})$ be its scalar extension under $(f_D/g,\psi)$ \eqref{eq:qHiggsCpxScExtFunct},
and let $(\CM,\utheta_{\CM})$ (resp.~$(\CM',\utheta_{\CM'})$) be
the $q$-Higgs module over $(\fD,\ut,\utheta_{\fD})$ (resp.~$(\fD',\ut',\utheta_{\fD'})$)
associated to $(M,\utheta_M)$ (resp.~$(M',\utheta_{M'})$). 
For each affine open $\Spf(\tD)\subset \fD$ and its pullback $\Spf(\tD')\subset \fD'$
under $\ff_D$, writing $f_{\tD}$ for the morphism $\tD\to \tD'$ induced by $f_D$,
we see that $E^{\psi}(f_{\tD})\colon E^{\ut\mu,\utheta_{\tD}}(\tD)
\to E^{\ut'\mu,\utheta_{\tD'}}(\tD')$ is a bialgebra homomorphism over $f_{\tD}$
as $E^{\psi}(f_D)$ is a bialgebra  homomorphism over $f_D$ (Proposition \ref{prop:FrSmqHiggDerivFunct}).
Therefore we obtain a morphism of complexes
\begin{equation}\label{eq:qHiggsCpxShvFunct}
q\Omega^{\bullet}_{f_D,\psi}(\CM,\utheta_{\CM})\colon
q\Omegab(\CM,\utheta_{\CM})\longrightarrow \ff_{D,\Zar*}(q\Omegab(\CM',\utheta_{\CM'}))
\end{equation}
by applying \eqref{eq:TwDrCpxScExt} for $f_{\tD}$ and $\psi$ to the sections of the
two $q$-Higgs complexes on $\Spf(\tD)$ and $\Spf(\tD')$ 
for each affine open $\Spf(\tD)$ of $\fD$. 
The morphism \eqref{eq:qHiggsCpxShvFunct} is compatible with compositions of $(f/g,\psi)$'s 
by the same property of the morphism \eqref{eq:TwDrCpxScExt}. Since 
\eqref{eq:TwDrCpxScExt} is compatible
with the product \eqref{eq:DRCpxProdMap} when $\psi$ is injective, the morphism 
\eqref{eq:qHiggsCpxShvFunct} is compatible with the product \eqref{eq:qHiggsCpxShvProd} 
when $\psi$ is injective \cite[14.16 (2)]{TsujiPrismQHiggs}. 
We have some weaker compatibility for a general
$\psi$ (see \cite[14.17]{TsujiPrismQHiggs}). By the compatibility
of the Frobenius pullbacks of de Rham complexes \eqref{eq:TwdRCpxFrobPB}
with scalar extensions mentioned before Example \ref{ex:qHiggsDerivFunct}, 
the morphism \eqref{eq:qHiggsCpxShvFunct}
is compatible with the Frobenius pullback morphisms
\eqref{eq:qHiggCpxShvFrobPB} for $(\CM,\utheta_{\CM})$ and $(\CM',\utheta_{\CM'})$
\cite[14.16 (1)]{TsujiPrismQHiggs}.
\end{construction}

\begin{definition}\label{def:qNilpQHiggs}
Let $((A,J)/R,\ut=(t_i)_{i\in\Lambda})$ be an admissible framed smooth $q$-pair
and let $D$ be the $q$-prismatic envelope of $(A,J)$. \par
(1) (\cite[11.19]{TsujiPrismQHiggs}) For $n\in\N$,
we say that an object $(M,\utheta_{M})$ of $q\HIG(D_n,\ut,\utheta_{D})$
(Definition \ref{def:qHiggsModCpx}) is {\it quasi-nilpotent} if for any $x\in M$ and $i\in \Lambda$,
there exists an integer $N\geq 1$ such that $\theta_{M,i}^N(x)=0$.
We write $q\HIG_{\qnilp}(D_n,\ut,\utheta_D)$ for the full subcategory
of $q\HIG(D_n,\ut,\utheta_D)$ consisting of quasi-nilpotent objects.\par
(2) Assume that $A/J$ is $p$-adically complete and separated, and
put $\fX=\Spf(A/J)$. We regard $D$ as an object of $(\fX/R)_{\prism}$
by the morphism $v_D\colon \Spf(D/\pq D)\to \fX$ induced 
by the homomorphism of $\delta$-pairs $(A,J)\to (D,\pq D)$. 
Let $i\in \Lambda$. 
Since $\theta_{A,i}(A)\subset \pq A\subset J$ (Proposition \ref{prop:qPrismEnvTheta}),
the automorphism $\gamma_{D,i}=\id_D+t_i\mu\theta_{D,i}$ of the bounded
prism $(D,\pq D)$ over $(R,\pq R)$ defines an automorphism of the object $(D,v_D)$ of $(\fX/R)_{\prism}$,
which we denote by $\gamma_{(D,v_D),i}$.
\end{definition}

\begin{theorem}[{\cite[11.10, 11.20]{TsujiPrismQHiggs}}]
\label{thm:PrismCrysqHiggsEquiv}
We have a canonical equivalence of categories 
\begin{equation} \label{eq:PrismCrysqHiggsEquiv}
\Crystal_{\prism}(\CO_{\fX/R,n})\simeq q\HIG_{\qnilp}(D_n,\ut,\utheta_D)
\end{equation}
satisfying the properties below, 
for each admissible framed smooth $q$-pair 
$((A,J)/R,\ut)$ with $A/J$ $p$-adically complete and separated, 
$\fX=\Spf(A/J)$, and the $q$-prismatic envelope $D$ 
of $(A,J)$, which is equipped with $\utheta_{D}$ and $v_D$.\par
(1) {\rm(\cite[11.33]{TsujiPrismQHiggs})} Let $\CF$ be an object of $\Crystal_{\prism}(\CO_{\fX/R,n})$,
and let $(M,\utheta_{M})$ be the object of $q\HIG(D_n,\ut,\utheta_D)$
corresponding to $\CF$ by \eqref{eq:PrismCrysqHiggsEquiv}.
Then we have $M=\CF(D,v_D)$ and the $\gamma_{D,i}$-semilinear
automorphism $\gamma_{M,i}=\id_M+t_i\mu\theta_{M,i}$ 
(Lemma \ref{lem:TwConnSemilinEnd}) of $M$
coincides with the $\gamma_{D,i}$-semilinear 
automorphism of $\CF(D,v_D)$ induced by the automorphism
$\gamma_{(D,v_D),i}$ of $(D,v_D)$.\par
(2) The equivalence \eqref{eq:PrismCrysqHiggsEquiv} is
compatible with the inclusions
$\Crystal_{\prism}(\CO_{\fX/R,n})\subset \Crystal_{\prism}(\CO_{\fX/R,n+1})$
and $q\HIG_{\qnilp}(D_n,\ut,\utheta_D)\subset 
q\HIG_{\qnilp}(D_{n+1},\ut,\utheta_D)$.\par
(3) {\rm (\cite[11.34 (1)]{TsujiPrismQHiggs})} The equivalence \eqref{eq:PrismCrysqHiggsEquiv} is compatible
with the Frobenius pullbacks $\varphi_n^*$ (Remark \ref{rmk:PrismCrystalFrobTensor} (2))
and $\varphi^*_{D_n}$ \eqref{eq:FrSmqHigFrobPB}.\par
(4) {\rm (\cite[11.34 (2)]{TsujiPrismQHiggs})}
The equivalence \eqref{eq:PrismCrysqHiggsEquiv} is compatible 
with tensor products (Remark \ref{rmk:PrismCrystalFrobTensor} (3)
and Definition \ref{def:TwConnectionDRcpx} (3)).\par
(5) {\rm (\cite[11.14, 11.37]{TsujiPrismQHiggs})} Let $(f/g,\psi)\colon ((A,J)/R,\ut)\to ((A',J')/R',\ut')$ be a
morphism of admissible framed smooth $q$-pairs such that 
$A/J$ and $A'/J'$ are $p$-adically complete and separated.
Let $D$ and $D'$ be the $q$-prismatic envelopes of $(A,J)$ and $(A',J')$,
let $f_D\colon D\to D'$ be the morphism of $q$-prisms induced by $f$,
put $\fX=\Spf(A/J)$ and $\fX'=\Spf(A'/J')$, and 
let $\off_{\prism}\colon (\fX'/R')^{\sim}_{\prism}\to (\fX/R)_{\prism}^{\sim}$ 
be the morphism of topos induced by $f$ and $g$.
Then the following diagram is commutative
up to a canonical isomorphism.
\begin{equation}
\xymatrix{
\Crystal_{\prism}(\CO_{\fX/R,n})\ar@{-}[r]^(.43){\sim}_(.43){\eqref{eq:PrismCrysqHiggsEquiv}}
\ar[d]_{\overline{\ff}_{\prism}^{-1}}^{\text{Definition }\ref{def:PrismaticSiteCrystal} \text{(3)}}
& q\HIG_{\qnilp}(D_n,\ut,\utheta_D) \ar@{^{(}->}[r]&
q\HIG(D,\ut,\utheta_D)
\ar[d]_{(f_D/g,\psi)^*}^{\eqref{eq:qHiggsCpxScExtFunct}}\\
\Crystal_{\prism}(\CO_{\fX'/R',n})\ar@{-}[r]^(.43){\sim}_(.43){\eqref{eq:PrismCrysqHiggsEquiv}}& 
q\HIG_{\qnilp}(D'_n,\ut',\utheta_{D'})\ar@{^{(}->}[r]&
q\HIG(D',\ut',\utheta_{D'})
}
\end{equation}
\end{theorem}

\begin{definition}\label{def:PrismCrysqHigsCpxShv}
Let $((A,J)/R,\ut=(t_i)_{i\in\Lambda})$ be an admissible framed smooth $q$-pair with
$A/J$ $p$-adically complete and separated, let $D$ be the $q$-prismatic envelope of $(A,J)$, 
and put $\fD=\Spf(D)$ and $\fX=\Spf(A/J)$. 
For $n\in \N$ and $\CF\in \Ob\Crystal_{\prism}(\CO_{\fX/R,n})$, we write
$(M_D(\CF),\utheta_{M_D(\CF)})$ for the $q$-Higgs module over $(D,\ut,\utheta_{D})$ 
corresponding to $\CF$ by the equivalence \eqref{eq:PrismCrysqHiggsEquiv}, and define
$(\CF_{\fD},\utheta_{\CF_{\fD}})$ to  be the $q$-Higgs module over $(\fD,\ut,\utheta_{\fD})$
associated to $(M_D(\CF),\utheta_{M_D(\CF)})$ (Construction \ref{constr:qHiggsModCpxShv} (1)).
For $\CF\in \Ob \hCrystal_{\prism}(\CO_{\fX/R})$, we define 
$(\CF_{\fD},\utheta_{\CF_{\fD}})$ to be the inverse limit
of $(\CF_{n\fD},\utheta_{\CF_{n\fD}})$
$(n\in \N)$, where $\CF_n=\CF\otimes_{\CO_{\fX/R}}\CO_{\fX/R,n}$ 
(Remark \ref{rmk:PrismCrystalFrobTensor} (1)),
and define the {\it $q$-Higgs complex} $q\Omegab(\CF_{\fD},\utheta_{\CF_{\fD}})$
to be the Koszul complex of $\utheta_{\CF_{\fD}}$, which is isomorphic 
to the inverse limit of $q\Omegab(\CF_{n\fD},\utheta_{\CF_{n\fD}})$ $(n\in \N)$.
We often abbreviate $q\Omegab(\CF_{\fD},\utheta_{\CF_{\fD}})$
to $q\Omegab(\CF_{\fD})$ to simplify the notation.
\end{definition}

Let $((A,J)/R,\ut=(t_i)_{i\in\Lambda})$ be an admissible framed smooth $q$-pair with
$A/J$ $p$-adically complete and separated, let $D$ be the $q$-prismatic envelope of $(A,J)$, 
and put $\fD=\Spf(D)$ and $\fX=\Spf(A/J)$. We can apply Construction \ref{constr:qHiggsModCpxShv}
(2), (3), and (4) to the $q$-Higgs modules over $(D,\ut,\utheta_D)$ associated to crystals
on $(\fX/R)_{\prism}$ as follows.

For $n\in \N$ and $\CF_1,\CF_2\in \Ob\Crystal_{\prism}(\CO_{\fX/R,n})$
(resp.~$\Ob \hCrystal_{\prism}(\CO_{\fX/R})$), we obtain a morphism 
\begin{equation}\label{eq:CrysqHiggsCpxProd}
q\Omegab(\CF_{1\fD})\otimes_Rq\Omega^{\bullet}(\CF_{2\fD})
\to q\Omegab((\CF_1\otimes_{\CO_{\fX/R,n}}\CF_2)_{\fD})
\;\;(\text{resp.~ } q\Omegab((\CF_1\hotimes_{\CO_{\fX/R}}\CF_2)_{\fD}))
\end{equation}
by applying \eqref{eq:qHiggsCpxShvProd} to $(M_D(\CF_{\nu}),\utheta_{M_D(\CF_{\nu})})$ $(\nu=1,2)$
(resp.~$(M_D(\CF_{\nu,n}),\utheta_{M_D(\CF_{\nu,n})}$) $(\nu=1,2)$, 
$\CF_{\nu,n}=\CF_{\nu}\otimes_{\CO_{\fX/R}}\CO_{\fX/R,n}$ $(n\in \N)$ and taking
the inverse limit over $n$.) Note that the equivalence \eqref{eq:PrismCrysqHiggsEquiv} 
preserves tensor products
(Theorem \ref{thm:PrismCrysqHiggsEquiv} (4)).
See Remark \ref{rmk:PrismCrystalFrobTensor} (3) for $\CF_1\hotimes_{\CO_{\fX/R}}\CF_2$.

For $n\in \N$ and $\CF\in \Ob\Crystal_{\prism}(\CO_{\fX/R,n})$
(resp.~$\CF\in \Ob\hCrystal_{\prism}(\CO_{\fX/R})$), we obtain a morphism 
\begin{equation}\label{eq:CrysqHiggCpxFrobPB}
q\Omega^{\bullet}(\CF_{\fD})\longrightarrow q\Omegab((\varphi_n^*\CF)_{\fD})
\;\;\text{(resp.~ }q\Omegab((\hvarphi^*\CF)_{\fD})
\end{equation}
functorial in $\CF$ and compatible with the product \eqref{eq:CrysqHiggsCpxProd}
by applying \eqref{eq:qHiggCpxShvFrobPB} to $(M_D(\CF),\utheta_{M_D(\CF)})$ 
(resp.~$(M_D(\CF_n),\utheta_{M_D(\CF_n)})$,
$\CF_n=\CF\otimes_{\CO_{\fX/R}}\CO_{\fX/R,n}$ $(n\in \N)$ and taking the inverse limit
over $n$).
Note that the equivalence \eqref{eq:PrismCrysqHiggsEquiv} is compatible with the Frobenius pullbacks
(Theorem \ref{thm:PrismCrysqHiggsEquiv} (3)). See Remark 
\ref{rmk:PrismCrystalFrobTensor} (2) for $\hvarphi^*\CF$. 

Under the same notation as Construction \ref{constr:qHiggsModCpxShv} (4), 
assume that $A/J$ and $A'/J'$ are $p$-adically complete
and separated, put $\fX=\Spf(A/J)$ and $\fX'=\Spf(A'/J')$, and let 
$\off_{\prism}\colon (\fX'/R')_{\prism}^{\sim}\to (\fX/R)_{\prism}^{\sim}$ 
denote the morphism of topos induced by $f$ and $g$.
For $n\in \N$ and $\CF\in \Ob\Crystal_{\prism}(\CO_{\fX/R,n})$
(resp.~$\CF\in \Ob\hCrystal_{\prism}(\CO_{\fX/R})$), we obtain a morphism 
\begin{equation}\label{eq:CrysqHiggsCpxFunct}
q\Omegab(\CF_{\fD})\longrightarrow \ff_{D,\Zar*}(q\Omegab((\off^{-1}_{\prism}\CF)_{\fD'}))
\end{equation}
by applying \eqref{eq:qHiggsCpxShvFunct} to $(M_D(\CF),\utheta_{M_D(\CF)})$
(resp.~$(M_D(\CF_n),\utheta_{M_D(\CF_n)})$, $\CF_n=\CF\otimes_{\CO_{\fX/R}}\CO_{\fX/R,n}$
$(n\in \N)$ and taking the inverse limit over $n$). 
Note that the equivalence \eqref{eq:PrismCrysqHiggsEquiv}
is compatible with the
pullbacks $\off_{\prism}^{-1}$ and $(f_D/g,\psi)^*$ (Theorem \ref{thm:PrismCrysqHiggsEquiv} (5)).
The morphism \eqref{eq:CrysqHiggsCpxFunct} 
is functorial in $\CF$, compatible with the Frobenius pullback \eqref{eq:CrysqHiggCpxFrobPB},
and compatible with compositions of $(f/g,\psi)$'s. When $\psi $ is injective,
it is also compatible with the product \eqref{eq:CrysqHiggsCpxProd}. 
For a general $\psi$, we have  some
weaker compatibility (\cite[14.17]{TsujiPrismQHiggs}).

\begin{theorem}[{\cite[13.9, 13.31]{TsujiPrismQHiggs}}]
\label{thm:PrismCohQHiggsCpx}
Let $((A,J)/R,\ut)$ be an admissible framed smooth $q$-pair with
$A/J$ $p$-adically complete and separated, put $\fX=\Spf(A/J)$,
and let $D$ be the $q$-prismatic envelope of $(A,J)$, which is equipped with
$\utheta_D$ and $v_D$. 
For $n\in \N$ and $\CF\in \Ob \Crystal_{\prism}(\CO_{\fX/R,n})$
(resp.~$\CF\in \Ob\hCrystal_{\prism}(\CO_{\fX/R})$),
we have a canonical isomorphism
in $D^+(\fX_{\Zar},R_n)$ (resp.~$D^+(\fX_{\Zar},R)$)
\begin{equation}\label{eq:PrismCohQHiggsCpx}
Ru_{\fX/R*}\CF\cong v_{D,\Zar*}(q\Omegab(\CF_{\fD},\utheta_{\CF_{\fD}}))
\end{equation}
functorial in $\CF$ and satisfying the properties below.
Here $u_{\fX/R}$ is the morphism of topos \eqref{eq:PrismToposZarProj}.\par
(1) In the first case, the isomorphism \eqref{eq:PrismCohQHiggsCpx} is compatible
with respect to $n$ such that $\CF\in \Ob \Crystal_{\prism}(\CO_{\fX/R,n})$.
The isomorphisms \eqref{eq:PrismCohQHiggsCpx} for $\CF\in \Ob\hCrystal_{\prism}(\CO_{\fX/R})$
and $\CF_n=\CF\otimes_{\CO_{\fX/R}}\CO_{\fX/R,n}$ $(n\in \N)$ are compatible with the
projection morphisms for both sides induced by $\CF\to \CF_n$.
\par
(2) {\rm(\cite[13.28 (1), 13.33 (1)]{TsujiPrismQHiggs})} 
The isomorphisms \eqref{eq:PrismCohQHiggsCpx} for $\CF$ and 
$\varphi_n^*\CF$ (resp.~$\hvarphi^*\CF$) 
are compatible with 
$Ru_{\fX/R*}\CF\to \varphi_*Ru_{\fX/R*}(\varphi_n^*\CF)$ (resp.~$\varphi_*Ru_{\fX/R*}(\hvarphi^*\CF)$)
induced by 
$\CF\to \varphi_n^*\CF$ (resp.~$\hvarphi^*\CF$); $x\mapsto x\otimes 1$ and
\eqref{eq:CrysqHiggCpxFrobPB}. Here $\varphi_*$ denotes the restriction of scalars
under the lifting of Frobenius of $R_n$ (resp.~$R$).\par
(3) {\rm (\cite[13.28 (2), 13.33 (2)]{TsujiPrismQHiggs})}
For $n\in \N$ and $\CF_{\nu}\in \Ob\Crystal_{\prism}(\CO_{\fX/R,n})$ 
(resp.~$\CF_{\nu}\in \Ob\hCrystal_{\prism}(\CO_{\fX/R})$) $(\nu=1,2)$,
the isomorphisms \eqref{eq:PrismCohQHiggsCpx} for $\CF_{\nu}$
and their tensor product
are compatible with the product $Ru_{\fX/R*}\CF_1\otimes^LRu_{\fX/R*}\CF_2
\to Ru_{\fX/R*}(\CF_1\otimes_{\CO_{\fX/R,n}}\CF_2)$
(resp.~$Ru_{\fX/R*}(\CF_1\hotimes_{\CO_{\fX/R}}\CF_2$))
and the product \eqref{eq:CrysqHiggsCpxProd}.
Here $\otimes^L$ denotes $\otimes^L_{R_n}$ (resp.~$\otimes^L_R$).
\par
(4) {\rm(\cite[14.32, 14.33]{TsujiPrismQHiggs})} 
Under the notation in Theorem \ref{thm:PrismCrysqHiggsEquiv} (5), 
the isomorphisms
\eqref{eq:PrismCohQHiggsCpx} 
for $(\CF,  ((A,J)/R,\ut))$ and 
$(\overline{\ff}_{\prism}^{-1}\CF, ((A',J')/R',\ut'))$ are compatible with 
$Ru_{\fX/R*}\CF\to Ru_{\fX/R*}R\overline{\ff}_{\prism*}\overline{\ff}_{\prism}^{-1}\CF
\cong R\overline{\ff}_{\Zar*}Ru_{\fX'/R'*}\overline{\ff}_{\prism}^{-1}\CF$ 
and \eqref{eq:CrysqHiggsCpxFunct}, where
$\overline{\ff}_{\Zar}$ denotes the morphism of topos
$\fX_{\Zar}^{\prime\sim}\to \fX_{\Zar}^{\sim}$ induced by $f$.
\end{theorem}

\begin{definition}
An {\it admissible framed embedding system over $\Z_p[[\mu]]$}
is a set of data
\begin{equation}\label{eq:AdFrEmbSys}
(\fX\xleftarrow{\pi_{\bcdot}} \fX_{\bcdot}=\Spf(\oA_{\bcdot})
\overset{\mfi_{\bcdot}}{\hookrightarrow}
\fY_{\bcdot}=\Spf(A_{\bcdot})/R, \ut_{\bcdot})\end{equation}
consisting of a $q$-prism $R$, a quasi-compact and separated
$p$-adic formal scheme $\fX$ over $\Spf(R/\pq R)$,
a Zariski hypercovering $\fX_{\bcdot}=([r]\mapsto \fX_{[r]}=\Spf(\oA_{[r]}))_{r\in\N}$
of $\fX$ by affine formal schemes, a cosimplicial admissible framed smooth $\delta$-pair 
$A_{\bcdot}=([r]\mapsto ((A_{[r]},J_{[r]}), \ut_{[r]}))_{r\in \N}$ over $R$ 
(Definition \ref{def:qprism} (3)), 
and a closed immersion
$\mfi_{\bcdot}\colon \fX_{\bcdot}\hookrightarrow \fY_{\bcdot}=\Spf(A_{\bcdot})$ of simplicial
formal schemes over $R$ defined by an isomorphism of cosimplicial
$R$-algebras $\oA_{\bcdot}\cong A_{\bcdot}/J_{\bcdot}$.
By taking the $q$-prismatic envelope $D_{[r]}$ of 
$(A_{[r]},J_{[r]})$ for each $r\in \N$, we obtain a closed 
immersion of simplicial formal schemes
$\ofD_{\bcdot}=\Spf(D_{\bcdot}/\pq D_{\bcdot})
\hookrightarrow \fD_{\bcdot}=\Spf(D_{\bcdot})$ 
lying over $\mfi_{\bcdot}\colon\fX_{\bcdot}\hookrightarrow \fY_{\bcdot}$,
which we call the {\it $q$-prismatic envelope of the embedding system}
\eqref{eq:AdFrEmbSys}.
We write $v_{D_{\bcdot}}$ for the morphism of 
simplicial formal schemes $\ofD_{\bcdot}\to \fX_{\bcdot}$.
 
A {\it morphism of admissible framed embedding systems over $\Z_p[[\mu]]$}
\begin{equation}\label{eq:MOrphiAdFrEmbSys}
(\fX'\leftarrow \fX'_{\bcdot}\hookrightarrow\fY'_{\bcdot}/R',\ut'_{\bcdot})\longrightarrow
(\fX\leftarrow \fX_{\bcdot}\hookrightarrow\fY_{\bcdot}/R,\ut_{\bcdot})\end{equation}
is a pair $((f_{\bcdot}/g,\psi_{\bcdot}),\off)$ consisting of 
a morphism $(f_{\bcdot}/g,\psi_{\bcdot})\colon ((A_{\bcdot},J_{\bcdot})/R,\ut_{\bcdot})\to
((A'_{\bcdot},J'_{\bcdot})/R',\ut'_{\bcdot})$ of cosimplicial framed smooth $q$-pairs
and a morphism of formal schemes $\off\colon \fX'\to \fX$ over $\Spf(g)$
such that the morphism $\off_{\bcdot}\colon \fX'_{\bcdot}\to \fX_{\bcdot}$ induced by $f_{\bcdot}$
is compatible with $\off$. The morphism $((f_{\bcdot}/g,\psi_{\bcdot}),\off)$ induces a
morphism between the $q$-prismatic envelopes 
$(\off_{D_{\bcdot}},\ff_{D_{\bcdot}})\colon (\ofD'_{\bcdot}\hookrightarrow \fD'
_{\bcdot})\to (\ofD_{\bcdot}\hookrightarrow \fD_{\bcdot})$.
\end{definition}

\begin{remark}[{\cite[15.1]{TsujiPrismQHiggs}}]
Any $q$-prism $R$ and any quasi-compact and separated $p$-adic smooth formal scheme $\fX$ over $\Spf(R/\pq R)$ have an admissible framed
embedding system as \eqref{eq:AdFrEmbSys}. 
\end{remark}

Let $(\fX\xleftarrow{\pi_{\bcdot}} \fX_{\bcdot}=\Spf(\oA_{\bcdot})
\overset{\mfi_{\bcdot}}{\hookrightarrow}
\fY_{\bcdot}=\Spf(A_{\bcdot})/R,\ut_{\bcdot})$ be an admissible framed embedding system over $\Z_p[[\mu]]$.
We define $((\fX_{\bcdot}/R)_{\prism}^{\sim},\CO_{\fX_{\bcdot}/R,n})$ 
to be the ringed topos associated to the simplicial ringed topos 
$([r]\mapsto ((\fX_{[r]}/R)_{\prism}^{\sim},\CO_{\fX_{[r]}/R,n}))_{r\in \N}$, and let 
$\theta_{\prism}$ denote the morphism of ringed topos
$((\fX_{\bcdot}/R)_{\prism}^{\sim},\CO_{\fX_{\bcdot}/R,n})\to
((\fX/R)_{\prism}^{\sim},\CO_{\fX/R,n})$ induced by $\pi_{\bcdot}$. 
We define the ringed topos
$(\fX_{\bcdot,\Zar}^{\sim},R_n)$ and the morphism of ringed topos
$\theta\colon (\fX_{\bcdot,\Zar}^{\sim},R_n)\to (\fX_{\Zar}^{\sim},R_n)$
similarly. Let $u_{\fX_{\bcdot}/R}\colon ((\fX_{\bcdot}/R)_{\prism}^{\sim},\CO_{\fX_{\bcdot}/R,n})
\to (\fX_{\bcdot,\Zar}^{\sim},R_n)$ denote the morphism of ringed topos associated to
the morphism of simplicial ringed topos
$([r]\mapsto u_{\fX_{[r]}/R}\colon 
((\fX_{[r]}/R)_{\prism}^{\sim},\CO_{\fX_{[r]}/R,n})\to (\fX_{[r],\Zar}^{\sim},R_n))$
for $n\in \N$. We also consider variants without the subscript $n$. 
For the $q$-prismatic envelope $\ofD_{\bcdot}\hookrightarrow \fD_{\bcdot}$
of the embedding system, the morphism of simplicial formal schemes
$v_{D_{\bcdot}}\colon \ofD_{\bcdot}\to \fX_{\bcdot}$ defines a morphism of topos
$v_{D_{\bcdot},\Zar}\colon \ofD_{\bcdot,\Zar}^{\sim}\to \fX_{\bcdot,\Zar}^{\sim}$.

\begin{theorem}[{\cite[15.5, 15.11]{TsujiPrismQHiggs}}]\label{thm:PrismCohQHiggsCpxGlb}
Under the notation above, let $n\in \N$ and $\CF\in \Ob \Crystal_{\prism}(\CO_{\fX/R,n})$
(resp.~$\CF\in\Ob \hCrystal_{\prism}(\CO_{\fX/R})$),
and let $\CF_{\bcdot}$ be the $\CO_{\fX_{\bcdot}/R,n}$-module (resp.~$\CO_{\fX_{\bcdot}/R}$-module) 
$\theta_{\prism}^{-1}(\CF)$, which consists of $\pi_{[r]\prism}^{-1}(\CF)
\in \Ob \Crystal_{\prism}(\CO_{\fX_{[r]}/R,n})$ (resp.~$\in\Ob \hCrystal_{\prism}(\CO_{\fX_{[r]}/R})$)
$(r\in \N)$.
Let $q\Omegab(\CF_{\bcdot \fD_{\bcdot}},\utheta_{\CF_{\bcdot \fD_{\bcdot}}})$ be the
complex on $\ofD_{\bcdot,\Zar}^{\sim}$ 
obtained by applying Definition \ref{def:PrismCrysqHigsCpxShv}
and \eqref{eq:CrysqHiggsCpxFunct} to $\CF_{[r]}$ $(r\in \N)$ and 
the structure morphisms among 
$((A_{[r]},J_{[r]}),\ut_{[r]})$ $(r\in \N)$. 
Then we have a canonical isomorphism 
in $D^+(\fX_{\bcdot,\Zar}^{\sim},R_n)$ (resp.~$D^+(\fX_{\bcdot,\Zar}^{\sim},R)$)
\begin{equation}\label{eq:PrismCohQHiggsCpxSimp}
Ru_{\fX_{\bcdot}/R*}\CF_{\bcdot}\cong 
v_{D_{\bcdot},\Zar*}(q\Omegab(\CF_{\bcdot \fD_{\bcdot}},\utheta_{\CF_{\bcdot \fD_{\bcdot}}}))\end{equation}
functorial in $\CF$ and satisfying the following properties.\par
(1) In the first case, the isomorphism \eqref{eq:PrismCohQHiggsCpxSimp} is compatible
with respect to $n$ such that $\CF\in \Ob \Crystal_{\prism}(\CO_{\fX/R,n})$.
The isomorphisms \eqref{eq:PrismCohQHiggsCpxSimp} for $\CF\in \Ob\hCrystal_{\prism}(\CO_{\fX/R})$
and $\CF_n=\CF\otimes_{\CO_{\fX/R}}\CO_{\fX/R,n}$ $(n\in \N)$ are compatible with the
projection morphisms for both sides induced by $\CF\to \CF_n$.\par
(2) {\rm(\cite[15.13]{TsujiPrismQHiggs})} 
The isomorphisms \eqref{eq:PrismCohQHiggsCpxSimp} for $\CF$ and 
$\varphi_n^*\CF$ (resp.~$\hvarphi^*\CF$) 
are compatible with 
$Ru_{\fX_{\bcdot}/R*}\CF_{\bcdot}\to \varphi_*Ru_{\fX_{\bcdot}/R*}(\varphi_n^*\CF)_{\bcdot}$ 
(resp.~$\varphi_*Ru_{\fX_{\bcdot}/R*}(\hvarphi^*\CF)_{\bcdot}$) induced by 
$\CF\to \varphi_n^*\CF$ (resp.~$\hvarphi^*\CF$); $x\mapsto x\otimes 1$ and
the morphism on $\ofD_{\bcdot,\Zar}^{\sim}$
\begin{equation}q\Omegab(\CF_{\bcdot\fD_{\bcdot}})\longrightarrow 
\varphi_*q\Omegab((\varphi_n^*\CF)_{\bcdot\fD_{\bcdot}})\;\;
{\text(resp.~ }\varphi_*q\Omegab((\hvarphi^*\CF)_{\bcdot\fD_{\bcdot}}))\end{equation}
obtained by applying \eqref{eq:CrysqHiggCpxFrobPB}
to $\CF_{[r]}$ $(r\in \N)$. Here $\varphi_*$ denotes the restriction of scalars
under the lifting of Frobenius of $R_n$ (resp.~$R$).\par
(3) {\rm(\cite[15.24, 15.28]{TsujiPrismQHiggs})}
Suppose that we are given a morphism $((f_{\bcdot}/g,\psi_{\bcdot}),\off)$ as
\eqref{eq:MOrphiAdFrEmbSys}, and let $\off_{\bcdot}$ and $\ff_{D_{\bcdot}}$ 
be the morphisms introduced after \eqref{eq:MOrphiAdFrEmbSys}.
Let $\off_{\prism}$ (resp.~$\off_{\bcdot\prism}$) be
 the morphism of topos $(\fX'/R')^{\sim}_{\prism}\to (\fX/R)^{\sim}_{\prism}$
(resp.~$(\fX'_{\bcdot}/R')^{\sim}_{\prism}\to (\fX_{\bcdot}/R)^{\sim}_{\prism}$)
induced by $\off$ (resp.~$\off_{\bcdot}$) and $g$.
Put $\CF':=\off_{\prism}^{-1}\CF$.
Then 
the isomorphisms \eqref{eq:PrismCohQHiggsCpxSimp} for 
$\CF$ and $\CF'$ are compatible with 
$Ru_{\fX_{\bcdot}/R*}\CF_{\bcdot}\to Ru_{\fX_{\bcdot}/R*}R\off_{\bcdot\prism*}\CF'_{\bcdot}
\cong R\off_{\bcdot,\Zar*}Ru_{\fX'_{\bcdot}/R'*}\CF'_{\bcdot}$ 
and the morphism on $\ofD_{\bcdot,\Zar}^{\sim}$
\begin{equation}\label{eq:CrysQHiggsCpxPB}
q\Omegab(\CF_{\bcdot\fD_{\bcdot}})\longrightarrow 
\ff_{D_{\bcdot},\Zar*}(q\Omegab(\CF'_{\bcdot\fD'_{\bcdot}}))\end{equation}
obtained by applying \eqref{eq:CrysqHiggsCpxFunct} to $\CF_{[r]}$ and $(f_{[r]}/g,\psi_{[r]})$
$(r\in \N)$.\par
(4) {\rm(\cite[15.17]{TsujiPrismQHiggs})}
Let $(\fX\xleftarrow{\pi_{\bcdot}}\fX_{\bcdot}\overset{\mfi^{\pone}_{\bcdot}}
\hookrightarrow \fY^{\pone}_{\bcdot}/R,\ut_{\bcdot}^{\pone})$ be
the admissible framed embedding system over $\Z_p[[\mu]]$
obtained by taking the fiber product of two copies of $(\fY_{\bcdot},\ut_{\bcdot})$
over $R$ (the $q$-prismatic envelope of $\mfi^{\pone}_{\bcdot}$ exists by 
Proposition \ref{prop:SuffCondExBdd2}), and let $\fD^{\pone}_{\bcdot}=\Spf(D^{\pone}_{\bcdot})$ denote
its $q$-prismatic envelope. Let $((p_{\nu\bcdot}/\id_R,\psi_{\nu\bcdot}),\id_{\fX})$ 
$(\nu=1,2)$ be the projection to $(\fX\xleftarrow{\pi_{\bcdot}}\fX_{\bcdot}
\overset{\mfi_{\bcdot}}\hookrightarrow \fY_{\bcdot}/R,\ut_{\bcdot})$
defined by the $\nu$th projection $(\fY^{\pone}_{\bcdot},\ut^{\pone}_{\bcdot})
\to (\fY_{\bcdot},\ut_{\bcdot})$. (For each $r\in \N$, we equip 
the index set $\Lambda_{[r]}^{\pone}$ of the coordinates
$\ut_{[r]}^{\pone}$ with the unique total order such that $\psi_{\nu}\colon
\Lambda_{[r]}\to \Lambda_{[r]}^{\pone}$
$(\nu=1,2)$ are maps of ordered sets and $\psi_1(i)\leq \psi_2(j)$
for every $i,j\in \Lambda_{[r]}$.) 
Let $n\in \N$ and 
$\CF_{\nu}\in \Ob \Crystal_{\prism}(\CO_{\fX/R,n})$
(resp.~$\CF_{\nu}\in \Ob \hCrystal_{\prism}(\CO_{\fX/R})$)
$(\nu=1,2)$, and let $\CF$ be $\CF_1\otimes_{\CO_{\fX/R,n}}\CF_2$
(resp.~$\CF_1\hotimes_{\CO_{\fX/R}}\CF_2$). Then the morphism
\eqref{eq:CrysQHiggsCpxPB} $q\Omegab(\CF_{\nu\bcdot\fD_{\bcdot}})\to
\fp_{\nu D_{\bcdot},\Zar*}(q\Omegab(\CF_{\nu\bcdot\fD^{\pone}_{\bcdot}}))$
for $\CF_{\nu}$ and $((p_{\nu\bcdot}/\id_R,\psi_{\nu\bcdot}),\id_{\fX})$,
and the product \eqref{eq:CrysqHiggsCpxProd} for $\CF_{1[r]}$, $\CF_{2[r]}$, and 
$(\mfi^{\pone}_{[r]}\colon \fX_{[r]}\hookrightarrow \fY^{\pone}_{[r]}/R,\ut_{[r]}^{\pone})$ $(r\in\N)$
induce a morphism of complexes
\begin{equation}
v_{D_{\bcdot},\Zar*}(q\Omegab(\CF_{1\bcdot\fD_{\bcdot}}))\otimes_R
v_{D_{\bcdot},\Zar*}(q\Omegab(\CF_{2\bcdot\fD_{\bcdot}}))
\longrightarrow
v_{D^{\pone}_{\bcdot},\Zar*}(q\Omegab(\CF_{\bcdot\fD^{\pone}_{\bcdot}})),
\end{equation}
and it describes the product
$Ru_{\fX_{\bcdot}/R*}\CF_{1\bcdot}
\otimes^LRu_{\fX_{\bcdot}/R*}\CF_{2\bcdot}
\to Ru_{\fX_{\bcdot}/R*}\CF_{\bcdot}$
via the isomorphisms \eqref{eq:PrismCohQHiggsCpxSimp} for $(\CF_{\nu},\fX\leftarrow\fX_{\bcdot}
\hookrightarrow \fY_{\bcdot}/R,\ut_{\bcdot})$ $(\nu=1,2)$
and for $(\CF, \fX\leftarrow \fX_{\bcdot}\hookrightarrow\fY^{\pone}_{\bcdot}/R,\ut^{\pone}_{\bcdot})$.
Here $\otimes^L$ denotes $\otimes^L_{R_n}$ (resp.~$\otimes^L_R$).
\end{theorem}

\begin{remark}[{\cite[15.6,15.9]{TsujiPrismQHiggs}}]
By taking $R\theta_*$ of \eqref{eq:PrismCohQHiggsCpxSimp}, we obtain an isomorphism in $D^+(\fX_{\Zar},R_n)$
(resp.~$D^+(\fX_{\Zar},R)$)
\begin{equation}\label{eq:PrismCohQHiggsCpxSimp2}
Ru_{\fX/R*}\CF\cong (([r]\mapsto \pi_{[r],\Zar*}v_{D_{[r]},\Zar*}(q\Omegab(\CF_{[r]\fD_{[r]}},
\utheta_{\CF_{[r]\fD_{[r]}}})))_{r\in\N})_s,
\end{equation}
where $(\;)_s$ denotes the simple complex associated to a cosimplicial complex.
\end{remark}

\section{Sheaves with action of a profinite group}\label{sec:Gsheaves}
In this section, we summarize basic facts on sheaves with
action of a profinite group on a site equipped with the
topology induced  by a pretopology
whose coverings consist of finite families of morphisms.
Recall that we have fixed two universes $\bV\in \bU$ and
chosen some set-theoretical conventions on sites, presheaves,
sheaves, and topos at the end of Introduction. We assume that 
profinite groups $G$, $G'$ and $G''$
appearing in this section always belong to $\bU$.

\begin{definition} \label{def:Gset}
Let $G$ be a profinite group and let $A$ be a ring belonging to $\bU$.\par
(1) By a {\it $G$-set} (resp.~{\it $G$-module}, resp.~{\it $A$-$G$-module}),
we mean a set (resp.~module, resp.~$A$-module) belonging to $\bU$
equipped with
a left action of $G$ continuous with respect to the discrete topology on 
the set (resp.~module, resp.~$A$-module). A {\it morphism of 
$G$-sets} (resp.~{\it $G$-modules}, resp.~{\it $A$-$G$-modules}) is a
$G$-equivariant map (resp.~homomorphism of modules, 
resp.~$A$-linear map). We write $G\Set$ (resp.~$G\Mod$, resp.~$A$-$G\Mod$)
for the category of $G$-sets (resp.~$G$-modules, resp.~$A$-$G$-modules).\par
\par
(2) By a {\it finite $G$-set}, we mean a $G$-set whose underlying set is finite
and belongs to $\bV$.
We define $G\fSet$ to be the category of  finite $G$-sets 
equipped with the topology induced by the pretopology 
defined by all finite surjective families of morphisms. 
For $S\in \Ob G\fSet$, we write $\Cov_{G\fSet}(S)$ for
the set consisting of all finite surjective families of morphisms
in $G\fSet$ with target $S$. Note that $G\fSet$ is $\bU$-small.
\end{definition}

\begin{remark}\label{rmk:LimitsGSets}
 Let $G$ and $A$ be as in Definition \ref{def:Gset}.\par
(1) Every $\bU$-small direct limit in $G\Set$ is representable; it is 
represented by the direct limit of the underlying sets equipped
with the action of $G$ by functoriality, which is continuous. 
Every $\bU$-small inverse limit in $G\Set$ is also representable; 
it is represented by the subset of the
inverse limit of the underlying sets with the action of $G$
by functoriality consisting of elements whose
stabilizers of the $G$-action are open. For a finite inverse limit,
the stabilizer of every element is open. Therefore a finite inverse limit
in $G\Set$ is compatible with that of the underlying sets.
This implies that a $G$-module is interpreted as a module object
of $G\Set$. Similarly we may regard $A$ equipped with the trivial $G$-action
as a ring object of $G\Set$, and then an $A$-$G$-module
is interpreted as an $A$-module object of $G\Set$.\par
(2) Every finite inverse limit in $G\fSet$ is representable; 
the finite inverse limit in $G\Set$ is represented by an object
of $G\fSet$ because 
only finite number of finite $G$-sets are involved in the construction.
\end{remark}

Let $G$ be a profinite group, and let 
$\CN(G)$ denote the set of all open normal subgroups of $G$.
For each $H\in \CN(G)$, we regard $G/H$ as an
object of $G\fSet$ by the left action of $G$, which factors through
$G/H$ and hence is continuous. Then the right action of $G/H$
on $G/H$ defines a right action of
$G/H$ on $G/H$ regarded as an object of $G\fSet$. Hence
for a sheaf of sets $\CF$ on $G\fSet$, the right action above
defines a left action of $G/H$ on $\CF(G/H)$. By taking the
direct limit over $H\in \CN(G)$ with respect to the map 
$\CF(G/H)\to \CF(G/H')$
induced by the projection map $G/H'\to G/H$ for $H$, $H'\in \CN(G)$
with $H'\subset H$, we obtain an object
$\varinjlim_{H\in \CN(G)}\CF(G/H)$ of $G\Set$, which will be 
denoted by $\CF(G)$. This construction is obviously functorial in $\CF$,
and therefore defines a functor 
\begin{equation}\label{eq:GfSetShfToGSet}
\rho_{G}^*\colon G\fSet^{\sim}\to G\Set,\;\;
\rho_{G}^*\CF=\CF(G)=\varinjlim_{H\in \CN(G)}\CF(G/H).\end{equation}

Conversely, given a $G$-set $T$, we obtain a presheaf $\CF$
on $G\fSet$ simply by restricting the presheaf on $G\Set$
represented by $T$, i.e., by $\CF(S)=\Hom_{G\Set}(S,T)$ $(S\in G\fSet)$.
We see that
$\CF$ is a sheaf as follows. Let $(S_{\lambda}\to S)_{\lambda\in \Lambda}$
be a finite surjective family of morphisms in $G\fSet$. 
It is straightforward to show that the sequence 
$$\Map(S,T)\to \prod_{\lambda\in\Lambda}\Map(S_{\lambda},T)
\rightrightarrows \prod_{(\lambda,\lambda')\in\Lambda^2}
\Map(S_{\lambda}\times_SS_{\lambda'},T)$$
is exact, where the right two maps are defined by the two projections
$S_{\lambda}\times_SS_{\lambda'}\to S_{\lambda}$, $S_{\lambda'}$. 
This remains exact after it is restricted to $\Hom_{G\Set}$ the sets
of $G$-equivariant maps since
a map $S\to T$ is $G$-equivariant if and only if its composition with the
surjective $G$-equivariant map $\sqcup_{\lambda\in\Lambda}S_{\lambda}\to S$ is 
$G$-equivariant. This construction is functorial in $T$, and we obtain
a functor in the opposite direction 
\begin{equation}\label{eq:GSetToGfSetShf}
\rho_{G*}\colon G\Set\to G\fSet^{\sim},\;\;
\rho_{G*}T=\Hom_{G\Set}(-,T).\end{equation}

We have the following well-known fact, which implies that 
the category $G\Set$ is a topos.

\begin{proposition} \label{prop:GsetGfsetSheaves}
The functor $\rho_{G*}$ is canonically regarded as a right adjoint of 
$\rho_G^*$, and the functors $\rho_{G*}$ and $\rho_G^*$
are equivalences of categories which are quasi-inverses of each other
by the adjunction.
\end{proposition}

\begin{proof}
Let $\CF$ be a sheaf of sets on $G\fSet$, and let $T$ be a $G$-set.
Given a morphism $\varphi\colon \CF\to \rho_{G*}T=\Hom_{G\Set}(-,T)$ in $G\fSet^{\sim}$,
we obtain a morphism $\psi\colon \rho_G^*\CF\to T$ in $G\Set$ by taking the 
direct limit over $H\in \CN(G)$ of the composition of $\varphi(G/H)$ with the $G/H$-equivariant
bijection $\ev_{1}\colon \Hom_{G\Set}(G/H,T)\to T^H; f\mapsto f(1)$.
Conversely, for a morphism
$\psi\colon \rho_G^*\CF\to T$ in $G\Set$,  one can construct a morphism 
$\varphi\colon \CF\to \rho_{G*}T=\Hom_{G\Set}(-,T)$ by sending 
$x\in \CF(S)$ to $\varphi(x)\colon s\mapsto \psi(\CF(\alpha_s)(x))\in \rho_{G*}T(S)$ 
for each $S\in \Ob (G\fSet)$, where $\alpha_s$ denotes a morphism 
$G/H\to S; \gamma H\mapsto \gamma s$ in $G\fSet$ for an $H\in \CN(G)$ stabilizing $s$;
we see that $\varphi(x)$ is $G$-equivariant as we may take $\alpha_{gs}$ to
be $\alpha_s(-\cdot g)$ for $g\in G$. It is straightforward to verify that the above constructions
give bijections between $\Hom_{G\fSet^{\sim}}(\CF,\rho_{G*}T)$
and $\Hom_{G\Set}(\rho_{G}^*\CF,T)$ which are inverses of each other.
They are obviously functorial in $\CF$ and $T$. It remains to show that 
the adjunction morphisms $\eta_{T}\colon \rho_G^*\rho_{G*}T\to T$ and
$\varepsilon_{\CF}\colon\CF\to\rho_{G*}\rho_G^*\CF$ are isomorphisms.
The former is given by the direct limit of 
the bijective maps $\ev_1\colon \Hom_{G\Set}(G/H,T)\xrightarrow{\cong}T^H$
over $H\in \CN(G)$. For $H\in \CN(G)$, the map $\varepsilon_{\CF}(G/H)$
is given by the inclusion map $\CF(G/H)\to\CF(G)^H=\Hom_{G\Set}(G/H,\CF(G))$,
which is bijective by Lemma \ref{lem:GfSetSheafDescent} below. Hence $\varepsilon_{\CF}$
is bijective as $G/H$ $(H\in \CN(G))$ form generators of the site $G\fSet$.
\end{proof}

\begin{remark}\label{rmk:GsetGfSetSheavesAdj}
By the proof of Proposition \ref{prop:GsetGfsetSheaves},
the adjunction morphism $\eta_T\colon \rho_G^*\rho_{G*}T\xrightarrow{\cong} T$
for $T\in \Ob G\Set$ is the direct limit of the bijections
$\Hom_{G\Set}(G/H,T)\xrightarrow{\cong}T^H;f\mapsto f(1)$ for $H\in \CN(G)$,
and the adjunction morphism $\varepsilon_{\CF}\colon 
\CF\xrightarrow{\cong}\rho_{G*}\rho_G^*\CF$ for $\CF\in \Ob G\fSet^{\sim}$
is explicitly given as follows. For $S\in \Ob G\fSet$ and
$x\in \CF(S)$, $\varepsilon_{\CF}(S)(x)
\in \rho_{G*}\rho_G^*\CF(S)=\Hom_{G\Set}(S,\varinjlim_{H\in\CN(G)}
\CF(G/H))$ is the map sending $s\in S$ to 
$\CF(\alpha_s)(x)$, where $\alpha_s$ is the morphism 
$G/H\to S;\gamma H\mapsto \gamma s$ in $G\fSet$ for an $H\in \CN(G)$
stabilizing $s$.
\end{remark}

\begin{lemma}\label{lem:GfSetSheafDescent}
For a sheaf of sets $\CF$ on $G\fSet$ and $H\in \CN(G)$,
the morphism $\CF(G/H)\to \CF(G)^{H}$ is bijective.
\end{lemma}

\begin{proof} It suffices to prove that, for $H$, $H'\in\CN(G)$ with $H'\subset H$,
the map $\CF(G/H)\to \CF(G/H')^{H/H'}$ induced by the projection $G/H'\to G/H$
is bijective. We note that this projection is a covering in $G\fSet$. 
A $G$-equivariant map
$G/H'\to G/H'\times_{G/H}G/H';g\mapsto (g,gh)$
for each $h\in H/H'$ induces a bijection 
$\sqcup_{h\in H/H'}G/H'\to G/H'\times_{G/H}G/H'$, whose
composition with the first (resp.~second) projection $p_0$ (resp.~
$p_1)$ to $G/H'$ is given by the identity map 
(resp.~the right multiplication by $h$) of $G/H'$ for each
$h\in H/H'$. This implies that the difference kernel of the
maps $\CF(p_i)\colon \CF(G/H')\to \CF(G/H'\times_{G/H}G/H')$ 
$(i=0,1)$ is the $H/H'$-invariant part, and completes the proof
since $\CF$ is a sheaf.
\end{proof}

Let $v\colon G'\to G$ be a continuous homomorphism of profinite groups.
Then we can define a functor 
\begin{equation}\label{eq:RestrGpActionFunctor}
\vf^*\colon G\fSet\longrightarrow G'\fSet\end{equation} 
by 
sending $S$ to $S$ equipped with the action of $G'$ via $v$. 
The functor $\vf^*$ obviously preserves finite surjective
families of morphisms and it also preserves finite inverse
limits by Remark \ref{rmk:LimitsGSets}. Hence $\vf^*$ defines a morphism of 
sites (\cite[IV 4.9]{SGA4}) and induces a morphism of topos
$\tv=(\tv^*,\tv_*)\colon G'\fSet^{\sim}\to G\fSet^{\sim}$. 
By composing $\tv_*$ with the equivalences in 
Proposition \ref{prop:GsetGfsetSheaves} for $G$ and $G'$, 
we obtain a functor 
\begin{equation}\label{eq:GSetFunct}
v_*\colon G'\Set\to G\Set; T\mapsto \varinjlim_{H\in \CN(G)}
\Map_{G'}(G/H,T)=\Map_{G',\cont}(G,T),\end{equation}
where $\Map_{G',\cont}$ denotes the set of continuous 
$G'$-equivariant maps. We define a functor
$v^*\colon G\Set\to G'\Set$ by sending $T$ to $T$
equipped with the action of $G'$ via $v$.
By Remark \ref{rmk:LimitsGSets} (1), the functor $v^*$ is left exact, i.e., 
preserves finite inverse limits.

\begin{proposition}\label{prop:GSetFunctoriality}
The functor $v^*$ is canonically regarded as a left adjoint
of $v_*$. Therefore the pair $\vS:=(v^*,v_*)$ defines a morphism
of topos $G'\Set\to G\Set$. The unit and counit
$\id_{G\Set}\to v_*v^*$ and $v^*v_*\to \id_{G'\Set}$ are
given by $T\to \Map_{G',\cont}(G,v^*T);x \mapsto (g\mapsto gx)$
and $v^*\Map_{G',\cont}(G,T')\to T'; \varphi\mapsto \varphi(1)$
for $T\in \Ob (G\Set)$ and $T'\in \Ob (G'\Set)$. 
\end{proposition}

\begin{proof}
Let $T$ be a $G$-set, and let $T'$ be a $G'$-set.
Then we have a bijection 
$\Hom_{G'\Set}(v^*T,T')\cong \Hom_{G\Set}(T,
\Map_{G',\cont}(G,T'))$ given by $\varphi\leftrightarrow\psi$,
$(\psi(x))(g)=\varphi(gx)$  and 
$\varphi(x)=(\psi(x))(1)$ $(x\in T, g\in G)$. It is straightforward to verify
that the maps in both directions are well-defined
and give the inverses of each other.
\end{proof}

For another continuous homomorphism $v'\colon G''\to G'$ of profinite groups,
we have $\vf^{\prime*}\circ \vf^*=(v\circ v')_{\text{f}}^*$,
which implies that we have canonical isomorphisms
$\tv\circ \tv^{\prime}\cong \widetilde{v\circ v'}$
and $\vS\circ \vS^{\prime}\cong (v\circ v')_{\mathbf{S}}$.

\begin{lemma}\label{lem:GsetPFComp}
For $T''\in G''\Set$, the canonical isomorphism 
$$\Map_{G',\cont}(G,\Map_{G'',\cont}(G',T''))=
v_*\circ v'_*(T'')\cong (v\circ v')_*(T'')=\Map_{G'',\cont}(G,T'')$$
is given by $\varphi\leftrightarrow \psi$, where 
$\psi(g)=(\varphi(g))(1)$ $(g\in G)$ and $(\varphi(g))(g')
=\psi(v(g')g)$ $(g\in G, g'\in G')$.
\end{lemma}

\begin{proof}
Let $\varepsilon'$ denote the unit isomorphism 
$\id_{G'\fSet^{\sim}}\xrightarrow{\cong} \rho_{G'*}\rho^*_{G'}$
(Proposition \ref{prop:GsetGfsetSheaves}). Then, by definition, the isomorphism 
$(v\circ v')_*\xrightarrow{\cong} v_*v'_*$ is
$\rho_G^*\tv_*\circ\varepsilon'\circ
\widetilde{v}'_*\rho_{G''*}\colon\rho_G^*\widetilde{v\circ v'}_*\rho_{G''*}
=\rho_G^*\tv_*\widetilde{v}'_*\rho_{G''*}
\xrightarrow{\cong}
\rho_G^*\widetilde{v}_*\rho_{G'*}\rho_{G'}^*\widetilde{v}'_*
\rho_{G''*}$.
By Remark \ref{rmk:GsetGfSetSheavesAdj}, the morphism 
$(\varepsilon'(\tv'_*\rho_{G''*}T''))(S')\colon
\Map_{G''\Set}(\vf^{\prime*}S',T'')=\tv'_*\rho_{G''*}T''(S')
\xrightarrow{\cong}
\rho_{G'*}\rho_{G'}^*\tv'_*\rho_{G''*}T''(S')
=\Map_{G'\Set}(S',(\tv'_*\rho_{G''*}T'')(G'))
=\Map_{G'\Set}(S',\Map_{G'',\cont}(v^{\prime*}G',T''))$
for $S'\in G'\fSet$
is given by $f\mapsto (s'\mapsto (g'\mapsto f(g's')))$.
Since $\rho_G^*\tv_*\CF'=(\tv_*\CF')(G)
=\varinjlim_{H\in \CN(G)}\CF'(\vf^*(G/H))$
for $\CF'\in G'\fSet^{\sim}$,
we obtain the description of the map
$(v\circ v')_*(T'')\xrightarrow{\cong}
v_*\circ v'_*(T'')$ in the claim by setting $S'=\vf^*(G/H)$
and taking the direct limit over $H\in \CN(G)$.
\end{proof}

\begin{definition}\label{def:Gsheaf}
Let $G$ be a profinite group, let $C$ be a site
whose topology is defined by a pretopology $\Cov_C(X)$ $(X\in \Ob C)$
consisting of finite families of morphisms, and let $\CA$ be a sheaf
of rings on $C$. \par
(1) We say that a left action of $G$ on a presheaf $\CP$ of sets 
(resp.~modules, resp.~$\CA$-modules) on $C$ is {\it continuous} if the action of $G$
on $\CP(X)$ with the discrete topology is continuous for every $X\in \Ob C$.\par
(2)
By a {\it $G$-sheaf} (resp.~{\it $G$-presheaf}) 
{\it of sets} (resp.~{\it modules}, resp.~{\it $\CA$-modules})
{\it $\CT$ on $C$}, we mean a sheaf (resp.~presheaf) of sets
(resp.~modules, resp.~$\CA$-modules) 
on $C$ equipped with a continuous left action of $G$.
A {\it morphism of $G$-sheaves} (resp.~{\it $G$-presheaves}) 
{\it of sets} (resp.~{\it modules}, 
resp.~{\it $\CA$-modules}) {\it on $C$} is a
morphism of sheaves (resp.~presheaves) of sets (resp.~modules, resp.~$\CA$-modules)
which is $G$-equivariant.
We write $G$-$C^{\sim}$ (resp.~$G\Mod(C)$,
resp.~$G\Mod(C,\CA)$) for the category of $G$-sheaves of sets
(resp.~modules, resp.~$\CA$-modules) on $C$.\par
(3) We define a site $C_G$ to be the product category
$C\times G\fSet$ with the topology generated by 
two sets of coverings $\Cov_{C_G}^h(X,S)$
and $\Cov_{C_G}^v(X,S)$ of each $(X,S)\in \Ob C_G$ defined as follows.
\begin{align*}
\Cov_{C_G}^h(X,S)&=\{((f_{\lambda},\id_S)\colon (X_{\lambda},S)
\to (X,S))_{\lambda\in \Lambda}\;\vert\; (f_{\lambda}\colon X_{\lambda}\to X)_{\lambda\in\Lambda}
\in \Cov_C(X)\}\\
\Cov_{C_G}^{v}(X,S)&=\{((\id_X,\alpha_{\lambda})\colon
(X,S_{\lambda})\to (X,S))_{\lambda\in\Lambda}\;\vert\; 
(\alpha_{\lambda}\colon S_{\lambda}\to S)_{\lambda\in\Lambda}
\in \Cov_{G\fSet}(S)\}
\end{align*}
We call $C_G$ the {\it site of finite $G$-sets above $C$}. 
\end{definition}

\begin{remark}\label{rmk:FinInvLimGSheaf}
Let $G$, $C$ and $\CA$ be as in Definition \ref{def:Gsheaf}. \par
(1) By Remark \ref{rmk:LimitsGSets} (1), we see that a finite inverse limit
in $G\hy C^{\sim}$ is given by the finite inverse limit of the underlying
sheaves of sets equipped with the action of $G$ defined by 
functoriality. This implies that a $G$-sheaf of modules 
(resp.~$\CA$-modules) on $C$ is interpreted as a module object
(resp.~an $\CA$-module object) of $G$-$C^{\sim}$, where
$\CA$ is equipped with the trivial $G$-action.\par
(2) For a $G$-presheaf of sets on $C$, the induced action of 
$G$ on the associated sheaf of sets is continuous.
By the construction of the sheaf associated to a presheaf
in \cite[II.3]{SGA4}, this is a consequence of 
the claim in Remark \ref{rmk:LimitsGSets} (1) on direct limits of $G$-sets 
and the following fact: For a $G$-presheaf $\CP$ of sets on $C$
and the sieve $R$ generated by a covering $(X_{\lambda}\to X)_{\lambda\in \Lambda}
\in \Cov_C(X)$ of an object $X$ of $C$, 
the action of $G$ on 
$\CP(R)\cong \Ker(\prod_{\lambda\in\Lambda}\CP(X_{\lambda})
\rightrightarrows \prod_{(\lambda,\lambda')\in \Lambda^2}
\CP(X_{\lambda}\times_XX_{\lambda'}))$
is continuous since the index set $\Lambda$ is finite.
\end{remark}

\begin{proposition}\label{prop:CGfSetSheaf}
Let $\CF$ be a presheaf of sets on $C_G$. Then $\CF$
is a sheaf of sets on $C_G$ if and only if $\CF$ satisfies 
the following two conditions. \par
(a) The sequence below is exact for every $(X,S)\in \Ob C_G$ and 
every $(X_{\lambda}\to X)_{\lambda\in\Lambda}\in \Cov_{C}(X)$,
i.e., $\CF(-,S)$ is a sheaf of sets on $C$ for every $S\in \Ob G\fSet$.
$$\CF(X,S)\to \prod_{\lambda\in\Lambda}\CF(X_{\lambda},S)
\rightrightarrows \prod_{(\lambda,\lambda')\in\Lambda^2}
\CF(X_{\lambda}\times_XX_{\lambda'},S)$$

(b) The sequence below is exact for every $(X,S)\in \Ob C_G$
and $(S_{\lambda}\to S)_{\lambda\in \Lambda}
\in \Cov_{G\fSet}(S)$, i.e., $\CF(X,-)$ is a sheaf of sets on $G\fSet$
for every $X\in \Ob C$.
$$\CF(X,S)\to \prod_{\lambda\in\Lambda}\CF(X,S_{\lambda})
\rightrightarrows \prod_{(\lambda,\lambda')\in\Lambda^2}
\CF(X,S_{\lambda}\times_SS_{\lambda'})$$
\end{proposition}

\begin{proof}
Since $\Cov^h_{C_G}(X,S)$ and $\Cov^v_{C_G}(X,S)$ are stable under
base changes, the claim follows from \cite[II Corollaire 2.3]{SGA4}.
\end{proof}

For a sheaf of sets $\CF$ on $C_G$, 
the presheaf of sets $\CF(X,-)$ on 
$G\fSet$ is a sheaf by Proposition \ref{prop:CGfSetSheaf}.
Hence, by associating $X$ to the $G$-set
$\rho_{G}^*(\CF(X,-))=\varinjlim_{H\in \CN(G)}\CF(X,G/H)$
\eqref{eq:GfSetShfToGSet},
we obtain a $G$-presheaf of sets on $C$, which is denoted by 
$\rho_{G,C}^*\CF$ in the following. We
see that $\rho_{G,C}^*\CF$ is a $G$-sheaf by using the assumption that
$C$ is generated by a pretopology consisting of finite families of morphisms.
Thus we obtain a functor 
\begin{equation}\label{eq:CGfSetShfToGShf}
\rho_{G,C}^*\colon C_G^{\sim}\longrightarrow G\hy C^{\sim}
\end{equation}
Conversely we have the following construction in the
opposite direction.
\begin{lemma}\label{lem:GShfToCGfSetShf}
Let $\CT$ be a $G$-sheaf of sets on $C$.
Then the presheaf  of sets $\CF$ on $C_G$ defined by 
$\CF(X,S)=(\rho_{G*}(\CT(X))(S)=\Hom_{G\Set}(S,\CT(X))$ 
\eqref{eq:GSetToGfSetShf} is a sheaf of sets on $C_G$.
\end{lemma}
\begin{proof}
For each $X\in \Ob C$, $\CF(X,-)=\rho_{G*}(\CT(X))$ is a sheaf of sets
on $G\fSet$.
For a covering $(X_{\lambda}\to X)_{\lambda\in \Lambda}\in \Cov_C(X)$ of 
$X\in \Ob C$, the sequence
$\CT(X)\to \prod_{\lambda\in\Lambda}\CT(X_{\lambda})\rightrightarrows
\prod_{(\lambda,\lambda')\in \Lambda^2}\CT(X_{\lambda}\times_XX_{\lambda'})$
is exact since $\CT$ is a $G$-sheaf, whence the
sequence
$\Hom_{G\Set}(S,\CT(X))\to
\prod_{\lambda\in\Lambda}\Hom_{G\Set}(S,\CT(X_{\lambda}))
\rightrightarrows
\prod_{(\lambda,\lambda')\in\Lambda^2}\Hom_{G\Set}(S,\CT(X_{\lambda}\times_XX_{\lambda'}))$
is exact. This completes the proof by Proposition \ref{prop:CGfSetSheaf}. 
\end{proof}
The construction in Lemma \ref{lem:GShfToCGfSetShf} is functorial
in $\CT$, and defines a functor 
\begin{equation}\label{eq:GShfToCGfSetShf}
\rho_{G,C*}\colon G\hy C^{\sim}\longrightarrow C_G^{\sim};\; (\rho_{G,C*}\CT)(X,-)=\rho_{G*}(\CT(X)).
\end{equation}

\begin{proposition}\label{prop:GsheafTopos}
The functor $\rho_{G,C*}$ is canonically regarded as a right adjoint
of $\rho_{G,C}^*$, and the functors $\rho_{G,C*}$ and $\rho_{G,C}^*$
are equivalences of categories which are quasi-inverses of each other
by the adjunction.
\end{proposition}
\begin{proof}
For a sheaf of sets $\CF$ on $C_G$, the adjunction
isomorphism $\CF(X,-)\xrightarrow{\cong} \rho_{G*}\rho_G^*(\CF(X,-))
=(\rho_{G,C*}\rho_{G,C}^*\CF)(X,-)$ (Proposition \ref{prop:GsetGfsetSheaves}) for each $X\in \Ob C$
defines an isomorphism $\varepsilon_{\CF}\colon\CF\xrightarrow{\cong}
\rho_{G,C*}\rho_{G,C}^*\CF$, which is functorial in $\CF$. 
Similarly, for a $G$-sheaf of sets $\CT$ on $C$,
the adjunction isomorphism 
$(\rho_{G,C}^*\rho_{G,C*}\CT)(X)
=\rho_G^*\rho_{G*}(\CT(X))\xrightarrow{\cong}\CT(X)$
(Proposition \ref{prop:GsetGfsetSheaves}) for each  $X\in \Ob C$ gives an isomorphism 
$\eta_{\CT}\colon \rho_{G,C}^*\rho_{G,C*}\CT\xrightarrow{\cong}\CT$,
which is functorial in $\CT$. We see that the compositions
$\rho_{G,C}^*\CF\xrightarrow{\rho_{G,C}^*\varepsilon_{\CF}}
\rho_{G,C}^*\rho_{G,C*}\rho_{G,C}^*\CF
\xrightarrow{\eta_{\rho_{G,C}^*\CF}}\rho_{G,C}^*\CF$
and $\rho_{G,C*}\CT\xrightarrow{\varepsilon_{\rho_{G,C*}\CT}}
\rho_{G,C*}\rho_{G,C}^*\rho_{G,C*}\CT
\xrightarrow{\rho_{G,C*}\eta_{\CT}}\rho_{G,C*}\CT$
are the identity morphisms by evaluating them on 
each $X\in \Ob C$. This completes the proof.
\end{proof}

By Proposition \ref{prop:GsheafTopos}, 
we see that $G$-$C^{\sim}$ is a topos. 
By Remark \ref{rmk:FinInvLimGSheaf} (1), we see that 
$G\Mod(C)$ and $G\Mod(C,\CA)$ are 
abelian categories with enough injectives
(\cite[II Proposition 6.7, Remarque 6.9]{SGA4}).

\begin{proposition}\label{prop:GSheafToposEx}
A sequence  $\CL\to \CM\to \CN$ in 
$G\Mod(C)$ (resp.~$G\Mod(C,\CA)$) is exact 
if and only it is exact as a sheaf of modules
(resp.~$\CA$-modules) on $C$.
\end{proposition}

\begin{proof}
It suffices to show that kernels and cokernels are preserved  under the functor taking the underlying 
sheaves of modules (resp.~$\CA$-modules). Let $f\colon \CM\to \CN$ be
a morphism in $G\Mod(C)$ (resp.~$G\Mod(C,\CA)$). The kernel $\CL$ of 
the morphism of sheaves of modules (resp.~$\CA$-modules) 
underlying $f$ is stable under the action of $G$ on $\CM$, and $\CL$ equipped with
the induced $G$-action, which is continuous, gives the kernel of 
the morphism $f$. Let $\mathcal C_{\text{pre}}$ be the  cokernel of $f$ regarded
as a morphism of presheaves of modules (resp.~$\CA$-modules).
Then the action of $G$ on $\mathcal C_{\text{pre}}$ induced by that on $\CN$ is continuous,
and it induces a continuous $G$-action on the sheaf of sets $\mathcal C$
associated to $\mathcal C_{\text{pres}}$ by Remark \ref{rmk:FinInvLimGSheaf} (2).
We see that the $G$-equivariant morphism $\CN\to \mathcal C$ gives the cokernel of $f$.
This completes the proof.
\end{proof}

Let $G'$ be a profinite group, let $C'$ be a site whose topology is defined by 
a pretopology  $\Cov_{C'}(X')$ $(X'\in \Ob C')$
consisting of finite families of morphisms,
and suppose that we are given a
continuous homomorphism $v\colon G'\to G$ and a functor
$u\colon C\to C'$ defining a morphism of sites.
Let $u$ also denote the morphism of topos
$(u^*, u_*)\colon C^{\prime\sim}\to C^{\sim}$ induced by $u$.

By Proposition \ref{prop:CGfSetSheaf}, the product functor 
$u\times \vf^*\colon C_G\to C'_{G'}$ is continuous, and 
therefore defines an adjoint pair of functors
$(\tv_u^*,\tv_{u*})\colon C_{G'}^{\prime\sim}\to C_G^{\sim}$. 
By composing $\tv_{u*}$ with the equivalences in 
Proposition \ref{prop:GsheafTopos} for $(C,G)$ and $(C',G')$, we obtain a functor 
\begin{equation}\label{eq:GSheafFunct}
v_{u*}\colon G'\hy C^{\prime\sim}\to G\hy C^{\sim};
v_{u*}\CT'=\CMap_{G',\cont}(G,u_*\CT')
\end{equation}
similarly to \eqref{eq:GSetFunct}. 
Here, for a $G'$-sheaf of sets $\CT$ on $C$, we write $\CMap_{G',\cont}(G,\CT)$
for the $G$-sheaf of sets $X\mapsto \Map_{G',\cont}(G,\CT(X))=v_*(\CT(X))$
on $C$. 
\begin{proposition}\label{prop:GSheafPB2}
(1) For a $G$-sheaf of sets $\CT$ on $C$, the action of $G$ on 
$u^*\CT$ is continuous.\par
(2) Let $v_u^*$ be the functor $G\hy C^{\sim}\to G'\hy C^{\prime\sim}$
defined by sending $\CT$ to $u^*\CT$ equipped with the action
of $G'$ via $v$, which is continuous by (1). 
Then the functor $v_u^*$ is canonically regarded as a left adjoint
of $v_{u*}$ \eqref{eq:GSheafFunct}. The unit and counit
$\id_{G\hy C^{\sim}}\to v_{u*}v_u^*$ and 
$v_u^*v_{u*}\to \id_{G'\hy C^{\prime \sim}}$ are given by 
$\CT\to \CMap_{G',\cont}(G,u_*u^*\CT); x\mapsto 
(g\mapsto \eta_{\CT}(gx))$ and
the morphism $u^*\CMap_{G',\cont}(G,u_*\CT')\to \CT'$
corresponding to $\CMap_{G',\cont}(G,u_*\CT')\to u_*\CT';
\varphi\mapsto \varphi(1)$ by the adjunction of $(u^*,u_*)$
for $\CT\in \Ob (G\hy C^{\sim})$ and $\CT'\in 
\Ob (G'\hy C^{\prime\sim})$,
where $\eta_{\CT}$ denotes the adjunction morphism
$\CT\to u_*u^*\CT$. 
\end{proposition}

\begin{proof}
(1) This follows from Remark \ref{rmk:LimitsGSets} (1) and  Remark \ref{rmk:FinInvLimGSheaf} (2) since 
$u^*\CT$ is the sheaf associated to the presheaf inverse image of $\CT$.\par
(2) The adjunction is given by the following bijections for $\CT\in \Ob (G\hy C^{\sim})$ and 
$\CT'\in \Ob (G'\hy C^{\prime\sim})$.
$$\Hom_{G'}(u^*\CT,\CT')
\cong \Hom_{G'}(\CT,u_*\CT')\cong \Hom_{G}(\CT, \CMap_{G',\cont}(G,u_*\CT')),$$
where the second map is given by $\varphi\leftrightarrow \psi$, $(\psi(x))(g)=\varphi(gx)$,
$\varphi(x)=(\psi(x))(1)$.
\end{proof}

Since $u^*$ preserves finite inverse limits, the functor $v_u^*$ 
defined in Proposition \ref{prop:GSheafPB2} (2) also preserves finite inverse limits 
by Remark \ref{rmk:FinInvLimGSheaf} (1). Hence the pairs
$(v_u^*,v_{u*})$ and $(\tv_{u}^*,\tv_{u*})$ define 
morphisms of topos 
\begin{equation}\label{eq:GSheafToposMorph}
v_u\colon G'\hy C^{\prime\sim}\to G\hy C^{\sim},\qquad
\tv_u\colon C^{\prime\sim}_{G'}\to C^{\sim}_{G},\end{equation}
respectively. 
When $C=C'$, $\Cov_{C}(X)=\Cov_{C'}(X)$ $(X\in \Ob C)$
and $u=\id_C$, then we write $v_C$ and $\tv_C$ for $v_u$
and $\tv_u$, respectively.

\begin{lemma}\label{lem:GfSetShfToGShfFunct}
The isomorphism of functors
$$\sigma\colon \rho_{G,C}^*\tv_{u*}
\xrightarrow{\cong}\rho_{G,C}^*\tv_{u*}\rho_{G',C'*}\rho_{G',C'}^*
= v_{u*}\rho_{G',C'}^*\colon C_{G'}^{\prime\sim}
\to G\hy C^{\sim}$$
is explicitly described as follows. For $\CF'\in \Ob(C^{\prime\sim}_{G'})$,
$X\in \Ob C$, and $N\in \CN(G)$, the image of $x\in \CF'(u(X),\vf^*(G/N))$ under
$$\sigma(\CF')(X)\colon 
\varinjlim_{H\in \CN(G)}\CF'(u(X),\vf^*(G/H))
\to\Map_{G',\cont}(G,\varinjlim_{H'\in \CN(G')}\CF'(u(X),G'/H'))$$
is the map sending $g\in G$ to $\CF'(\id_{u(X)}, \alpha_g)(x)$,
where $\alpha_g$ denotes the morphism
$G'/N'\to \vf^*(G/N)$ in $G'\fSet$ sending $1$ to $g$
for an $N'\in \CN(G')$ stabilizing $gN\in \vf^*(G/N)$.
\end{lemma}

\begin{proof}
This follows from Remark \ref{rmk:GsetGfSetSheavesAdj} and the proof of 
Proposition \ref{prop:GsheafTopos}.
\end{proof}

Suppose that we are given another pair of $v'\colon G''\to G'$ and
$u'\colon C'\to C''$ satisfying the same conditions as $v$ and $u$.
Then we have 
$(u'\times \vf^{\prime*})\circ (u\times \vf^*)=
(u'\circ u)\times (v\circ v')_{\text{f}}^*$, which implies that we have
canonical isomorphisms 
\begin{equation}\label{eq:GSheavesFunctCocyc}
\tv_u\circ \tv'_{u'}\cong
\widetilde{(v\circ v')}_{u'\circ u},\qquad
v_u\circ v'_{u'}\cong (v\circ v')_{u'\circ u}.
\end{equation}

\begin{lemma}\label{lem:GSheafPFComp}
For a $G''$-sheaf $\CT''$ of sets on $C''$ and $X\in \Ob C$, the canonical isomorphism 
\begin{multline*}
\Map_{G',\cont}(G,\Map_{G'',\cont}(G',\CT''(u'(u(X))))=
v_{u*}\circ v'_{u'*}(\CT'')(X)\\
\cong (v\circ v')_{u'\circ u*}(\CT'')(X)=\Map_{G'',\cont}(G,
\CT''(u'\circ u(X)))
\end{multline*}
 is given by the same formula as Lemma \ref{lem:GsetPFComp}.
\end{lemma}

\begin{proof}
The isomorphism $(v\circ v')_{u'\circ u*}
\xrightarrow{\cong}v_{u*}\circ v'_{u'*}$ is given by 
$\sigma\circ \tv'_{u'*}\rho_{G'',C''*}\colon
\rho_{G,C}^*\tv_{u*}\tv'_{u'*}\rho_{G'',C''*}
\xrightarrow{\cong}\rho_{G,C}^*\tv_{u*}\rho_{G',C'*}
\rho_{G',C'}^*\tv'_{u'*}\rho_{G'',C''*}$
for the isomorphism $\sigma$ in Lemma \ref{lem:GfSetShfToGShfFunct}.
Therefore we obtain the claim by applying Lemma \ref{lem:GfSetShfToGShfFunct}
to $\CF'=\tv'_{u'*}\rho_{G'',C''*}\CT''$, for which we have
$\CF'(u(X),S')=(\rho_{G'',C''*}\CT'')(u'\circ u(X),\vf^{\prime*}S')
=\Map_{G''}(\vf^{\prime*}S',\CT''(u'\circ u(X)))$ $(S'\in \Ob G'\fSet)$.
\end{proof}

Let $G$ be a profinite group, and let 
$C$  be a site whose topology is generated by a pretopology
$\Cov_C(X)$ $(X\in \Ob C)$ consisting of finite families of morphisms.
Then the homomorphisms $\iota_G\colon \{1\}\to G$
and $\pi_G\colon G\to \{1\}$ induce morphisms of topos
$\iota_{G,C}\colon C^{\sim}\to G\hy C^{\sim}$
and $\pi_{G,C}\colon G\hy C^{\sim}\to C^{\sim}$
such that $\pi_{G,C}\circ \iota_{G,C}\cong\id_{C^{\sim}}$. 
By \eqref{eq:GSheafFunct} and 
Proposition \ref{prop:GSheafPB2} (2), the direct images and the inverse images
under these morphisms are explicitly given as follows, where
$\CT\in \Ob (C^{\sim})$ and $\CT'\in \Ob (G\hy C^{\sim})$: 
$\iota_{G,C*}(\CT)=\Map_{\cont}(G,\CT)$, 
$\iota_{G,C}^*(\CT')=$the sheaf of sets underlying $\CT'$, 
$\pi_{G,C*}(\CT')=\CT^{\prime G}$, and $\pi_{G,C}^*(\CT)=$
the sheaf $\CT$ with the trivial action of $G$. 
Let $\CA$ be a sheaf of rings on $C$, and we write
$\CA$ for its inverse image by $\pi_{G,C}$. 

\begin{proposition}\label{pro:TrivIndFunctEx}
The functor $\iota_{G,C*}\colon \Mod(C,\CA)\to G\Mod(C,\CA)$
is exact.
\end{proposition}

\begin{proof}
Let $\varphi\colon \CM\to \CM'$ be an epimorphism of $\CA$-modules on $C$.
Let $X\in \Ob C$ and $f\in \Map_{\cont}(G,\CM'(X))$. Then
$f$ factors through a map $f_H\colon G/H\to \CM'(X)$ for some $H\in \CN(G)$. 
Since $G/H$ is finite, there exists a covering 
$(\alpha_{\lambda}\colon X_{\lambda}\to X)_{\lambda\in \Lambda}
\in \Cov_C(X)$ such that, for each $\lambda\in\Lambda$, 
the image of the composition $G/H\xrightarrow{f_H}\CM'(X)
\xrightarrow{\CM'(\alpha_{\lambda})}\CM'(X_{\lambda})$
lies in the image of $\varphi(X_{\lambda})\colon\CM(X_{\lambda})\to \CM'(X_{\lambda})$, 
whence there exists a
map $\tf_H\colon G/H\to \CM(X_{\lambda})$ satisfying 
$\varphi(X_{\lambda})\circ \tf_H=\CM'(\alpha_{\lambda})\circ f_H$.
\end{proof}

\begin{corollary}\label{cor:IndGModGacyc}
For a sheaf of $\CA$-modules $\CM$ on $C$, the morphism
$$\CM\cong \pi_{G,C*}\iota_{G,C*}\CM\to
R\pi_{G,C*}(\iota_{G,C*}\CM)$$ 
is an isomorphism in $D^+(C,\CA)$. 
\end{corollary}

Thanks to this corollary, we see that the right derived functor
of $\pi_{G,C*}\colon G\Mod(C,\CA)\to \Mod(C,\CA)$ can 
be computed by the complex obtained by taking the inhomogeneous
cochain complex of the section on each $X\in \Ob C$ as follows.

For a sheaf of $\CA$-modules $\CM$ on $C$ and $n\in \N$, 
we define the presheaf $\Map_{\cont}(G^n,\CM)$ of $\CA$-modules on $C$
by $X\mapsto \Map_{\cont}(G^n,\CM(X))$ $(X\in \Ob C)$.
\begin{lemma}\label{lem:GMapSheaf}
The presheaf  $\Map_{\cont}(G^n,\CM)$ is a sheaf.\par
\end{lemma}

\begin{proof}
This follows from the following observation. For $X\in \Ob C$
and $(X_{\lambda}\to X)_{\lambda\in \Lambda}\in\Cov_C(X)$, 
the map $G^n\to \CM(X)$ is continuous  if its composition
with $\CM(X)\to \CM(X_{\lambda})$ is continuous for every $\lambda\in\Lambda$
because the set $\Lambda$ is finite.
\end{proof}

Let $\CM$ be a $G$-sheaf of $\CA$-modules on $C$ (Definition \ref{def:Gsheaf} (2)). For $n\in \N$,
we define a $G$-sheaf of $\CA$-modules $K^n(G,\CM)$ on $C$ to be 
the sheaf of $\CA$-modules $\Map_{\cont}(G^{n+1},\CM)$ on 
$C$ equipped with the action of $G$ defined by 
$(g\cdot f)(g_0,\ldots, g_n)=g\cdot f(g^{-1}g_0,\ldots, g^{-1}g_n)$
for $X\in \Ob C$, $f\in\Map_{\cont}(G^{n+1},\CM(X))$, and $g, g_0,\ldots, g_n\in G$;
the action is continuous since every $f$ factors through $(G/H)^{n+1}$
and $\CM(X)^H$ for some $H\in \CN(G)$. We define homomorphisms of $G$-sheaves
of $\CA$-modules $d^n\colon K^n(G,\CM)\to K^{n+1}(G,\CM)$
$(n\in \N)$ and $\varepsilon\colon \CM\to K^0(G,\CM)$ by 
$(d^nf)(g_0,\ldots, g_{n+1})=\sum_{i=0}^{n+1}(-1)^if(g_0,\ldots,\check{g}_i,\ldots, g_{n+1})$
and $\varepsilon(x)(g_0)=x$ for $X\in \Ob C$, 
$f\in K^n(G,\CM)(X)$, $x\in \CM(X)$, and $g_0,\ldots, g_{n+1}\in G$. 
It is straightforward to see $d^{n+1}\circ d^n=0$ $(n\in \N)$ and
$d^0\circ \varepsilon=0$. 

\begin{lemma}\label{lem:GModResol}
The complex $\CM\xrightarrow{\varepsilon} K^{\bullet}(G,\CM)$
is homotopy equivalent to zero as a complex of sheaves of $\CA$-modules
on $C$.
\end{lemma}
 
 \begin{proof} Put $K^{-1}(G,\CM)=\CM=\Map_{\cont}(G^0,\CM)$, where $G^0$ denotes the
 trivial group. Then the $\CA$-linear morphisms $K^n(G,\CM)\to K^{n-1}(G,\CM)$
 $(n\in \N)$  induced by the continuous maps $G^n\to G^{n+1};(g_0,\ldots, g_{n-1})
 \mapsto (1,g_0,\ldots, g_{n-1})$ give the desired homotopy.
\end{proof}

\begin{lemma}\label{lem:GModResolnd}
The morphism
$\rho\colon \iota_{G,C}^*K^n(G,\CM)\to \Map_{\cont}(G^n,\CM)$
in $\Mod(C,\CA)$ defined by $\rho(f)(g_1,\ldots, g_n)=f(1,g_1,\ldots, g_n)$
for $X\in \Ob  C$, $f\in K^n(G,\CM)(X)$, and $g_1,\ldots, g_n\in G$
induces an isomorphism 
$\tau\colon K^n(G,\CM)\xrightarrow{\cong}\iota_{G,C*}\Map_{\cont}(G^n,\CM)$
in $G\Mod(C,\CA)$.
\end{lemma}

\begin{proof} Let $X$ be an object of $C$. By the proof of Proposition \ref{prop:GSheafPB2} (2),
the $\CA(X)$-linear map 
$\tau(X)\colon \Map_{\cont}(G^{n+1},\CM(X))\to \Map_{\cont}(G,\Map_{\cont}(G^n,\CM(X)))$ is given by 
$$\{(\tau(X)(f))(g)\}(g_1,\ldots, g_n)
=(g\cdot f)(1,g_1,\ldots, g_n)=g\cdot f(g^{-1},g^{-1}g_1,\ldots,
g^{-1}g_n)$$
for $f\in \Map_{\cont}(G^{n+1},\CM(X))$ and $g,g_1,\ldots,g_n\in G$. 
We see that the inverse of $\tau(X)$ is given by sending 
$F$ to $f$ defined by $f(g_0,g_1,\ldots,g_n)=g_0\cdot F(g_0^{-1})(g_0^{-1}g_1,\ldots,g_0^{-1}g_n)$.
\end{proof}

We define the inhomogeneous cochain complex $C^{\bullet}(G,\CM)$
of $\CM$ by associating to each $X$ the inhomogeneous cochain complex
$C^{\bullet}(G,\CM(X))$  of the $\CA(X)$-$G$-module $\CM(X)$. 
By Lemma \ref{lem:GMapSheaf}, this is a complex of sheaves of $\CA$-modules.
We see that the restriction under the continuous map
$G^n\to G^{n+1}$ sending $(g_1,\ldots, g_n)$ to $(1,g_1,g_1g_2,\ldots, 
g_1g_2\cdots g_n)$ induces an isomorphism of complexes
$\pi_{G,C*}K^{\bullet}(G,\CM)\xrightarrow{\cong}C^{\bullet}(G,\CM)$.
By Lemma \ref{lem:GModResol}, Lemma \ref{lem:GModResolnd}, and 
Corollary \ref{cor:IndGModGacyc}, we obtain  the following isomorphisms in 
$D^+(C,\CA)$. 
\begin{equation}
R\pi_{G,C*}\CM\xrightarrow{\;\cong\;} R\pi_{G,C*}K^{\bullet}(G,\CM)
\xleftarrow{\;\cong\;} \pi_{G,C*}K^{\bullet}(G,\CM)\xrightarrow{\;\cong\;} C^{\bullet}(G,\CM)
\end{equation}

When $G=\Map(\Lambda,\Z_p)$ for some finite set $\Lambda$ and
$\gamma_i$ $(i\in \Lambda)$ denotes its element sending $j\in\Lambda$ to $1$ if $j=i$ and 
to $0$ otherwise, we show that the right derived functor of $\pi_{G,C*}$ can be computed by the 
Koszul complex with respect to $\gamma_i-1$ $(i\in \Lambda)$ similarly to 
the usual group cohomology, and study the functoriality of the description with respect to $C$ and $\Lambda$. 

To start with, we introduce a Koszul complex on a ringed site and discuss its
functoriality with respect to $\Lambda$ as above.

Let $(C,\CA)$ be a ringed site, and let $\Lambda$ be a finite set.
Let  $\Gamma_{\Lambda}^{\disc}$ be the group $\Map(\Lambda,\Z)$,
and let $\gamma_i$ $(i\in \Lambda)$ be the map $\Lambda\to \Z$
sending $j$ to $1$ if $j=i$ and to $0$ otherwise. 
Let $\CM$ be an $\CA$-module endowed with an $\CA$-linear
action of $\Gamma_{\Lambda}^{\disc}$. 
We define $K^{\bullet}_{\Lambda}(\CM)$ to be the Koszul complex of $\CM$ with
respect to the actions of $\gamma_i-1$ $(i\in \Lambda)$ on $\CM$
commuting with each other. We have $K^r_{\Lambda}(\CM)=\CM\otimes_{\Z}
\wedge^r\Z^{(\Lambda)}$ and $d^r(m\otimes e_{\bmI})=
\sum_{i\in \Lambda}(\gamma_i-1)(m)\otimes e_i\wedge
e_{\bmI}$ for $r\in \N$, $X\in \Ob C$, $m\in \CM(X)$, and $\bmI\in \Lambda^r$, 
where $e_i$ $(i\in \Lambda)$ denotes the standard basis 
of $\Z^{(\Lambda)}=\oplus_{i\in \Lambda}\Z$, and
$e_{\bmI}=e_{i_1}\wedge\cdots\wedge e_{i_r}$ for $\bmI=(i_1,\ldots, i_r)\in \Lambda^r$.
For a $\Gamma^{\disc}_{\Lambda}$-equivariant morphism  
$f\colon \CM\to \CM'$ of $\CA$-modules on $C$ with $\CA$-linear $\Gamma^{\disc}_{\Lambda}$-action,
we write $K^{\bullet}_{\Lambda}(f)$ for the morphism of complexes 
$K^{\bullet}_{\Lambda}(\CM)\to K^{\bullet}_{\Lambda}(\CM')$ of 
$\CA$-modules on $C$ induced by $f$ and the identity maps of 
$\wedge^r\Z^{(\Lambda)}$ $(r\in \N)$. 
Let $\CM_{\nu}$ $(\nu=1,2)$ be $\CA$-modules with $\CA$-linear
$\Gamma^{\disc}_{\Lambda}$-action, and let $\CM_1\otimes_{\CA}\CM_2$
be equipped with the diagonal $\CA$-linear action of $\Gamma^{\disc}_{\Lambda}$:
$\gamma(m_1\otimes m_2)=\gamma(m_1)\otimes\gamma(m_2)$
($X\in \Ob C$, $m_{\nu}\in \CM_{\nu}(X)$ $(\nu=1,2)$, $\gamma\in \Gamma^{\disc}_{\Lambda}$).
Then one can define a morphism of complexes of $\CA$-modules
\begin{equation}\label{eq:KoszulCpxProd}
K_{\Lambda}^{\bullet}(\CM_1)\otimes_{\CA}K_{\Lambda}^{\bullet}(\CM_2)
\longrightarrow K_{\Lambda}^{\bullet}(\CM_1\otimes_{\CA}\CM_2)
\end{equation}
by sending $(m_1\otimes e_{\bmI_1})\otimes(m_2\otimes e_{\bmI_2})$
to $(m_1\otimes\gamma_{\bmI_1}(m_2))\otimes e_{\bmI_1}\wedge e_{\bmI_2}$
for $X\in \Ob C$, $r_{\nu}\in \N$, $m_{\nu}\in \CM_{\nu}(X)$, and $\bmI_{\nu}\in \Lambda^{r_{\nu}}$
$(\nu=1,2)$. Here $\gamma_{\bmI}$ denotes
$\prod_{n=1}^r\gamma_{i_n}\in\Gamma^{\disc}_{\Lambda}$ for
$r\in \N$ and $\bmI=(i_1,\ldots, i_r)\in \Lambda^r$.
We write $z_1\wedge_{\Gamma^{\disc}_{\Lambda}}z_2$
for the image of $z_1\otimes z_2$ under \eqref{eq:KoszulCpxProd}. Then, for another $\CA$-module
$\CM_3$ with $\CA$-linear $\Gamma_{\Lambda}^{\disc}$-action,
we have $(z_1\wedge_{\Gamma^{\disc}_{\Lambda}}z_2)\wedge_{\Gamma^{\disc}_{\Lambda}}z_3
=z_1\wedge_{\Gamma^{\disc}_{\Lambda}}(z_2\wedge_{\Gamma^{\disc}_{\Lambda}}z_3)$.
The product morphism \eqref{eq:KoszulCpxProd} is obviously functorial in $\CM_1$ and $\CM_2$.

Let $\psi\colon \Lambda\to \Lambda'$ be a map of finite sets, and 
let $\Gamma^{\disc}_{\psi}$ denote the  homomorphism 
$\Gamma^{\disc}_{\Lambda'}\to \Gamma^{\disc}_{\Lambda}$ defined by 
the composition with $\psi$. We have $\Gamma^{\disc}_{\psi}(\gamma_{i'})
=\prod_{i\in \psi^{-1}(i')}\gamma_i$ for $i'\in \Lambda'$.
We assume that we are given a total 
order on $\Lambda$,  and let $\Lambda_{\psi, i}^{<}$ for $i\in \Lambda$ denote the 
subset of $\Lambda$ consisting of $j\in \Lambda$ 
satisfying  $\psi(j)=\psi(i)$ and $j<i$. 
We  define
$\gamma_{\psi,i}^{<}$ for $i\in \Lambda$  to be the product 
of $\gamma_j$ $(j\in \Lambda_{\psi,i}^{<})$, 
and define $\gamma_{\psi,\bmI}^{<}$ $(r\in \N, \bmI=(i_1,\ldots, i_r)\in \Lambda^r)$
to be $\prod_{\nu=1}^r\gamma_{\psi,i_{\nu}}^{<}$. 
For an $\CA$-module $\CM$ with $\CA$-linear action of $\Gamma^{\disc}_{\Lambda}$,
we define $\CA$-linear homomorphisms 
$K^r_{\psi}(\CM)\colon K^r_{\Lambda}(\CM)\to K^r_{\Lambda'}(\CM)$ $(r\in \Z)$
by $K^r_{\psi}(\CM)(m\otimes e_{\bmI})=\gamma_{\psi,\bmI}^{<}(m)\otimes e_{\psi^r(\bmI)}$ 
for $r\in \N$, $X\in \Ob C$, $m\in \CM(X)$, and $\bmI\in \Lambda^r$.
Here $\psi^r$ denotes the product $\Lambda^r\to \Lambda^{\prime r}$ of $\psi$,
and we define the action of $\Gamma^{\disc}_{\Lambda'}$ 
on $\CM$ in the codomain via $\Gamma_{\psi}^{\disc}$. 
We often abbreviate $\psi^r$ $(r\in \N)$ to $\psi$ in the following.
The homomorphism $K_{\psi}^r(\CM)$ is obviously functorial in $\CM$.
If $\psi$ is injective, the element $\gamma_{\psi,\bmI}^{<}\in \Gamma^{\disc}_{\Lambda}$ is 
the unit for every $r\in \N$ and $\bmI\in \Lambda^r$, which implies, in 
particular, that $K_{\psi}^r(\CM)$ does not depend on the choice of 
a total order of $\Lambda$.

\begin{lemma}\label{lem:KoszulFunct}
The homomorphisms $K^r_{\psi}(\CM)$ $(r\in \N)$ define a morphism of complexes\linebreak
$K^{\bullet}_{\psi}(\CM)\colon K^{\bullet}_{\Lambda}(\CM)\to 
K^{\bullet}_{\Lambda'}(\CM)$.
\end{lemma}

\begin{proof}
For $r\in \N$, $\bmI\in \Lambda^r$, $X\in \Ob C$, and $m\in \CM(X)$, we have
\begin{align*}
K^{r+1}_{\psi}(\CM)\circ d^r(m\otimes e_{\bmI})&=\sum_{i\in\Lambda}
\gamma_{\psi,i}^{<}\gamma_{\psi,\bmI}^{<}(\gamma_i-1)(m)\otimes e_{\psi(i)}\wedge e_{\psi^r(\bmI)},
\\
d^r\circ K^r_{\psi}(\CM)(m\otimes e_{\bmI})
&=\sum_{i'\in \Lambda'}(\Gamma^{\disc}_{\psi}(\gamma_{i'})-1)
\gamma_{\psi,\bmI}^{<}(m)\otimes e_{i'}\wedge e_{\psi^r(\bmI)}.
\end{align*}
Hence it suffices to prove that the sum
$\sum_{i\in \psi^{-1}(i')}\gamma_{\psi,i}^{<}(\gamma_i-1)$ coincides with $\Gamma^{\disc}_{\psi}(\gamma_{i'})-1
=\prod_{i\in \psi^{-1}(i')}\gamma_i-1$
in $\Z[\Gamma^{\disc}_{\Lambda}]$ for $i'\in \Lambda'$.
By setting $\psi^{-1}(i')=\{i_1<\cdots<i_s\}$, this is simply verified as 
$\sum_{\nu=1}^{s}
(\gamma_{i_1}\cdots\gamma_{i_{\nu-1}})(\gamma_{i_{\nu}}-1)
=\gamma_{i_1}\cdots\gamma_{i_s}-1$.
\end{proof}

\begin{lemma}\label{lem:KoszFunctComp}
Let $\psi'\colon \Lambda'\to \Lambda''$ be another map of finite sets,
and assume that we are given a total order on $\Lambda'$ such that
$\psi\colon \Lambda\to \Lambda'$ preserves orders. 
Then we have $K_{\psi'}^{\bullet}(\CM)\circ K_{\psi}^{\bullet}(\CM)
=K^{\bullet}_{\psi'\circ\psi}(\CM)$.
\end{lemma}

\begin{proof}
For $r\in \N$, $\bmI\in \Lambda^r$, $X\in \Ob C$, and $m\in \CM(X)$, we have
\begin{align*}
K_{\psi'}^{\bullet}(\CM)\circ K_{\psi}^{\bullet}(\CM)
(m\otimes e_{\bmI})&=\Gamma^{\disc}_{\psi}(\gamma_{\psi',\psi(\bmI)}^{<})
\gamma_{\psi,\bmI}^{<}(m)\otimes e_{\psi'(\psi(\bmI))},\\
K^{\bullet}_{\psi'\circ\psi}(\CM)(m\otimes e_{\bmI})
&=\gamma_{\psi'\circ\psi,\bmI}^{<}(m)\otimes e_{\psi'\circ\psi(\bmI)}.
\end{align*}
Hence we are reduced to showing
$\Gamma_{\psi}^{\disc}(\gamma^{<}_{\psi',\psi(i)})
\gamma_{\psi,i}^{<}=\gamma^{<}_{\psi'\circ\psi,i}$ for $i\in\Lambda$.
By using $\Gamma_{\psi}^{\disc}(\gamma_{i'})=
\prod_{j\in\psi^{-1}(i')}\gamma_j$ for $i'\in \Lambda'$,
we see that $\Gamma_{\psi}^{\disc}(\gamma^{<}_{\psi',\psi(i)})$
(resp.~$\gamma_{\psi,i}^{<}$) is the product of 
$\gamma_j$ over $j\in (\psi'\circ\psi)^{-1}(\psi'\circ\psi(i))$ satisfying
$j<i$ and $\psi(j)<\psi(i)$ (resp.~$\psi(j)=\psi(i)$). 
This implies the desired equality.
\end{proof}

Let $u\colon (C',\CA')\to (C,\CA)$ be a morphism of ringed sites, 
let $\CM$ be an $\CA$-module on $C$
with $\CA$-linear action of $\Gamma^{\disc}_{\Lambda}$, let
$\CM'$ be an $\CA'$-module on $C'$
with $\CA'$-linear action of $\Gamma^{\disc}_{\Lambda'}$,
and let $f\colon \CM\to u_*\CM'$ be an $\CA$-linear
homomorphism compatible with the actions of $\Gamma^{\disc}_{\Lambda}$
and  $\Gamma^{\disc}_{\Lambda'}$ via $\Gamma^{\disc}_{\psi}$. 
We define $K^{\bullet}_{\psi}(f)$ to be the composition 
\begin{equation}\label{eq:KoszulFunctMorph}
K_{\Lambda}^{\bullet}(\CM)
\xrightarrow{K^{\bullet}_{\psi}(\CM)}K_{\Lambda'}^{\bullet}(\CM)
\xrightarrow{K^{\bullet}_{\Lambda'}(f)}
K_{\Lambda'}^{\bullet}(u_*\CM')=u_*K_{\Lambda'}^{\bullet}(\CM').
\end{equation}
Let $g\colon u^*\CM\to \CM'$ be the morphism corresponding to $f$
by adjunction. Then the morphism $K_{\psi}^{\bullet}(g)\colon 
u^*K_{\Lambda}^{\bullet}(\CM)=K_{\Lambda}^{\bullet}(u^*\CM)
\to K_{\Lambda'}^{\bullet}(\CM')$ corresponds to 
$K_{\psi}^{\bullet}(f)$ by adjunction.

\begin{remark}\label{rmk:KoszulCpxProdFunct}
(1) 
Let $u\colon (\CC',\CA')\to (\CC,\CA)$ be a morphism of ringed sites,
let $\psi\colon \Lambda\to \Lambda'$ be an injective map of finite sets,
and suppose that $\Lambda$ is equipped with a total order.
 Let $\CM_{\nu}$ (resp.~$\CM'_{\nu}$) $(\nu=1,2)$ be $\CA$-modules
(resp.~$\CA'$-modules) with $\CA$-linear $\Gamma_{\Lambda}^{\disc}$-action
(resp.~$\CA'$-linear $\Gamma_{\Lambda'}^{\disc}$-action), and let $f_{\nu}$
$(\nu=1,2)$ be $\CA$-linear maps $\CM_{\nu}\to u_*\CM'_{\nu}$
compatible with the actions of $\Gamma_{\Lambda}^{\disc}$
and $\Gamma_{\Lambda'}^{\disc}$ via $\Gamma_{\psi}^{\disc}$.
Let $f$ be the $\CA$-linear map $\CM_1\otimes_{\CA}\CM_2
\xrightarrow{f_1\otimes f_2}
(u_*\CM_1')\otimes_{\CA}(u_*\CM_2')\to u_*(\CM_1'\otimes_{\CA'}\CM_2')$,
which is compatible with the diagonal actions of $\Gamma_{\Lambda}^{\disc}$
and $\Gamma_{\Lambda'}^{\disc}$ via $\Gamma_{\psi}^{\disc}$. 
Then the following diagram is commutative.
\begin{equation}\label{eq:KoszulProductFunct}
\xymatrix{
K_{\Lambda}^{\bullet}(\CM_1)\otimes_{\CA}K_{\Lambda}^{\bullet}(\CM_2)
\ar[rr]^{\eqref{eq:KoszulCpxProd}}
\ar[d]_{K^{\bullet}_{\psi}(f_1)\otimes K^{\bullet}_{\psi}(f_2)}&&
K_{\Lambda}^{\bullet}(\CM_1\otimes_{\CA}\CM_2)
\ar[d]^{K^{\bullet}_{\psi}(f)}\\
(u_*K_{\Lambda'}^{\bullet}(\CM'_1))\otimes_{\CA}
(u_*K_{\Lambda'}^{\bullet}(\CM'_2))\ar[r]&
u_*(K_{\Lambda'}^{\bullet}(\CM'_1)\otimes_{\CA'}
K_{\Lambda'}^{\bullet}(\CM'_2))
\ar[r]^(.55){\eqref{eq:KoszulCpxProd}}&
u_*K_{\Lambda'}^{\bullet}(\CM_1'\otimes_{\CA'}\CM_2')
}
\end{equation}
The proof of the claim is reduced to the two cases $u=\id_{(C,\CA)}$
and $(\psi,f_{\nu})=(\id_{\Lambda'},\id_{u_*\CM_{\nu}'})$.
The first case immediately follows from the definition of 
$K_{\psi}^{\bullet}(\CM_{\nu})$ $(\nu=1,2)$ and the
functoriality of \eqref{eq:KoszulCpxProd}. The second case follows from the
$\Gamma_{\Lambda'}^{\disc}$-equivariance of the
$\CA$-linear map $(u_*\CM_1')\otimes_{\CA}(u_*\CM_2')
\to u_*(\CM_1'\otimes_{\CA'}\CM_2')$ with respect to the
$\Gamma_{\Lambda'}^{\disc}$-actions via the second factors.

(2) Suppose that we are given the following commutative diagrams of ringed
sites and finite ordered sets such that the maps $\chi\colon \Lambda_1\sqcup \Lambda_2\to \Lambda$
and $\tchi\colon \tLambda_1\sqcup \tLambda_2\to \tLambda$ induced by 
$\chi_{\nu}$ and $\tchi_{\nu}$ are injective.
\begin{equation}
\xymatrix@R=15pt{
(\CC,\CA)&
\ar[l]_{p_{\nu}} (\CC',\CA')\\
(\tCC,\tCA)\ar[u]_{u}&
\ar[l]_{\tp_{\nu}} (\tCC',\tCA')\ar[u]_{u'}
}\qquad\qquad
\xymatrix@R=15pt{
\Lambda_{\nu}\ar[d]_{\psi_{\nu}}
\ar[r]^{\chi_{\nu}}&
\Lambda\ar[d]^{\psi}\\
\tLambda_{\nu}\ar[r]^{\tchi_{\nu}}&
\tLambda
}\qquad (\nu\in \{1,2\})
\end{equation}
Let $\CM_{\nu}$ (resp.~$\CM_{\nu}'$) $(\nu=1,2)$ be 
$\CA$-modules with $\Gamma_{\Lambda_{\nu}}^{\disc}$-action
(resp.~$\CA'$-modules with $\Gamma_{\Lambda}^{\disc}$-action),
and let $\tCM_{\nu}$ (resp.~$\tCM_{\nu}'$) $(\nu=1,2)$ be 
$\tCA$-modules with $\Gamma_{\tLambda_{\nu}}^{\disc}$-action
(resp.~$\tCA'$-modules with $\Gamma_{\tLambda}^{\disc}$-action).
For $\nu\in \{1,2\}$, let $g_{\nu}\colon p_{\nu}^*\CM_{\nu}\to \CM_{\nu}'$
(resp.~$\tg_{\nu}\colon \tp_{\nu}^*\tCM_{\nu}\to \tCM'_{\nu}$) be 
a $\Gamma^{\disc}_{\chi_{\nu}}$-equivariant $\CA'$-linear
(resp.~$\Gamma^{\disc}_{\tchi_{\nu}}$-equivariant $\tCA'$-linear)
map, and let $h_{\nu}\colon u^*\CM_{\nu}\to \tCM_{\nu}$
(resp.~$h'_{\nu}\colon u^{\prime*}\CM'_{\nu}\to \tCM'_{\nu}$)
be a  $\Gamma_{\psi_{\nu}}^{\disc}$-equivariant $\tCA$-linear
(resp.~$\Gamma_{\psi}^{\disc}$-equivariant $\tCA'$-linear)
map such that $h'_{\nu}\circ u^{\prime*}g_{\nu}=\tg_{\nu}\circ \tp_{\nu}^*h_{\nu}$. 
Then the following diagram is commutative
\begin{equation}
\xymatrix@R=15pt{
u^{\prime*}(p_1^*(K_{\Lambda_1}^{\bullet}(\CM_1))\otimes_{\CA'}
p_2^*(K_{\Lambda_2}^{\bullet}(\CM_2)))\ar[r]
\ar[d]_{\tp_1^*K_{\psi_1}^{\bullet}(h_1)\otimes \tp_2^*K_{\psi_2}^{\bullet}(h_2)}
&
u^{\prime*}K_{\Lambda}^{\bullet}(\CM_1'\otimes_{\CA'}\CM_2')
\ar[d]^{K_{\psi}^{\bullet}(h_1'\otimes h_2')}
\\
\tp_1^*K_{\tLambda_1}^{\bullet}(\tCM_1)\otimes_{\tCA'}
\tp_2^*K_{\tLambda_2}^{\bullet}(\tCM_2)
\ar[r]&
K^{\bullet}_{\tLambda}(\tCM_1'\otimes_{\tCA'}\tCM_2'),
}
\end{equation}
where the upper horizontal morphism is defined by the composition of  the tensor 
product of $K_{\chi_{\nu}}^{\bullet}(g_{\nu})\colon 
p_{\nu}^*K^{\bullet}_{\Lambda_{\nu}}(\CM_{\nu})\to K^{\bullet}_{\Lambda}(\CM'_{\nu})$
$(\nu=1,2)$ with \eqref{eq:KoszulCpxProd} for $\CM'_{\nu}$ $(\nu=1,2)$, and 
the lower one is defined similarly using $\tg_{\nu}$. 
By replacing $\CM_{\nu}$, $\tCM_{\nu}$, $\CM'_{\nu}$, $g_{\nu}$, and $h_{\nu}$
by their pullbacks on $(\tCC',\tCA')$, the proof is reduced to the case where
all of the morphisms  of ringed topos above are the identity morphisms.
We regard $\Lambda_1\sqcup \Lambda_2$ (resp.~$\tLambda_1\sqcup\tLambda_2$)
as a subset of $\Lambda$ (resp.~$\tLambda)$ via the injective map 
$\chi$ (resp.~$\tchi$). For $\nu\in\{1,2\}$ and $i\in \Lambda\backslash \Lambda_{\nu}$
the action of $\gamma_i\in \Gamma_{\Lambda}^{\disc}$ on $\CM_{\nu}$
via $\Gamma_{\chi_{\nu}}^{\disc}$ is trivial. Hence,  for
$m_{\nu}\in \CM_{\nu}$, $r_{\nu}\in \N$, and $\bmI_{\nu}\in \Lambda_{\nu}^{r_{\nu}}$
$(\nu=1,2)$, the image of $m_1\otimes e_{\bmI_1}\otimes m_2\otimes e_{\bmI_2}$
under the upper horizontal morphism is 
$x=g_1(m_1)\otimes g_2(m_2)\otimes e_{\bmI_1}\wedge e_{\bmI_2}$.
The same claim holds for the lower horizontal morphism. Since
$(\Lambda_{\nu})_{\psi_{\nu},i}^{<}=\Lambda_{\psi,i}^{<}\cap \Lambda_{\nu}$
for $\nu\in \{1,2\}$ and $i\in \Lambda_{\nu}$, the above remark
on the action of $\Gamma_{\Lambda}^{\disc}$ on $\CM_{\nu}$ implies
that the image of $x$ above under the right vertical morphism 
$K_{\psi}^{\bullet}(h_1'\otimes h_2')$
is 
$h_1'g_1(\gamma_{\psi_1,\bmI_1}^{<}(m_1))\otimes
h_2'g_2(\gamma_{\psi_2,\bmI_2}^{<}(m_2))
\otimes e_{\psi_1(\bmI_1)}\wedge e_{\psi_2(\bmI_2)}$.
This completes the proof because the image of 
$m_{\nu}\otimes e_{\bmI_{\nu}}$ under
$K^{\bullet}_{\psi_{\nu}}(h_{\nu})$ is 
$h_{\nu}(\gamma_{\psi_{\nu},\bmI_{\nu}}^{<}(m_{\nu}))\otimes
e_{\psi_{\nu}(\bmI_{\nu})}$.
\end{remark}

\begin{lemma}\label{lem:GammaKoszCompos}
 Under the notation before Remark \ref{rmk:KoszulCpxProdFunct}, 
 assume that we are given 
a morphism $u'\colon (C'',\CA'')\to (C',\CA')$ of ringed sites,
a map of finite sets $\psi'\colon \Lambda'\to \Lambda''$,
and a total order on $\Lambda'$ such that $\psi\colon \Lambda\to \Lambda'$
preserves orders. Let $\CM''$ be an $\CA''$-module on $C''$ with
$\CA''$-linear action of $\Gamma_{\Lambda''}^{\disc}$,
and let $f'\colon \CM'\to u'_*\CM''$ be an $\CA'$-linear homomorphism
compatible with the actions of $\Gamma_{\Lambda'}^{\disc}$
and $\Gamma_{\Lambda''}^{\disc}$ via $\Gamma_{\psi'}^{\disc}$. 
Then we have $u_*K^{\bullet}_{\psi'}(f')\circ K_{\psi}^{\bullet}(f)=
K^{\bullet}_{\psi'\circ\psi}(u_*f'\circ f)$.
\end{lemma}

\begin{proof}
By the functoriality of $K_{\psi}^{\bullet}(\CM)$ and 
$K_{\psi'}^{\bullet}(\CM')$ with respect to $\CM$ and $\CM'$, 
the claim is immediately reduced to 
Lemma \ref{lem:KoszFunctComp}.
\end{proof}

We are ready to study the right derived functor of $\pi_{G,C*}$
when $G=\Map(\Lambda,\Z_p)$ in terms of the Koszul complex.
Let $(C,\CA)$ be a ringed site whose topology is defined by a pretopology 
$\Cov_C(X)$ $(X\in\Ob C)$ consisting of finite families of morphisms,
let $\CA$ be a sheaf of rings on $C$, and let
$\Lambda$ be a finite set.
Let $\Gamma_{\Lambda}$ denote
the abelian profinite group $\Map(\Lambda,\Z_p)$. For $i\in\Lambda$,
let $\gamma_i$ be the map $\Lambda\to \Z_p$ sending $j$ to $1$
if $j=i$ and to $0$ otherwise. To simplify the notation, we write
$\iota_{\Lambda,C}$ and $\pi_{\Lambda,C}$
for the morphisms of topos $\iota_{\Gamma_{\Lambda},C}$
and $\pi_{\Gamma_{\Lambda},C}$, respectively.

Let $\CM$ be a $\Gamma_{\Lambda}$-sheaf of $\CA$-modules on $C$.
We define an action $[-]$ of $\Gamma_{\Lambda}^{\disc}$ on the $\Gamma_{\Lambda}$-sheaf of 
$\CA$-modules
$\iota_{\Lambda,C*}\iota_{\Lambda,C}^*\CM
=\CMap_{\cont}(\Gamma_{\Lambda},\CM)$
by  
\begin{equation}\label{eq:SecGammaAction}
([\gamma]f)(g)=\gamma\cdot f(\gamma^{-1}g)
\end{equation}
for $X\in \Ob C$, $f\in \Map_{\cont}(\Gamma_{\Lambda},\CM(X))$,
$g\in \Gamma_{\Lambda}$, and $\gamma\in \Gamma^{\disc}_{\Lambda}$.
By Remark \ref{rmk:FinInvLimGSheaf} (1) and Proposition \ref{prop:GsheafTopos}, 
we can apply the general construction of Koszul complex 
on a ringd site to $\iota_{\Lambda,C*}\iota_{\Lambda,C}^*\CM$ with the above
$\Gamma^{\disc}_{\Lambda}$-action and obtain the complex
$K^{\bullet}_{\Lambda}(\iota_{\Lambda,C*}\iota_{\Lambda,C}^*\CM)$ of 
$\Gamma_{\Lambda}$-sheaves of $\CA$-modules on  $C$.

\begin{proposition}\label{prop:GammaRepKResol}
If $p^N\CM=0$ for some positive integer $N$, then 
the complex $K^{\bullet}_{\Lambda}(\iota_{\Lambda,C*}\iota_{\Lambda,C}^*\CM)$ 
in $\Gamma_{\Lambda}\Mod(C,\CA)$
gives a resolution of $\CM$ via the adjunction morphism 
$\CM\to \iota_{\Lambda,C*}\iota_{\Lambda,C}^*\CM$.
\end{proposition}

\begin{proof}
By considering sections over each $X\in \Ob C$, we are reduced
to the case where $C$ is a one point category and $\CA=\Z/p^n\Z$ for
a positive integer $n$. (See Proposition \ref{prop:GSheafPB2} (2) for the description
of the adjunction morphism.) Let $M$ be a $\Gamma_{\Lambda}$-module
over $\Z/p^n\Z$, and let $K_{\Lambda}^{\bullet}(\Map_{\cont}(\Gamma_{\Lambda},M))$
be the Koszul complex of $\Map_{\cont}(\Gamma_{\Lambda},M)$ with respect
to the $[-]$-action of $\Gamma_{\Lambda}^{\disc}$.
The adjunction morphism $\eta_{\Lambda,M}\colon M\to \Map_{\cont}(\Gamma_{\Lambda},M)$
is given by $x\mapsto f_x$, $f_x(\gamma)=\gamma x$. 
We prove the claim by induction on $\sharp \Lambda$. 
\par
Assume $\sharp \Lambda=1$ and let  $i_0$ be the unique element of $\Lambda$.
For $x\in M$, we have $([\gamma_{i_0}]f_x)(\gamma)=
\gamma_{i_0}f_x(\gamma_{i_0}^{-1}\gamma)=\gamma_{i_0}(\gamma_{i_0}^{-1}\gamma)x
=\gamma x=f_x(\gamma)$ for $\gamma\in \Gamma_{\Lambda}$.
Conversely, if $f\in \Map_{\cont}(\Gamma_{\Lambda},M)$ satisfies
$[\gamma_{i_0}]f=f$, then we have 
$\gamma_{i_0}^r(f(1))=([\gamma_{i_0}^r]f)(\gamma_{i_0}^r)=f(\gamma_{i_0}^r)$ $(r\in \N)$,
which implies $f=f_x$, $x=f(1)$ by the continuity of $f$. 
For $f\in \Map_{\cont}(\Gamma_{\Lambda},M)$, suppose that the map 
$f$ factors through $\Gamma_{\Lambda}/\Gamma_{\Lambda}^{p^m}$
and the action of $\Gamma_{\Lambda}^{p^m}$ on 
the finite set $f(\Gamma_{\Lambda})$ is trivial.
Then we see that
the map $h_+\colon \{\gamma_{i_0}^{r}\,\vert\, r\in \Z, r>0\}\to M$
defined by $h_+(\gamma_{i_0}^{r})=-\gamma_{i_0}^r\sum_{s=1}^r\gamma_{i_0}^{-s}f(\gamma_{i_0}^s)$
factors through a map $\overline{h}\colon \Gamma_{\Lambda}/\Gamma_{\Lambda}^{p^{n+m}}\to M$
and its composition $h$ with the projection map from $\Gamma_{\Lambda}$ satisfies
$([\gamma_{i_0}]-1)h=f$. This completes the proof in the case $\sharp \Lambda=1$.
\par
Assume $\sharp \Lambda \geq 2$, and the claim holds for $\Lambda_1\subset\Lambda$
with $\sharp \Lambda_1=\sharp\Lambda-1$. Put $\Lambda_0=\Lambda\backslash \Lambda_1=\{i_0\}$.
We identify $\Gamma_{\Lambda}$ with $\Gamma_{\Lambda_0}\times\Gamma_{\Lambda_1}$
by the isomorphism $\gamma\mapsto (\gamma\vert_{\Lambda_0},\gamma\vert_{\Lambda_1})$.
Put $N=\Map_{\cont}(\Gamma_{\Lambda_0},M)$.  Then the action of $\Gamma_{\Lambda_1}$
on $M$ induces a continuous action of $\Gamma_{\Lambda_1}$ on $N$, and we have an
isomorphism $\Phi\colon \Map_{\cont}(\Gamma_{\Lambda},M)
\xrightarrow{\cong}\Map_{\cont}(\Gamma_{\Lambda_1},N)$ given by 
$f\mapsto F$, $f(\gamma_0,\gamma_1)=(F(\gamma_1))(\gamma_0)$,
$(\gamma_0,\gamma_1)\in\Gamma_{\Lambda}=\Gamma_{\Lambda_0}\times\Gamma_{\Lambda_1}$.
The isomorphism $\Phi$ is equivariant with respect to the $[-]$-action 
of $\Gamma_{\Lambda_1}^{\disc}$ (resp.~$\Gamma_{\Lambda_0}^{\disc}$) on the domain,
and the $[-]$-action of $\Gamma_{\Lambda_1}^{\disc}$ (resp.~the action of $\Gamma_{\Lambda_0}^{\disc}$
induced by its $[-]$-action on $N$) on the codomain. Hence we have morphisms of complexes
\begin{multline*}
K^{\bullet}_{\Lambda}(\Map_{\cont}(\Gamma_{\Lambda},M))
\xrightarrow{\cong}\text{fiber}\left(K^{\bullet}_{\Lambda_1}(\Map_{\cont}(\Gamma_{\Lambda_1},N))
\xrightarrow{[\gamma_{i_0}]-1}K^{\bullet}_{\Lambda_1}(\Map_{\cont}(\Gamma_{\Lambda_1},N))\right)\\
\longleftarrow\text{fiber}(N\xrightarrow{[\gamma_{i_0}]-1} N)
\xleftarrow{\eta_{\Lambda_0,M}} M,
\end{multline*}
where the second arrow is the morphism induced by $\eta_{\Lambda_1,N}$,
which is a quasi-isomorphism by assumption, and the third arrow is a
quasi-isomorphism by the case $\sharp \Lambda=1$ proven above. 
We obtain the claim by observing $\Phi\circ\eta_{\Lambda,M}=\eta_{\Lambda_1,N}\circ \eta_{\Lambda_0,M}$.
\end{proof}

Since $\iota_{\Lambda,C}^*\CM$ is the sheaf of 
$\CA$-modules underlying $\CM$, we have an action of $\Gamma^{\disc}_{\Lambda}$ on 
$\iota_{\Lambda,C}^*\CM$ via $\Gamma_{\Lambda}$. The $[-]$-action of 
$\Gamma_{\Lambda}^{\disc}$ on the $\Gamma_{\Lambda}$-sheaf of $\CA$-modules
$\iota_{\Lambda,C*}\iota_{\Lambda,C}^*\CM$
induces  an action of $\Gamma^{\disc}_{\Lambda}$ on 
its direct image under $\pi_{\Lambda,C}$. 

\begin{lemma}\label{lem:KoszResolGalInv}
The isomorphism 
$\iota_{\Lambda,C}^*
\CM\cong \pi_{\Lambda,C*}\iota_{\Lambda,C*}
\iota_{\Lambda,C}^*\CM$ 
given by $\id_{C^{\sim}}\cong \pi_{\Lambda,C}\circ \iota_{\Lambda,C}$
is $\Gamma_{\Lambda}^{\disc}$-equivariant with respect to the 
$\Gamma_{\Lambda}^{\disc}$-actions
above.
\end{lemma}

\begin{proof} By Lemma \ref{lem:GSheafPFComp}, 
the isomorphism in the claim is given by 
$\CM\xrightarrow{\cong}\CMap_{\cont}(\Gamma_{\Lambda},\CM)^{\Gamma_{\Lambda}}; x\mapsto c_x$,
where $c_x(g)=x$ for $X\in \Ob C$, $x\in \CM(X)$, and $g\in \Gamma_{\Lambda}$. 
The claim holds by $([\gamma]c_x)(g)
=\gamma\cdot c_x(\gamma^{-1}g)=\gamma x=c_{\gamma x}(g)$.
\end{proof}

By applying the general construction of Koszul complex on a ringed site
to the $\CA$-module $\iota_{\Lambda,C}^*\CM$ with the 
$\Gamma^{\disc}_{\Lambda}$-action,
we obtain a complex $K^{\bullet}_{\Lambda}(\iota_{\Lambda,C}^*\CM)$.
If $p^N\CM=0$ for some positive integer $N$, then,
by Proposition \ref{prop:GammaRepKResol}, 
Corollary \ref{cor:IndGModGacyc}, and 
Lemma \ref{lem:KoszResolGalInv},  we obtain the following isomorphisms
in $D^+(C,\CA)$.
\begin{equation}\label{eq:GSheafCohKCpx}
R\pi_{\Lambda,C*}\CM
\cong R\pi_{\Lambda,C*}K^{\bullet}_{\Lambda}
(\iota_{\Lambda,C*}\iota_{\Lambda,C}^*\CM)
\cong \pi_{\Lambda,C*}K^{\bullet}_\Lambda
(\iota_{\Lambda,C*}\iota_{\Lambda,C}^*\CM)
\cong K^{\bullet}_{\Lambda}
(\iota_{\Lambda,C}^*\CM)
\end{equation}

\begin{remark}\label{rmk:GShfCohKoszProd}
We abbreviate $\iota_{\Lambda,C}$, $\pi_{\Lambda,C}$, 
$\Gamma_{\Lambda}$, $\Gamma_{\Lambda}^{\disc}$, and $\otimes_{\CA}$ to 
$\iota$, $\pi$, $\Gamma$, $\Gamma^{\disc}$, and $\otimes$,
respectively, in this remark. Let $\CM_{\nu}$ $(\nu=1,2)$ be $\Gamma$-sheaves
of $\CA$-modules on $C$. We see that the $\CA$-linear map
\begin{equation}\label{eq:TensorGammadiscAction}
\iota_*\iota^*(\CM_1)\otimes\iota_*\iota^*(\CM_2)\to \iota_*\iota^*(\CM_1\otimes\CM_2)
\end{equation}
obtained by taking the right adjoint of the tensor product 
$\iota^*\iota_*(\iota^*\CM_1)\otimes\iota^*\iota_*(\iota^*\CM_2)
\to \iota^*\CM_1\otimes\iota^*\CM_2$ of counit maps, is
$\Gamma^{\disc}$-equivariant for the $[-]$-action
by the following explicit description of \eqref{eq:TensorGammadiscAction}:
By Proposition \ref{prop:GSheafPB2} (2), 
the counit map $\iota^*\iota_*(\iota^*\CM_{\nu})
=\Map_{\cont}(\Gamma,\iota^*\CM_{\nu})\to \iota^*\CM_{\nu}$
sends $\varphi$ to $\varphi(1)$ for $\nu=1,2$, and therefore 
the right adjoint in question
$\Map_{\cont}(\Gamma,\iota^*\CM_1)
\otimes
\Map_{\cont}(\Gamma,\iota^*\CM_2)
\to \Map_{\cont}(\Gamma,\iota^*(\CM_1\otimes\CM_2))$
sends $\varphi_1\otimes\varphi_2$ to
$\psi$ defined by $\psi(\gamma)=
(\gamma\cdot \varphi_1)(1)\otimes(\gamma\cdot\varphi_2)(1)
=\varphi_1(\gamma)\otimes\varphi_2(\gamma)$
$(\gamma\in \Gamma)$. By applying \eqref{eq:KoszulProductFunct} to 
the $\Gamma^{\disc}$-equivariant maps $\iota^*\CM_{\nu}\xrightarrow{\cong}
\pi_*(\iota_*\iota^*\CM_{\nu})$ (Lemma \ref{lem:KoszResolGalInv}) and using 
the $\Gamma^{\disc}$-equivariance of \eqref{eq:TensorGammadiscAction}
above, we obtain the following commutative diagram, where $K$ denotes $K_{\Lambda}^{\bullet}$.
\begin{equation}\label{eq:KoszulGammaInvProdComp}
\xymatrix@R=15pt{
K(\iota^*\CM_1)\otimes
K(\iota^*\CM_2)\ar[r]^{\eqref{eq:KoszulCpxProd}}\ar[d]&
K((\iota^*\CM_1)\otimes(\iota^*\CM_2))\ar[r]^{\cong}\ar[d]&
K(\iota^*(\CM_1\otimes\CM_2))\ar[d]\\
\pi_*(K(\iota_*\iota^*\CM_1)\otimes K(\iota_*\iota^*\CM_2))
\ar[r]^(.53){\eqref{eq:KoszulCpxProd}}&
\pi_*(K(\iota_*\iota^*\CM_1\otimes\iota_*\iota^*\CM_2))
\ar[r]^{\eqref{eq:TensorGammadiscAction}}&
\pi_*(K(\iota_*\iota^*(\CM_1\otimes\CM_2)))
}
\end{equation}
Since the unit maps $\CN\to \iota_*\iota^*\CN$ for 
$\CN=\CM_1,\CM_2$, and $\CM_1\otimes\CM_2$ 
are compatible with \eqref{eq:TensorGammadiscAction} above,
we obtain the following compatibility of \eqref{eq:GSheafCohKCpx} with products.
\begin{equation}\label{eq:GSheafCohKCpxProdComp}
\xymatrix@R=15pt@C=60pt{
R\pi_*\CM_1\otimes^LR\pi_*\CM_2
\ar[r]^(.45){\cong}_(.45){\eqref{eq:GSheafCohKCpx}}\ar[d]^{-\; \cup\;- }&
K_{\Lambda}^{\bullet}(\iota^*\CM_1)\otimes^LK_{\Lambda}^{\bullet}(\iota^*\CM_2)\ar[d]^{\eqref{eq:KoszulCpxProd}}\\
R\pi_*(\CM_1\otimes\CM_2)
\ar[r]^{\cong}_{\eqref{eq:GSheafCohKCpx}}&
K_{\Lambda}^{\bullet}(\iota^*(\CM_1\otimes\CM_2))
}
\end{equation}
\end{remark}

We discuss the functoriality of the isomorphisms \eqref{eq:GSheafCohKCpx}
with respect to $(C,\CA)$ and $\Lambda$.
Let $\Lambda'$ be a finite set, let $C'$ be a site whose topology is defined by 
a pretopology $\Cov_{C'}(X')$ $(X'\in \Ob C')$ consisting of finite families
of morphisms, and suppose that we are given a map $\psi\colon \Lambda\to 
\Lambda'$ and a functor $u\colon C\to C'$ defining a morphism of sites and therefore
inducing a morphism of topos $(u^*,u_*)\colon C^{\prime\sim}\to C^{\sim}$, which will be also denoted by $u$.
Let $\Gamma_{\psi}$ denote the continuous homomorphism $\Gamma_{\Lambda'}
\to\Gamma_{\Lambda}; \gamma\mapsto \gamma\circ\psi$ induced by $\psi$.
We write $u_{\psi}=(u_{\psi}^*,u_{\psi*})$ for the morphism of 
topos $\Gamma_{\Lambda'}\hy C^{\prime\sim}\to \Gamma_{\Lambda}\hy C^{\sim}$
induced by the pair $(u,\Gamma_{\psi})$
\eqref{eq:GSheafToposMorph}. 
Let $\CA'$ be a sheaf of rings  on $C'$, let $\CA'$ also denote its pullback
by $\pi_{\Lambda',C'}$, and suppose that we are given a morphism 
$u^*(\CA)\to \CA'$ of sheaves of rings on $C'$.
Under these settings, we have the following commutative diagram
of ringed topos.
\begin{equation}\label{eq:KoszFunctToposDiag}
\xymatrix@C=50pt{
(C^{\prime\sim},\CA')\ar[r]^(.45){\iota_{\Lambda',C'}}\ar[d]^{u}
&(\Gamma_{\Lambda'}\hy C^{\prime\sim},\CA')\ar[r]^{\pi_{\Lambda',C'}}\ar[d]^{u_{\psi}}
&(C^{\prime\sim},\CA')\ar[d]^{u}\\
(C^{\sim},\CA)\ar[r]^(.45){\iota_{\Lambda,C}}
&(\Gamma_{\Lambda}\hy C^{\sim},\CA)\ar[r]^{\pi_{\Lambda,C}}
&(C^{\sim},\CA)
}
\end{equation}
We choose and fix a total order on the set $\Lambda$. 
To simplify the notation, we abbreviate $\iota_{\Lambda,C}$, $\pi_{\Lambda,C}$,
$\Gamma_{\Lambda}$,  $\iota_{\Lambda',C'}$, $\pi_{\Lambda',C'}$, 
and $\Gamma_{\Lambda'}$ to 
$\iota$, $\pi$, $\Gamma$, $\iota'$, $\pi'$, and $\Gamma'$, respectively,
in the following.

\begin{lemma}\label{lem:ForgetGammaBC}
For a $\Gamma'$-sheaf of $\CA'$-modules $\CM'$ on $C'$, 
The base change morphism 
$\iota^*u_{\psi*}\CM'\to u_*\iota^{\prime*}\CM'$
with respect to the left square of \eqref{eq:KoszFunctToposDiag}
is given by $\CMap_{\Gamma',\cont}(\Gamma,u_*\CM')\to u_*\CM'; \varphi\mapsto \varphi(1)$.\par
\end{lemma}

\begin{proof} 
The base change morphism is the composition  
$\iota^*u_{\psi*}\CM'\to \iota^*u_{\psi*}\iota^{\prime}_*\iota^{\prime*}\CM'
\cong \iota^*\iota_*u_*\iota^{\prime*}\CM'\to u_*\iota^{\prime*}\CM'$,
where the first (resp.~third) morphism is induced by the unit (resp.~counit) of $\iota'$
(resp.~$\iota$). By Lemma \ref{lem:GSheafPFComp} and Proposition \ref{prop:GSheafPB2} (2),
this is explicitly given by the composition 
$$\CMap_{\Gamma',\cont}(\Gamma,u_*\CM')
\to \CMap_{\Gamma',\cont}(\Gamma,u_*\CMap_{\cont}(\Gamma',\CM'))
\cong \CMap_{\cont}(\Gamma,u_*\CM')\to u_*\CM'$$
sending a section $\varphi$ on an object of $C$ as
$\varphi\mapsto \{g\mapsto (g'\mapsto g'\varphi(g))\}
\mapsto (g\mapsto \varphi(g))\mapsto \varphi(1).$
\end{proof}

\begin{proposition}\label{prop:GSheafKoszulFunct}
Let $\CM$ be a $\Gamma$-sheaf of $\CA$-modules on $C$,
let $\CM'$ be a $\Gamma'$-sheaf of $\CA'$-modules on $C'$,
and let $f\colon \CM\to u_{\psi*}\CM'$ be a morphism of 
$\Gamma$-sheaves of $\CA$-modules on $C$. \par
(1) The composition $\of\colon \iota^*\CM
\xrightarrow{\iota^*f}\iota^*u_{\psi*}\CM'
\to u_*\iota^{\prime*}\CM'$ in $\Mod(C,\CA)$ 
is compatible with the actions of $\Gamma_{\Lambda}^{\disc}$
and $\Gamma_{\Lambda'}^{\disc}$ via
$\Gamma^{\disc}_{\psi}\colon \Gamma^{\disc}_{\Lambda'}
\to \Gamma^{\disc}_{\Lambda}$, where the second
morphism in the composition $\of$ is the base change
with respect to the left square of \eqref{eq:KoszFunctToposDiag}.\par
(2) The composition $\oof\colon \iota_*\iota^*\CM
\xrightarrow{\iota_*(\of)}\iota_*u_*\iota^{\prime*}\CM'\cong
u_{\psi*}\iota'_*\iota^{\prime*}\CM'$ in $\Gamma_{\Lambda}\Mod(C,\CA)$ 
is compatible with the $[-]$-actions 
\eqref{eq:SecGammaAction} of $\Gamma^{\disc}_{\Lambda}$ and 
$\Gamma^{\disc}_{\Lambda'}$ via $\Gamma_{\psi}^{\disc}$.
\par
(3) The following diagrams are commutative.
\begin{equation}\label{eq:IotaUnitFunct}
\xymatrix{
\iota_*\iota^*\CM\ar[r]^(.45){\oof}& u_{\psi*}\iota_*'\iota^{\prime*}\CM'\\
\CM\ar[u]^{\eta_{\iota}(\CM)}\ar[r]^(.45)f&
u_{\psi*}\CM'\ar[u]_{u_{\psi*}\eta_{\iota'}(\CM')}
}
\end{equation}
\begin{equation}\label{eq:PiIotaIdFunct}
\xymatrix{
\pi_*\iota_*\iota^*\CM\ar[r]^(.45){\pi_*\oof}&
\pi_*u_{\psi*}\iota'_*\iota^{\prime*}\CM'
\ar[r]^{\cong}&
u_*\pi_*'\iota_*'\iota^{\prime*}\CM'\\
\iota^*\CM\ar[u]^{\cong}_{\alpha_{C}(\CM)}\ar[rr]^{\of}&&
u_*\iota^{\prime*}\CM'
\ar[u]^{\cong}_{u_*\alpha_{C'}(\CM')}
}
\end{equation}
\end{proposition}

\begin{proof}
By the equalities $\of=\overline{\id_{u_{\psi*}\CM'}}\circ\iota^*f$
and $\oof=\overline{\overline{\id_{u_{\psi*}\CM'}}}\circ \iota_*\iota^*f$, 
every claim is immediately reduced to the case $f=\id_{u_{\psi*}\CM'}$.\par
(1) The claim follows from Lemma \ref{lem:ForgetGammaBC}; we have 
$(\Gamma^{\disc}_{\psi}(\gamma')\varphi)(1)=\varphi(\Gamma^{\disc}_{\psi}(\gamma'))
=\gamma'(\varphi(1))$ for $\varphi\in \Map_{\Gamma',\cont}(\Gamma,u_*\CM')$ and
$\gamma'\in \Gamma^{\disc}_{\Lambda'}$.\par
(2) By Lemmas  \ref{lem:ForgetGammaBC} and \ref{lem:GsetPFComp}, the morphism
$\overline{\overline{\id_{u_*\CM'}}}$
is given by the composition 
$$\CMap_{\cont}(\Gamma,\CMap_{\Gamma'}(\Gamma,u_*\CM'))
\to \CMap_{\cont}(\Gamma,u_*\CM')\cong \CMap_{\Gamma',\cont}(\Gamma,\CMap_{\cont}
(\Gamma',u_*\CM'))$$
sending a section $\varphi$ on an object of $C$ as
$\varphi\mapsto (g\mapsto \{\varphi(g)\}(1))\mapsto 
(g\mapsto (g'\mapsto \{\varphi(\Gamma_{\psi}(g')g)\}(1))).$
The claim is verified by the following computation for $\gamma'\in \Gamma_{\Lambda'}^{\disc}$
and $\gamma=\Gamma_{\psi}^{\disc}(\gamma')$:
$$\{([\gamma]\varphi)(\Gamma_{\psi}(g')g)\}(1)=
\{\gamma\cdot \varphi(\gamma^{-1}\Gamma_{\psi}(g')g)\}(1)
=\{\varphi(\gamma^{-1}\Gamma_{\psi}(g')g)\}(\gamma)
=\gamma'\{\varphi(\Gamma_{\psi}(\gamma^{\prime-1}g')g)\}(1).$$

(3) The morphism $\oof$ for $f=\id_{u_{\psi_*}\CM'}$ is given by the composition 
$\iota_*\iota^*u_{\psi*}
\to\iota_*\iota^*u_{\psi*}\iota_*^{\prime}\iota^{\prime*}
\cong\iota_*\iota^*\iota_*u_*\iota^{\prime*}\to 
\iota_*u_*\iota^{\prime*}\cong u_{\psi*}\iota^{\prime}_*\iota^{\prime*}$
defined by the unit of $\iota'$ and the counit of $\iota$, and
the commutativity of \eqref{eq:IotaUnitFunct} is reduced to 
the fact that the composition $\iota_*u_*\iota^{\prime*}
\to \iota_*\iota^*\iota_*u_*\iota^{\prime*}
\to \iota_*u_*\iota^{\prime*}$ given by the unit and counit
of $\iota$ is the identity morphism. The commutativity of
\eqref{eq:PiIotaIdFunct} follows from the construction
of $\oof$ from $\of$ and  the fact that
we have isomorphisms
$u_*\cong \pi_*\iota_*u_*\cong \pi_*u_{\psi*}\iota'_*
\cong u_*\pi'_*\iota'_*\cong u_*$
whose composition  is the identity morphism
by \eqref{eq:GSheavesFunctCocyc}.
\end{proof}

\begin{remark} \label{rmk:KoszulResolProd}
For $\nu\in \{1,2,3\}$, let $\CM_{\nu}$ be a $\Gamma_{\Lambda}$-sheaf of $\CA$-modules on $C$,
let $\CM'_{\nu}$ be a $\Gamma_{\Lambda'}$-sheaf of $\CA'$-modules on $C'$,
and let $f_{\nu}\colon \CM_{\nu}\to u_{\psi*}\CM'_{\nu}$ be a morphism of 
$\Gamma_{\Lambda}$-sheaves of $\CA$-modules on $C$. We define 
$\of_{\nu}$ and $\oof_{\nu}$ by applying the constructions of Proposition 
\ref{prop:GSheafKoszulFunct} (1) and  (2) to $f_{\nu}$. Suppose that we are
given a $\Gamma_{\Lambda}$-equivariant $\CA$-linear morphism 
$g\colon \CM_1\otimes_{\CA}\CM_2\to \CM_3$ and a $\Gamma_{\Lambda'}$-equivariant 
$\CA'$-linear morphism 
$g'\colon \CM'_1\otimes_{\CA'}\CM'_2\to \CM'_3$ making the following diagram commutative.
\begin{equation}\label{eq:GammaSheafProdFunct}
\xymatrix@R=10pt@C=40pt{
\CM_1\otimes_{\CA}\CM_2\ar[r]^(.4){f_1\otimes f_2}\ar[d]_{g}
&(u_{\psi*}\CM'_1)\otimes_{\CA}(u_{\psi*}\CM'_2)\ar[r]&
u_{\psi*}(\CM'_1\otimes_{\CA'}\CM'_2)\ar[d]^{u_{\psi*}g'}\\
\CM_3\ar[rr]^{f_3}&& u_{\psi*}\CM_3'
}
\end{equation}
Then the following two diagrams are commutative.
\begin{equation}\label{eq:GammaSheafProdFunct2}
\xymatrix@R=10pt@C=40pt{
\iota^*\CM_1\otimes_{\CA}\iota^*\CM_2\ar[r]^(.42){\of_1\otimes \of_2}\ar[d]_{\iota^*g}
&(u_*\iota^{\prime*}\CM'_1)\otimes_{\CA}(u_*\iota^{\prime*}\CM'_2)\ar[r]&
u_*\iota^{\prime*}(\CM'_1\otimes_{\CA'}\CM'_2)\ar[d]^{u*\iota^{\prime*}g'}\\
\iota^*\CM_3\ar[rr]^{\of_3}&& u_*\iota^{\prime*}\CM_3'
}
\end{equation}
\begin{equation}\label{eq:GammaSheafProdFunct3}
\xymatrix@R=10pt{
\iota_*\iota^*\CM_1\otimes_{\CA}\iota_*\iota^*\CM_2\ar[r]^(.41){\oof_1\otimes \oof_2}\ar[d]
&(u_{\psi*}\iota'_*\iota^{\prime*}\CM'_1)\otimes_{\CA}(u_{\psi*}\iota'_*\iota^{\prime*}\CM'_2)\ar[r]&
u_{\psi*}\iota'_*\iota^{\prime*}(\CM'_1\otimes_{\CA'}\CM'_2)\ar[d]^{u_{\psi*}\iota'_*\iota^{\prime*}g'}\\
\iota_*\iota^*(\CM_1\otimes_{\CA}\CM_2)\ar[r]^(.6){\iota_*\iota^*g}&
\iota_*\iota^*\CM_3\ar[r]^{\oof_3}& u_{\psi*}\iota'_*\iota^{\prime*}\CM_3'
}
\end{equation}
The proof is reduced to the case $g$ and $g'$ are the identity morphisms, and
$f_3$ is the composition of the upper horizontal morphisms of 
\eqref{eq:GammaSheafProdFunct}. Then the claim for \eqref{eq:GammaSheafProdFunct2}
follows from the fact that the base change morphism $\iota^*u_{\psi*}\to u_*\iota^{\prime*}$
is compatible with the lax monoidal structures. We obtain the claim 
for \eqref{eq:GammaSheafProdFunct3} from that for \eqref{eq:GammaSheafProdFunct2}
by applying $\iota_*$ and noting that the isomorphism $\iota_*u_*\cong u_{\psi*}\iota'_*$
is compatible with the lax monoidal structures.
\end{remark}

We keep the notation and assumption in Proposition \ref{prop:GSheafKoszulFunct}.
Since $\alpha_C(\CM)$ (resp.~$\alpha_{C'}(\CM')$)
is $\Gamma_{\Lambda}^{\disc}$ (resp.~$\Gamma_{\Lambda'}^{\disc}$)-equivariant
(Lemma \ref{lem:KoszResolGalInv}),
the claims (1), (2), and the commutative diagram \eqref{eq:PiIotaIdFunct} in Proposition 
\ref{prop:GSheafKoszulFunct} allow us to apply Lemma \ref{lem:GammaKoszCompos}
to the composition of $\oof$ and $\alpha_{C}(\CM)$
and to that of $\alpha_{C'}(\CM')$ and $\of$, obtaining the 
following commutative diagram of complexes in $\Mod(C,\CA)$.
\begin{equation}\label{eq:GsheafKoszulCohFunct}
\xymatrix@C=50pt{
\pi_*K^{\bullet}_{\Lambda}(\iota_*\iota^*\CM)
\ar[r]^(.47){\pi_*K^{\bullet}_{\psi}(\oof)}&
\pi_*u_{\psi*}K^{\bullet}_{\Lambda'}(\iota'_*\iota^{\prime*}\CM')
\ar[r]^{\cong}&
u_*\pi'_*K^{\bullet}_{\Lambda'}(\iota'_*\iota^{\prime*}\CM')\\
K^{\bullet}_{\Lambda}(\iota^*\CM)
\ar[u]_{\cong}^{K^{\bullet}_{\Lambda}(\alpha_{C}(\CM))}
\ar[rr]^{K^{\bullet}_{\psi}(\of)}&&
u_*K^{\bullet}_{\Lambda'}(\iota^{\prime*}\CM')
\ar[u]^{\cong}_{u_*K^{\bullet}_{\Lambda'}(\alpha_{C'}(\CM'))}
}
\end{equation}
Combining \eqref{eq:GsheafKoszulCohFunct} with
the commutative diagram \eqref{eq:IotaUnitFunct}, we obtain a 
commutative diagram in $D^+(C,\CA)$
\begin{equation}\label{eq:GammaSheafCohFunct}
\xymatrix{
R\pi_*\CM\ar[r]^(.45){R\pi_*f}&
R\pi_*u_{\psi*}\CM'\ar[r]&
R\pi_*Ru_{\psi*}\CM'
\ar[r]^{\cong}&
Ru_*R\pi^{\prime}_*\CM'\\
K^{\bullet}_{\Lambda}(\iota^*\CM)
\ar[u]_{\cong}^{\eqref{eq:GSheafCohKCpx}}
\ar[rr]^{K^{\bullet}_{\psi}(\of)}&&
u_*K^{\bullet}_{\Lambda'}(\iota^{\prime*}\CM')
\ar[r]&
Ru_*K^{\bullet}_{\Lambda'}(\iota^{\prime*}\CM').
\ar[u]^{\cong}_{\eqref{eq:GSheafCohKCpx}}&
}
\end{equation}

We see that the commutative diagrams
\eqref{eq:IotaUnitFunct}, \eqref{eq:PiIotaIdFunct}, \eqref{eq:GsheafKoszulCohFunct}, 
and \eqref{eq:GammaSheafCohFunct} are
compatible with composition of $f$'s as follows.
Suppose that we are given another pair 
$(\psi',u')$ of a map of finite sets $\Lambda'\to \Lambda''$
and a morphism of ringed sites $(C'',\CA'')\to(C',\CA')$ 
satisfying the same conditions as the pair $(\psi,u)$ 
considered above, and a total order on 
$\Lambda'$ such that the map $\psi\colon 
\Lambda\to \Lambda'$ preserves orders.
We have a commutative diagram \eqref{eq:KoszFunctToposDiag} for
the pair $(\psi',u')$ and also for the pair $(\psi'',u''):=(\psi'\circ\psi, u\circ u')$.
We abbreviate $\iota_{\Lambda'',C''}$ and $\pi_{\Lambda'',C''}$
to $\iota''$ and $\pi''$, respectively. 

\begin{proposition}\label{prop:KoszulFunctCocycCond}
Let $\CM$, $\CM'$, and $\CM''$ be objects
of $\Gamma_{\Lambda}\Mod(C,\CA)$, 
$\Gamma_{\Lambda'}\Mod(C',\CA')$,
and $\Gamma_{\Lambda''}\Mod(C'',\CA'')$,
respectively, and suppose that 
we are given morphisms $f\colon 
\CM\to u_{\psi*}\CM'$ in $\Gamma_{\Lambda}\Mod(C,\CA)$
and $f'\colon \CM'\to u'_{\psi'*}\CM''$ in 
$\Gamma_{\Lambda'}\Mod(C',\CA')$.
We define $f''$ to be the composition 
$\CM\xrightarrow{f}u_{\psi*}\CM'
\xrightarrow{u_{\psi*}f'}u_{\psi*}u'_{\psi'*}\CM''
\cong u''_{\psi''*}\CM''$, and define
$(\of,\oof)$, $(\of',\oof')$, and $(\of'', \oof'')$
by applying the construction of Proposition \ref{prop:GSheafKoszulFunct} (1) and (2) to 
$(u,\psi,f)$, $(u',\psi',f')$, and $(u'',\psi'',f'')$, respectively.\par
(1) Via the canonical isomorphism $u_*\circ u'_*\cong u''_*$
and $u_{\psi*}\circ u'_{\psi'*}\cong u''_{\psi''*}$, we have 
$\of''=u_*\of'\circ \of$ and $\oof''=u_{\psi*}\oof'\circ \oof$.\par
(2) The composition of the diagrams 
\eqref{eq:GsheafKoszulCohFunct}
(resp.~\eqref{eq:GammaSheafCohFunct}) for $(u,\psi,f)$ and $(u',\psi',f')$
coincides with that for $(u'',\psi'',f'')$ under the equalities
in (1).
\end{proposition}

\begin{proof}
(1) The second equality immediately follows from 
the first one, which is reduced to the case $f=\id$
and $f'=\id$, where the claim is nothing but
the compatibility of the base change morphisms 
with the composition of the left squares of 
\eqref{eq:KoszFunctToposDiag} for $(u,\psi)$ and $(u',\psi')$.\par
(2) The equalities in (1) imply that the composition
of the diagrams \eqref{eq:IotaUnitFunct} (resp.~\eqref{eq:PiIotaIdFunct}) for
$(u,\psi,f)$ and $(u',\psi',f')$ coincides with that 
for $(u'',\psi'',f'')$. Hence the claim follows
from Lemma \ref{lem:GammaKoszCompos} as we have assumed that
$\psi\colon \Lambda\to \Lambda'$ preserves
orders. 
\end{proof}

\section{Relative Breuil-Kisin-Fargues modules}
\label{sec:RelativeBKF}
Let $C$ be a perfectoid field of mixed characteristic $(0,p)$ containing
all $p$-power roots of unity, let $\CO$ be the ring of 
integers of $C$, let $\CO^{\flat}$ be the tilt of $\CO$:
$\varprojlim_{\N, \text{Frob}}\CO/p\CO
=\varprojlim_{\N, x\mapsto x^p}\CO$,
which is a perfect ring, let $A_{\inf}$ be $W(\CO^{\flat})$,
which is equipped with a lifting of Frobenius $\varphi$
of the absolute Frobenius of $W(\CO^{\flat})/pW(\CO^{\flat})=\CO^{\flat}$,
and let $\theta\colon A_{\inf}\to \CO$ be the
Fontaine's period map. Since $A_{\inf}$ is $p$-torsion 
free as $\CO^{\flat}$ is perfect, 
the lifting of Frobenius $\varphi$ of $A_{\inf}$
defines a $\delta$-structure on $A_{\inf}$. 
We fix a compatible system 
$\varepsilon=(\zeta_n)_{n\in \N}\in \CO^{\flat}$ of 
primitive $p^n$th roots of unity in $\CO$, and
define $\mu\in A_{\inf}$ to be $[\varepsilon]-1$.
Let $\Z_p[[q-1]]$, $\mu$, $\eta$, and $[n]_q$ $(n\in \Z)$
be as in the first paragraph of \S\ref{sec:PrismCrysQHiggs}.
We regard $A_{\inf}$ as a $\delta$-$\Z_p[[q-1]]$-algebra
by the homomorphism sending $q$ to $[\varepsilon]$. 
The composition $\theta\circ\varphi^{-1}$ induces an 
isomorphism $A_{\inf}/[p]_qA_{\inf}\xrightarrow{\cong}\CO$.
We regard every $p$-adic formal scheme over $\CO$
also as a formal scheme over $A_{\inf}$ via
$\theta\circ\varphi^{-1}$ in the following.
Since $A_{\inf}$ is $(p,\pq)$-adically complete and 
separated, the pair  $(A_{\inf},\pq A_{\inf})$ is a $q$-prism
(Definition \ref{def:qprism} (1)). 

Let $\fX$ be a separated, smooth, $p$-adic formal scheme over $\CO$,
and let $X$ be the adic generic fiber of $\fX$.
\begin{definition}[{\cite[Definition 5.1]{MT}}]
\label{def:FprojCrystalBKFCat}
We define $\BKF(\fX)$ to be the category of {\it relative Breuil-Kisin-Fargues modules
without Frobenius on $\fX$}, which is the full subcategory of the category
of locally finite free $\bA_{\inf,X}$-modules on $X_{\proet}$ consisting
of objects ``trivial modulo $<\mu$".
\end{definition}

In this section, we  recall the construction of the fully faithful functor
(\cite[Theorem 5.15]{MT})
\begin{equation}\label{eq:PrismCrysToBKF}
\bM_{\BKF,\fX}\colon \Crystal^{\fproj}_{\prism}(\fX/A_{\inf})\longrightarrow \BKF(\fX),
\end{equation}
and discuss its compatibility with the inverse image functors
(Proposition \ref{prop:BKFFunctComp}). 
See Definition \ref{def:PrismaticSiteCrystal} (2) for the definition of $\Crystal^{\fproj}_{\prism}(\fX/A_{\inf})$.

Let $\CB_{\fX}$ be the the full subcategory of $X_{\proet}$ consisting
of affinoid perfectoids $V\in \Ob X_{\proet}$ such that the image of $V\to X$
is contained in the adic generic fiber $U=\Spa(A[\frac{1}{p}],A)$ 
of some affine open $\fU=\Spf(A)\subset \fX$. We equip
$\CB_{\fX}$ with the topology induced by that of $X_{\proet}$. 
Since every object of $X_{\proet}$ admits a covering by an object
of $\CB_{\fX}$ (\cite[Corollary 4.7]{ScholzepHTRigid}), the
inclusion functor $\iota_X\colon \CB_{\fX}\to X_{\proet}$ is
continuous and cocontinuous, the restriction functor 
$\iota_X^*\colon X_{\proet}^{\sim}\to\CB_{\fX}^{\sim};
\iota_X^*\CG=\CG\circ \iota_X$ is an equivalence of categories
(\cite[III Th\'eor\`eme 4.1 and its proof]{SGA4},
and its quasi-inverse is given by 
$\iota_{X*}\colon \CB_{\fX}^{\sim}
\to X_{\proet}^{\sim}; (\iota_{X*}\CH)(W)=\varprojlim_{V\in (\CB_{\fX})_{/W}}
\CH(V)$, where $(\CB_{\fX})_{/W}$ is the full subcategory of $(X_{\proet})_{/W}$
consisting of $V\to W$, $V\in \Ob \CB_{\fX}$.

For $V\in \Ob \CB_{\fX}$, put $A_V^+=\Gamma(V,\hCO_{V}^+)$. Then
we have $\bA_{\inf,X}(V)=A_{\inf}(A_V^+)$
(\cite[Theorem 6.5 (i)]{ScholzepHTRigid}), and the morphism 
$v_{\fX,V}\colon 
\Spf(A_{\inf}(A_V^+)/[p]_qA_{\inf}(A_V^+))
=\Spf(A_V^+)\to \fU=\Spf(A)\to \fX$ is independent of the choice of 
$\fU$ whose adic generic fiber contains the image of $V$ in $X$ since 
$\fX$ is assumed to be separated. This construction gives
a functor 
\begin{equation}\alpha_{\inf,\fX}\colon 
\CB_{\fX}\to (\fX/A_{\inf})_{\prism}; V\mapsto
(\bA_{\inf,X}(V),v_{\fX,V}).
\end{equation}

\begin{lemma}
For $\CF\in \Ob (\Crystal^{\fproj}_{\prism}(\fX/A_{\inf}))$, the 
presheaf $\CF_{\BKF}=\CF\circ\alpha_{\inf,\fX}$ is a sheaf 
of $\iota_X^*\bA_{\inf,X}$-modules on $\CB_{\fX}$.
\end{lemma}

\begin{proof}
Let $V\in \Ob \CB_{\fX}$, let $(\CB_{\fX})_{/V}$ be the category of 
objects of $\CB_{\fX}$ lying above $V$ equipped with the topology
induced by that of $\CB_{\fX}$ via the functor 
$j_V\colon (\CB_{\fX})_{/V}\to \CB_{\fX}; (W\to V)\mapsto W$.
The functor $j_V$ is continuous and cocontinuous
(\cite[III Proposition 5.2 2)]{SGA4}). Let $\hj_V^*$
denote the functor $\CB_{\fX}^{\wedge}
\to (\CB_{\fX})_{/V}^{\wedge};\CP\mapsto \CP\circ j_V$,
and let $j_V^*$ denote its restriction
$\CB_{\fX}^{\sim}\to (\CB_{\fX})_{/V}^{\sim}$. Then
we have 
\begin{multline*}(\hj_V^*\CF_{\BKF})(W\to V)
=\CF(\bA_{\inf,X}(W),v_{\fX,W})
\xleftarrow{\cong} \CF(\bA_{\inf,X}(V),v_{\fX,V})\otimes_{\bA_{\inf,X}(V)}\bA_{\inf,X}(W)\\
=\CF_{\BKF}(V)\otimes_{\bA_{\inf,X}(V)}
(j_{V}^*\iota_X^*\bA_{\inf,X})(W).\end{multline*}
Since $\CF_{\BKF}(V)$ is a finite projective $\bA_{\inf,X}(V)$-module,
this implies that $\hj_V^*\CF_{\BKF}$ is a sheaf of 
$j_V^*\iota_X^*\bA_{\inf,X}$-modules on $(\CB_{\fX})_{/V}$. Since 
$j_V^*$ and $\hj_V^*$ commute with the sheafifying functors $a(-)$
(\cite[III Proposition 2.3 2)]{SGA4}), we obtain
$$(a\CF_{\BKF})(V)
=(j_V^*a\CF_{\BKF})(\id_V)\xleftarrow{\cong}
(a\hj_V^*\CF_{\BKF})(\id_V)
\xleftarrow{\cong} (\hj_V^*\CF_{\BKF})(\id_V)=\CF_{\BKF}(V).$$
\end{proof}

\begin{definition}\label{BKFFunctor}
We define the functor $\bM_{\BKF,\fX}$ \eqref{eq:PrismCrysToBKF} 
by $\bM_{\BKF,\fX}(\CF)=\iota_{X*}(\CF\circ \alpha_{\inf,\fX})$
(\cite[Theorem 5.15]{MT}). 
\end{definition}

\begin{lemma}\label{lem:BKFModLocalStr}
Let  $\CF\in \Ob(\Crystal^{\fproj}_{\prism}(\fX/A_{\inf}))$, and put
$\bM=\bM_{\BKF,\fX}(\CF)$. Then, for any $V\in \Ob\CB_{\fX}$, the 
morphism $\bM(V)\otimes_{\bA_{\inf,X}(V)}
\bA_{\inf,X}\vert_V\xrightarrow{\cong}\bM\vert_{V}$
is an isomorphism on $(X_{\proet})_{/V}^{\sim}$
\end{lemma}

\begin{proof}
This follows from the isomorphism
$\CF_{\BKF}(V)\otimes_{\bA_{\inf,X}(V)}
\bA_{\inf,X}(V')\xrightarrow{\cong}\CF_{\BKF}(V')$
for every object $V'$ of $\CB_{\fX}$ lying over $V$,
where $\CF_{\BKF}=\CF\circ \alpha_{\inf,\fX}$. 
\end{proof}

\begin{remark}\label{rmk:PrismTOBKFFrobTensor}
(1) For every $V\in \Ob \CB_{\fX}$, we have $A_V^+=A_V^{\circ}$
(\cite[Example 1.6 (iii)]{MT}), which implies that the 
Frobenius of $\bA_{\inf,X}(V)=A_{\inf}(A_V^+)$ is an automorphism. 
Therefore, the Frobenius of $\bA_{\inf, X}$ is an automorphism, and
for $\CF\in \Ob(\Crystal^{\fproj}_{\prism}(\fX/A_{\inf}))$ and the object
$\varphi^*\CF=\CF\otimes_{\CO_{\fX/A_{\inf}},\varphi}\CO_{\fX/A_{\inf}}$
of $\Crystal^{\fproj}_{\prism}(\fX/A_{\inf})$ (Remark \ref{rmk:PrismCrystalFrobTensor} (2)),
the homomorphism $\CF\to \varphi^*\CF;x\mapsto x\otimes 1$
induces an isomorphism
\begin{equation}\label{eq:BKFFunctMapFrobPB}
\bM_{\BKF,\fX}(\CF)\xrightarrow{\;\cong\;} \bM_{\BKF,\fX}(\varphi^*\CF)
\end{equation}
semilinear over $\varphi\colon \bA_{\inf,X}\xrightarrow{\cong}\bA_{\inf,X}$.
Taking the scalar extension under the Frobenius of $\bA_{\inf,X}$, we obtain 
an $\bA_{\inf,X}$-linear isomorphism
\begin{equation}
\varphi^*(\bM_{\BKF,\fX}(\CF))=\bM_{\BKF,\fX}(\CF)\otimes_{\bA_{\inf,X},\varphi}
\bA_{\inf,X}\xrightarrow{\;\cong\:}\bM_{\BKF,\fX}(\varphi^*\CF).
\end{equation}
\par
(2) For $\CF,\CG\in \Ob(\Crystal^{\fproj}_{\prism}(\fX/A_{\inf}))$ and
$\CF\otimes_{\CO_{\fX/A_{\inf}}}\CG\in \Ob(\Crystal^{\fproj}_{\prism}(\fX/A_{\inf}))$
(Remark \ref{rmk:PrismCrystalFrobTensor} (3)), the $\CO_{\fX/A_{\inf}}$-bilinear
map $\CF\times\CG\to \CF\otimes_{\CO_{\fX/A_{\inf}}}\CG$ induces
an $\bA_{\inf,X}$-bilinear map 
$\bM_{\BKF,\fX}(\CF)\times\bM_{\BKF,\fX}(\CG)
\to \bM_{\BKF,\fX}(\CF\otimes_{\CO_{\fX/A_{\inf}}}\CG)$,
which yields an isomorphism
\begin{equation}\label{eq:BKFFunctorTensor}
\bM_{\BKF,\fX}(\CF)\otimes_{\bA_{\inf,X}}\bM_{\BKF,\fX}(\CG)\xrightarrow{\;\cong\;}
\bM_{\BKF,\fX}(\CF\otimes_{\CO_{\fX/A_{\inf}}}\CG)
\end{equation}
because $\CF(\alpha_{\inf,\fX}(V))\otimes_{\bA_{\inf,X}(V)}
\CG(\alpha_{\inf,\fX}(V))=(\CF\otimes_{\CO_{\fX/A_{\inf}}}\CG)(\alpha_{\inf,\fX}(V))$
for $V\in \Ob\CB_{\fX}$ 
(Remark \ref{rmk:PrismCrystalFrobTensor} (3)).
\end{remark}

We summarize  some basic facts on $\bM$ used in the proof of 
Theorem \ref{thm:PrismCohAinfCohLocComp}.
They are easily derived from some properties of the sheaves $\hCO_{X}^+$ and
$\bA_{\inf,X}$ on $X_{\proet}$ proven in \cite{ScholzepHTRigid}.

\begin{definition}\label{def:AlmostIsom}
For an $A_{\inf}$-module or a sheaf of $A_{\inf}$-modules on a topos $M$,
we say that $M$ is {\it almost zero} and write $M\approx 0$ if it is annihilated
by $[\varphi^{-r}(\varepsilon-1)]$ for every $r\in \N$. 
The subcategory of almost zero modules or sheaves is
 stable under kernels, cokernels, and extensions. 
A morphism $f\colon M\to N$ of
$A_{\inf}$-modules or  sheaves of $A_{\inf}$-modules on a topos is said to be
an {\it almost isomorphism} if the kernel and the cokernel of $f$ are almost zero.
Almost isomorphisms are stable under compositions. 
\end{definition}

Let $w_X\colon X_{\proet}^{\sim}\to X_{\et}^{\sim}$ denote
the morphism of topos induced by the morphism of sites
defined by the inclusion functor $X_{\et}\to X_{\proet}$.
Following \cite[\S5]{MT}, we say that a sheaf of abelian groups
$\CF$ on $X_{\proet}$ is {\it discrete} if it is the pullback of 
a sheaf of abelian groups on $X_{\et}$ under $w_X$. 
For a sheaf of abelian groups $\CG$ on $X_{\et}$, 
the adjunction morphism $\CG\to Rw_{X*}w_X^*\CG$
is an isomorphism \cite[Corollary 3.17 (i)]{ScholzepHTRigid}.
This implies that $\CF$ is discrete if and only if the adjunction
morphism $w_X^*w_{X*}\CF\to \CF$ is an isomorphism, and
that discrete sheaves on $X_{\proet}$ are stable under extensions. 

\begin{proposition}\label{prop:BKFProperties}
Let $\CF\in \Ob(\Crystal^{\fproj}_{\prism}(\fX/A_{\inf}))$,
and put $\bM=\bM_{\BKF,\fX}(\CF)$ and $\bM_m=\bM/(p,\pq)^{m+1}\bM$
for $m\in \N$. Let $V=$`` $\varprojlim_i$" $V_i$ be an object of $\CB_{\fX}$.\par
(1) The morphism $\bM\to \varprojlim_m\bM_m$ is an isomorphism.\par
(2) The homomorphism 
$\bM(V)/(p,\pq)^{m+1}\bM(V)\to \bM_m(V)$ is an almost isomorphism for $m\in \N$.\par
(3) The sheaf $\bM_m$ is discrete for $m\in \N$.\par
(4) The homomorphism $\varinjlim_i H^r(V_i,\bM_m)\to H^r(V,\bM_m)$
is an isomorphism for $r\in \N$, and the latter cohomology is almost zero if $r>0$.
\end{proposition}

\begin{proof}
Put $\tI=pA_{\inf}+\pq A_{\inf}$ to simplify the notation. 
Let $W$ be an object of $\CB_{\fX}$. Since $\bM(W)$ is a finite
projective $\bA_{\inf,X}(W)$-module and the sequence
$p$, $\pq$ is $\bA_{\inf,X}(W)$-regular, the homomorphisms
$\bM(W)/\tI\bM(W)\to \tI^{m+1}\bM(W)/\tI^{m+2}\bM(W)$
induced by the multiplication by $p^i\pq^{m+1-i}$
on $\bM(W)$ for $i\in \N\cap [0,m+1]$ induce an exact sequence
\begin{equation}\label{eq:BKFSectionGr}
0\longrightarrow
\bigoplus_{i=0}^{m+1}\bM(W)/\tI\bM(W)\longrightarrow
\bM(W)/\tI^{m+2}\bM(W)
\longrightarrow
\bM(W)/\tI^{m+1}\bM(W)
\longrightarrow 0.
\end{equation}
Varying $W$, taking the associated sheaf on $\CB_{\fX}$,
and applying $\iota_{X*}$, we obtain an exact sequence
\begin{equation}\label{eq:BKFGr}
0\longrightarrow \bigoplus_{i=0}^{m+1}\bM_0
\longrightarrow \bM_{m+1}\longrightarrow \bM_m\longrightarrow 0.
\end{equation}
Note that the functors $\iota_X^*\colon X_{\proet}^{\sim}\to \CB_{\fX}^{\sim}$
and $\widehat{\iota_X}^*\colon X_{\proet}^{\wedge}\to \CB_{\fX}^{\wedge};
\CP\mapsto \CP\circ \iota_X$ commute with the functors $a$
sending presheaves to their associated sheaves: $a\widehat{\iota_X}^*\cong \iota_X^*a$
by \cite[III Proposition 2.3 2)]{SGA4}.

(1) By Lemma \ref{lem:BKFModLocalStr}, it suffices to prove
$\varprojlim_m\bA_{\inf,X}/\tI^{m+1}\bA_{\inf,X}\cong\bA_{\inf,X}$.
This follows from 
$\bA_{\inf,X}=\varprojlim_rW_r(\hCO_{X^{\flat}}^{+})$,
$W_r(\hCO_{X^{\flat}}^{+})\cong\bA_{\inf,X}/p^r\bA_{\inf,X}$,
and $\bA_{\inf,X}/p^r\bA_{\inf,X}\cong\varprojlim_m\bA_{\inf,X}/(\pq^m,p^r)$.
Since $\bA_{\inf,X}$ is $p$-torsion free, the last isomorphism is
reduced to the case $r=1$. In this case, it is derived from 
\cite[Lemma 5.11 (i)]{ScholzepHTRigid}
which implies that $\frac{\varepsilon-1}{\varphi^{-1}(\varepsilon-1)}$ is $\hCO_{X^{\flat}}^{+}$-regular
and the kernel of the projection to the
$n$th component $\hCO_{X^{\flat}}^{+}\to \CO_X^+/p\CO_X^+$
is generated by $\left(\frac{\varepsilon-1}{\varphi^{-1}(\varepsilon-1)}\right)^{p^{n-1}}$. \par
(2) By Lemma \ref{lem:BKFModLocalStr}, we have 
$\bM_m(V)\cong \bM(V)\otimes_{\bA_{\inf,X}(V)}(\bA_{\inf,X}/\tI^{m+1}\bA_{\inf,X})(V)$.
Hence it suffices to prove the claim for $\bM=\bA_{\inf, X}$. 
By \eqref{eq:BKFSectionGr} for $W=V$ and 
\eqref{eq:BKFGr}, the claim is reduced to the case $m=0$.
By \cite[Theorem 6.5 (i), Lemma 4.10 (i), (iii)]{ScholzepHTRigid},
we have isomorphisms $\bA_{\inf,X}(V)/\pq\bA_{\inf,X}(V)\xrightarrow{\cong}\hCO_X^+(V)$
and $\bA_{\inf,X}/\pq\bA_{\inf,X}\xrightarrow{\cong}\hCO_X^+$,
and an almost isomorphism $\hCO_X^+(V)/p\hCO_X^+(V)\xrightarrow{\approx}\CO^+_X/p\CO^+_X(V)
\cong \hCO_X^+/p\hCO_X^+(V)$.\par
(3) By \eqref{eq:BKFGr}, the claim is reduced to the case $m=0$ by induction.
Since $\bM/p\bM$ is a locally free $\bA_{\inf,X}/p\bA_{\inf,X}=\CO_{X^{\flat}}^{+}$-module,
$\varphi^{-1}(\pq)=\mu/\varphi^{-1}(\mu)$ is $\hCO_{X^{\flat}}^+$-regular, 
and $\bM_0\cong(\bM/p\bM)/\varphi^{-1}(\pq)^p(\bM/p\bM)$, 
it is further reduced to the claim that $\bM/(p,\varphi^{-1}(\pq))\bM$ is discrete,
which is a consequence of $\bA_{\inf,X}/(p,\varphi^{-1}(\pq))\bA_{\inf,X}\cong 
\CO_X^+/p\CO_X^+$
since $\bM$ is trivial modulo $\varphi^{-1}(\pq)$ (\cite[Definition 5.1]{MT}).\par
(4) The first claim follows from (3) and \cite[Lemma 3.16]{ScholzepHTRigid}.
Similarly to the proof of (3), the exact sequence
\eqref{eq:BKFGr} and  the triviality of $\bM$ modulo $\varphi^{-1}(\pq)$
reduces the claim to $H^r(V,\CO_X^+/p\CO_X^+)\approx 0$ $(r>0)$
\cite[the proof of Lemma 4.10 (v)]{ScholzepHTRigid}.
\end{proof}

Let us discuss the compatibility of $\bM_{\BKF,\fX}$ 
with inverse image functors. Let $\fX'$
be another separated, smooth, $p$-adic formal scheme over 
$\CO$, and let $g\colon \fX'\to \fX$ be a morphism
over $\CO$. We define $X'$, $\iota_{X'}\colon \CB_{\fX'}
\hookrightarrow X'_{\proet}$ and 
$v_{\fX',V'}\colon \Spf(A_{V'}^+)\to \fX'$
for $V'\in \Ob \CB_{\fX'}$ associated to $\fX'$
in the same way as $X$, $\iota_X$ and $v_{\fX,V}$ 
defined above associated to $\fX$. 
Let $\bmg\colon X'\to X$ be the morphism of 
adic spaces associated to $g$. We define 
$\CB_g$ to be the category of morphisms 
$u\colon V'\to V$ $(V\in \Ob \CB_{\fX}, V'\in \Ob \CB_{\fX'})$
compatible with $\bmg$ such that there exists
a pair of affine opens $\fU'\subset \fX'$ and $\fU\subset \fX$
such that $g(\fU')\subset \fU$ and the adic
generic fibers of $\fU$ and $\fU'$ contain
the images of $V$ and $V'$ in $X$ and $X'$, respectively.
Let $\CF\in \Ob(\Crystal_{\prism}^{\fproj}(\fX/A_{\inf}))$
and put $\CF'=g_{\prism}^{-1}\CF$ 
(Definition \ref{def:PrismaticSiteCrystal} (3)), 
which belongs
to $\Crystal_{\prism}^{\fproj}(\fX'/A_{\inf})$.
Put $\CF_{\BKF}=\CF\circ \alpha_{\inf,\fX}$,
$\CF'_{\BKF}=\CF'\circ\alpha_{\inf,\fX'}$, 
$\bM=\bM_{\BKF,\fX}(\CF)=\iota_{X*}\CF_{\BKF}$,
and $\bM'=\bM_{\BKF,\fX'}(\CF')=\iota_{X'*}\CF'_{\BKF}$.
For $u\colon V'\to V\in \Ob\CB_g$, the morphism 
$\Spf(A_{V'}^+)\to \Spf(A_V^+)$ 
is compatible with $g\colon \fX'\to \fX$
via $v_{\fX,V}$ and $v_{\fX',V'}$, thereby 
defining a morphism 
$(\bA_{\inf,X'}(V'),g\circ v_{\fX',V'})
\to (\bA_{\inf,X}(V), v_{\fX,V})$ in
$(\fX/A_{\inf})_{\prism}$, which is obviously
functorial in $u$. Evaluating $\CF$ on the morphism,
we obtain a morphism $\CF_{\BKF}(u)\colon 
\CF_{\BKF}(V)\to \CF'_{\BKF}(V')$ functorial in $u$.

\begin{proposition}\label{prop:BKFFunctComp}
(1) There exists a unique $\bA_{\inf,X}$-linear morphism $\varepsilon_{g,\CF}\colon 
\bM\to \bmg_{\proet*}\bM'$ such that for any $W\in \Ob X_{\proet}$
and a left commutative diagram below with 
$u\colon V'\to V\in \Ob \CB_g$
and the left (resp.~right) vertical morphism belonging
to $X_{\proet}$ (resp.~$X'_{\proet})$, 
the right diagram below is commutative.
\begin{equation*}
\xymatrix{
W&\ar[l] W\times_XX'\\
V\ar[u]&\ar[l]_u V'\ar[u]
}\qquad
\xymatrix@C=50pt{
\bM(W)\ar[r]^(.42){\varepsilon_{g,\CF}(W)} \ar[d]&\bM(W\times_XX')\ar[d]\\
\bM(V)\ar[r]^{\CF_{\BKF}(u)} &\bM(V')
}
\end{equation*}

(2) The left adjoint $\eta_{g,\CF}\colon 
\bmg_{\proet}^*\bM:=
\bA_{\inf,X'}\otimes_{\bmg^{-1}_{\proet}(\bA_{\inf,X})}
\bmg_{\proet}^{-1}(\bM)\to \bM'$ 
of the $\bA_{\inf,X}$-linear morphism $\varepsilon_{g,\CF}$
is an isomorphism. This is functorial in $\CF$ and gives an 
isomorphism 
\begin{equation}\eta_g\colon \bmg_{\proet}^*\circ\bM_{\BKF,\fX}
\xrightarrow{\;\cong\;} \bM_{\BKF,\fX'}\circ g_{\prism}^*.
\end{equation}
\par

(3) Let $g'\colon \fX''\to \fX'$ be an $\CO$-morphism 
from a separated, smooth, $p$-adic formal scheme $\fX''$ 
over $\CO$, and let $\bmg'\colon X''\to X'$
be its adic generic fiber. Then the following diagram is commutative
\begin{equation}
\xymatrix@C=50pt{
\bmg_{\proet}^{\prime*}\bmg_{\proet}^*\bM_{\BKF,\fX}
\ar[r]^{\cong}_{\bmg^{\prime*}_{\proet}(\eta_g)}&
\bmg^{\prime*}_{\proet}\bM_{\BKF,\fX'}g_{\prism}^*
\ar[r]^{\cong}_{\eta_{g'}\circ g_{\prism}^*}&
\bM_{\BKF,\fX''}g^{\prime*}_{\prism}g^*_{\prism}\\
(\bmg\circ\bmg')_{\proet}^*\bM_{\BKF,\fX}
\ar@{-}[u]^{\cong}
\ar[rr]^{\cong}_{\eta_{g\circ g'}}&&
\bM_{\BKF,\fX''}(g\circ g')^*_{\prism}
\ar@{-}[u]_{\cong}
}
\end{equation}
\end{proposition}

\begin{proof}
(1) Any $W\in \Ob X_{\proet}$ admits a 
covering $(V_{\alpha}\to W)_{\alpha\in A}$
by objects $V_{\alpha}$ of $\CB_{\fX}$ and,
for each $\alpha\in A$, $V_{\alpha}'=V_{\alpha}\times_XX'$
is covered by objects of $\CB_{\fX'}$ as 
$(V'_{\alpha\beta}\to V_{\alpha}')_{\beta\in A_{\alpha}}$
in such a way that the composition $u_{\alpha\beta}\colon V'_{\alpha\beta}
\to V'_{\alpha}\to V_{\alpha}$ belongs to $\CB_g$.
This implies the uniqueness. 
Put $W'=W\times_XX'$, 
$A'=\{(\alpha,\beta)\,\vert\,
\alpha\in A, \beta\in A_{\alpha}\}$, 
let $\varphi\colon A'\to A$ be the morphism 
defined by $(\alpha,\beta)\mapsto \alpha$,
and put $V_{\alpha_1\alpha_2}=V_{\alpha_1}\times_WV_{\alpha_2}$
$(\alpha_1,\alpha_2\in A)$
and $V'_{\alpha'_1\alpha'_2}=
V'_{\alpha'_1}\times_{W'}V'_{\alpha'_2}$
$(\alpha'_1,\alpha_2'\in A'$). 
Note that $(V'_{\alpha'}\to V'_{\varphi(\alpha')}\to W')_{\alpha'\in A'}$
is a covering of $W'$. 
We can further take a covering $V_{\alpha_1\alpha_2;\gamma}\to 
V_{\alpha_1\alpha_2}$ $(\gamma\in A_{\alpha_1\alpha_2})$
by objects of $\CB_{\fX}$ for each $(\alpha_1,\alpha_2)
\in A^2$, and then a covering
$V'_{\alpha'_1\alpha_2';\gamma\gamma'}
\to V'_{\alpha'_1\alpha'_2;\gamma}:=V_{\alpha_1\alpha_2;\gamma}\times_{V_{\alpha_1\alpha_2}}
V'_{\alpha_1'\alpha_2'}
$
$(\gamma'\in A_{\alpha_1'\alpha_2';\gamma})$
by objects of $\CB_{\fX'}$ for $\alpha_i'\in A'$,
$\alpha_i=\varphi(\alpha_i')$ $(i=1,2)$, and $\gamma\in A_{\alpha_1\alpha_2}$
such that the composition 
$u_{\alpha'_1\alpha'_2;\gamma\gamma'}\colon V'_{\alpha'_1\alpha'_2;\gamma\gamma'}
\to V'_{\alpha'_1\alpha'_2;\gamma}\to V_{\alpha_1\alpha_2;\gamma}$
belongs to $\CB_g$. 
The morphisms $\CF_{\BKF}(u_{\alpha'})$
and $\CF_{\BKF}(u_{\alpha'_1\alpha_2';\gamma\gamma'})$
induce an  $\bA_{\inf,X}(W)$-linear morphism 
$\varepsilon(W)\colon \bM(W)\to \bM'(W')$ since the domain
(resp.~the codomain) is the difference kernel of
$\prod\CF_{\BKF}(V_{\alpha})
\rightrightarrows \prod\CF_{\BKF}(V_{\alpha_1\alpha_2;\gamma})$
(resp.~$\prod \CF'_{\BKF}(V'_{\alpha'})
\rightrightarrows \prod \CF'_{\BKF}(V'_{\alpha'_1\alpha'_2;\gamma\gamma'})$.
We see that $\varepsilon(W)$ satisfies the desired property
for any morphism from an object $u\colon V'\to V$ of $\CB_g$
to $W'\to W$ over $\bmg\colon X'\to X$ by adding $V$ and a covering of $V\times_XX'$ containing
$V'$ to the coverings
$(V_{\alpha}\to W)_{\alpha\in A}$ and $(V'_{\alpha'}\to W')_{\alpha'\in A'}$,
respectively.  This compatibility with $\CF_{\BKF}(u)$ allows us to show that 
$\varepsilon(W)$ is functorial in $W$ and therefore defines
an $\bA_{\inf,X}$-linear morphism $\bM\to \bmg_{\proet*}\bM'$.\par
(2) It suffices to prove that the restriction of 
$\eta_{g,\CF}\colon\bmg_{\proet}^*\bM\to \bM'$ to $(X'_{\proet})_{/V'}$ 
is an isomorphism for any $u\colon V'\to V\in \Ob \CB_g$.
Put $V''=V\times_XX'$, let $u'\colon V'\to V''$
be the morphism in $X'_{\proet}$ induced by $u$,
let $j_{u'}\colon (X'_{\proet})_{/V'}^{\sim}
\to (X'_{\proet})_{/V''}^{\sim}$ be the morphism 
of topos induced by $u'$, and let $\bmg_{\proet,V}$ be the
morphism of topos
$(X'_{\proet})_{/V''}^{\sim}\to (X_{\proet})_{/V}^{\sim}$
induced by the morphism of sites
$(X_{\proet})_{/V}\to (X'_{\proet})_{/V''}$
defined by taking the pullback by $\bmg\colon X'\to X$. 
Then the restriction of $\eta_{g,\CF}$ to 
$(X_{\proet})_{/V'}$ is  the adjoint
of $\bM\vert_V\xrightarrow{\varepsilon_{g,\CF}\vert_{V}}
 \bmg_{\proet, V*}(\bM'\vert_{V''})
\to \bmg_{\proet,V*}j_{u'*}j_{u'}^*(\bM'\vert_{V''})
=\bmg_{\proet,V*}j_{u'*}(\bM'\vert_{V'})$,
whose section over $V$ is simply given by $\CF_{\BKF}(u)$. 
Hence the composition of $(\eta_{g,\CF}\vert_{V'})(V')\colon
(\bmg_{\proet}^*\bM)(V')
\to \bM'(V')$ with $\CF_{\BKF}(V)
=\bM(V)\to \bmg_{\proet*}\bmg_{\proet}^*\bM(V)=
\bmg_{\proet}^*\bM(V'')
\xrightarrow{\bmg_{\proet}^*\bM(u')}\bmg_{\proet}^*\bM(V')$ coincides with 
$\CF_{\BKF}(u)\colon \CF_{\BKF}(V)
\to \CF'_{\BKF}(V')$. 
By Lemma \ref{lem:BKFModLocalStr}, we have
$(\bmg_{\proet}^*\bM)\vert_{V'}
\cong j_{u'}^*\bmg_{\proet,V}^*(\bM\vert_V)
\cong\CF_{\BKF}(V)\otimes_{\bA_{\inf,X}(V)}
\bA_{\inf,X'}\vert_{V'}$
and 
$\bM'\vert_{V'}
\cong\CF'_{\BKF}(V')\otimes_{\bA_{\inf,X'}(V')}
\bA_{\inf,X'}\vert_{V'}$. 
Now the above observation
on $(\eta_{g,\CF}\vert_{V'})(V')$ implies that 
$\eta_{g,\CF}\vert_{V'}$ is given by the isomorphism 
$\CF_{\BKF}(V)\otimes_{\bA_{\inf,X}(V)}
\bA_{\inf,X'}\vert_{V'}
\xrightarrow{\cong} \CF'_{\BKF}(V')
\otimes_{\bA_{\inf,X'}(V')}\bA_{\inf,X'}\vert_{V'}$
induced by $\CF_{\BKF}(u)$.\par
(3) Let $\CF\in \Ob(\Crystal^{\fproj}_{\prism}(\fX/A_{\inf}))$, and
put $\CF'=g_{\prism}^{-1}\CF$, $\CF''=(g^{\prime}_{\prism})^{-1}\CF'
=(g\circ g')_{\prism}^{-1}\CF$, $\bM=\bM_{\BKF,\fX}(\CF)$,
$\bM'=\bM_{\BKF,\fX'}(\CF')$, and $\bM''=\bM_{\BKF,\fX''}(\CF'')$.
Then it suffices to prove that the 
composition 
$\bmg_{\proet*}(\varepsilon_{g',\CF'})\circ
\varepsilon_{g,\CF}\colon 
\bM\to \bmg_{\proet*}\bM'
\to \bmg_{\proet*}\bmg_{\proet*}^{\prime}\bM''$
coincides with $\varepsilon_{g\circ g',\CF}
\colon \bM\to (\bmg\circ\bmg')_{\proet*}\bM''$
via the canonical isomorphism between the
codomains. For any $W\in \Ob X_{\proet}$,
$W'=W\times_XX'\in \Ob X'_{\proet}$, 
and $W''=W'\times_{X'}X''\in \Ob X''_{\proet}$,
there exist coverings
$(V_{\alpha}\to W)_{\alpha\in A}$,
$(V'_{\alpha'}\to W')_{\alpha'\in A'}$,
and $(V''_{\alpha''}\to W'')_{\alpha''\in A''}$
by objects of $\CB_{\fX}$, $\CB_{\fX'}$, and $\CB_{\fX''}$,
respectively, maps $\varphi\colon A'\to A$
and $\varphi'\colon A''\to A'$, and 
morphisms $u_{\alpha'}\colon V'_{\alpha'}\to V_{\varphi(\alpha')}
\in \Ob \CB_g$ $(\alpha'\in A')$ 
(resp.~$u'_{\alpha''}\colon V''_{\alpha''}\to 
V'_{\varphi'(\alpha'')}\in \Ob \CB_{g'}$
$(\alpha''\in A'')$) compatible with $W'\to W$
(resp.~$W''\to W'$) and satisfying $u_{\varphi(\alpha'')}\circ u'_{\alpha''}
\in \Ob \CB_{g\circ g'}$ for every $\alpha''\in A''$.  
Hence, by the characterization
of the morphism $\varepsilon_{g,\CF}$ in (1),
the claim is reduced to $\CF'_{\BKF}(u'_{\alpha''})
\circ \CF_{\BKF}(u_{\varphi(\alpha'')})=
\CF_{\BKF}(u_{\varphi(\alpha'')}\circ
u'_{\alpha''})$, which immediately follows
from the definition.
\end{proof}

\begin{remark}
(1) For $\CF\in \Ob(\Crystal^{\fproj}_{\prism}(\fX/A_{\inf}))$, the following
diagram is commutative.
\begin{equation}\label{eq:BKFFunctFrobFunct}
\xymatrix{
\bM_{\BKF,\fX}(\CF)\ar[rr]^{\varepsilon_{g,\CF}}
\ar[d]^{\eqref{eq:BKFFunctMapFrobPB}}
&&
\bmg_{\proet*}\bM_{\BKF,\fX'}(g_{\prism}^{-1}(\CF))
\ar[d]^{\eqref{eq:BKFFunctMapFrobPB}}\\
\bM_{\BKF,\fX}(\varphi^*\CF)
\ar[r]^(.4){\varepsilon_{g,\varphi^*\CF}}&
\bmg_{\proet*}\bM_{\BKF,\fX'}(g_{\prism}^{-1}(\varphi^*\CF))
\ar@{=}[r]&
\bmg_{\proet*}\bM_{\BKF,\fX'}(\varphi^*(g_{\prism}^{-1}(\CF)))
}
\end{equation}
Indeed, for any $W\in \Ob(X_{\proet})$, taking
$u_{\alpha\beta}\colon V'_{\alpha\beta}\to V_{\alpha}$
as in the proof of Proposition \ref{prop:BKFFunctComp} (1), the commutativity
for the sections on $W$ is reduced to the compatibility of 
$\CF_{\BKF}(u_{\alpha\beta})$ and $(\varphi^*\CF)_{\BKF}(u_{\alpha\beta})$
with the map $\CF(P)\to \varphi^*\CF(P);x\mapsto x\otimes 1$
for $P=\alpha_{\inf,\fX}(V_{\alpha})$
and $g_{\prism}\circ\alpha_{\inf,\fX'}(V'_{\alpha\beta})$.\par
(2) For $\CF, \CG\in \Ob(\Crystal_{\prism}^{\fproj}(\fX/A_{\inf}))$,
the following diagram is commutative, where the upper horizontal
map is induced by $\varepsilon_{g,\CF}$ and $\varepsilon_{g,\CG}$.
\begin{equation}\label{eq:BKFFunctorTensorFunct}
\xymatrix{
\bM_{\BKF,\fX}(\CF)\otimes_{\bA_{\inf,X}}
\bM_{\BKF,\fX}(\CG)\ar[r]\ar[d]_{\cong}^{\eqref{eq:BKFFunctorTensor}}&
\bmg_{\proet*}(\bM_{\BKF,\fX'}(g_{\prism}^{-1}(\CF))\otimes_{\bA_{\inf,X'}}
\bM_{\BKF,\fX'}(g_{\prism}^{-1}(\CG)))\ar[d]_{\cong}^{\eqref{eq:BKFFunctorTensor}}\\
\bM_{\BKF,\fX}(\CF\otimes_{\CO_{\fX/A_{\inf}}}\CG)
\ar[r]^(.4){\varepsilon_{g,\CF\otimes\CG}}&
\bmg_{\proet*}\bM_{\BKF,\fX'}(g_{\prism}^{-1}(\CF)\otimes_{\CO_{\fX'/A_{\inf}}}g_{\prism}^{-1}(\CG))
}
\end{equation}
It suffices to prove the claim after replacing the tensor products $\otimes$
on the upper line by the products $\times$. Then, similarly to (1) above,
the commutativity for the sections on $W\in \Ob(X_{\proet})$ is reduced
to the compatibility of $\CH_{\BKF}(u_{\alpha\beta})$ for
$\CH=\CF$, $\CG$ and $\CF\otimes_{\CO_{\fX/A_{\inf}}}\CG$ with the
product $\CF(P)\times \CG(P)\to (\CF\otimes_{\CO_{\fX/A_{\inf}}}\CG)(P)$
for $P=\alpha_{\inf,\fX}(V_{\alpha})$ and 
$g_{\prism}\circ\alpha_{\inf,\fX'}(V'_{\alpha\beta})$.
\end{remark}

Put $\bA_{\inf,X,m}=\bA_{\inf,X}/(p,\pq)^{m+1}$ for $m\in \N$,
and let $\ubA_{\inf,X}$ denote the inverse system of sheaves
of rings $(\bA_{\inf,X,m})_{m\in \N}$ on $X_{\proet}$.

\begin{definition}\label{eq:AinfCoh}
For $\CF\in \Ob \Crystal_{\prism}^{\fproj}(\fX/A_{\inf})$
and $\bM=\bM_{\BKF,\fX}(\CF)$, we define 
$A\Omega_{\fX}(\bM)$ to be $L\eta_{\mu}R\nu_{\fX*}
R\varprojlim_{\N}\ubM\in D(\fX_{\Zar},A_{\inf})$
(\cite[\S6]{MT}), where  $\nu_{\fX}$ denotes the
morphism of topos $X_{\proet}^{\sim}
\to \fX_{\Zar}^{\sim}$ and $\ubM$ is the $\ubA_{\inf,X}$-module
$\bM\otimes_{\bA_{\inf,X}}\ubA_{\inf,X}
=(\bM/(p,\pq)^{m+1}\bM)_{m\in \N}$.
(Note that we have $R\varprojlim_m(\bM/(p^n,\pq^m)\bM)=\bM/p^n\bM$
for $n>0$  by \cite[Proposition 5.4 (i), Example 5.2 (ii), Proposition 5.13]{MT}.
This implies $R\varprojlim_{\N}\ubM=R\varprojlim_n\bM/p^n\bM$.)
\end{definition}

\begin{remark}\label{rmk:FrobProdAOmega}
(1) For $\CF\in \Ob \Crystal_{\prism}^{\fproj}(\fX/A_{\inf})$, the
semilinear map \eqref{eq:BKFFunctMapFrobPB} induces an isomorphism in $D(\fX_{\Zar},A_{\inf})$
\begin{equation}\label{eq:AinfCohBKFFrobPB}
A\Omega_{\fX}(\bM_{\BKF,\fX}(\CF))\overset{\cong}\longrightarrow
\varphi_*L\eta_{[p]_q}A\Omega_{\fX}(\bM_{\BKF,\fX}(\varphi^*\CF)),
\end{equation}
where $\varphi_*$ denotes the restriction of scalars under
$\varphi\colon A_{\inf}\to A_{\inf}$, as follows.
Let $\ubM$ and $\ubM_{\varphi}$ be the $\ubA_{\inf,X}$-modules
$\bM_{\BKF,\fX}(\CF)\otimes_{\bA_{\inf,X}}\ubA_{\inf,X}$ and 
$\bM_{\BKF,\fX}(\varphi^*\CF)\otimes_{\bA_{\inf,X}}\ubA_{\inf,X}$,
respectively.
Then the isomorphism is obtained by composing
\begin{multline*}
L\eta_{\mu}R\nu_{\fX*}R\varprojlim_{\N} \ubM
\xrightarrow{\cong}
L\eta_{\mu}R\nu_{\fX*}R\varprojlim_{\N} \uvarphi_*\ubM_{\varphi}
\cong L\eta_{\mu}\varphi_*R\nu_{\fX*}R\varprojlim_{\N} \ubM_{\varphi}\\
\cong \varphi_*L\eta_{\varphi(\mu)}R\nu_{\fX*}R\varprojlim_{\N} \ubM_{\varphi}
\cong \varphi_*L\eta_{\pq}L\eta_{\mu}R\nu_{\fX*}R\varprojlim_{\N} \ubM_{\varphi}.
\end{multline*}
Here $\uvarphi_*$ denotes the restriction of 
scalars under $\uvarphi=(\varphi\mod (p,\pq)^{m+1})_{m\in \N}\colon
\ubA_{\inf,X}\to \ubA_{\inf,X}$, the first isomorphism is
induced by \eqref{eq:BKFFunctMapFrobPB},
and the last isomorphism is
obtained by \cite[Lemma 6.11]{BMS} and $\varphi(\mu)=\pq\mu$.

(2) For $\CF_1, \CF_2\in \Ob\Crystal_{\prism}^{\fproj}(\fX/A_{\inf})$
and $\CF_3=\CF_1\otimes_{\CO_{\fX/A_{\inf}}}\CF_2
\in \Ob\Crystal_{\prism}^{\fproj}(\fX/A_{\inf})$ 
(Remark \ref{rmk:PrismCrystalFrobTensor} (3)), we have
a morphism in $D(\fX_{\Zar},A_{\inf})$
\begin{equation}\label{eq:BKFAinfCohProd}
A\Omega_{\fX}(\bM_{\BKF,\fX}(\CF_1))\otimes^L_{A_{\inf}}
A\Omega_{\fX}(\bM_{\BKF,\fX}(\CF_2))\longrightarrow
A\Omega_{\fX}(\bM_{\BKF,\fX}(\CF_3))
\end{equation}
defined by the composition below, where 
$\ubM_{\nu}=\bM_{\BKF,\fX}(\CF_{\nu})\otimes_{\bA_{\inf,X}}\ubA_{\inf,X}$
for $\nu\in\{1,2,3\}$.
\begin{multline*}
(L\eta_{\mu}R\nu_{\fX*}R\varprojlim_{\N}\ubM_1)\otimes_{A_{\inf}}^L
(L\eta_{\mu}R\nu_{\fX*}R\varprojlim_{\N}\ubM_2)
\longrightarrow L\eta_{\mu}((R\nu_{\fX*}R\varprojlim_{\N}\ubM_1)\otimes_{A_{\inf}}^L
(R\nu_{\fX*}R\varprojlim_{\N}\ubM_2))\\
\longrightarrow L\eta_{\mu}R\nu_{\fX*}R\varprojlim_{\N}(\ubM_1\otimes^L_{\ubA_{\inf,X}}\ubM_2)
\longrightarrow L\eta_{\mu}R\nu_{\fX*}R\varprojlim_{\N}\ubM_3.
\end{multline*}
The first map is defined by \cite[Proposition 6.7]{BMS}, the second one is induced by the cup product
\cite[0B6C]{Stacks},
and the third one is given by \eqref{eq:BKFFunctorTensor}.
\end{remark}

\begin{lemma}\label{lem:DecalageMorphRingedTopos}
Let $f\colon (E',\CA')\to (E,\CA)$ be a flat morphism of ringed topos,
let $\CI$ be an invertible ideal of $\CA$, and put $\CI'=f^*\CI\subset \CA'$.\par
(1) We have a natural isomorphism of functors \cite[Lemma 6.14]{BMS}
\begin{equation}
\alpha_{f,\CI}\colon Lf^*L\eta_{\CI}\xrightarrow{\cong} L\eta_{\CI'}Lf^*\colon D(E,\CA)\to D(E',\CA').
\end{equation}
By taking the adjoint of the composition of 
$\alpha_{f,\CI}\circ Rf_*$ with $L\eta_{\CI'}\circ\text{(counit)}\colon
L\eta_{\CI'}Lf^*Rf_*\to L\eta_{\CI'}$, we obtain a  morphism of functors
\begin{equation}\label{eq:DeclageDirectImage}
\beta_{f,\CI}\colon L\eta_{\CI}Rf_*\to Rf_*L\eta_{\CI'}\colon D(E',\CA')\to D(E,\CA).
\end{equation}

(2) Let $g\colon (E'',\CA'')\to (E',\CA')$ be another flat morphism of ringed topos,
and put $\CI''=g^*\CI'=(f\circ g)^*\CI\subset \CA''$. Then the following diagrams
are commutative.
\begin{equation}
\xymatrix@R=10pt@C=50pt{
Lg^*Lf^*L\eta_{\CI}\ar@{-}[d]^{\cong}
\ar[r]^{\cong}_{Lg^*\circ \alpha_{f,\CI}}
&
Lg^*L\eta_{\CI'}Lf^*
\ar[r]_{\alpha_{g,\CI'}\circ Lf^*}^{\cong}
&
L\eta_{\CI''}Lg^*Lf^*\ar@{-}[d]_{\cong}
\\
L(f\circ g)^*L\eta_{\CI}
\ar[rr]^{\cong}_{\alpha_{f\circ g,\CI}}
&&
L\eta_{\CI''}L(f\circ g)^*
}
\end{equation}
\begin{equation}\label{eq:DecalageDirectImComp}
\xymatrix@R=10pt@C=50pt{
L\eta_{\CI}Rf_*Rg_*\ar[r]^{\beta_{f,\CI}\circ Rg_*}\ar@{-}[d]_{\cong}&
Rf_*L\eta_{\CI'}Rg_*\ar[r]^{Rf_*\circ \beta_{g,\CI'}}&
Rf_*Rg_*L\eta_{\CI''}\ar@{-}[d]^{\cong}\\
L\eta_{\CI}R(f\circ g)_*\ar[rr]^{\beta_{f\circ g,\CI}}&&
R(f\circ g)_*L\eta_{\CI''}}
\end{equation}

(3) Let $\CJ$ be another invertible ideal of $\CA$, and put $\CJ'=f^*\CJ\subset \CA'$. Then the
isomorphisms $L\eta_{\CI\CJ}\cong L\eta_{\CI}L\eta_{\CJ}$
and $L\eta_{\CI'\CJ'}\cong L\eta_{\CI'}L\eta_{\CJ'}$ \cite[Lemma 6.11]{BMS}
make the following diagrams commutative.
\begin{equation}
\xymatrix@R=10pt@C=50pt{
Lf^*L\eta_{\CI}L\eta_{\CJ}\ar[r]^{\cong}_{\alpha_{f,\CI}\circ L\eta_{\CJ}}\ar@{-}[d]_{\cong}&
L\eta_{\CI'}Lf^*L\eta_{\CJ}\ar[r]^{\cong}_{L\eta_{\CI'}\circ\alpha_{f,\CJ}}&
L\eta_{\CI'}L\eta_{\CJ'}Lf^*\ar@{-}[d]^{\cong}\\
Lf^*L\eta_{\CI\CJ}\ar[rr]^{\cong}_{\alpha_{f,\CI\CJ}}&&
L\eta_{\CI'\CJ'}Lf^*}
\end{equation}
\begin{equation}\label{eq:DecalageCompDerictIm}
\xymatrix@R=10pt@C=50pt{
L\eta_{\CI}L\eta_{\CJ}Rf_*\ar@{-}[d]_{\cong}
\ar[r]^{L\eta_{\CI}\circ \beta_{f,\CJ}}&
L\eta_{\CI}Rf_*L\eta_{\CJ'}
\ar[r]^{\beta_{f,\CI}\circ L\eta_{\CJ'}}&
Rf_*L\eta_{\CI'}L\eta_{\CJ'}\ar@{-}[d]^{\cong}\\
L\eta_{\CI\CJ}Rf_{*}\ar[rr]^{\beta_{f,\CI\CJ}}&&
Rf_*L\eta_{\CI'\CJ'}
}
\end{equation}
\end{lemma}
We follow the notation and assumption before Proposition 
\ref{prop:BKFFunctComp}. 
Put $\ubM=\bM\otimes_{\bA_{\inf,X}}\ubA_{\inf,X}$
and $\ubM'=\bM'\otimes_{\bA_{\inf,X'}}\ubA_{\inf,X'}$.
Then the morphism $\varepsilon_{g,\CF}\colon \bM\to\bmg_{\proet*}\bM'$
in Proposition \ref{prop:BKFFunctComp} (1) induces  morphisms
\begin{equation}\label{eq:BKFZarProjFunct}
R\nu_{\fX*}R\varprojlim_{\N}\ubM
\to R\nu_{\fX*}R\varprojlim_{\N}R\bmg^{\N^{\circ}}_{\proet*}
\ubM'
\cong R\nu_{\fX*}R\bmg_{\proet*}R\varprojlim_{\N}\ubM'
\cong Rg_{\Zar*}R\nu_{\fX'*}R\varprojlim_{\N}\ubM'
\end{equation}
in $D(\fX_{\Zar},A_{\inf})$.
By taking $L\eta_{\mu}$ and composing it with $L\eta_{\mu}Rg_{\Zar*}
\to Rg_{\Zar*}L\eta_{\mu}$ \eqref{eq:DeclageDirectImage}, 
we obtain a morphism 
\begin{equation}\label{eq:AOmegaFunct}
A\Omega_{\fX}(\bM)\longrightarrow Rg_{\Zar*}(A\Omega_{\fX'}(\bM')).
\end{equation}
By Proposition \ref{prop:BKFFunctComp} (3) and 
\eqref{eq:DecalageDirectImComp}, 
the morphism 
\eqref{eq:AOmegaFunct}
satisfies the obvious cocycle condition with
respect to composition of $g$'s.

\begin{remark}
(1) We see that the morphism \eqref{eq:AOmegaFunct} is compatible with 
\eqref{eq:AinfCohBKFFrobPB}
as follows. Let $\CF\in \Ob \Crystal^{\fproj}_{\prism}(\fX/A_{\inf})$,
let $\CF'$ be $g_{\prism}^{-1}(\CF)$, and put $\CF_{\varphi}=\varphi^*\CF$
and $\CF'_{\varphi}=\varphi^*\CF'$ (Remark \ref{rmk:PrismCrystalFrobTensor} (2)). 
We have $g_{\prism}^{-1}(\CF_{\varphi})=\CF'_{\varphi}$. We define $\bM$ and $\bM_{\varphi}$
(resp.~$\bM'$ and $\bM'_{\varphi})$ to be the images of $\CF$ and $\CF_{\varphi}$
(resp.~$\CF'$ and $\CF'_{\varphi}$) under $\bM_{\BKF,\fX}$
(resp.~$\bM_{\BKF,\fX'}$). Let $\ubM$ and $\ubM_{\varphi}$
(resp.~$\ubM'$ and $\ubM'_{\varphi}$) be their scalar extensions under
$\bA_{\inf,X}\to\ubA_{\inf,X}$ (resp.~$\bA_{\inf,X'}\to \ubA_{\inf,X'}$).
Then the commutative diagram \eqref{eq:BKFFunctFrobFunct} induces a commutative 
diagram in $D((X_{\proet}^{\sim})^{\N^{\circ}},\ubA_{\inf,X})$
\begin{equation}
\xymatrix@R=10pt{
\ubM\ar[rr]\ar[d]&&
R\bmg_{\proet*}^{\N^{\circ}}\ubM'\ar[d]\\
\varphi_*\ubM_{\varphi}\ar[r]&
\varphi_*R\bmg_{\proet*}^{\N^{\circ}}\ubM'_{\varphi}\ar@{-}[r]^{\cong}&
R\bmg_{\proet*}^{\N^{\circ}}\varphi_*\ubM'_{\varphi},
}
\end{equation}
where $\varphi_*$ denotes the restriction of scalars under the Frobenius.
By taking $R\nu_{\fX*}R\varprojlim_{\N}$ and composing it with 
$R\nu_{\fX*}R\varprojlim_{\N}R\bmg_{\proet*}^{\N^{\circ}}
\cong Rg_{\Zar*}R\nu_{\fX'*}R\varprojlim_{\N}$
for the right two terms, we obtain a commutative diagram 
in $D(\fX_{\Zar},A_{\inf})$
\begin{equation}
\xymatrix@R=10pt@C=50pt{
R\nu_{\fX*}R\varprojlim_{\N}\ubM\ar[r]^(.45){\eqref{eq:BKFZarProjFunct}}\ar[d]_{\cong}&
Rg_{\Zar*}R\nu_{\fX'*}R\varprojlim_{\N}\ubM'\ar[d]^{\cong}\\
\varphi_*R\nu_{\fX*}R\varprojlim_{\N}\ubM_{\varphi}\ar[r]^(.45){\eqref{eq:BKFZarProjFunct}}&
\varphi_*Rg_{\Zar*}R\nu_{\fX'*}R\varprojlim_{\N}\ubM'_{\varphi}.
}
\end{equation}
By taking $L\eta_{\mu}$, using 
$L\eta_{\mu}\varphi_*\xrightarrow{\cong}\varphi_*L\eta_{\varphi(\mu)}$,
$L\eta_{\lambda}Rg_{\Zar*}\to Rg_{\Zar*}L\eta_{\lambda}$ for $\lambda=\mu,\varphi(\mu)$
\eqref{eq:DeclageDirectImage}, $L\eta_{\varphi(\mu)}\cong L\eta_{\pq}L\eta_{\mu}$,
and applying \eqref{eq:DecalageDirectImComp} (resp.~\eqref{eq:DecalageCompDerictIm}) to 
$Rg_{\Zar*}\circ \varphi_*\cong \varphi_*\circ Rg_{\Zar*}$
(resp.~$Rg_{\Zar*}$ and $L\eta_{\varphi(\mu)}\cong L\eta_{\pq}L\eta_{\mu}$), we obtain
the desired commutative diagram.
\begin{equation}
\xymatrix@R=10pt@C=40pt{
A\Omega_{\fX}(\bM)\ar[rr]^(.45){\eqref{eq:AOmegaFunct}}
\ar[d]_{\eqref{eq:AinfCohBKFFrobPB}}^{\cong}&&
Rg_{\Zar*}A\Omega_{\fX}(\bM')\ar[d]_{\eqref{eq:AinfCohBKFFrobPB}}^{\cong}\\
\varphi_*L\eta_{\pq}A\Omega_{\fX}(\bM_{\varphi})
\ar[r]^(.45){\eqref{eq:AOmegaFunct}}
&
\varphi_*L\eta_{\pq}Rg_{\Zar*}A\Omega(\bM'_{\varphi})
\ar[r]^{\eqref{eq:DeclageDirectImage}}
&
Rg_{\Zar*}\varphi_*L\eta_{\pq}A\Omega_{\fX}(\bM'_{\varphi})
}
\end{equation}

(2)  We see that the morphism \eqref{eq:AOmegaFunct} is compatible with the 
product \eqref{eq:BKFAinfCohProd} as follows. Let $\CF_{\nu}$ ($\nu\in\{1,2\}$) be objects of
$\Crystal_{\prism}^{\fproj}(\fX/A_{\inf})$, put $\CF_3=\CF_1\otimes_{\CO_{\fX/A_{\inf}}}\CF_2$,
and let $\CF'_{\nu}$ denote $g_{\prism}^{-1}(\CF_{\nu})$ for $\nu\in \{1,2,3\}$.
We have $\CF_3'=g_{\prism}^{-1}\CF_3$ (Remark \ref{rmk:PrismCrystalFrobTensor} (3)). 
For $\nu\in \{1,2,3\}$, we define 
$\bM_{\nu}$ (resp.~$\bM_{\nu}'$) to be $\bM_{\BKF,\fX}(\CF_{\nu})$ (resp.~$\bM_{\BKF,\fX'}(\CF_{\nu}')$),
and put $\ubM_{\nu}=\bM_{\nu}\otimes_{\bA_{\inf,X}}\ubA_{\inf,X}$
(resp.~$\ubM'_{\nu}=\bM'_{\nu}\otimes_{\bA_{\inf,X'}}\ubA_{\inf,X'}$).
Then we see that the obvious analogue of \eqref{eq:BKFFunctorTensorFunct} 
for $\ubM_{\nu}$ and $\ubM_{\nu}'$
$(\nu\in\{1,2,3\})$ holds by considering the reduction modulo
$(p,\pq)^{m+1}$ $(m\in \N)$ of the left adjoint of \eqref{eq:BKFFunctorTensorFunct}
with respect to 
$\bmg_{\proet}$.
Therefore
by the compatibility of cup products with composition of morphisms of ringed topos
\cite[0FPN]{Stacks} and that of $\bmg_{\proet*}\ubM_{\nu}'\to R\bmg_{\proet*}\ubM_{\nu}'$
with the products (which is verified by taking the left adjoint  with respect to $\bmg_{\proet}$
and going back to the definition of the cup product of $R\bmg_{\proet*}$ (\cite[0B6C]{Stacks}), 
we obtain a commutative diagram
\begin{equation}\label{eq:BKFZarProjProdFunct}
\xymatrix@R=10pt{
R\nu_{\fX*}R\varprojlim_{\N}\ubM_1
\otimes_{A_{\inf}}^L
R\nu_{\fX*}R\varprojlim_{\N}\ubM_2
\ar[r]^(.43){\eqref{eq:BKFZarProjFunct}}
\ar[dd]
&
Rg_{\Zar*}R\nu_{\fX'*}R\varprojlim_{\N}\ubM_1'
\otimes^L_{A_{\inf}}
Rg_{\Zar*}R\nu_{\fX'*}R\varprojlim_{\N}\ubM_2'\ar[d]\\
& Rg_{\Zar*}(R\nu_{\fX'*}R\varprojlim_{\N}\ubM_1'
\otimes^L_{A_{\inf}}R\nu_{\fX'*}R\varprojlim_{\N}\ubM_2')\ar[d]\\
R\nu_{\fX*}R\varprojlim_{\N}\ubM_3
\ar[r]^{\eqref{eq:BKFZarProjFunct}} 
&Rg_{\Zar*}R\nu_{\fX'*}R\varprojlim_{\N}\ubM_3'.
}
\end{equation}
We see that the morphism $L\eta_{\mu}Rg_{\Zar*}\to Rg_{\Zar*}L\eta_{\mu}$
\eqref{eq:DeclageDirectImage} is compatible with the lax symmetric monoidal
structures by taking the left adjoint with respect to $g_{\Zar}$, exchanging 
$Lg_{\Zar}^*$ and $L\eta_{\mu}$ by the
isomorphism $Lg_{\Zar}^*L\eta_{\mu}\cong L\eta_{\mu}Lg_{\Zar}^*$
(\cite[Lemma 6.14]{BMS}), which is compatible with the lax symmetric monoidal structures
(\cite[Proposition 6.7]{BMS}), and going back to the definition of 
the cup product of $Rg_{\Zar*}$.
Therefore, by taking $L\eta_{\mu}$ of \eqref{eq:BKFZarProjProdFunct},
and composing it with $L\eta_{\mu}Rg_{\Zar*}\to Rg_{\Zar*}L\eta_{\mu}$ considered above
for the right three terms, we obtain the following commutative diagram as desired.
\begin{equation}
\xymatrix@R=10pt{
A\Omega_{\fX}(\bM_1)\otimes^L_{A_{\inf}}
A\Omega_{\fX}(\bM_2)\ar[r]^(.41){\eqref{eq:AOmegaFunct}}
\ar[dd]_{\eqref{eq:BKFAinfCohProd}}
&
Rg_{\Zar*}A\Omega_{\fX'}(\bM'_1)\otimes^L_{A_{\inf}}
Rg_{\Zar*}A\Omega_{\fX'}(\bM'_2)\ar[d]\\
&
Rg_{\Zar*}(A\Omega_{\fX'}(\bM'_1)\otimes^L_{A_{\inf}}
A\Omega_{\fX'}(\bM'_2))
\ar[d]^{\eqref{eq:BKFAinfCohProd}}
\\
A\Omega_{\fX}(\bM_3)\ar[r]^{\eqref{eq:AOmegaFunct}}&
Rg_{\Zar*}A\Omega_{\fX'}(\bM'_3).
}
\end{equation}
\end{remark}

\section{Pro\'etale sites and framed embeddings}\label{sec:ProetGammaZarShv}
We retain the settings introduced in the first paragraph
of \S\ref{sec:RelativeBKF}. 
\begin{definition}\label{def:AdmFramedSmEmbed}
(1) A {\it small framed embedding over $A_{\inf}$}
$(\mfi\colon \fX\hookrightarrow \fY, \ut=(t_i)_{i\in\Lambda})$
is a set of data consisting of  a $p$-adic smooth affine formal
scheme $\fX=\Spf(A)$ over $\CO$, a framed smooth 
$\delta$-$A_{\inf}$-algebra $(B,\ut=(t_i)_{i\in\Lambda})$ 
(Definition \ref{def:qprism} (2))
equipped with a $(p,\pq)$-adic topology,
and a closed immersion 
$\mfi\colon \fX\hookrightarrow \fY=\Spf(B)$ over $A_{\inf}$,
which is equivalent to a surjective homomorphism 
$\mfi^*\colon B\to A$ of $A_{\inf}$-algebras, 
satisfying the following two conditions.
Let $t_{A,i}$ $(i\in \Lambda)$ denote the image of $t_i\in B$
in $A$ under $\mfi^*$.
\begin{equation}\label{cond:frameAinf}
\text{$t_i\in B^{\times}$ for every $i\in \Lambda$.}
\end{equation}
\begin{equation}\label{cond:frameAinf2}
\text{\parbox[t]{0.8\linewidth}{There exists $\Lambda_A\subset \Lambda$
such that $t_{A,i}$ $(i\in \Lambda_A)$ form $p$-adic coordinates of $A$ 
over $\CO$ (Definition \ref{def:IadicProperties} (2)). }}
\end{equation}
When $\fX=\Spf(A)$ is given, we call 
$(\mfi\colon \fX\hookrightarrow \fY, \ut=(t_i)_{i\in\Lambda})$
a {\it small framed embedding of $\fX$ over $A_{\inf}$}.\par
(2)  A {\it morphism $(g,h,\psi)$ of small framed embeddings over} $A_{\inf}$
from $(\mfi'\colon \fX'\to \fY'=\Spf(B'),\ut'=(t'_{i'})_{i'\in\Lambda'})$
to $(\mfi\colon \fX\to \fY=\Spf(B), \ut=(t_{i})_{i\in\Lambda})$ is a triplet consisting of 
morphism $g\colon \fX'\to \fX$ over $\CO$, a morphism 
$h\colon \fY'\to \fY$ over $A_{\inf}$, and a map of ordered sets
$\psi\colon \Lambda\to \Lambda'$ such that $h\circ \mfi'=\mfi\circ g$
and $(h^*,\psi)$ is a morphism of framed smooth $\delta$-$A_{\inf}$-algebras
(Definition \ref{def:qprism} (2)) from 
$(B,\ut)$ to $(B',\ut')$, i.e., $h^*(t_i)=t'_{\psi(i)}$ for every $i\in \Lambda$. 
\end{definition}

Let $\fX=\Spf(A)$ be a $p$-adic smooth affine formal scheme over $\CO$
admitting invertible $p$-adic coordinates. Then there exists 
a small framed embedding $\mathtt i=(\mfi\colon \fX\to \fY=\Spf(B), \ut=(t_i)_{i\in\Lambda})$ of $\fX$ over $A_{\inf}$, which we choose in the following. Put $\Gamma_{\Lambda}=\Map(\Lambda,\Z_p)$, and let
$X$ denote the adic generic fiber $\Spa(A[\frac{1}{p}], A)$ of $\fX$. 
In this section, we will construct and study a morphism of topos 
from the pro\'etale topos $X_{\proet}^{\sim}$ to the topos $\Gamma_{\Lambda}\hy\fX^{\sim}_{\Zar}$
of $\Gamma_{\Lambda}$-sheaves of sets on 
$\fX_{\Zar}$ (Definition \ref{def:Gsheaf} (2), Proposition \ref{prop:GsheafTopos})
\begin{equation}\label{eq:ProetGammaShvProj}
\nu_{\fX,\ut}\colon X_{\proet}^{\sim}\to \Gamma_{\Lambda}\hy\fX_{\Zar}^{\sim}\end{equation}
by evaluating sheaves on $X_{\proet}$
on an inverse system of finite \'etale adic spaces over the adic generic fiber 
of each open formal subscheme of $\fX$ obtained by 
adjoining $p$-power roots of $t_i$ $(i\in \Lambda)$ via the embedding $\mfi$. 
See \eqref{eq:ProetGammaShvProjDescrip}.

Put $t_{A,i}=\mfi^*(t_i)$ for $i\in \Lambda$. Regarding $A$ as an algebra
over $\CO[T_i^{\pm 1} (i\in \Lambda)]$
by the $\CO$-homomorphism defined by $T_i\mapsto 
t_{A,i}$, we define an inductive system of $A$-algebras
$(A_n)_{n\in \N}$ by the integral closures of $A$ in 
the finite \'etale $A[\frac{1}{p}]$-algebras
\begin{equation}
A[\tfrac{1}{p}]\otimes_{\CO[T_i^{\pm 1} (i\in \Lambda)]}
\CO[T_i^{\pm 1/p^n} (i\in \Lambda)]\quad
(n\in \N).\end{equation} For $n\in \N$, the homomorphism
$A_n\to A_{n+1}$ is injective, and we regard
$A_n$ as an $A$-subalgebra of $A_{n+1}$
in the following. 

The inductive system of $A$-algebras
$(A_n)_{n\in \N}$ is equipped with the action of 
$\Gamma_{\Lambda}$ defined by 
$\gamma(T_i^{1/p^n})=\zeta_n^{\gamma(i)}T_i^{1/p^n}$
$(i\in \Lambda)$ for $\gamma\in \Gamma_{\Lambda}$, 
where $\zeta_n$ is the 
primitive $p^n$th root of unity in $\CO$ fixed at the beginning of 
\S\ref{sec:RelativeBKF}. 
For a finite $\Gamma_{\Lambda}$-set $S$
(Definition \ref{def:Gset} (2)), we define an $A$-algebra
$A_S$ to be 
$\varinjlim_{n\in \N}\Map_{\Gamma_{\Lambda}}(S,A_n)$.
This construction is contravariant in $S$; a morphism of
finite $\Gamma_{\Lambda}$-sets $\alpha\colon S\to S'$
induces a homomorphism of $A$-algebras
$A_{S'}\to A_S$ by the composition with $\alpha$.
It is obvious that if $S$ is the disjoint union of
finite number of finite $\Gamma_{\Lambda}$-sets 
$S_{\alpha}$, then $A_S$ is the product of $A_{S_{\alpha}}$. 
For $n$, $n'\in \N$ with $n'\geq n$, 
we have $A_n=(A_{n'})^{p^n\Gamma_{\Lambda}}$
as it holds after inverting $p$. This implies
that we have $A_S=\Map_{\Gamma_{\Lambda}}(S,A_n)$
if the action of $p^n\Gamma_{\Lambda}$ on $S$ is trivial.
In particular, we have 
$A_{\Gamma_{\Lambda}/p^n\Gamma_{\Lambda}}
=\Map_{\Gamma_{\Lambda}}
(\Gamma_{\Lambda}/p^n\Gamma_{\Lambda},A_n)
\cong A_n; f\mapsto f(1)$. We see that 
$A_S$ is integrally closed in $A_S[\frac{1}{p}]$
noting that the underlying set of $S$ is finite.

\begin{lemma}\label{lem:FiniteEtalAdicSpacesTriv}
Let $S$ be a finite $\Gamma_{\Lambda}$-set and let $n\in \N$
such that the action of $p^n\Gamma_{\Lambda}$ on $S$ is trivial.
Then we have a canonical isomorphism of $A_n[\frac{1}{p}]$-algebras
$A_S[\frac{1}{p}]\otimes_{A[\frac{1}{p}]}A_n[\frac{1}{p}]
\cong\Map(S,A_n[\frac{1}{p}])$ functorial in $S$ with trivial
$p^n\Gamma_{\Lambda}$-action and compatible with $n$.
\end{lemma}

\begin{proof}
By the flatness of $A[\frac{1}{p}]\to A_n[\frac{1}{p}]$,
we have $A_{S}[\frac{1}{p}]\otimes_{A[\frac{1}{p}]}
A_n[\frac{1}{p}]\cong
\Map_{\Gamma_{\Lambda}}(S, (A_n\otimes_AA_n)[\frac{1}{p}])$,
where $\Gamma_{\Lambda}$ acts on $A_n\otimes_AA_n$
via the left factor. For the right-hand side, we have a 
$\Gamma_{\Lambda}$-equivariant isomorphism
$(A_n\otimes_AA_n)[\frac{1}{p}]
\xrightarrow{\cong}\Map(\Gamma_{\Lambda}/p^n\Gamma_{\Lambda},
A_n[\frac{1}{p}])$ sending $x\otimes y$ to the map $f$ defined by 
$f(\gamma)=\gamma(x)y$, where $\Gamma_{\Lambda}$
acts on the codomain via the right action on 
$\Gamma_{\Lambda}/p^n\Gamma_{\Lambda}$,
and a bijection 
$\Map_{\Gamma_{\Lambda}}(S,\Map(\Gamma_{\Lambda}/p^n\Gamma_{\Lambda},
A_n[\frac{1}{p}]))\xrightarrow{\cong}
\Map(S,A_n[\frac{1}{p}])$ defined by 
$f\mapsto (s\mapsto f(s)(1))$. 
\end{proof}

\begin{lemma}\label{lem:KummerCov}
(1) For $S\in \Ob \Gamma_{\Lambda}\fSet$, $A_S[\frac{1}{p}]$
is a finite \'etale $A[\frac{1}{p}]$-algebra.\par
(2) For a covering $(S_{\alpha}\to S)_{\alpha\in I}$ in $\Gamma_{\Lambda}\fSet$
(Definition \ref{def:Gset} (2)), the morphism of $A[\frac{1}{p}]$-schemes
$\sqcup_{\alpha\in I}
\Spec (A_{S_{\alpha}}[\frac{1}{p}])\to \Spec (A_S[\frac{1}{p}])$ is surjective.\par
(3) For morphisms $S_i\to S_0$ $(i=1,2)$ in $\Gamma_{\Lambda}\fSet$
and $S_3=S_1\times_{S_0}S_2$, 
the $A[\frac{1}{p}]$-algebra homomorphism 
$(A_{S_1}\otimes_{A_{S_0}}A_{S_2})[\frac{1}{p}]
\to A_{S_3}[\frac{1}{p}]$ is an isomorphism.
\end{lemma}

\begin{proof}
By taking the scalar extension under the faithfully flat homomorphism
$A[\frac{1}{p}]\to A_n[\frac{1}{p}]$ for $n\in \N$ such that
$p^n\Gamma_{\Lambda}$-acts trivially on the relevant
finite $\Gamma_{\Lambda}$-sets and using Lemma \ref{lem:FiniteEtalAdicSpacesTriv}, 
we are reduced to showing
the corresponding claims for $\Map(S,A_n[\frac{1}{p}])$,
which are all obvious. Note that we have 
$\Spec(\Map(S,A_n[\frac{1}{p}]))
=\sqcup_{S}\Spec(A_n[\frac{1}{p}])$.
\end{proof}

By Lemma \ref{lem:KummerCov} (1), we have a covariant functor from the 
category of  finite $\Gamma_{\Lambda}$-sets
$\Gamma_{\Lambda}\fSet$ to that of 
finite \'etale adic spaces over $X$ sending 
$S$ to $X_S=\Spa(A_S[\frac{1}{p}], A_S)$. 
Since the reduction of $A$ modulo the maximal ideal of $\CO$ is noetherian, 
every open formal subscheme of $\fX$ is quasi-compact.
Therefore the topology of the Zariski site $\fX_{\Zar}$ of $\fX$ is defined by 
the pretopology $\{(\fU_{\alpha}\subset \fU)_{\alpha\in I}
\,\vert\, \sharp I<\infty, \fU_{\alpha}\in \Ob\fX_{\Zar}\,(\alpha\in I), \fU=\cup_{\alpha\in I}\fU_{\alpha}\}$
$(\fU\in \Ob \fX_{\Zar})$. By using this pretopology,
we can define the site $(\fX_{\Zar})_{\Gamma_{\Lambda}}$
of finite $\Gamma_{\Lambda}$-sets above
$\fX_{\Zar}$ (Definition \ref{def:Gsheaf} (3)). 
Recall that we have an equivalence $\rho_{\Gamma_{\Lambda},\fX_{\Zar}}^{-1}
=(\rho_{\Gamma_{\Lambda},\fX_{\Zar}*}, \rho_{\Gamma_{\Lambda},\fX_{\Zar}}^*)$
from the topos 
$(\fX_{\Zar})_{\Gamma_{\Lambda}}^{\sim}$
to the category 
$\Gamma_{\Lambda}\hy\fX_{\Zar}^{\sim}$
of $\Gamma_{\Lambda}$-sheaves of sets on 
$\fX_{\Zar}$ (Definition \ref{def:Gsheaf} (2)) by 
Proposition \ref{prop:GsheafTopos}. 
We can define a functor
$\nu_{\fX,\ut}^+\colon(\fX_{\Zar})_{\Gamma_{\Lambda}}\to X_{\proet}$
by sending $(\fU,S)$ to $U_S:=X_S\times_XU$,
where $U$ denotes the adic generic fiber of $\fU$.
It maps coverings defining the topology of $(\fX_{\Zar})_{\Gamma_{\Lambda}}$
(Definition \ref{def:Gsheaf} (3)) to coverings in the site $X_{\proet}$
by Lemma \ref{lem:KummerCov} (2), and preserves finite inverse limits
by Lemma \ref{lem:KummerCov} (3). Therefore it is continuous 
by Proposition \ref{prop:CGfSetSheaf}, and
defines a morphism of sites. Let $\tnu_{\fX,\ut}
=(\tnu_{\fX,\ut}^*,\tnu_{\fX,\ut*})$ denote 
the associated morphism of topos 
$X_{\proet}^{\sim}\to (\fX_{\Zar})_{\Gamma_{\Lambda}}^{\sim}$,
and put $\nu_{\fX,\ut}=\rho_{\Gamma_{\Lambda}, \fX_{\Zar}}^{-1}
\circ \tnu_{\fX,\ut}\colon X_{\proet}^{\sim}
\to\Gamma_{\Lambda}\hy\fX_{\Zar}^{\sim}$.
Recall that we have morphisms of topos 
$\pi_{\Gamma_{\Lambda},\fX_{\Zar}}\colon 
\Gamma_{\Lambda}\hy\fX_{\Zar}^{\sim}
\to \fX_{\Zar}^{\sim}$ and 
$\iota_{\Gamma_{\Lambda},\fX_{\Zar}}\colon 
\fX_{\Zar}^{\sim}\to \Gamma_{\Lambda}\hy\fX_{\Zar}^{\sim}$
introduced before Proposition \ref{pro:TrivIndFunctEx},
for which we write $\pi_{\Lambda,\fX}$ and $\iota_{\Lambda,\fX}$,
respectively, in the following. We have a canonical isomorphism
$\pi_{\Lambda,\fX}\circ \iota_{\Lambda,\fX}\cong 
\id_{\fX_{\Zar}^{\sim}}$.
Since the composition of the
functor $(\fX_{\Zar})_{\{1\}}\to (\fX_{\Zar})_{\Gamma_{\Lambda}}$
induced by $\id_{\fX_{\Zar}}$ and $\Gamma_{\Lambda}\to \{1\}$
and the functor $\nu_{\fX,\ut}^+\colon (\fX_{\Zar})_{\Gamma_{\Lambda}}
\to X_{\proet}$ sends the pair of $\fU\in \Ob\fX_{\Zar}$ and
a one point set to $U$, we see that the composition
$\pi_{\Lambda,\fX}\circ \nu_{\fX,\ut}\colon X_{\proet}^{\sim}
\to \fX_{\Zar}^{\sim}$ is canonically isomorphic to the
projection morphism, which is denoted by $\nu_{\fX}$ in the following.
\begin{equation}\label{eq:ProetGammaShfProj}
\xymatrix@C=40pt{
&X_{\proet}^{\sim}\ar[d]_{\nu_{\fX,\ut}}\ar[dr]^{\nu_{\fX}}&\\
\fX_{\Zar}^{\sim}\ar[r]^(.45){\iota_{\Lambda,\fX}}
\ar@/_1pc/[rr]_{\id_{\fX_{\Zar}^{\sim}}}&
\Gamma_{\Lambda}\hy\fX_{\Zar}^{\sim}
\ar[r]^{\pi_{\Lambda,\fX}}
& \fX_{\Zar}^{\sim}
}
\end{equation}

The direct image functor $\nu_{\fX,\ut*}$ is explicitly given as follows.
Let $\fU$ be an open formal subscheme of $\fX$,
 let $U$ be its adic generic fiber,  and put 
 $U_n=U_{\Gamma_{\Lambda}/p^n\Gamma_{\Lambda}}$ $(n\in \N)$.
 Then, by the definition of $\tnu_{\fX,\ut*}$ and 
$\rho_{\Gamma_{\Lambda},\fX_{\Zar}}^*$
\eqref{eq:CGfSetShfToGShf}, we have 
\begin{equation}\label{eq:ProetGammaShvProjDescrip}
(\nu_{\fX,\ut*}\CF)(\fU)
=\varinjlim_{H\in \CN(\Gamma_{\Lambda})}
(\tnu_{\fX,\ut*}\CF)(\fU,\Gamma_{\Lambda}/H)\\
\cong\varinjlim_{H\in\CN(\Gamma_{\Lambda})}
\CF(U_{\Gamma_{\Lambda}/H})
\cong\varinjlim_n\CF(U_n)
\end{equation}
for $\CF\in \Ob X_{\proet}^{\sim}$.
The action of $\Gamma_{\Lambda}$ on 
$(\nu_{\fX,\ut*}\CF)(\fU)$ is given by the right action of
$\Gamma_{\Lambda}$ on $U_n$.

We next
discuss the functoriality of 
$\nu_{\fX,\ut}$  with respect to 
$\mti=(\mfi \colon \fX\hookrightarrow \fY, \ut)$.
Let $\mti'=(\mfi'\colon \fX'=\Spf(A')\hookrightarrow \fY'=\Spf(B'),
\ut'=(t'_{i'})_{i'\in\Lambda'})$ be another
small framed embedding over $A_{\inf}$
(Definition \ref{def:AdmFramedSmEmbed} (1)), and let 
$\mtg=(g,h,\psi)\colon \mti'=(\mfi'\colon \fX'\hookrightarrow \fY',\ut')
\to \mti=(\mfi\colon \fX\hookrightarrow \fY,\ut)$ 
be a morphism of small framed embeddings over
$A_{\inf}$ (Definition \ref{def:AdmFramedSmEmbed} (2)).

Let $X'$ denote the adic space $\Spa(A'[\frac{1}{p}],A')$,
and let $\bmg\colon X'\to X$ be the morphism of adic spaces
associated to $g\colon \fX'\to\fX$. Let $t'_{A',i'}\in A^{\prime \times}$
$(i'\in \Lambda')$ be the image of $t'_{i'}\in B^{\prime \times}$ under
$\mfi^{\prime*}$, put $\Gamma_{\Lambda'}=\Map(\Lambda',\Z_p)$.
We have the commutative diagram of topos \eqref{eq:ProetGammaShfProj} for $\fX'$.
Let $\Gamma_{\psi}\colon \Gamma_{\Lambda'}\to \Gamma_{\Lambda}$
be the homomorphism defined by the composition with $\psi\colon 
\Lambda\to \Lambda'$.  We have morphisms of topos \eqref{eq:GSheafToposMorph}
\begin{equation}
(\widetilde{\Gamma_{\psi}})_{g_{\Zar}}\colon
(\fX'_{\Zar})^{\sim}_{\Gamma_{\Lambda'}}\to
(\fX_{\Zar})^{\sim}_{\Gamma_{\Lambda}},\qquad
(\Gamma_{\psi})_{g_{\Zar}}\colon 
\Gamma_{\Lambda'}\hy\fX_{\Zar}^{\prime\sim}
\to \Gamma_{\Lambda}\hy\fX_{\Zar}^{\sim}
\end{equation}
and isomorphisms \eqref{eq:KoszFunctToposDiag}
\begin{equation}
\pi_{\Lambda,\fX}\circ(\Gamma_{\psi})_{g_{\Zar}}\cong 
g_{\Zar}\circ\pi_{\Lambda',\fX'},\qquad
(\Gamma_{\psi})_{g_{\Zar}}\circ\iota_{\Lambda,\fX}\cong
\iota_{\Lambda,\fX}\circ g_{\Zar}.
\end{equation}
We define $A_n'$ $(n\in \N)$, $A'_{S'}$,
and $X'_{S'}$ $(S'\in \Ob\Gamma_{\Lambda'}\fSet)$
 in the same
way as $A_n$, $A_S$, and $X_S$
by using $\mti'=(\mfi', \ut')$ instead of $\mti=(\mfi,\ut)$.
Then the $\CO$-algebra homomorphisms 
$g^*\colon A\to A'$ and 
$\CO[T_i^{\pm 1/p^n} (i\in\Lambda)]\to
\CO[T_{i'}^{\prime\pm 1/p^n} (i'\in \Lambda')];
T_i^{\pm 1/p^n}\mapsto T_{\psi(i)}^{\prime\pm1/p^n}$
induce a homomorphism of inductive systems of $\CO$-algebras
$(g_n^*)_n\colon (A_{n})_n\to (A'_n)_n$, which is
equivariant with respect to $\Gamma_{\psi}\colon \Gamma_{\Lambda'}
\to  \Gamma_{\Lambda}$. For $S\in \Ob \Gamma_{\Lambda}\fSet$
and $S'=(\Gamma_{\psi})^*_{\text{f}}S\in \Ob\Gamma_{\Lambda'}\fSet$
\eqref{eq:RestrGpActionFunctor}, which is $S$ with an action of $\Gamma_{\Lambda'}$ via
$\Gamma_{\psi}$, $(g_n^*)_n$ induces a homomorphism 
$g_S^*\colon A_S\to A'_{S'}$ over $g^*\colon A\to A'$ functorial
in $S$. 
Then $g_S^*$ induces a morphism 
$\bmg_S\colon X'_{S'}\to X_S$ over $\bmg\colon X'\to X$
functorial in $S$. For the two compositions of functors
$(\fX_{\Zar})_{\Gamma_{\Lambda}}
\xrightarrow{g^*\times(\Gamma_{\psi})^*_{\text{f}}}
(\fX'_{\Zar})_{\Gamma_{\Lambda'}}
\xrightarrow{\nu_{\fX',\ut'}^+}X^{\prime}_{\proet}$
and 
$(\fX_{\Zar})_{\Gamma_{\Lambda}}
\xrightarrow{\nu_{\fX,\ut}^+}
X_{\proet}\xrightarrow{\bmg^*} X'_{\proet}$, we have
\begin{align*}
\nu_{\fX',\ut'}^+\circ (g^*\times(\Gamma_{\psi})_{\text{f}}^*)(\fU,S)
&=\nu_{\fX',\ut'}^+(\fU',S')=X'_{S'}\times_{X'}U',\\
\bmg^*\circ\nu_{\fX,\ut}^+(\fU,S)&=\bmg^*(X_S\times_XU)
=(X_S\times_XU)\times_XX'\cong X_S\times_XU'
\end{align*}
for $(\fU,S)\in \Ob (\fX_{\Zar})_{\Gamma_{\Lambda}}
=\Ob\fX_{\Zar}\times\Ob \Gamma_{\Lambda}\fSet$,
$\fU'=\fU\times_{\fX}\fX'$, and the adic generic fibers $U$ and $U'$
of $\fU$ and $\fU'$, respectively.  Hence the morphism 
$\bmg_S\times_{\bmg}\id_{U'}$ for each $(\fU,S)$ defines morphisms of
functors 
\begin{align}
&\Xi_{g,\psi}^+\colon\nu_{\fX',\ut'}^+\circ (g^*\times(\Gamma_{\psi})_{\text{f}}^*)
\longrightarrow \bmg^*\circ\nu_{\fX,\ut}^+\colon (\fX_{\Zar})_{\Gamma_{\Lambda}}
\longrightarrow X_{\proet}',\label{eq:XProetToGammaSiteFunct}\\
&\tXi_{g,\psi*}\colon\tnu_{\fX,\ut*}\circ \bmg_{\proet*}
\longrightarrow (\widetilde{\Gamma_{\psi}})_{g_{\Zar}*}\circ\tnu_{\fX',\ut'*}\colon 
X_{\proet}^{\prime\sim}\longrightarrow
(\fX_{\Zar})_{\Gamma_{\Lambda}}^{\sim},
\label{eq:XproetToGammafSetShfFunct}
\\
&\Xi_{g,\psi*}\colon\nu_{\fX,\ut*}\circ \bmg_{\proet*}
\longrightarrow (\Gamma_{\psi})_{g_{\Zar}*}\circ \nu_{\fX',\ut'*}\colon 
X_{\proet}^{\prime\sim}\longrightarrow \Gamma_{\Lambda}\hy\fX_{\Zar}^{\sim}.
\label{eq:XproetToGammaShfFunct}
\end{align}

\begin{equation}\label{eq:ProetToZarFacorization}
\xymatrix@C=50pt{
X^{\prime\sim}_{\proet}\ar[r]_{\tnu_{\fX',\ut'*}}
\ar[d]_{\bmg_{\proet*}}
\ar@/^1.5pc/[rr]^{\nu_{\fX',\ut'*}}
&
(\fX'_{\Zar})_{\Gamma_{\Lambda'}}^{\sim}
\ar[d]^{(\widetilde{\Gamma_{\psi}})_{g_{\Zar}*}}
\ar[r]^{\sim}_{\rho_{\Gamma_{\Lambda'},\fX'_{\Zar}}^*}
&
\Gamma_{\Lambda'}\hy\fX_{\Zar}^{\prime\sim}
\ar[d]^{(\Gamma_{\psi})_{g_{\Zar}*}}
\ar[r]_(.53){\pi_{\Lambda',\fX'*}}&
\fX^{\prime\sim}_{\Zar}\ar[d]^{g_{\Zar*}}\\
X_{\proet}^{\sim}
\ar[r]_{\tnu_{\fX,\ut*}}
\ar@{=>}[ur]_{\tXi_{g,\psi*}}
\ar@/_2pc/[rr]_{\nu_{\fX,\ut*}}
&
(\fX_{\Zar})_{\Gamma_{\Lambda}}^{\sim}
\ar[r]^{\sim}_{\rho_{\Gamma_{\Lambda},\fX_{\Zar}}^*}
&
\Gamma_{\Lambda}\hy\fX_{\Zar}^{\sim}
\ar[r]_(.53){\pi_{\Lambda,\fX*}}
&
\fX_{\Zar}^{\sim}
}
\end{equation}
Since $X_S=X$ when $S$ is a one point set, we 
see that the following diagram is commutative.
\begin{equation}\label{eq:ProetGammaSheafProjFunct}
\xymatrix@C=45pt@R=15pt{
\pi_{\Lambda,\fX*}\circ \nu_{\fX,\ut*}\circ \bmg_{\proet*}
\ar[r]^(.47){\pi_{\Lambda,\fX*}(\Xi_{g,\psi*})}_(.47){\eqref{eq:XproetToGammaShfFunct}}
\ar@{-}[d]^{\cong}
&
\pi_{\Lambda,\fX*}\circ(\Gamma_{\psi})_{g_{\Zar}*}\circ
\nu_{\fX',\ut'*}\ar@{-}[r]^(.55){\cong}&
g_{\Zar*}\circ\pi_{\Lambda',\fX'*}\circ \nu_{\fX',\ut'*}
\ar@{-}[d]^{\cong}
\\
\nu_{\fX*} \circ \bmg_{\proet*}
\ar@{-}[rr]^{\cong}&&
g_{\Zar*} \circ\nu_{\fX'*}
}
\end{equation}
\begin{remark}\label{rmk:GammaCohFunctCocyc}
The homomorphisms $g_n^*\colon A_n\to A_n'$ $(n\in\N)$ and
$g_S^*\colon A_S\to A'_{S'}$ ($S\in\Ob \Gamma_{\Lambda}\fSet$, 
$S'=(\Gamma_{\psi})_{\text{f}}^*S$) satisfy the
cocycle condition for composition of $(g,h,\psi)$'s. Therefore
it also holds for the morphisms $\Xi_{g,\psi}^+$,
$\tXi_{g,\psi*}$, and $\Xi_{g,\psi*}$.
\end{remark}

\section{Comparison map to $A_{\inf}$-cohomology with coefficients: the local case}
\label{sec:AinfcohCompMap}
We follow the settings introduced in the first paragraph
of \S\ref{sec:RelativeBKF}.
Since the pair  $(A_{\inf},\pq A_{\inf})$ is a $q$-prism,
we can apply the results recalled in \S\ref{sec:PrismCrysQHiggs}
to prismatic sites with base $(A_{\inf},\pq A_{\inf})$. 

For a $p$-adic smooth affine formal scheme $\fX=\Spf(A)$ over $\CO$
admitting invertible $p$-adic coordinates (Definition \ref{def:IadicProperties} (2)), 
an object $\CF$ of $\Crystal_{\prism}^{\fproj}(\fX/A_{\inf})$
(Definition \ref{def:PrismaticSiteCrystal} (2)), and
$\bM:=\bM_{\BKF,\fX}(\CF)$ (Definition \ref{BKFFunctor}), 
we will construct a canonical morphism
\begin{equation}\label{eq:PrismCohAinfCohCompMap}
\kappa_{\CF}\colon Ru_{\fX/A_{\inf}*}\CF\longrightarrow A\Omega_{\fX}(\bM)
\end{equation}
in $D(\fX_{\Zar},A_{\inf})$ functorial in $\CF$ and in $\fX$.

Let $\fX=\Spf(A)$ be a $p$-adic smooth affine formal scheme over $\CO$
admitting invertible $p$-adic coordinates. Then there exists 
a small framed embedding $\mathtt i=(\mfi\colon \fX\to \fY=\Spf(B), \ut=(t_i)_{i\in\Lambda})$ of $\fX$
over $A_{\inf}$ (Definition \ref{def:AdmFramedSmEmbed} (1)), which we choose in the following. 
Put $t_{A,i}=\mfi^*(t_i)$ for $i\in \Lambda$.

Since $(B,\ut)$ is a framed smooth $\delta$-$A_{\inf}$-algebra
(Definition \ref{def:qprism} (2)), 
we have a $t_i\mu$-derivation $\theta_{B,i}$ 
$(i\in \Lambda)$ of $B$ over $A_{\inf}$ $\delta$-compatible with 
respect to $t_i^{p-1}\eta$ as in  Proposition \ref{prop:qPrismEnvTheta}, and 
an endomorphism $\gamma_{B,i}=\id_B+t_i\mu\theta_{B,i}$
of the $\delta$-$A_{\inf}$-algebra $B$ associated to it. We have 
$\theta_{B,i}(t_j)=\pq t_j$ if $j=i$, $0$ otherwise, and 
$\gamma_{B,i}(t_j)=q^pt_j$ if $j=i$, $t_j$ otherwise.  
Let $J$ be the kernel of $\mfi^*\colon B\to A$. Then $((B,J),\ut)$ is an admissible framed
smooth $\delta$-pair over $A_{\inf}$ 
(Definition \ref{def:qprism} (3)) 
by \eqref{cond:frameAinf2} and 
Proposition \ref{prop:SuffCondExBdd}. 
See also \cite[4.13]{TsujiPrismQHiggs}.
Let $D$ be the $q$-prismatic envelope of $(B,J)$
(Definition \ref{def:qprism} (1)). 
As in Proposition \ref{prop:qPrismEnvTheta}
and Definition \ref{def:qHiggsModCpx}, we have 
 a $t_i\mu$-derivation 
 $\theta_{D,i}$ $(i\in \Lambda)$ of $D$ over $A_{\inf}$ 
 $\delta$-compatible with $t_i^{p-1}\eta$, and an endomorphism
 $\gamma_{D,i}=\id_D+t_i\mu\theta_{D,i}$ of 
 the $\delta$-$A_{\inf}$-algebra $D$.
 Put $D_{m}=D/(p,\pq)^{m+1}D$, $\oD=D/\pq D$,
 $\fD_m=\Spec(D_m)$, $\fD=\Spf(D)$,
 and $\ofD=\Spf(\oD)$.
 Let $v_D$ denote the morphism $\ofD\to \fX$
 defined by $A=B/J\to \oD$. 
 
 Put $A_{\inf,m}=A_{\inf}/(p,\pq)^{m+1}$ and 
$\bA_{\inf,X,m}=\bA_{\inf,X}/(p,\pq)^{m+1}$ for $m\in \N$.
The constant sheaves associated to $A_{\inf}$ and $A_{\inf,m}$
on a topos are also denoted by $A_{\inf}$ and $A_{\inf,m}$.
We write $\uA_{\inf}$ and $\ubA_{\inf,X}$
for the inverse systems $(A_{\inf,m})_{m\in \N}$
and $(\bA_{\inf,X,m})_{m\in\N}$. 
Then we have the following morphisms of ringed topos, where 
$T^{\N^{\circ}}$ for a topos $T$ denotes the topos of inverse
systems in $T$ indexed by $\N$ (\S\ref{sec:DTopos} after \eqref{eq:DtoposFBF}). 
See \eqref{eq:ProetGammaShfProj}.
\begin{gather}
((X_{\proet}^{\sim})^{\N^{\circ}},\ubA_{\inf,X})
\xrightarrow{\;\nu\;}
((\fX_{\Zar}^{\sim})^{\N^{\circ}},\uA_{\inf})
\xleftarrow{\:v\;}
((\fD_{\Zar}^{\sim})^{\N^{\circ}},\uA_{\inf})\label{eq:ProetZarMorph1}\\
((X_{\proet}^{\sim})^{\N^{\circ}},\ubA_{\inf,X})
\xrightarrow{\;\nu_{\infty}\;}
((\Gamma_{\Lambda}\hy\fX_{\Zar}^{\sim})^{\N^{\circ}},\uA_{\inf})
\overset{\pi}{\underset{\iota}{\rightleftarrows}}
((\fX_{\Zar}^{\sim})^{\N^{\circ}},\uA_{\inf})\label{eq:ProetZarMorph2}\\
\pi\circ \nu_{\infty}\cong \nu,\quad\pi\circ\iota\cong\id
\label{eq:ProetZarMorph3}
\end{gather}

Let $\CF$ be an object of $\Crystal_{\prism}^{\fproj}(\fX/A_{\inf})$
(Definition \ref{def:PrismaticSiteCrystal} (2)), and 
put $\bM:=\bM_{\BKF,\fX}(\CF)$ (Definition \ref{BKFFunctor}).
We define $\CF_m$ and $\bM_m$ to be the reduction modulo
$(p,[p]_q)^{m+1}$ of $\CF$ and $\bM$, respectively. 
The $\CO_{\fX/A_{\inf},m}$-module $\CF_m$ is a crystal of $\CO_{\fX/A_{\inf},m}$-modules
on $(\fX/A_{\inf})_{\prism}$ (Remark \eqref{rmk:PrismCrystalFrobTensor} (1)).
Let $\ubM$ denote the $\ubA_{\inf,X}$-module $(\bM_m)_{m\in \N}$.
We will construct a morphism 
\begin{equation}\label{eq:PrismCohAinfCohCompMapEmb}
\kappa_{\mti,\CF}\colon Ru_{\fX/A_{\inf}*}\CF\longrightarrow
L\eta_{\mu}R\varprojlim_{\N}R\nu_*\ubM\cong A\Omega_{\fX}(\bM)
\end{equation}
in $D(\fX_{\Zar},A_{\inf})$ by using Theorem \ref{thm:PrismCohQHiggsCpx},
and applying \eqref{eq:GSheafCohKCpx} to $\iota_{\Lambda,\fX}$
and $\pi_{\Lambda,\fX}$ in \eqref{eq:ProetGammaShfProj}
and $\nu_{\fX,\ut*}\bM_m$. We will then show the functoriality of \eqref{eq:PrismCohAinfCohCompMapEmb}
with respect to $\mathtt i=(\mfi\colon\fX\hookrightarrow \fY, \ut)$
(Proposition \ref{prop:PrismAinfCompFunct}),
which allows us to show that \eqref{eq:PrismCohAinfCohCompMapEmb}
is independent of the choice of $\mti$ (Proposition \ref{prop:PrismCohAinfCompMapIndep}).

We start by applying Theorem \ref{thm:PrismCohQHiggsCpx}
to $\CF$.
 Let $(M_m,\utheta_{M_m}=(\theta_{M_m,i})_{i\in\Lambda})$ be the object of 
 $q\HIG_{\qnilp}(D_m,\ut,\utheta_{D})$
 associated to $\CF_m$ by the equivalence of categories
in Theorem \ref{thm:PrismCrysqHiggsEquiv}, 
let $(\CM_m,\utheta_{\CM_m}=(\theta_{\CM_m,i})_{i\in\Lambda})$
be the $q$-Higgs module over
$(\fD, \ut,\utheta_{\fD})$ associated to $(M_m,\utheta_{M_m})$,
and let $q\Omegab(\CM_m,\utheta_{\CM_m})$ 
be the $q$-Higgs complex of $(\CM_m,\utheta_{\CM_m})$
(Construction \ref{constr:qHiggsModCpxShv} (1)).
We write $(\uCM,\utheta_{\uCM})$ for the inverse system of $q$-Higgs
modules $(\CM_m,\utheta_{\CM_m})_{m\in \N}$, and
$q\Omega^{\bullet}(\uCM, \utheta_{\uCM})$
for the complex of $\uA_{\inf}$-modules
$(q\Omega^{\bullet}(\CM_m,\utheta_{\CM_m}))_{m\in \N}$
on $(\fD_{\Zar}^{\sim})^{\N^{\circ}}$.
Then, by Definition \ref{def:PrismCrysqHigsCpxShv} and 
Theorem \ref{thm:PrismCohQHiggsCpx}, we have the following
isomorphism in $D^+(\fX_{\Zar},A_{\inf})$.
\begin{equation}\label{eq:AinfCrysProj}
a_{\mti,\CF}\colon Ru_{\fX/A_{\inf}*}\CF\xrightarrow{\;\cong\;} \varprojlim_{\N}
v_*(q\Omegab(\uCM,\utheta_{\uCM}))
 \end{equation}

 By applying 
Proposition \ref{prop:GammaRepKResol} and Lemma \ref{lem:KoszResolGalInv}
to $\nu_{\infty*}\ubM$, we obtain a resolution 
\begin{equation}\label{eq:BKFGammaRepResol}
\beta_{\mti,\CF}\colon \nu_{\infty*}\ubM\longrightarrow
K_{\Lambda}^{\bullet}(\iota_*\iota^*\nu_{\infty*}\ubM)
\quad\text{in}\;\;
C^+((\Gamma_{\Lambda}\hy\fX_{\Zar}^{\sim})^{\N^{\circ}},
\uA_{\inf})
\end{equation}
and an isomorphism
\begin{equation}\label{eq:GammaGalCohIsomMap}
\alpha_{\mti,\CF}\colon K^{\bullet}_{\Lambda}(\iota^*\nu_{\infty*}\ubM)
\xrightarrow{\;\cong\;}
\pi_*K_{\Lambda}^{\bullet}(\iota_*\iota^*\nu_{\infty*}\ubM)
\quad\text{in}\;\;
C^+((\fX_{\Zar}^{\sim})^{\N^{\circ}},\uA_{\inf}).
\end{equation}
By Corollary \ref{cor:IndGModGacyc} and \eqref{eq:InvSySFBF}, the resolution 
\eqref{eq:BKFGammaRepResol} yields an isomorphism
\begin{equation}\label{eq:BKFGammaCohKoszulResol}
R\pi_*(\nu_{\infty*}\ubM)\xrightarrow[R\pi_*(\beta_{\mti,\CF})]{\cong}
R\pi_*K^{\bullet}_{\Lambda}(\iota_*\iota^*\nu_{\infty*}\ubM)
\cong 
\pi_*K^{\bullet}_{\Lambda}(\iota_*\iota^*\nu_{\infty*}\ubM).
\end{equation}
By composing \eqref{eq:GammaGalCohIsomMap} 
and \eqref{eq:BKFGammaCohKoszulResol} with 
\begin{equation}\label{eq:BKFProetGammaCoh}
 R\pi_*(\nu_{\infty*}\ubM)
 \longrightarrow R\pi_*R\nu_{\infty*}\ubM
 \xleftarrow[\eqref{eq:ProetZarMorph3}]{\;\cong\;} R\nu_*\ubM,
 \end{equation}
 and taking $L\eta_{\mu}R\varprojlim_{\N}$, we obtain the following
morphism in $D(\fX_{\Zar},A_{\inf})$. 
 \begin{equation}\label{eq:AOmegaKoszul}
 b_{\mti,\CF}\colon
L\eta_{\mu}\varprojlim_{\N}K_{\Lambda}^{\bullet}(
 \iota^*\nu_{\infty*}\ubM)
 \longrightarrow
 L\eta_{\mu}R\varprojlim_{\N}R\nu_*\ubM\cong A\Omega_{\fX}(\bM)
  \end{equation}

 \begin{remark}\label{rmk:KoszAinfCohFrobProdComp}
 (1) Let $\CF$ be an object of $\Crystal_{\prism}^{\fproj}(\fX/A_{\inf})$,
 and put $\CF_{\varphi}=\varphi^*\CF\in \Ob \Crystal_{\prism}^{\fproj}(\fX/A_{\inf})$
 (Remark \ref{rmk:PrismCrystalFrobTensor} (2)). We write $\bM$ and $\bM_{\varphi}$
 for $\bM_{\BKF,\fX}(\CF)$ and $\bM_{\BKF,\fX}(\CF_{\varphi})$, respectively,
 and put $\ubM=\bM\otimes_{\bA_{\inf,X}}\ubA_{\inf,X}$
 and $\ubM_{\varphi}=\bM_{\varphi}\otimes_{\bA_{\inf,X}}\ubA_{\inf,X}$.
 We obtain from \eqref{eq:BKFFunctMapFrobPB} 
 a homomorphism 
 \begin{equation}\label{eq:BKFGammaShfMorph}
 \ubM\longrightarrow \ubM_{\varphi}\end{equation} 
 semilinear over the Frobenius of $\ubA_{\inf,X}$. 
 By applying Proposition \ref{prop:GSheafKoszulFunct} (1) and (2) to 
$\psi=\id\colon\Lambda\to\Lambda$, $u=\id\colon \fX_{\Zar}\to \fX_{\Zar}$,
 $\varphi\colon \uA_{\inf}\to \uA_{\inf}$, and the morphism of $\uA_{\inf}$-modules
\begin{equation}\label{eq:BKFGammaShfFrob}
\nu_{\infty*}\ubM \longrightarrow \nu_{\infty*}\varphi_*\ubM_{\varphi}
 \cong\varphi_*\nu_{\infty*}\ubM_{\varphi}\end{equation}
 on $\Gamma_{\Lambda}\hy\fX_{\Zar}^{\sim}$, where
 $\varphi_*$ denote the restriction of scalars under Frobenius,
we obtain morphisms
\begin{align}
K^{\bullet}_{\Lambda}(\iota^*\nu_{\infty*}\ubM)
 &\longrightarrow \varphi_*K^{\bullet}_{\Lambda}(\iota^*\nu_{\infty*}\ubM_{\varphi}),
 \label{eq:BKFGammaShfKosFrob}\\
 K^{\bullet}_{\Lambda}(\iota_*\iota^*\nu_{\infty*}\ubM)
&\longrightarrow\varphi_*K^{\bullet}_{\Lambda}(\iota_*\iota^*\nu_{\infty*}\ubM_{\varphi}).
\label{eq:BKFGammaShfKosFrob2}
\end{align}
 (Note that the base change morphism 
 $\iota^*\varphi_*\nu_{\infty*}\ubM_{\varphi}\to 
 \varphi_*\iota^*\nu_{\infty*}\ubM_{\varphi}$ 
 is given by the identity of the underlying $\uA_{\inf}$-modules
 on $\fX_{\Zar}$.) 
By \eqref{eq:IotaUnitFunct} and \eqref{eq:PiIotaIdFunct}, 
we see that the following diagrams are commutative.
\begin{equation}\label{eq:KoszGammaCohFrobComp}
\xymatrix@R=10pt@C=60pt{
K^{\bullet}_{\Lambda}(\iota^*\nu_{\infty*}\ubM)
\ar[r]^(.46){\eqref{eq:BKFGammaShfKosFrob}}
\ar[d]_{\alpha_{\mti,\CF}\;\;\eqref{eq:GammaGalCohIsomMap}}
&
\varphi_*K^{\bullet}_{\Lambda}(\iota^*\nu_{\infty*}\ubM_{\varphi})
\ar[d]^{\varphi_*\alpha_{\mti,\CF_{\varphi}}\;\;\eqref{eq:GammaGalCohIsomMap}}
\\
\pi_*K^{\bullet}_{\Lambda}(\iota_*\iota^*\nu_{\infty*}\ubM)
\ar[r]^(.46){\eqref{eq:BKFGammaShfKosFrob2}}
&
\varphi_*\pi_*K^{\bullet}_{\Lambda}(\iota_*\iota^*\nu_{\infty*}\ubM_{\varphi})
}
\end{equation}
\begin{equation}\label{eq:BKFGammaResolFrobComp}
\xymatrix@R=10pt@C=60pt{
\nu_{\infty*}\ubM
\ar[r]^{\eqref{eq:BKFGammaShfFrob}}
\ar[d]_{\beta_{\mti,\CF}\;\;\eqref{eq:BKFGammaRepResol}}
&
\varphi_*\nu_{\infty*}\ubM_{\varphi}
\ar[d]^{\beta_{\mti,\CF_{\varphi}}\;\;\eqref{eq:BKFGammaRepResol}}\\
K^{\bullet}_{\Lambda}(\iota_*\iota^*\nu_{\infty*}\ubM)
\ar[r]^{\eqref{eq:BKFGammaShfKosFrob2}}
&
\varphi_*K^{\bullet}_{\Lambda}(\iota_*\iota^*\nu_{\infty*}\ubM_{\varphi})
}
\end{equation}

 (2) Let $\CF_{\nu}$ $(\nu\in \{1,2\})$ be objects of
 $\Crystal_{\prism}^{\fproj}(\fX/A_{\inf})$, and put
 $\CF_3=\CF_1\otimes_{\CO_{\fX/A_{\inf}}}\CF_2
 \in\Ob\Crystal_{\prism}^{\fproj}(\fX/A_{\inf})$
 (Remark \ref{rmk:PrismCrystalFrobTensor} (3)). For $\nu\in \{1,2,3\}$,
 we write $\bM_{\nu}$ for $\bM_{\BKF,\fX}(\CF_{\nu})$,
 and put $\ubM_{\nu}=\bM_{\nu}\otimes_{\bA_{\inf,X}}\ubA_{\inf,X}$.
 Then we have an isomorphism 
 \begin{equation}\label{eq:BKFInvSysTensor}
\ubM_1\otimes_{\ubA_{\inf,X}}\ubM_2\cong \ubM_3
\end{equation}
of $\ubA_{\inf,X}$-modules induced by \eqref{eq:BKFFunctorTensor}.
By applying the construction of the upper horizontal morphism in 
\eqref{eq:KoszulGammaInvProdComp} to 
$\nu_{\infty*}\ubM_{\nu}$ $(\nu=1,2)$ and combining it 
with the the morphism 
\begin{equation}\label{eq:BKFGammaInvSysTensor}
\nu_{\infty*}\ubM_1\otimes_{\uA_{\inf}}\nu_{\infty*}\ubM_2
\to \nu_{\infty*}\ubM_3
\end{equation}
induced by \eqref{eq:BKFInvSysTensor}, 
we obtain a morphism
\begin{equation}\label{eq:BKFInvSySKoszProd1}
K^{\bullet}_{\Lambda}(\iota^*\nu_{\infty*}\ubM_1)
\otimes_{\uA_{\inf}}
K^{\bullet}_{\Lambda}(\iota^*\nu_{\infty*}\ubM_2)
\longrightarrow 
K^{\bullet}_{\Lambda}(\iota^*\nu_{\infty*}\ubM_3).
\end{equation}
By applying the argument in Remark \ref{rmk:GShfCohKoszProd} to $\nu_{\infty*}\ubM_{\nu}$ $(\nu=1,2)$
and combining it with \eqref{eq:BKFGammaInvSysTensor},  we obtain a morphism
\begin{equation}\label{eq:BKFInvSySKoszProd2}
K^{\bullet}_{\Lambda}(\iota_*\iota^*\nu_{\infty*}\ubM_1)
\otimes_{\uA_{\inf}}
K^{\bullet}_{\Lambda}(\iota_*\iota^*\nu_{\infty*}\ubM_2)
\longrightarrow
K^{\bullet}_{\Lambda}(\iota_*\iota^*\nu_{\infty*}\ubM_3),
\end{equation}
and see that the morphisms $\alpha_{\mti,\CF_{\nu}}$ \eqref{eq:GammaGalCohIsomMap}
and $\beta_{\mti,\CF_{\nu}}$ \eqref{eq:BKFGammaRepResol} ($\nu\in \{1,2,3\}$) 
are compatible with the products, i.e., the following diagrams are commutative.
\begin{equation}\label{eq:KoszulGammaCohProjFunct}
\xymatrix@R=15pt{
K_{\Lambda}^{\bullet}(\iota^*\nu_{\infty*}\ubM_1)
\otimes_{\uA_{\inf}}
K_{\Lambda}^{\bullet}(\iota^*\nu_{\infty*}\ubM_2)
\ar[d]_{\alpha_{\mti,\CF_1}\otimes_{\uA_{\inf}} \alpha_{\mti,\CF_2}}
\ar[r]^(.65){\eqref{eq:BKFInvSySKoszProd1}}&
K_{\Lambda}^{\bullet}(\iota^*\nu_{\infty*}\ubM_3)
\ar[d]^{\alpha_{\mti,\CF_3}}\\
\pi_*K_{\Lambda}^{\bullet}(\iota_*\iota^*\nu_{\infty*}\ubM_1)
\otimes_{\uA_{\inf}}
\pi_*K_{\Lambda}^{\bullet}(\iota_*\iota^*\nu_{\infty*}\ubM_2)
\ar[r]^(.67){\eqref{eq:BKFInvSySKoszProd2}}
&
\pi_*K_{\Lambda}^{\bullet}(\iota_*\iota^*\nu_{\infty*}\ubM_3)
}
\end{equation}
\begin{equation}
\label{eq:BKFGammaModResolFunct}
\xymatrix@R=15pt{
\nu_{\infty*}\ubM_1\otimes_{\uA_{\inf}}\nu_{\infty*}\ubM_2
\ar[d]_{\beta_{\mti,\CF_1}\otimes_{\uA_{\inf}} \beta_{\mti,\CF_2}}
\ar[r]^(.6){\eqref{eq:BKFGammaInvSysTensor}}&
\nu_{\infty*}\ubM_3
\ar[d]^{\beta_{\mti,\CF_3}}\\
K^{\bullet}_{\Lambda}(\iota_*\iota^*\nu_{\infty*}\ubM_1)
\otimes_{\uA_{\inf}}
K^{\bullet}_{\Lambda}(\iota_*\iota^*\nu_{\infty*}\ubM_2)
\ar[r]^(.67){\eqref{eq:BKFInvSySKoszProd1}}
&
K^{\bullet}_{\Lambda}(\iota_*\iota^*\nu_{\infty*}\ubM_3)
}
\end{equation}
 \end{remark}

 We construct the morphism \eqref{eq:PrismCohAinfCohCompMapEmb}
 by relating the codomain of $a_{\mti, \CF}$ \eqref{eq:AinfCrysProj} with the 
 domain of $b_{\mti,\CF}$ \eqref{eq:AOmegaKoszul}. 
  We follow the notation introduced in the construction of $\nu_{\fX,\ut}$
 \eqref{eq:ProetGammaShvProj} in \S\ref{sec:ProetGammaZarShv}. 
  Let $\fU=\Spf(A_{\fU})$ be an open affine formal subscheme of $\fX$,
 let $U$ be its adic generic fiber $\Spa(A_{\fU}[\frac{1}{p}],A_{\fU})$,
 and put $U_n=U_{\Gamma_{\Lambda}/p^n\Gamma_{\Lambda}}$,
 which is isomorphic to $\Spa(A_{\fU,n}[\frac{1}{p}],\CA_{\fU,n})$,
 where $A_{\fU,n}$ is defined in the same way as $A_n$ by replacing
 $A$ by $A_{\fU}$. 
 Let $U_{\infty}$ be the object
`` $\varprojlim_n$"$U_n$ of $X_{\proet}$, which is an affinoid
perfectoid (\cite[Definition 4.3 (i)]{ScholzepHTRigid}) by 
\eqref{cond:frameAinf2}, and put 
$A_{U_{\infty}}^+=\Gamma(U_{\infty},\hCO_{X}^+)=
(\varinjlim_n A_{\fU,n})^{\wedge}$ 
(\cite[Lemma 4.10 (iii)]{ScholzepHTRigid}), where
$\widehat{\;\;}$ denotes the $p$-adic completion.
By \eqref{eq:ProetGammaShvProjDescrip} and
Proposition \ref{prop:BKFProperties} (4), we have
\begin{equation}\label{eq:ProetGammaModIm}
(\nu_{\infty*}\ubM)(\fU)
\xrightarrow{\;\cong\;}
\varinjlim_n\ubM(U_n)
\xrightarrow{\;\cong\;}\ubM(U_{\infty}).
\end{equation}
The action of $\Gamma_{\Lambda}$ on 
$(\nu_{\infty*}\ubM)(\fU)$ is given by the right action of
$\Gamma_{\Lambda}$ on $U_{\infty}$.

We write $X_n$, $X_{\infty}$, and $A_{\infty}$
for $U_n$, $U_{\infty}$, and $A_{U_{\infty}}^+$
when $\fU=\fX$. 
Let $t_{A,i,n}$ $(n\in \N)$ be the image of $T_{i}^{1/p^n}$ in $A_n$,
and let $t_{A,i}^{\flat}$ be the element $(t_{A,i,n})_{n\in \N}$
of the tilt $A_{\infty}^{\flat}=\varprojlim_{\N,x\mapsto x^p}A_{\infty}
=\varprojlim_{\N,\text{Frob}}A_{\infty}/pA_{\infty}$ of $A_{\infty}$.
We define a $\delta$-$A_{\inf}$-algebra structure on 
$A_{\inf}[\uT^{\pm 1}]=A_{\inf}[T_i^{\pm 1} (i\in \Lambda)]$ by $\delta(T_i)=0$ 
$(i\in \Lambda)$, and define an
action of $\Gamma_{\Lambda}^{\disc}=\Map(\Lambda,\Z)$
on the $\delta$-$A_{\inf}$-algebra $A_{\inf}[\uT^{\pm 1}]$ by 
$\gamma(T_i)=[\varepsilon^{\gamma(i)}]^pT_i$ 
$(i\in \Lambda,\gamma\in\Gamma_{\Lambda}^{\disc})$.
We define an action of $\Gamma_{\Lambda}^{\disc}$ on
the $\delta$-$A_{\inf}$-algebra $B$ by 
$\gamma(x)=(\prod_{i\in\Lambda}\gamma_{B,i}^{\gamma(i)})(x)$
$(x\in B, \gamma\in\Gamma_{\Lambda}^{\disc})$. 
Then we have a commutative diagram of $A_{\inf}$-algebras with
$\Gamma_{\Lambda}^{\disc}$-action 
\begin{equation}\label{eq:BToAinfDiag}
\xymatrix{
A_{\inf}(A_{\infty})/[p]_qA_{\inf}(A_{\infty})\cong A_{\infty}&\ar[l] A\cong B/J&\ar[l] B\\
A_{\inf}(A_{\infty})\ar[u]&&\ar[ll]A_{\inf}[\uT^{\pm 1}],\ar[u]
}
\end{equation}
where the right vertical (resp.~bottom horizontal) homomorphism 
is defined by $T_i\mapsto t_i$ (resp.~$[t_{A,i}^{\flat}]^p$) for 
$i\in\Lambda$. Since 
the right vertical homomorphism is $(p,\pq)$-adically \'etale
(Definition \ref{def:IadicProperties} (1)), the kernel of the left vertical one is 
$(p,\pq)$-adically nilpotent, and $A_{\inf}(A_{\infty})$
(resp.~$A_{\infty}$) is $(p,\pq)$-(resp.~$p$-)adically complete and separated, there exists
a  unique homomorphism 
$B\to A_{\inf}(A_{\infty})$ making the above diagram commutative,
and it is $\Gamma^{\disc}_{\Lambda}$-equivariant.
Since the bottom horizontal map is a $\delta$-homomorphism,
Proposition \ref{prop:DeltaStruEtaleMap} (2) implies that 
this underlies a $\Gamma_{\Lambda}^{\disc}$-equivariant
homomorphism of $\delta$-pairs
$(B,J)\to (A_{\inf}(A_{\infty}),\pq A_{\inf}(A_{\infty}))$
over the $q$-prism $(A_{\inf},\pq A_{\inf})$, which extends
uniquely to a 
morphism of $q$-prisms $D\to A_{\inf}(A_{\infty})$ over $A_{\inf}$.
It is $\Gamma^{\disc}_{\Lambda}$-equivariant
for the action of $\Gamma^{\disc}_{\Lambda}$ on the 
$\delta$-$A_{\inf}$-algebra $D$ defined by 
$\gamma(x)=\prod_{i\in\Lambda}\gamma_{D,i}^{\gamma(i)}(x)$
$(x\in D,\gamma\in \Gamma^{\disc}_{\Lambda})$, which 
is the unique extension of the action of $\Gamma^{\disc}_{\Lambda}$
on the $\delta$-$A_{\inf}$-algebra $B$ to the $\delta$-$A_{\inf}$-algebra $D$. 
Thus we obtain a $\Gamma_{\Lambda}^{\disc}$-equivariant
morphism  in $(\fX/A_{\inf})_{\prism}$
\begin{equation}\label{eq:MorphAinfToD}
p_{D,\ut}\colon (A_{\inf}(A_{\infty}),\pq A_{\inf}(A_{\infty}))
\to (D,\pq D).
\end{equation}
Note that the action of $\Gamma_{\Lambda}^{\disc}$ on $D$ defines
an action of $\Gamma_{\Lambda}^{\disc}$ on the object 
$(D,\pq D)$ of $(\fX/A_{\inf})_{\prism}$ as observed in Definition 
\ref{def:qNilpQHiggs} (2).

Let $\fU=\Spf(A_{\fU})$ be an open affine formal subscheme of $\fX$,
and let $\ofD_{\fU}=\Spf(\oD_{\fU})$ and $\fD_{\fU}=\Spf(D_{\fU})$
be the open affine formal subschemes of $\ofD$ and $\fD$,
respectively, whose underlying set is $v_D^{-1}(\fU)$. Then, since
the morphism $\Spf(A_{U_{\infty}}^+)\to\Spf(A_{\infty})
\to \fX=\Spf(A)$ factors through $\fU=\Spf(A_{\fU})$, the composition 
$A_{\inf}(A^+_{U_{\infty}})\to A_{\inf}(A_{\infty})
\xrightarrow{p_{D,\ut}} D$ in $(\fX/A_{\inf})_{\prism}$
factors uniquely as $A_{\inf}(A^+_{U_{\infty}})
\xrightarrow{p_{D_{\fU},\ut}} D_{\fU}\to D$.
The morphism $p_{D_{\fU},\ut}$ induces $A_{\inf}$-linear maps
\begin{multline}\label{eq:CrysBKFMap}
(v_{D*}\CM_m)(\fU)
=\CM_m(\fD_{\fU})
=\CF_m(D_{\fU})
\\
\xrightarrow{\CF_m(p_{D_{\fU},\ut})}
\CF_m(A_{\inf}(A^+_{U_{\infty}}))=
\bM(U_{\infty})/(p,\pq)^{m+1}\bM(U_{\infty})
\longrightarrow \bM_m(U_{\infty}).
\end{multline}
Since the morphism $p_{D,\ut}$ is $\Gamma_{\Lambda}^{\disc}$-equivariant,
the morphism $p_{D_{\fU},\ut}$ is also $\Gamma_{\Lambda}^{\disc}$-equivariant.
By \cite[(13.4)]{TsujiPrismQHiggs}, we see that $v_{D*}(1+t_i\mu\theta_{\CM_m,i})(\fU)$
on $(v_{D*}\CM_m)(\fU)$ is compatible
with the action of 
$\gamma_i\in \Gamma_{\Lambda}^{\disc}$
on $\bM_m(U_{\infty})$ via \eqref{eq:CrysBKFMap}, 
where $\gamma_i$ is defined
by $\gamma_i(j)=1$ if $j=i$ and $0$ otherwise,
as before \eqref{eq:KoszulCpxProd}.
We define the action of $\Gamma_{\Lambda}^{\disc}$
on $v_{D*}\CM_m$ by letting $\gamma_i$
act by $v_{D*}(1+t_i\mu\theta_{\CM_m,i})$ for $i\in \Lambda$.

The morphisms \eqref{eq:ProetGammaModIm} and 
\eqref{eq:CrysBKFMap} are functorial in $\fU$ by their constructions. Therefore they
induce a $\Gamma_{\Lambda}^{\disc}$-equivariant morphism of $\uA_{\inf}$-modules on 
$\fX_{\Zar}$
\begin{equation}\label{eq:qHigBKFmap}
\delta_{\mti,\CF}\colon 
v_*\uCM=(v_{D*}\CM_m)_{m\in \N}\longrightarrow
\iota^*\nu_{\infty*}\ubM,
\end{equation}
which yields a morphism of complexes of $\uA_{\inf}$-modules on $\fX_{\Zar}$
\begin{equation}\label{eq:qHigBKFKoszulmap}
K^{\bullet}_{\Lambda}(\delta_{\mti,\CF})\colon
K^{\bullet}_{\Lambda}(v_*\uCM)\longrightarrow
K^{\bullet}_{\Lambda}(\iota^*\nu_{\infty*}\ubM).
\end{equation}

\begin{lemma}\label{lem:qdRKoszulComp}
The homomorphism 
$v_{D*}(q\Omega^r(\CM_m,\utheta_{\CM_m}))
\to K_{\Lambda}^r(v_{D*}\CM_m)$ $(r\in \N)$
sending $m\otimes\omega_{i_1}\wedge\cdots\wedge\omega_{i_r}$
to $m\otimes \mu^r\prod_{\nu=1}^rt_{i_{\nu}}e_{i_1}\wedge\cdots\wedge e_{i_r}$
defines a morphism of complexes of $\uA_{\inf}$-modules
on $\fX_{\Zar}$
\begin{equation}\label{eq:qHiggscpxInvSysKoszul}
\gamma_{\mti,\CF}\colon v_*(q\Omegab(\uCM,\utheta_{\uCM}))
\longrightarrow K_{\Lambda}^{\bullet}(v_*\uCM),
\end{equation}
which induces an isomorphism
\begin{equation}\label{eq:qdRcpxKoszul}
\hgamma_{\mti,\CF}\colon\varprojlim_{\N} v_*(q\Omegab(\uCM,\utheta_{\uCM}))
\xrightarrow{\;\cong\;}
\eta_{\mu}K_{\Lambda}^{\bullet}(\varprojlim_{\N}v_*\uCM).
\end{equation}
\end{lemma}
\begin{proof} The former follows from the definition of the action of 
$\gamma_i$ on $v_{D*}\CM_m$ and $\theta_{\CM_{m},i}\circ t_j\mu\cdot \id_{\CM_m}
=t_j\mu\cdot \id_{\CM_m}\circ\theta_{\CM_m,i}$ for $i\neq j$. Since $\gamma_i-1$ on $v_{D*}\CM_m$
is $A_{\inf}$-linear and trivial modulo $\mu$, we see
$(\eta_{\mu}K^{\bullet}_{\Lambda}(\varprojlim_{\N}v_*\uCM))^r
=\mu^r K_{\Lambda}^r(\varprojlim_{\N}v_*\uCM)$.
This implies the latter claim.\end{proof}

\begin{remark}\label{rmk:qdRcpxKoszulFrobProdComp}
(1) We can verify the compatibility of \eqref{eq:qHigBKFKoszulmap}
and \eqref{eq:qHiggscpxInvSysKoszul} with Frobenius
pullback as follows. Let $\CF\in \Ob \Crystal_{\prism}^{\fproj}(\fX/A_{\inf})$
and put $\CF_{\varphi}=\varphi^*\CF\in \Ob \Crystal_{\prism}^{\fproj}(\fX/A_{\inf})$
(Remark \ref{rmk:PrismCrystalFrobTensor} (2)). 
As before \eqref{eq:AinfCrysProj}, we define 
$(\uCM,\utheta_{\uCM})=(\CM_m,\utheta_{\CM_m})_{m\in \N}$
(resp.~$(\uCM_{\varphi},\utheta_{\uCM_{\varphi}})=
(\CM_{\varphi,m},\utheta_{\CM_{\varphi,m}})_{m\in \N}$)
to be the inverse system of $q$-Higgs modules over $(\fD,\ut,\utheta_{\fD})$
associated to $(\CF/(p,\pq)^{m+1}\CF)_{m\in\N}$
(resp.~$(\CF_{\varphi}/(p,\pq)^{m+1}\CF_{\varphi})_{m\in \N}$)
by Definition \ref{def:PrismCrysqHigsCpxShv}.
By Theorem \ref{thm:PrismCrysqHiggsEquiv} (3) and 
Construction \ref{constr:qHiggsModCpxShv} (3), we have 
$\CM_{\varphi,m}=\varphi_{\fD_m}^*\CM_m=
\CM_m\otimes_{\CO_{\fD_m},\varphi_{\fD_m}}\CO_{\fD_m}$ and 
$\theta_{\CM_{\varphi,m},i}(y\otimes 1)=
\theta_{\CM_m,i}(y)\otimes t_i^{p-1}\pq$
for $i\in \Lambda$ and a local section 
$y$ of $\CM_m$ on an affine open of $\fD$.
By the definition of the $\Gamma_{\Lambda}^{\disc}$-action
on $v_{D*}\CM_m$ and $v_{D*}\CM_{\varphi,m}$ given after
\eqref{eq:CrysBKFMap}, we see that the morphism 
$v_{D*}\CM_m\to \varphi_*v_{D*}\CM_{\varphi,m};x
\mapsto x\otimes 1$ is compatible with the actions of
$\Gamma_{\Lambda}^{\disc}$, and therefore induces 
a morphism of complexes
\begin{equation}\label{eq:qHigKoszulInvSysFrobPB}
K_{\Lambda}^{\bullet}(v_{*}\uCM)\longrightarrow
\varphi_*K_{\Lambda}^{\bullet}(v_{*}\uCM_{\varphi}).
\end{equation}
It is straightforward to verify that the morphisms
$\delta_{\mti,\CF}$ and $\delta_{\mti,\CF_{\varphi}}$
\eqref{eq:qHigBKFmap} 
are compatible with the morphisms between their domains
and codomains induced by $\CM_{m}\to\varphi_*\CM_{\varphi,m};
x\mapsto x\otimes1$ $(m\in \N)$ and \eqref{eq:BKFGammaShfFrob}, respectively.
Therefore the morphisms $K_{\Lambda}^{\bullet}(\delta_{\mti,\CF})$ and 
$K_{\Lambda}^{\bullet}(\delta_{\mti,\CF_{\varphi}})$
\eqref{eq:qHigBKFKoszulmap} are compatible with \eqref{eq:qHigKoszulInvSysFrobPB} and 
\eqref{eq:BKFGammaShfKosFrob}, i.e., the following diagram is commutative.
\begin{equation}\label{eq:HiggsKoszulBKFFrobComp}
\xymatrix@R=10pt@C=60pt{
K^{\bullet}_{\Lambda}(v_*\uCM)
\ar[r]^{\eqref{eq:qHigKoszulInvSysFrobPB}}
\ar[d]_{K_{\Lambda}^{\bullet}(\delta_{\mti,\CF})}
&
\varphi_*K^{\bullet}_{\Lambda}(v_*\uCM_{\varphi})
\ar[d]^{\varphi_*K^{\bullet}_{\Lambda}(\delta_{\mti,\CF_{\varphi}})}
\\
K_{\Lambda}^{\bullet}(\iota^*\nu_{\infty*}\ubM)
\ar[r]^{\eqref{eq:BKFGammaShfKosFrob}}&
\varphi_*K_{\Lambda}^{\bullet}(\iota^*\nu_{\infty*}\ubM_{\varphi})
}
\end{equation}
Let $q\Omega^{\bullet}(\uCM,\utheta_{\uCM})$
(resp.~$q\Omega^{\bullet}(\uCM_{\varphi},\utheta_{\uCM_{\varphi}})$)
denote the inverse system  of $q$-Higgs complexes \linebreak
$(q\Omega^{\bullet}(\CM_m,\utheta_{\CM_m}))_{m\in \N}$
(resp.~$(q\Omega^{\bullet}(\CM_{\varphi,m},\utheta_{\CM_{\varphi,m}}))_{m\in \N}$).
Then, by applying \eqref{eq:CrysqHiggCpxFrobPB} 
to $\CF/(p,\pq)^{m+1}\CF$ $(m\in \N)$, we obtain a 
morphism of complexes
\begin{equation}\label{eq:qHiggsInvSysFrobPB}
v_*(q\Omega^{\bullet}(\uCM,\utheta_{\uCM}))
\longrightarrow \varphi_*v_*(q\Omega^{\bullet}(\uCM_{\varphi},\utheta_{\uCM_{\varphi}})),
\end{equation}
which is explicitly given by $x\otimes\omega_{\bmI}
\mapsto x\otimes \pq^rt_{\bmI}^{p-1}\otimes\omega_{\bmI}$
for a section $x$ of $v_{D*}\CM_m$ on an affine open of $\fX$,
$r\in \N$, $\bmI=(i_n)_{1\leq n\leq r}\in\Lambda^r$, 
$\omega_{\bmI}=\omega_{i_1}\wedge\ldots\wedge \omega_{i_r}$,
and $t_{\bmI}=\prod_{n=1}^rt_{i_n}$. 
By simple explicit computation, we can verify that
the morphisms \eqref{eq:qHiggsInvSysFrobPB} and 
\eqref{eq:qHigKoszulInvSysFrobPB} are compatible
with $\gamma_{\mti,\CF}$ and $\gamma_{\mti,\CF_{\varphi}}$
\eqref{eq:qHiggscpxInvSysKoszul}, i.e., the following diagram
is commutative. 
\begin{equation}\label{eq:qHiggsKoszFrobComp}
\xymatrix@R=10pt@C=60pt{
v_*(q\Omega^{\bullet}(\uCM,\utheta_{\uCM}))
\ar[r]^{\eqref{eq:qHiggsInvSysFrobPB}}
\ar[d]_{\gamma_{\mti,\CF}}
&
 \varphi_*v_*(q\Omega^{\bullet}(\uCM_{\varphi},\utheta_{\uCM_{\varphi}}))
\ar[d]^{\varphi_*\gamma_{\mti,\CF_{\varphi}}}\\
K_{\Lambda}^{\bullet}(v_{*}\uCM)
\ar[r]^{\eqref{eq:qHigKoszulInvSysFrobPB}}
&
\varphi_*K_{\Lambda}^{\bullet}(v_{*}\uCM_{\varphi})
}
\end{equation}
Indeed
the image of $x\otimes\omega_{\bmI}$ under the
composition of
$\eqref{eq:qHiggsInvSysFrobPB}$ 
and $\varphi_*\gamma_{\mti,\CF_{\varphi}}$ 
(resp.~$\gamma_{\mti,\CF}$ and \eqref{eq:qHigKoszulInvSysFrobPB})
is computed as $x\otimes\omega_{\bmI}
\mapsto x\otimes\pq^rt_{\bmI}^{p-1}\otimes\omega_{\bmI}
\mapsto x\otimes \pq^rt_{\bmI}^{p-1}\mu^rt_{\bmI}\otimes e_{\bmI}$
(resp.~$\mapsto\mu^rt_{\bmI}x\otimes e_{\bmI}
\mapsto x\otimes \varphi(\mu)^rt_{\bmI}^p\otimes e_{\bmI}$).

(2) We keep the notation and assumption 
in Remark \ref{rmk:KoszAinfCohFrobProdComp} (2). 
As before \eqref{eq:AinfCrysProj}, we define 
$(\CM_{\nu,m},\utheta_{\CM_{\nu,m}})$ $(m\in \N,\nu\in \{1,2,3\})$
to be the $q$-Higgs module over $(\fD,\ut,\utheta_{\fD})$
associated to $\CF_{\nu}/(p,\pq)^{m+1}\CF_{\nu}$ 
by Definition \ref{def:PrismCrysqHigsCpxShv}.
Put $(\uCM_{\nu}, \utheta_{\uCM_{\nu}})=
(\CM_{\nu,m},\utheta_{\CM_{\nu,m}})_{m\in \N}$
and $q\Omega^{\bullet}(\uCM_{\nu},\utheta_{\uCM_{\nu}})
=(q\Omega^{\bullet}(\CM_{\nu,m},\utheta_{\CM_{\nu,m}}))_{m\in \N}$.
Then the composition of \eqref{eq:CrysBKFMap} is obviously compatible with the
products $v_*\uCM_1(\fU)\times v_*\uCM_2(\fU)
\to v_*\uCM_3(\fU)$ and $\ubM_1(U_{\infty})\times
\ubM_2(U_{\infty})\to \ubM_3(U_{\infty})$. Therefore
the morphisms $\delta_{\mti,\CF_{\nu}}$ $(\nu\in \{1,2,3\})$
\eqref{eq:qHigBKFmap} are compatible with \eqref{eq:BKFGammaInvSysTensor}
and $v_*\uCM_1\otimes_{\uA_{\inf}}v_*\uCM_2
\to v_*\uCM_3$, which is $\Gamma_{\Lambda}^{\disc}$-equivariant 
by Theorem \ref{thm:PrismCrysqHiggsEquiv} (4) and the remark on the endomorphisms $\gamma$'s after
\eqref{eq:TensorSingleConnection}. 
Hence the morphisms $K_{\Lambda}^{\bullet}(\delta_{\mti,\CF_{\nu}})$ $(\nu\in \{1,2,3\})$
\eqref{eq:qHigBKFKoszulmap} are compatible with the product \eqref{eq:KoszulCpxProd}
\begin{equation}\label{eq:qHiggsInvSysKoszProd}
K_{\Lambda}^{\bullet}(v_*\uCM_1)\otimes_{\uA_{\inf}}
K_{\Lambda}^{\bullet}(v_*\uCM_2)\longrightarrow
 K_{\Lambda}^{\bullet}(v_*\uCM_3)
\end{equation}
and \eqref{eq:BKFInvSySKoszProd1}, i.e., the following diagram 
is commutative.
\begin{equation}\label{eq:qHigCpxKoszProdComp}
\xymatrix@C=60pt@R=15pt{
K_{\Lambda}^{\bullet}(v_*\uCM_1)\otimes_{\uA_{\inf}}
K_{\Lambda}^{\bullet}(v_*\uCM_2)
\ar[d]_{K_{\Lambda}^{\bullet}(\delta_{\mti,\CF_1})\otimes_{\uA_{\inf}}
K_{\Lambda}^{\bullet}(\delta_{\mti,\CF_2})}^{\eqref{eq:qHigBKFKoszulmap}}
\ar[r]^(.6){\eqref{eq:qHiggsInvSysKoszProd}}
&
K_{\Lambda}^{\bullet}(v_*\uCM_3)
\ar[d]^{K_{\Lambda}^{\bullet}(\delta_{\mti,\CF_3})}_{\eqref{eq:qHigBKFKoszulmap}}\\
K^{\bullet}_{\Lambda}(\iota^*\nu_{\infty*}\ubM_1)
\otimes_{\uA_{\inf}}
K^{\bullet}_{\Lambda}(\iota^*\nu_{\infty*}\ubM_2)
\ar[r]^(.65){\eqref{eq:BKFInvSySKoszProd1}}
&
K^{\bullet}_{\Lambda}(\iota^*\nu_{\infty*}\ubM_3)
}
\end{equation}
On the other hand, we see that the products
\eqref{eq:CrysqHiggsCpxProd} 
\begin{equation}\label{eq:qHiggsInvSysProd}
q\Omega^{\bullet}(\uCM_1,\utheta_{\uCM_1})
\otimes_{\uA_{\inf}}
q\Omega^{\bullet}(\uCM_2,\utheta_{\uCM_2})
\longrightarrow
q\Omega^{\bullet}(\uCM_3,\utheta_{\uCM_3})
\end{equation}
and  \eqref{eq:qHiggsInvSysKoszProd}
are compatible with $\gamma_{\mti,\CF_{\nu}}$ $(\nu\in\{1,2,3\})$
\eqref{eq:qHiggscpxInvSysKoszul}, i.e., the diagram
\begin{equation}\label{eq:qDRKoszProdComp}
\xymatrix@R=15pt@C=50pt{
q\Omega^{\bullet}(\uCM_1,\utheta_{\uCM_1})
\otimes_{\uA_{\inf}}
q\Omega^{\bullet}(\uCM_2,\utheta_{\uCM_2})
\ar[d]_{\gamma_{\mti,\CF_1}\otimes_{\uA_{\inf}}\gamma_{\mti,\CF_2}}
^{\eqref{eq:qHiggscpxInvSysKoszul}}
\ar[r]^(.65){\eqref{eq:qHiggsInvSysProd}}
&
q\Omega^{\bullet}(\uCM_3,\utheta_{\uCM_3})
\ar[d]^{\gamma_{\mti,\CF_3}}_{\eqref{eq:qHiggscpxInvSysKoszul}}
\\
K_{\Lambda}^{\bullet}(v_*\uCM_1)\otimes_{\uA_{\inf}}
K_{\Lambda}^{\bullet}(v_*\uCM_2)
\ar[r]^(.65){\eqref{eq:qHiggsInvSysKoszProd}}
&
K_{\Lambda}^{\bullet}(v_*\uCM_3)
}
\end{equation}
is commutative by going
back to the definitions of the morphisms and products. 
\end{remark}

The composition of $\hgamma_{\mti,\CF}$ \eqref{eq:qdRcpxKoszul} with 
$\eta_{\mu}\varprojlim_{\N}K_{\Lambda}^{\bullet}(\delta_{\mti,\CF})$ \eqref{eq:qHigBKFKoszulmap}
gives  a morphism of complexes of $A_{\inf}$-modules on $\fX_{\Zar}$
\begin{equation}\label{eq:qdRKoszulMorLoc}
c_{\mti,\CF}\colon
\varprojlim_{\N}v_*(q\Omegab(\uCM,\utheta_{\uCM}))
\longrightarrow
\eta_{\mu}K^{\bullet}_{\Lambda}(\varprojlim_{\N}\iota^*
\nu_{\infty*}\ubM).
\end{equation}
Composing $c_{\mti,\CF}$ \eqref{eq:qdRKoszulMorLoc} with 
$a_{\mti,\CF}$ \eqref{eq:AinfCrysProj} and 
$b_{\mti,\CF}$ \eqref{eq:AOmegaKoszul}, we obtain the desired
morphism \eqref{eq:PrismCohAinfCohCompMapEmb}.
\begin{equation}\label{eq:PrismCohAinfCohCompMap2}
\kappa_{\mti,\CF}\colon 
Ru_{\fX/A_{\inf}*}\CF\longrightarrow
A\Omega_{\fX}(\bM)\cong L\eta_{\mu}R\varprojlim_{\N}R\nu_{*}\ubM
\end{equation}

\begin{proposition}
Put $\CF_{\varphi}=\varphi^*\CF\in \Ob \Crystal^{\fproj}_{\prism}(\fX/A_{\inf})$
(Remark \ref{rmk:PrismCrystalFrobTensor} (2)) and $\bM_{\varphi}=\bM_{\BKF,\fX}(\CF_{\varphi})$.
Then the following diagram is commutative, where the left vertical 
morphism is induced by $\CF\to\CF_{\varphi};x\mapsto x\otimes 1$.
\begin{equation}
\xymatrix@R=10pt@C=60pt{
Ru_{\fX/A_{\inf}*}\CF\ar[r]^{\kappa_{\mti,\CF}}\ar[d]&
A\Omega_{\fX}(\bM)\ar[d]^{\eqref{eq:AinfCohBKFFrobPB}}\\
\varphi_*Ru_{\fX/A_{\inf}*}\CF_{\varphi}
\ar[r]^{\varphi_*\kappa_{\mti,\CF_{\varphi}}}&
\varphi_*A\Omega_{\fX}(\bM_{\varphi})}
\end{equation}
\end{proposition}

\begin{proof}
We follow the notation introduced in Remarks \ref{rmk:KoszAinfCohFrobProdComp} (1)
and \ref{rmk:qdRcpxKoszulFrobProdComp} (1). 
By \eqref{eq:qHiggsKoszFrobComp}, \eqref{eq:HiggsKoszulBKFFrobComp}, 
\eqref{eq:KoszGammaCohFrobComp}, and \eqref{eq:BKFGammaResolFrobComp}, 
we see that the morphisms \eqref{eq:qdRKoszulMorLoc}
\begin{align*}
&(\eta_{\mu}\varprojlim_{\N}\alpha_{\mti,\CF})\circ c_{\mti,\CF}
\colon \varprojlim_{\N} v_*(q\Omega^{\bullet}(\uCM,\utheta_{\uCM}))\longrightarrow 
\eta_{\mu}\varprojlim_{\N} \pi_*K^{\bullet}_{\Lambda}(\iota_*\iota^*\nu_{\infty*}\ubM),\\
&L\eta_{\mu}R\varprojlim_{\N}R\pi_*(\beta_{\mti,\CF})\colon 
L\eta_{\mu}R\varprojlim_{\N} R\pi_*(\nu_{\infty*}\ubM)
\xrightarrow{\;\;\cong\;\;}
L\eta_{\mu}R\varprojlim_{\N} R\pi_*K^{\bullet}_{\Lambda}(\iota_*\iota^*\nu_{\infty*}\ubM)
\end{align*}
and the corresponding ones for $\CF_{\varphi}$
are compatible with the Frobenius pullbacks induced by 
\eqref{eq:qHiggsInvSysFrobPB}, \eqref{eq:BKFGammaShfKosFrob2},
and \eqref{eq:BKFGammaShfFrob} together with the natural morphisms
$\eta_{\mu}\varphi_*\to \varphi_*\eta_{\mu}$ and $L\eta_{\mu}\varphi_*\to \varphi_*L\eta_{\mu}$. 
We obtain the claim by combining the two 
compatibility and Theorem \ref{thm:PrismCohQHiggsCpx} (2), and by noting that the morphism
$L\eta_{\mu}R\varprojlim_{\N}R\pi_*(\nu_{\infty*}
\ubM)\to L\eta_{\mu}R\varprojlim_{\N}R\nu_*
\ubM$ and the corresponding one for $\CF_{\varphi}$ are compatible 
with the Frobenius pullbacks.
\end{proof}

\begin{proposition}\label{prop:PrismAinfCohLocMapProdComp}
Let $\CF_{\nu}$ $(\nu\in \{1,2\})$ be objects of
 $\Crystal_{\prism}^{\fproj}(\fX/A_{\inf})$, and put
 $\CF_3=\CF_1\otimes_{\CO_{\fX/A_{\inf}}}\CF_2
 \in\Ob\Crystal_{\prism}^{\fproj}(\fX/A_{\inf})$
 (Remark \ref{rmk:PrismCrystalFrobTensor} (3)).
 Then the following diagram is commutative.
 \begin{equation}\label{eq:PrismAinfCohLocCompFrProd}
\xymatrix@R=15pt@C=70pt{
Ru_{\fX/A_{\inf}*}\CF_1\otimes^L_{A_{\inf}}Ru_{\fX/A_{\inf}*}\CF_2
\ar[r]^(.52){\kappa_{\mti,\CF_1}\otimes^L_{A_{\inf}}\kappa_{\mti,\CF_2}}
\ar[d]
&
A\Omega_{\fX}(\bM_1)\otimes_{A_{\inf}}^LA\Omega_{\fX}(\bM_2)
\ar[d]^{\eqref{eq:BKFAinfCohProd}}\\
Ru_{\fX/A_{\inf}*}\CF_3
\ar[r]^{\kappa_{\mti,\CF_3}}
&
A\Omega_{\fX}(\bM_3).
}
\end{equation}
\end{proposition}

\begin{proof}
We follow the notation introduced in Remarks \ref{rmk:KoszAinfCohFrobProdComp} (2)
and \ref{rmk:qdRcpxKoszulFrobProdComp} (2). Recall that we have morphisms of complexes
\begin{gather*}
v_*(q\Omega^{\bullet}(\uCM_{\nu},\utheta_{\uCM_{\nu}}))
\xrightarrow{\gamma_{\mti,\CF_{\nu}}}
K^{\bullet}_{\Lambda}(v_*\uCM_{\nu})
\xrightarrow{K^{\bullet}_{\Lambda}(\delta_{\mti,\CF_{\nu}})}
K^{\bullet}_{\Lambda}(\iota^*\nu_{\infty*}\ubM_{\nu})
\xrightarrow{\alpha_{\mti,\CF_{\nu}}} \pi_*K^{\bullet}_{\Lambda}(\iota_*\iota^*\nu_{\infty*}\ubM_{\nu}),\\
\nu_{\infty*}\ubM_{\nu}\xrightarrow{\beta_{\mti,\CF_{\nu}}}
K_{\Lambda}^{\bullet}(\iota_*\iota^*\nu_{\infty*}\ubM_{\nu})
\end{gather*}
compatible with products 
as \eqref{eq:qDRKoszProdComp}, \eqref{eq:qHigCpxKoszProdComp}, 
\eqref{eq:KoszulGammaCohProjFunct}, and \eqref{eq:BKFGammaModResolFunct}. 
Therefore the morphisms \eqref{eq:qdRKoszulMorLoc}
\begin{align*}
&(\eta_{\mu}\varprojlim_{\N}\alpha_{\mti,\CF_{\nu}})\circ c_{\mti,\CF_{\nu}}
\colon \varprojlim_{\N} v_*(q\Omega^{\bullet}(\uCM_{\nu},\utheta_{\uCM_{\nu}}))\longrightarrow 
\eta_{\mu}\varprojlim_{\N} \pi_*K^{\bullet}_{\Lambda}(\iota_*\iota^*\nu_{\infty*}\ubM_{\nu}),\\
&L\eta_{\mu}R\varprojlim_{\N}R\pi_*(\beta_{\mti,\CF_{\nu}})\colon 
L\eta_{\mu}R\varprojlim_{\N}R\pi_* (\nu_{\infty*}\ubM_{\nu})
\xrightarrow{\;\;\cong\;\;}
L\eta_{\mu}R\varprojlim_{\N} R\pi_*K^{\bullet}_{\Lambda}(\iota_*\iota^*\nu_{\infty*}\ubM_{\nu})
\end{align*}
are compatible with the products induced by 
\eqref{eq:qHiggsInvSysProd}, \eqref{eq:BKFInvSySKoszProd2},
and \eqref{eq:BKFGammaInvSysTensor}. 
We obtain the claim by combining the two 
compatibility and Theorem \ref{thm:PrismCohQHiggsCpx} (3), and by noting that the morphisms
$L\eta_{\mu}R\varprojlim_{\N}R\pi_*(\nu_{\infty*}
\ubM_{\nu})\to L\eta_{\mu}R\varprojlim_{\N}R\nu_*
\ubM_{\nu}$ are compatible with the products.
\end{proof}

We next
discuss the functoriality of 
\eqref{eq:PrismCohAinfCohCompMap2} with respect to 
$\mti=(\mfi \colon \fX\hookrightarrow \fY, \ut)$.
Let $\mti'=(\mfi'\colon \fX'=\Spf(A')\hookrightarrow \fY'=\Spf(B'),
\ut'=(t'_{i'})_{i'\in\Lambda'})$ be another
small framed embedding over $A_{\inf}$
(Definition \ref{def:AdmFramedSmEmbed} (1)), and let 
$\mtg=(g,h,\psi)\colon \mti'=(\mfi'\colon \fX'\hookrightarrow \fY',\ut')
\to \mti=(\mfi\colon \fX\hookrightarrow \fY,\ut)$ 
be a morphism of small framed embeddings over
$A_{\inf}$ (Definition \ref{def:AdmFramedSmEmbed} (2)).
The functoriality of \eqref{eq:PrismCohAinfCohCompMap2} is stated as follows.

\begin{proposition}\label{prop:PrismAinfCompFunct}
Let $\CF\in \Ob \Crystal^{\fproj}_{\prism}(\fX/A_{\inf})$,
and put $\CF'=g_{\prism}^{-1}\CF$ (Definition \ref{def:PrismaticSiteCrystal} (3)),
which belongs to $\Crystal_{\prism}^{\fproj}(\fX'/A_{\inf})$,
$\bM=\bM_{\BKF,\fX}(\CF)$, and 
$\bM'=\bM_{\BKF,\fX'}(\CF')$ (Definition \ref{BKFFunctor}).
Then the following diagram is commutative. 
Note that we have $g_{\Zar}\circ u_{\fX'/A_{\inf}}=u_{\fX/A_{\inf}}\circ g_{\prism}$
\eqref{eq:PrismToposZarProjFunct}.
\begin{equation}\label{eq:PrismAinfCompFunct}
\xymatrix{
Ru_{\fX/A_{\inf}*}\CF\ar[r]\ar[d]^{\kappa_{\mti,\CF}}_{\eqref{eq:PrismCohAinfCohCompMap2}}&
Ru_{\fX/A_{\inf}*}Rg_{\prism*}\CF'\ar[r]^{\cong}&
Rg_{\Zar*}Ru_{\fX'/A_{\inf}*}\CF'
\ar[d]^{Rg_{\Zar*}(\kappa_{\mti',\CF'})}_{\eqref{eq:PrismCohAinfCohCompMap2}}\\
A\Omega_{\fX}(\bM)\ar[rr]^{\eqref{eq:AOmegaFunct}}&&
Rg_{\Zar*}A\Omega_{\fX'}(\bM').}
\end{equation}
\end{proposition}

By the construction of \eqref{eq:PrismCohAinfCohCompMap2}, 
the proof of Proposition \ref{prop:PrismAinfCompFunct} is
reduced to verifying the functoriality of \eqref{eq:AinfCrysProj}, 
\eqref{eq:AOmegaKoszul}, 
and \eqref{eq:qdRKoszulMorLoc} with respect to $\mtg=(g,h,\psi)$. 

We define $\theta_{B',i'}$, $\gamma_{B',i'}$, $D'$, $\theta_{D',i'}$, $\gamma_{D',i'}$
$(i'\in\Lambda')$, $\fD'$,
$D'_m$, $\fD'_m$ $(m\in \N)$, $\oD'$,  $\ofD'$, and
$v_{D'}\colon \ofD'\to \fX'$  as in the paragraph after
\eqref{eq:PrismCohAinfCohCompMap} by using $\mti'=(\mfi'\colon \fX'\to \fY',\ut')$.
Let $h_{D}\colon D\to D'$ be the morphism of bounded prisms over
the prism $(A_{\inf},\pq A_{\inf})$ induced by $h$, and put 
$\fh_{D}=\Spf(h_D)\colon \fD'\to \fD$. 

We use  the
functoriality of the morphism of topos
$\nu_{\fX,\ut}\colon X_{\proet}^{\sim}\to \Gamma_{\Lambda}\hy\fX_{\Zar}^{\sim}$
discussed in \S\ref{sec:ProetGammaZarShv}.  
We follow the notation introduced in \S\ref{sec:ProetGammaZarShv}.
Put $\bA_{\inf,X',m}=\bA_{\inf,X'}/(p,\pq)^{m+1}$,
and let $\ubA_{\inf,X'}$ denote the inverse system 
$(\bA_{\inf,X',m})_{m\in \N}$. Then we have the following diagrams of ringed topos.
\begin{gather}
\xymatrix@R=15pt{
((X_{\proet}^{\prime\sim})^{\N^{\circ}},\ubA_{\inf,X'})
\ar[r]^(.54){\nu'}\ar[d]^{\bmg}&
((\fX_{\Zar}^{\prime\sim})^{\N^{\circ}},\uA_{\inf})\ar[d]^{g}&
\ar[l]_{v'}
((\fD_{\Zar}^{\prime\sim})^{\N^{\circ}},\uA_{\inf})\ar[d]^{h_{\fD}}
\\
((X_{\proet}^{\sim})^{\N^{\circ}},\ubA_{\inf,X})
\ar[r]^(.54){\nu}&
((\fX_{\Zar}^{\sim})^{\N^{\circ}},\uA_{\inf})&
\ar[l]_{v}
((\fD_{\Zar}^{\sim})^{\N^{\circ}},\uA_{\inf})
}
\label{eq:ProetZarMorphFunct1}\\
\xymatrix@R=15pt{
((X_{\proet}^{\prime\sim})^{\N^{\circ}},\ubA_{\inf,X'})
\ar[r]^{\nu'_{\infty}}\ar[d]^{\bmg}&
((\Gamma_{\Lambda'}\hy\fX_{\Zar}^{\prime\sim})^{\N^{\circ}},\uA_{\inf})
\ar@<.5ex>[r]^(.55){\pi'}\ar[d]^{g_{\psi}}&
\ar@<.5ex>[l]^(.45){\iota'}
((\fX_{\Zar}^{\prime\sim})^{\N^{\circ}},\uA_{\inf})\ar[d]^{g}
\\
((X_{\proet}^{\sim})^{\N^{\circ}},\ubA_{\inf,X})
\ar[r]^{\nu_{\infty}}&
((\Gamma_{\Lambda}\hy\fX_{\Zar}^{\sim})^{\N^{\circ}},\uA_{\inf})
\ar@<.5ex>[r]^(.55){\pi}&
\ar@<.5ex>[l]^(.45){\iota}
((\fX_{\Zar}^{\sim})^{\N^{\circ}},\uA_{\inf})}
\label{eq:ProetZarMorphFunct2}
\end{gather}
They are commutative up to canonical isomorphisms 
except the lower left square, for which
we have a morphism \eqref{eq:XproetToGammaShfFunct}
\begin{equation}\label{eq:XproetToGammaShfFunctAinf}
\Xi_{\mathtt g}\colon \nu_{\infty*}\circ\bmg_*
\longrightarrow g_{\psi*}\circ\nu'_{\infty*},\end{equation}
and the composition 
$\uA_{\inf}\to\nu_{\infty*}\ubA_{\inf,X}
\to\nu_{\infty*}\bmg_*\ubA_{\inf,X'}
\xrightarrow{\Xi_{\mathtt g}(\ubA_{\inf,X'})}
g_{\psi*}\nu'_{\infty*}\ubA_{\inf,X'}$
coincides with $\uA_{\inf}
\to g_{\psi*}\uA_{\inf}\to g_{\psi*}\nu'_{\infty*}\ubA_{\inf,X'}$.

We start with the functoriality of \eqref{eq:AinfCrysProj}.
We define $(M_m',\utheta_{M_m'})$,
$(\CM'_m,\utheta_{\CM_m'})$,
$q\Omegab(\CM_m',\utheta_{\CM'_m})$, $(\uCM',\utheta_{\uCM'})$, and 
$q\Omegab(\uCM',\utheta_{\uCM'})$ in the same way as before
\eqref{eq:AinfCrysProj} 
by using $\mti'=(\mfi'\colon \fX'\to \fY',\ut')$ and $\CF'_m=\CF'/(p,\pq)^{m+1}\CF'$. 
Then we have a morphism 
in $C^+((\fX^{\sim}_{\Zar})^{\N^{\circ}},\uA_{\inf})$
\begin{equation}\label{eq:qdRInvSysFunct}
\sigma_{\mtg,\CF}\colon v_*(q\Omegab(\uCM,\utheta_{\uCM}))
\to v_*h_{\fD*}(q\Omegab(\uCM',\utheta_{\uCM'}))
\cong g_*v'_*(q\Omegab(\uCM',\utheta_{\uCM'}))
\end{equation}
induced by \eqref{eq:CrysqHiggsCpxFunct}, and
obtain the following commutative
diagram just by applying Theorem \ref{thm:PrismCohQHiggsCpx} (4)
to $\mtg=(g,h,\psi)$ and $\CF$.
\begin{equation}\label{eq:qdRPrismZarProjCompFunctAinf}
\xymatrix{
Ru_{\fX/A_{\inf}*}\CF
\ar[r]
\ar[d]^{\cong}_{a_{\mti,\CF}\;\;\eqref{eq:AinfCrysProj}}
&Ru_{\fX/A_{\inf}*}Rg_{\prism*}\CF'
\ar[r]^{\cong}
& 
Rg_{\Zar*}Ru_{\fX'/A_{\inf}*}\CF'
\ar[d]_{\cong}^{Rg_{\Zar*}(a_{\mti',\CF'})\;\;\eqref{eq:AinfCrysProj}}
\\
\varprojlim_{\N}v_*(q\Omegab(\uCM,\utheta_{\uCM}))
\ar[rr]^{\varprojlim_{\N}\sigma_{\mtg,\CF}\;\;\eqref{eq:qdRInvSysFunct}}&&
Rg_{\Zar*}\varprojlim_{\N}v'_*(q\Omegab(\uCM',\utheta_{\uCM'})).
}
\end{equation}

\begin{remark}\label{rmk:qdRInvSysFunctCocyc}
The morphisms $\sigma_{\mtg,\CF}$'s \eqref{eq:qdRInvSysFunct} satisfy
the cocycle condition for  composition of $\mtg$'s by the
remark after the construction of \eqref{eq:CrysqHiggsCpxFunct}.
\end{remark}

\begin{remark}\label{rmk:qdRInvSysFunctFrobComp}
Put $\CF_{\varphi}=\varphi^*\CF\in \Ob \Crystal^{\fproj}_{\prism}(\fX/A_{\inf})$
(Remark \ref{rmk:PrismCrystalFrobTensor} (2)). Then, by 
\cite[Remark 14.16 (1)]{TsujiPrismQHiggs}, the
morphisms $\sigma_{\mtg,\CF}$ and $\sigma_{\mtg,\CF_{\varphi}}$
\eqref{eq:qdRInvSysFunct} are compatible with the Frobenius pullbacks 
\eqref{eq:qHiggsInvSysFrobPB} for
$\CF$ and $\CF'$. Note that we have $g_{\prism}^{-1}\CF_{\varphi}
\cong\varphi^*\CF'$ (Remark \ref{rmk:PrismCrystalFrobTensor} (2)).
\end{remark}

Let us study the functoriality of \eqref{eq:AOmegaKoszul}
with respect to $\mtg=(g,h,\psi)$. 
We define $\bM'_m$ $(m\in \N)$ to be $\bM'/(p,\pq)^{m+1}\bM'$
similarly to $\bM_m$, and write $\ubM'$ for the $\ubA_{\inf,X'}$-module
$(\bM'_m)_{m\in \N}$ similarly to $\ubM$.
The morphism 
$\varepsilon_{g,\CF}\colon \bM\to \bmg_{\proet*}\bM'$ in 
Proposition \ref{prop:BKFFunctComp} (1)  induces 
\begin{equation}\label{eq:BKFIndsysFunctMap}
\uvarepsilon_{g,\CF}
\colon \ubM\longrightarrow \bmg_*\ubM'
\quad\text{in}\;\;\Mod((X_{\proet}^{\sim})^{\N^{\circ}}, \ubA_{\inf,X})
\end{equation} and then 
\begin{equation}\label{eq:BKFGammaRepFunct}
\tau_{\mtg,\CF}\colon 
\nu_{\infty*}\ubM
\xrightarrow{\;\nu_{\infty*}(\uvarepsilon_{g,\CF})\;}
\nu_{\infty*}\bmg_*\ubM'
\xrightarrow[\;\eqref{eq:XproetToGammaShfFunctAinf}\;]
{\Xi_{\mtg}(\ubM')}
g_{\psi*}\nu'_{\infty*}\ubM'
\quad \text{in}\;\;
\Mod((\Gamma_{\Lambda}\hy\fX_{\Zar}^{\sim})^{\N^{\circ}},
\uA_{\inf}).
\end{equation}
By applying the construction of $\of$ and $\oof$ from $f$ in 
Proposition \ref{prop:GSheafKoszulFunct} (1)  and (2) to 
$\tau_{\mtg,\CF}$, $g_{\Zar}\colon \fX_{\Zar}^{\prime\sim}\to \fX_{\Zar}^{\sim}$,
and $\psi\colon \Lambda\to\Lambda'$, we obtain 
\begin{align}\label{eq:BKFGammaProjFunct}
&\otau_{\mtg,\CF}\colon 
\iota^*
\nu_{\infty*}\ubM
\xrightarrow{\;\iota^*(\tau_{\mtg,\CF})\;}
\iota^*g_{\psi*}
\nu'_{\infty*}\ubM'
\longrightarrow g_*\iota^{\prime*}\nu'_{\infty*}\ubM'\\
&\ootau_{\mtg,\CF}\colon \iota_*\iota^*\nu_{\infty*}\ubM
\xrightarrow{\;\iota_*(\otau_{\mtg,\CF})\;}
\iota_*g_*\iota^{\prime*}\nu_{\infty*}'\ubM'
\cong g_{\psi*}\iota^{\prime}_*\iota^{\prime*}\nu'_{\infty*}\ubM'
\label{eq:BKFGammaProjFunct2}
\end{align}
equivariant with respect to 
$\Gamma_{\psi}^{\disc}\colon \Gamma_{\Lambda'}^{\disc}
\to \Gamma_{\Lambda}^{\disc}$. 
By applying \eqref{eq:KoszulFunctMorph} to the morphisms 
$\otau_{\mtg,\CF}$ and $\ootau_{\mtg,\CF}$, we obtain morphisms
\begin{align}
&K_{\psi}^{\bullet}(\otau_{\mtg,\CF})\colon K_{\Lambda}^{\bullet}(\iota^*\nu_{\infty*}\ubM)
\to g_*K^{\bullet}_{\Lambda'}(\iota^{\prime*}\nu_{\infty*}'\ubM')
\quad\text{in}\;\;C^+((\fX_{\Zar}^{\sim})^{\N^{\circ}},\uA_{\inf}),
\label{eq:BKFKoszulFunctMap}\\
&K_{\psi}^{\bullet}(\ootau_{\mtg,\CF})
\colon K_{\Lambda}^{\bullet}(\iota_*\iota^*\nu_{\infty*}\ubM)
\to g_{\psi*}K^{\bullet}_{\Lambda'}(\iota'_*\iota^{\prime*}\nu_{\infty*}'\ubM')
\quad\text{in}\;\;C^+((\Gamma_{\Lambda}\hy\fX_{\Zar}^{\sim})^{\N^{\circ}},\uA_{\inf}),
\label{eq:BKFKoszulFunctMap2}
\end{align}
and see that the following diagrams are commutative by 
Proposition \ref{prop:GSheafKoszulFunct} (3) and Lemma \ref{lem:GammaKoszCompos}.
\begin{equation}\label{eq:BKFGammaCohResolFunct}
\xymatrix@C=60pt{
\nu_{\infty*}\ubM\ar[r]^(.4){\beta_{\mti,\CF}}_(.4){\eqref{eq:BKFGammaRepResol}}
\ar[d]^{\tau_{\mtg,\CF}}
&
K_{\Lambda}^{\bullet}(\iota_*\iota^*\nu_{\infty*}\ubM)
\ar[d]^{K^{\bullet}_{\psi}(\ootau_{\mtg,\CF})}\\
g_{\psi*}\nu_{\infty*}'\ubM'\ar[r]^(.4){g_{\psi*}(\beta_{\mti',\CF'})}_(.4){\eqref{eq:BKFGammaRepResol}}&
g_{\psi*}K_{\Lambda'}^{\bullet}(\iota'_*\iota^{\prime*}\nu'_{\infty*}\ubM')
}
\end{equation}
\begin{equation}\label{eq:BKFGammaCohFunct}
\xymatrix@C=40pt{
K_{\Lambda}^{\bullet}(\iota^*\nu_{\infty*}\ubM)
\ar[rr]^{\alpha_{\mti,\CF}}_{\eqref{eq:GammaGalCohIsomMap}}
\ar[d]_{K^{\bullet}_{\psi}(\otau_{\mtg,\CF})}
&&
\pi_*K_{\Lambda}^{\bullet}(\iota_*\iota^*\nu_{\infty*}\ubM)
\ar[d]^{\pi_*K^{\bullet}_{\psi}(\ootau_{\mtg,\CF})}
\\
g_*K_{\Lambda'}^{\bullet}(\iota^{\prime*}\nu'_{\infty*}\ubM')
\ar[r]^(.47){g_*\alpha_{\mti',\CF'}}_(.47){\eqref{eq:GammaGalCohIsomMap}}
&
g_{*}\pi'_*K_{\Lambda'}^{\bullet}(\iota'_*\iota^{\prime*}\nu'_{\infty*}\ubM')
\ar[r]^{\cong}
&
\pi_*g_{\psi*}K_{\Lambda'}^{\bullet}(\iota'_*\iota^{\prime*}\nu'_{\infty*}\ubM')
}
\end{equation}

We obtain the following commutative diagram by using \eqref{eq:ProetGammaSheafProjFunct}
for the lower rectangle and by combining \eqref{eq:BKFGammaCohResolFunct} and 
\eqref{eq:BKFGammaCohFunct} as \eqref{eq:GammaSheafCohFunct}
for the upper rectangle.
\begin{equation}\label{eq:GammaKoszulBKFZarProjFunct}
\xymatrix@C=50pt{
K^{\bullet}_{\Lambda}(\iota^*\nu_{\infty*}\ubM)
\ar[r]^(.47){K_{\psi}^{\bullet}(\otau_{\mtg,\CF})}
\ar[d]^{\cong}_{\eqref{eq:BKFGammaCohKoszulResol}^{-1}\eqref{eq:GammaGalCohIsomMap}}
&
g_*K_{\Lambda'}^{\bullet}(\iota^{\prime*}\nu'_{\infty*}\ubM')
\ar[r]
&
Rg_*K_{\Lambda'}^{\bullet}(\iota^{\prime*}\nu'_{\infty*}\ubM')
\ar[d]_{\cong}^{Rg_*\eqref{eq:BKFGammaCohKoszulResol}^{-1}\eqref{eq:GammaGalCohIsomMap}}
\\
R\pi_*(\nu_{\infty*}\ubM)
\ar[r]^(.43){R\pi_*(\tau_{\mtg,\CF})}
\ar[d]_{\eqref{eq:BKFProetGammaCoh}}
&
R\pi_*Rg_{\psi*}(\nu'_{\infty*}\ubM')
\ar@{-}[r]^{\cong}
&
Rg_*R\pi'_*(\nu'_{\infty*}\ubM')
\ar[d]^{Rg_*\eqref{eq:BKFProetGammaCoh}}
\\
R\nu_*\ubM
\ar[r]^{R\nu_*(\uvarepsilon_{g,\CF})}
&
R\nu_*R\bmg_*\ubM'
\ar@{-}[r]^{\cong}
&
Rg_*R\nu'_*\ubM'
}
\end{equation}

By taking $L\eta_{\mu}R\varprojlim_{\N}$ and composing it with
$L\eta_{\mu}Rg_{\Zar*}\to Rg_{\Zar*}L\eta_{\mu}$
\eqref{eq:DeclageDirectImage}, 
we obtain a commutative diagram 
\begin{equation}\label{eq:AOmegaKoszulFunct}
\xymatrix{
L\eta_{\mu}\varprojlim\limits_{\N}K_{\Lambda}^{\bullet}
(\iota^*\nu_{\infty*}\ubM)
\ar[r]^{L\eta_{\mu}\varprojlim\limits_{\N}K_{\psi}^{\bullet}(\otau_{\mtg,\CF})}
\ar[d]_(.6){b_{\mti,\CF}\;\;\eqref{eq:AOmegaKoszul}}
&
L\eta_{\mu} Rg_{\Zar*}
\varprojlim\limits_{\N}K_{\Lambda'}^{\bullet}(\iota^{\prime*}
\nu'_{\infty*}\ubM')
\ar[r]
&
Rg_{\Zar*}L\eta_{\mu} 
\varprojlim\limits_{\N}K_{\Lambda'}^{\bullet}(\iota^{\prime*}
\nu'_{\infty*}\ubM')
\ar[d]_(.6){Rg_{\Zar*}(b_{\mti',\CF'})\;\;\eqref{eq:AOmegaKoszul}}
\\
A\Omega_{\fX}(\bM)
\ar[rr]^{\eqref{eq:AOmegaFunct}}
&&
Rg_{\Zar*}(A\Omega_{\fX'}(\bM')).
}
\end{equation}

\begin{remark}\label{rmk:BKFGammaKoszulFuncCocyc}
By Remark \ref{rmk:GammaCohFunctCocyc}, 
Proposition \ref{prop:BKFFunctComp} (3), and Proposition 
\ref{prop:KoszulFunctCocycCond} (1), the morphisms
$\varepsilon_{g,\CF}$, $\uvarepsilon_{g,\CF}$, 
$\tau_{\mtg,\CF}$, $\otau_{\mtg,\CF}$, and $\ootau_{\mtg,\CF}$
satisfy the cocycle condition for composition of $\mtg$'s. 
Lemma \ref{lem:GammaKoszCompos} shows
that $K^{\bullet}_{\psi}(\otau_{\mtg,\CF})$ and $K^{\bullet}_{\psi}(\ootau_{\mtg,\CF})$
also satisfy the cocycle condition for composition of $\mtg$'s. 
\end{remark}

\begin{remark}\label{rmk:BKFGammaKoszulFunctFrobPB}
Put $\CF_{\varphi}=\varphi^*\CF\in \Ob \Crystal^{\fproj}_{\prism}(\fX/A_{\inf})$
(Remark \ref{rmk:PrismCrystalFrobTensor} (2)). 
Then the morphisms $\tau_{\mtg,\CF}$, $K_{\psi}^{\bullet}(\otau_{\mtg,\CF})$, 
and $K_{\psi}^{\bullet}(\ootau_{\mtg,\CF})$, and those for $\CF_{\varphi}$
are compatible with the Frobenius pullbacks \eqref{eq:BKFGammaShfFrob}, 
\eqref{eq:BKFGammaShfKosFrob}, and \eqref{eq:BKFGammaShfKosFrob2}
for $\CF$ and $\CF'$; 
the claim for $\tau_{\mtg,\CF}$ and $\tau_{\mtg,\CF_{\varphi}}$ follows from
\eqref{eq:BKFFunctFrobFunct} and 
the remark after \eqref{eq:XproetToGammaShfFunctAinf}, and 
it implies the remaining ones by Proposition 
\ref{prop:KoszulFunctCocycCond} (1) and Lemma \ref{lem:GammaKoszCompos}.
\end{remark}

It remains to prove the functoriality of \eqref{eq:qdRKoszulMorLoc}.
Let $\sigma_{\mtg,\CF}^0$ denote the degree $0$-part
$v_*\uCM\to g_*v'_*\uCM'$ of $\sigma_{\mtg,\CF}$
\eqref{eq:qdRInvSysFunct}.
By the construction of \eqref{eq:qHiggsCpxShvFunct} 
and \eqref{eq:CrysqHiggsCpxFunct}, and the 
formula \eqref{eq:TwConnScExtFormula2},
we see that $\sigma^0_{\mtg, \CF}$ is equivariant with respect
to the homomorphism $\Gamma_{\psi}^{\disc}\colon
\Gamma_{\Lambda'}^{\disc}\to \Gamma_{\Lambda}^{\disc};
\gamma'\mapsto \gamma'\circ\psi$. By comparing
the definition of the morphism \eqref{eq:TwDrCpxScExt} used in
the construction of $\sigma_{\mtg,\CF}$ \eqref{eq:qdRInvSysFunct}
and the definition of $K_{\psi}^{\bullet}(\CM)$ given before
Lemma \ref{lem:KoszulFunct}, we obtain a commutative diagram
\begin{equation}\label{eq:KoszulqdRFunctComp}
\xymatrix@C=60pt{
v_*(q\Omegab(\uCM, \utheta_{\uCM}))
\ar[r]^(.47){\sigma_{\mtg,\CF}\;\;\eqref{eq:qdRInvSysFunct}}
\ar[d]^{\gamma_{\mti,\CF}}_{\eqref{eq:qHiggscpxInvSysKoszul}}
&
g_*v'_*(q\Omegab(\uCM', \utheta_{\uCM'}))\ar[d]^{g_*\gamma_{\mti',\CF'}}_{\eqref{eq:qHiggscpxInvSysKoszul}}\\
K^{\bullet}_{\Lambda}(v_*\uCM)\ar[r]^{K_{\psi}^{\bullet}(\sigma_{\mtg,\CF}^{0})
\;\;\eqref{eq:KoszulFunctMorph}}
&
g_*K^{\bullet}_{\Lambda'}(v'_*\uCM').
}
\end{equation}
Note that we have $\gamma_{\psi,\bmI}^{<}(t_{i_{\nu}})=t_{i_{\nu}}$
for $r\in \N$, $\bmI=(i_1,\ldots, i_r)\in \Lambda^r$, 
and $\nu\in \N\cap [1,r]$ when the components of 
$\psi(\bmI)$ are mutually different.

\begin{lemma}\label{lem:qdRBKFMorphFunct}
The following diagram is commutative.
\begin{equation}\label{eq:qdRBKFMorphFunct}
\xymatrix@C=80pt{
v_*\uCM
\ar[r]^{\sigma_{\mtg,\CF}^0}_{\eqref{eq:qdRInvSysFunct}}
\ar[d]_{\delta_{\mti,\CF}}^{\eqref{eq:qHigBKFmap}}
&
g_*v'_*\uCM'
\ar[d]^{g_*(\delta_{\mti',\CF'})}_{\eqref{eq:qHigBKFmap}}
\\
\iota^*\nu_{\infty*}\ubM
\ar[r]_(.45){\eqref{eq:BKFGammaProjFunct}}^(.45){\otau_{\mtg,\CF}}
&
g_*\iota^{\prime*}\nu'_{\infty*}\ubM^{\prime}
}
\end{equation}
\end{lemma}
We give a proof of Lemma \ref{lem:qdRBKFMorphFunct}
after finishing the proof of Proposition \ref{prop:PrismAinfCompFunct}.
By taking the Koszul complexes of \eqref{eq:qdRBKFMorphFunct} 
with respect to the actions of 
$\Gamma_{\Lambda}^{\disc}$ and $\Gamma_{\Lambda'}^{\disc}$ and
using Lemma \ref{lem:GammaKoszCompos}, we obtain 
a commutative diagram
\begin{equation}\label{eq:qdRBKFMorphKoszulFunct}
\xymatrix@C=80pt{
K_{\Lambda}^{\bullet}(v_*\uCM)
\ar[r]^{K_{\psi}^{\bullet}(\sigma_{\mtg,\CF}^0)}_{\eqref{eq:qdRInvSysFunct}}
\ar[d]_{K_{\Lambda}^{\bullet}(\delta_{\mti,\CF})}^{\eqref{eq:qHigBKFKoszulmap}}
&
g_*K_{\Lambda'}^{\bullet}(v'_*\uCM')
\ar[d]^{g_*(K_{\Lambda'}^{\bullet}(\delta_{\mti',\CF'}))}_{\eqref{eq:qHigBKFKoszulmap}}
\\
K_{\Lambda}^{\bullet}(\iota^*\nu_{\infty*}\ubM)
\ar[r]_(.45){\eqref{eq:BKFKoszulFunctMap}}^(.45){K_{\psi}^{\bullet}(\otau_{\mtg,\CF})}
&
g_*K_{\Lambda'}^{\bullet}(\iota^{\prime*}\nu'_{\infty*}\ubM^{\prime}).
}
\end{equation}

By composing \eqref{eq:KoszulqdRFunctComp} and \eqref{eq:qdRBKFMorphKoszulFunct},
and taking $\varprojlim_{\N}$, we obtain a commutative diagram
\begin{equation}\label{eq:qdRKoszulGammaCohFunct}
\xymatrix{
\varprojlim_{\N}v_*(q\Omegab(\uCM,\utheta_{\uCM}))
\ar[d]_{c_{\mti,\CF}\;\;\eqref{eq:qdRKoszulMorLoc}}
\ar[rr]^{\varprojlim_{\N} \sigma_{\mtg,\CF}\;\;\eqref{eq:qdRInvSysFunct}}
&&
g_{\Zar*}\varprojlim_{\N}v'_*(q\Omegab(\uCM',\utheta_{\uCM'}))
\ar[d]^{g_{\Zar*}(c_{\mti',\CF'})\;\;\eqref{eq:qdRKoszulMorLoc}}
\\
\eta_\mu \varprojlim_{\N}K^{\bullet}_{\Lambda}(\iota^*\nu_{\infty*}\ubM)
\ar[r]^{\eta_{\mu}\varprojliml_{\N}K^{\bullet}_{\psi}(\otau_{\mtg,\CF})}
&
\eta_{\mu}g_{\Zar*}\varprojlim_{\N}K_{\Lambda'}(\iota^{\prime*}\nu_{\infty*}'\ubM')
\ar[r]
&
g_{\Zar*}\eta_{\mu}\varprojlim_{\N}K_{\Lambda'}(\iota^{\prime*}\nu_{\infty*}'\ubM').
}
\end{equation}

By combining \eqref{eq:qdRPrismZarProjCompFunctAinf}, 
\eqref{eq:qdRKoszulGammaCohFunct}, and \eqref{eq:AOmegaKoszulFunct}, 
we obtain the commutative diagram \eqref{eq:PrismAinfCompFunct} 
in Proposition \ref{prop:PrismAinfCompFunct}.

\begin{proof} [Proof of Lemma \ref{lem:qdRBKFMorphFunct}]
To simplify the notation,  we write $\Gamma$, $\Gamma'$
and $\Psi$ for $\Gamma_{\Lambda}$, $\Gamma_{\Lambda'}$, and 
$\Gamma_{\psi}\colon \Gamma_{\Lambda'}\to \Gamma_{\Lambda}$, and put
$\Gamma_n=\Gamma/p^n\Gamma$ and $\Gamma'_n=\Gamma'/p^n\Gamma'$.
Let $\Psi_n$ denote the homomorphism $\Gamma'_n\to \Gamma_n$ induced by 
$\Psi$.  We write $\tnu_{\infty}$ and $\rho$ for the morphisms
of ringed topos $((X_{\proet}^{\sim})^{\N^{\circ}},\ubA_{\inf,X})
\to (((\fX_{\Zar})^{\sim}_{\Gamma})^{\N^{\circ}},\uA_{\inf})$
and $((\Gamma\hy\fX_{\Zar}^{\sim})^{\N^{\circ}},\uA_{\inf})
\to(((\fX_{\Zar})^{\sim}_{\Gamma})^{\N^{\circ}},\uA_{\inf})
$
defined by $\tnu_{\fX,\ut}$ and $\rho_{\Gamma_{\Lambda},\fX_{\Zar}}$
(see before \eqref{eq:ProetGammaShfProj}),
and write them with a prime for those
associated to $\mti'$. Similarly to $g_{\psi}$, let $\tg_{\psi}$ denote 
the morphism of ringed topos 
$(((\fX'_{\Zar})^{\sim}_{\Gamma'})^{\N^{\circ}},\uA_{\inf})
\to (((\fX_{\Zar})^{\sim}_{\Gamma})^{\N^{\circ}},\uA_{\inf})
$ defined by $(\widetilde{\Gamma_{\psi}})_{g_{\Zar}}$
\eqref{eq:ProetToZarFacorization}.

(Step 1) As in the construction of $\nu_{\fX,\ut}$ before \eqref{eq:ProetGammaShfProj},
for $(\fU',S')\in \Ob (\fX'_{\Zar})_{\Gamma'}$, 
let $U'$ denote the adic generic fiber of $\fU'$, and  put 
$U'_{S'}=X'_{S'}\times_{X'}U'$.  We define $U_n'$ $(n\in\N)$ to be $U'_{\Gamma'_n}$
similarly as before \eqref{eq:ProetGammaShvProjDescrip}.
For $\fU\in \Ob \fX_{\Zar}$ and $\fU'=\fU\times_{\fX}\fX'\in \Ob \fX'_{\Zar}$,
let $\bmg_{\fU,n}$ be the morphism 
$U'_{\Psi_{\trmf}^*\Gamma_n}=X'_{\Psi_{\trmf}^*\Gamma_n}\times_{X'}U'
\to X_{\Gamma_n}\times_XU'=U_n\times_XX'$
induced by $\bmg_{\Gamma_n}\colon X'_{\Psi_{\text{\rm f}}^*\Gamma_n}
\to X_{\Gamma_n}$
(see the construction of \eqref{eq:XproetToGammaShfFunct}),
let $a_{\fU,n}$ be the morphism $U'_n=U'_{\Gamma_n'}\to U'_{\Psi_{\trmf}^*\Gamma_n}$
corresponding to the morphism $\Psi_n\colon \Gamma_n'\to \Psi_{\trmf}^*\Gamma_n$
in $\Gamma'\fSet$, and let $b_{\fU,n}$ be the composition $\bmg_{\fU,n}\circ a_{\fU,n}$.
 We have
$\iota^*\nu_{\infty*}\ubM(\fU)\cong\varinjlim_n\ubM(U_n)$
and $g_*\iota^{\prime*}\nu'_{\infty*}\ubM'(\fU)
=\varinjlim_n\ubM'(U'_n)$ by 
\eqref{eq:ProetGammaShvProjDescrip}. 
We assert that, under this description, the morphism
$\otau_{g,\CF}(\fU)\colon\iota^*\nu_{\infty*}\ubM(\fU)
\to  g_*\iota^{\prime*}\nu'_{\infty*}\ubM'(\fU)$
is given by the composition of $\uvarepsilon_{g,\CF}(U_n)
\colon\ubM(U_n)\to 
\bmg_*\ubM'(U_n)=\ubM'(U_n\times_XX')$
with the morphism $\ubM'(b_{\fU,n})\colon 
\ubM'(U_n\times_XX')\to \ubM'(U'_n)$.
By definition, the morphism 
$\tXi_{g,\psi*}(\fU,\Gamma_n)\colon
\tnu_{\infty*}\bmg_*\ubM'(\fU,\Gamma_n)
=\ubM'(U_n\times_XX')
\to \tg_{\psi*}\tnu_{\infty*}'\ubM'(\fU,\Gamma_n)
=\ubM'(U'_{\Psi_{\trmf}^*\Gamma_n})$
\eqref{eq:XproetToGammafSetShfFunct} is 
$\ubM'(\bmg_{\fU,n})$. By Lemmas
\ref{lem:GfSetShfToGShfFunct} and 
\ref{lem:ForgetGammaBC}, the composition
\begin{equation*}
\xymatrix@R=10pt{
\iota^*\rho^*\tg_{\psi*}\tnu_{\infty*}'\ubM'(\fU)
\ar[r]^{\cong}
&
\iota^*g_{\psi*}\rho^{\prime*}\tnu_{\infty*}'\ubM'(\fU)
\ar[r]
&
g_*\iota^{\prime*}\rho^{\prime*}\tnu'_{\infty*}\ubM'(\fU)
\\
\varinjlim_n\ubM'(U'_{\Psi_{\trmf}^*\Gamma_n})
\ar@{=}[u]
&
\Map_{\Gamma',\cont}(\Gamma,
\varinjlim_n\ubM'(U'_n))
\ar@{=}[u]
&
\varinjlim_n \ubM'(U'_n)
\ar@{=}[u]
}
\end{equation*}
is given by $\varinjlim_n\ubM'(a_{\fU,n})$. 
This implies the claim because $\otau_{g,\CF}$ is the
composition of \linebreak $\iota^*\nu_{\infty*}(\uvarepsilon_{g,\CF})
\colon \iota^*\nu_{\infty*}\ubM
\to \iota^*\nu_{\infty*}\bmg_*\ubM'$ with the image of
$\ubM'$ under the composition
\begin{multline*}
\iota^*\nu_{\infty*}\bmg_*
=\iota^*\rho^*\tnu_{\infty*}\bmg_*
\xrightarrow{\iota^*\rho^*(\tXi_{g,\psi*})}
\iota^*\rho^*\tg_{\psi*}\tnu'_{\infty*}
\xrightarrow{\,\cong\,} \iota^*g_{\psi*}\rho^{\prime*}\tnu_{\infty*}'
=\iota^*g_{\psi*}\nu'_{\infty*}
\to g_*\iota^{\prime*}\nu'_{\infty*}.
\end{multline*}

(Step 2) We keep the notation above, and assume that 
$\fU$ is affine. Put $U'_{\infty}=$``\,$\varprojlim_n$"$U'_n\in \Ob X'_{\proet}$, and 
let $b_{\fU,\infty}$ be the morphism ``\,$\varprojlim_n$" $b_{\fU,n}\colon 
U_{\infty}'\to U_{\infty}\times_XX'$.
Then the composition 
$c_{\fU,\infty}\colon U'_{\infty}\xrightarrow{b_{\fU,\infty}} U_{\infty}\times_XX'\to U_{\infty}$
belongs to the category $\CB_g$ introduced before
Proposition \ref{prop:BKFFunctComp} and defines a morphism 
$\fg_{\fU,\infty}\colon 
(\bA_{\inf,X'}(U'_{\infty}),g\circ v_{\fX',U_{\infty}'})
\to (\bA_{\inf,X}(U_{\infty}),v_{\fX,U_{\infty}})$ in $(\fX/A_{\inf})_{\prism}$.
Let $\fD_{\fU}$ (resp.~$\fD'_{\fU'}$) be the affine formal subsheme
of $\fD$ (resp.~$\fD'$) whose underlying topological space
is $v_{D}^{-1}(\fU)\subset \ofD$ (resp.~$v_{D'}^{-1}(\fU')\subset \ofD'$). 
We claim that the diagram
\begin{equation}
\xymatrix@C=50pt{
(\bA_{\inf,X'}(U'_{\infty}),g\circ v_{\fX',U_{\infty}'})
\ar[r]^(.53){\fg_{\fU,\infty}}
\ar[d]^{p_{D'_{\fU'},\ut'}}
&
(\bA_{\inf,X}(U_{\infty}),v_{\fX,U_{\infty}})
\ar[d]^{p_{D_{\fU},\ut}}
\\
(\fD'_{\fU'},g\circ v_{D'_{\fU'}})
\ar[r]^{\fh_{D_{\fU}}}
&(\fD_{\fU},v_{D_{\fU}})
}
\end{equation}
in $(\fX/A_{\inf})_{\prism}$ is commutative, where
$v_{D_{\fU}}=v_D\vert_{v_D^{-1}(\fU)}$,
$v_{D'_{\fU'}}=v_{D'}\vert_{v_{D'}^{-1}(\fU')}$,
$\fh_{D_{\fU}}=\fh_{D}\vert_{\fD'_{\fU'}}$,
and we define the vertical morphisms as before
\eqref{eq:CrysBKFMap}.
To prove it, we may replace $\fh_{D_{\fU}}$ by $\fh_D$.
Then, since the remaining morphisms are compatible with those for 
$\fX$ and $\fX'$, the claim is reduced to the case $\fU=\fX$.
In this case, the morphism $c_{\fX,\infty}$ is induced by 
$\varinjlim g_n^*\colon 
\varinjlim_nA_n\to \varinjlim_n A_n'$
by the definition of $\bmg_S$ in the paragraph defining
\eqref{eq:XProetToGammaSiteFunct}.
Let $g^*_{\infty}\colon A_{\infty}\to A'_{\infty}$ be the
$p$-adic completion of $\varinjlim g_n^*$.
By the definition of $g_n^*$ $(n\in\N)$, 
we see that the $A_{\inf}$-algebra homomorphisms
$g_{\infty}^*$, $A_{\inf}(g_{\infty}^*)$,
$g^*\colon A\to A'$, $h^*\colon B\to B'$, and
$A_{\inf}[T_i^{\pm 1} (i\in\Lambda)]
\to A_{\inf}[T_{i'}^{\prime\pm 1} (i'\in \Lambda')];
T_i\mapsto T'_{\psi(i)}$ $(i\in \Lambda)$
define a morphism between the diagrams
\eqref{eq:BToAinfDiag} for $(\mathtt{i}\colon \fX\to \fY, \ut)$
and $(\mathtt{i}'\colon \fX'\to \fY',\ut')$. 
This implies that the unique homomorphisms $B\to A_{\inf}(A_{\infty})$
and $B'\to A_{\inf}(A'_{\infty})$ making these diagrams
commutative are compatible with $h^*$ and $A_{\inf}(g_{\infty}^*)$.
Therefore their unique extensions 
$D\to A_{\inf}(A_{\infty})$ and $D'\to A_{\inf}(A'_{\infty})$
to $\delta$-$A_{\inf}$-algebra homomorphisms are 
compatible with $\fh_{D}^*$ and $A_{\inf}(g_{\infty}^*)$. \par
(Step 3) By the characterization of $\varepsilon_{g,\CF}$
in Proposition \ref{prop:BKFFunctComp} (1),  the composition 
$\ubM(b_{\fU,\infty})\circ \uvarepsilon_{g,\CF}(\fU)\colon
\ubM(U_{\infty})\to \ubM'(U'_{\infty})$
considered in (Step 1) is compatible with the homomorphism 
$\bM(U_{\infty})\to \bM'(U'_{\infty})$ obtained by taking
the section of $\CF$ over the morphism $\fg_{\fU,\infty}$
introduced in (Step 2). Therefore, by (Step 1), (Step 2),
and the construction of \eqref{eq:CrysBKFMap} for 
$\CF$, $(\mfi\colon \fX\to \fY,\ut)$, $\fU\subset \fX$
and $\CF'$, $(\mfi'\colon \fX'\to \fY',\ut')$, $\fU'\subset\fX'$,
we are reduced to  showing that the inverse system of 
homomorphisms $\CM_m(D_{\fU})
\cong\CF_m(D_{\fU},v_{D_{\fU}})
\xrightarrow{\CF_m(\fh_{D_{\fU}})}
\CF'_m(D'_{\fU'},v_{D'_{\fU'}})\cong
\CM'_m(D'_{\fU'})$ $(m\in \N)$
coincides with $\sigma_{\mtg,\CF}^0(\fU)$.
This is true because $\sigma_{\mtg,\CF}^0(\fU)$ is given by the scalar
extension of  $\CF_m(\fh_{D})
\colon M_m=\CF_m(\fD,v_D)
\to \CF'_m(\fD',v_{D'})=M'_m$ to
$\fh_{D_{\fU}}^*\colon D_{\fU}\to D'_{\fU'}$
for each $m\in \N$.
\end{proof}

\begin{proposition}\label{prop:PrismCohAinfCompMapIndep}
The morphism $\kappa_{\mti,\CF}$ 
\eqref{eq:PrismCohAinfCohCompMap2}
is independent of the choice of a small framed 
embedding $\mti=(\mfi\colon \fX\to \fY, \ut)$ over $A_{\inf}$. 
\end{proposition}

\begin{definition}\label{def:PrismCohAinfCohLocCompMap}
We define the morphism $\kappa_{\CF}$ \eqref{eq:PrismCohAinfCohCompMap} to 
be $\kappa_{\mti,\CF}$ independent of the choice of $\mti$.
It is functorial in $\fX$ by Proposition \ref{prop:PrismAinfCompFunct}. 
\end{definition}

\begin{proof}[Proof of Proposition \ref{prop:PrismCohAinfCompMapIndep}]
Let $\mti'=(\mfi'\colon \fX\to \fY'=\Spf(B'), \ut'=(t'_{i'})_{i'\in\Lambda'})$ be
another small framed embedding of $\fX$ over $A_{\inf}$.
We define $\fY''=\Spf(B'')$ to be the fiber product of $\fY$ and $\fY'$ over $\Spf(A_{\inf})$.
Let $t''_i$ $(i\in \Lambda)$ (resp.~$t''_{i'}$ $(i'\in\Lambda')$) be the
pullback of $t_i\in B$  (resp.~$t'_{i'}\in B'$) to $B''$, and let $\Lambda''$ be
the disjoint union of $\Lambda$ and $\Lambda'$ equipped with the
unique total order compatible with those on $\Lambda$ and $\Lambda'$
and satisfying $i\leq i'$ for any $i\in \Lambda$ and $i'\in \Lambda'$.
Then the closed immersion  $\mfi''\colon \fX\to \fY''$ induced by 
$\mfi$ and $\mfi'$ endowed with $\ut''=(t''_j)_{j\in \Lambda''}$ is a
small framed embedding of $\fX$ over $A_{\inf}$
(Definition \ref{def:AdmFramedSmEmbed} (1)). The identity
morphism of $\fX$, the projections $\fY''\to \fY$, $\fY'$, and
the inclusions $\Lambda$, $\Lambda'\to \Lambda''$ define 
morphisms of small framed embeddings over $A_{\inf}$
(Definition \ref{def:AdmFramedSmEmbed} (2))
$(\mfi'',\ut'')\to \mti$, $\mti'$. 
By applying Proposition \ref{prop:PrismAinfCompFunct} to these two
morphisms, we see that 
$\kappa_{\mti,\CF}$ and $\kappa_{\mti',\CF}$ \eqref{eq:PrismCohAinfCohCompMap2}
coincide.
\end{proof}

\section{Comparison with $A_{\inf}$-cohomology with coefficients:
the local case}
\label{sec:compAinfcohLocal}

We prove the following theorem.

\begin{theorem}\label{thm:PrismCohAinfCohLocComp}
Let $\fX=\Spf(A)$ be a $p$-adic smooth affine formal scheme over $\CO$
admitting invertible $p$-adic coordinates (Definition \ref{def:IadicProperties} (2)), 
let $\CF$ be an object of $\Crystal_{\prism}^{\fproj}(\fX/A_{\inf})$
(Definition \ref{def:PrismaticSiteCrystal} (2)), and
put $\bM:=\bM_{\BKF,\fX}(\CF)$ (Definition \ref{BKFFunctor}).
Then the morphism $\kappa_{\CF}\colon Ru_{\fX/A_{\inf}*}\CF\longrightarrow A\Omega_{\fX}(\bM)$ 
(Definition \ref{def:PrismCohAinfCohLocCompMap})
is an isomorphism.
\end{theorem}

By the assumption on $\fX$, there exists a small framed embedding
$\mti=(\mfi\colon \fX\to \fY=\Spf(B),\ut=(t_{i})_{i\in\Lambda})$
of $\fX$ over $A_{\inf}$ such that $\mfi$ is a lifting of $\fX$, i.e., 
$\mfi$ induces an isomorphism $\fX\cong \fY\times_{\Spf(A_{\inf})}\Spf(\CO)$,
and $t_{i}$ $(i\in \Lambda)$ form $(p,\pq)$-adic
coordinates of $B$ over $A_{\inf}$ (Definition \ref{def:IadicProperties} (2)).
We prove that $\kappa_{\mti,\CF}$ \eqref{eq:PrismCohAinfCohCompMap2} is an isomorphism.

It should be possible to prove that the composition $b_{\mti,\CF}\circ c_{\mti,\CF}\colon
\varprojlim_{\N}v_*(q\Omega^{\bullet}(\uCM,\utheta_{\uCM}))
\to A\Omega_{\fX}(\bM)$ \eqref{eq:AOmegaKoszul}, \eqref{eq:qdRKoszulMorLoc} 
is an isomorphism by comparing
its derived section  $R\Gamma(\fU,-)$ on each affine open $\fU$ of $\fX$
with the second claim in \cite[Theorem 6.1]{MT}. However we give a modified
proof which avoids checking the commutativity of $R\Gamma(\fU,-)$ 
and $L\eta_{\mu}$ as \cite[Proposition 6.6]{MT}; we show that $b_{\mti,\CF}$
\eqref{eq:AOmegaKoszul} is an isomorphism (Proposition \ref{prop:KoszulAOmega}) 
by adapting the proof of 
\cite[Theorem 5.14]{MT} to sheaves on $\fX_{\Zar}$; we see that $c_{\mti,\CF}$
\eqref{eq:qdRKoszulMorLoc} is an isomorphism by applying \cite[Proposition 1.29 and Corollary 2.25]{MT}
to the derived section $R\Gamma(\fU,c_{\mti,\CF})$ on each affine open $\fU$ of $\fX$,
but we give a modified proof unifying the proofs of the two claims in \cite{MT}
to make the proof of the theorem self-contained (Proposition \ref{prop:qDolbeaultGalCoh}).

We follow the notation introduce in the construction of $\kappa_{\mti,\CF}$.
We start by proving the following proposition.

\begin{proposition}\label{prop:KoszulProetProjRlim}
The cone of the morphism in $D^+(\fX_{\Zar},A_{\inf})$
\begin{equation}\label{eq:KoszulBKFLocCohMap}
\varprojlim_{\N}K_{\Lambda}^{\bullet}(\iota^*\nu_{\infty*}\ubM)
\longrightarrow R\varprojlim_{\N} R\nu_*\ubM
\end{equation}
induced by \eqref{eq:GammaGalCohIsomMap}, \eqref{eq:BKFGammaCohKoszulResol}, 
and \eqref{eq:BKFProetGammaCoh} is
annihilated by any element of $\Ker(A_{\inf}\to W(k))$.
\end{proposition}

\begin{lemma}\label{lem:DerLimAzero}
Let $\underline{\CK}^{\bullet}=(\CK^{\bullet}_m)_{m\in \N}$ be a complex of $\uA_{\inf}$-modules
on $\fX_{\Zar}$ bounded below,  and assume that $\CH^r(\underline{\CK}^{\bullet})$
$(r\in \Z)$ are almost zero (Definition \ref{def:AlmostIsom}). Then $R\varprojlim_{\N} 
\underline{\CK}^{\bullet}$
is annihilated by any element of $\Ker(A_{\inf}\to W(k))$.
\end{lemma}

\begin{proof}
The same argument as the proof of \cite[Lemma 109]{TsujiSimons}
for $A_{\inf}/p^m$-modules works as follows. Letting $J:=\Ker(A_{\inf}\to W(k))$,
we see $\CH^r(J/p^m\otimes_{A_{\inf}/p^m}\CK_m^{\bullet})
\xleftarrow{\cong} J/p^m\otimes_{A_{\inf}/p^m}\CH^r(\CK_m^{\bullet})=0$
by taking stalks and using $J/p^m=\varinjlim_n[\varphi^{-n}(\varepsilon-1)](A_{\inf}/p^m)$
and the fact that $A_{\inf}/p^m$ is $[\varepsilon-1]$-torsion free.
This implies the claim since the multiplication by an element 
of $J$ on $(\CK^{\bullet}_m)_{m\in \N}$ factors through
$(J/p^m\otimes_{A_{\inf}/p^m}\CK_m^{\bullet})_{m\in \N}$, which 
is acyclic as shown above. 
\end{proof}

For an open affine formal subscheme $\fU=\Spf(A_{\fU})$ of $\fX$, we use the 
notation $U$, $U_n$, $A_{\fU,n}$, $U_{\infty}$, and $A_{U_{\infty}}^+$ 
introduced before \eqref{eq:ProetGammaModIm}. 
We write $X_n$, $X_{\infty}$, and $A_{\infty}$ for $U_n$, $U_{\infty}$
and $A_{U_{\infty}}^+$ for $\fU=\fX$. 

\begin{lemma}\label{lem:AinfOpenSub}
 Let $\fU=\Spf(A_{\fU})$ be an open affine formal subscheme of $\fX$.\par
(1) The morphisms 
$\Spec(A_{\fU}/p)\leftarrow \Spec(A_{U_{\infty}}^+/p)
\to \Spec(\bA_{\inf,X}(U_{\infty})/(p,\pq)^{m+1})$ are homeomorphisms.\par
(2) The morphism 
$\Spec(\bA_{\inf,X}(U_{\infty})/(p,\pq)^{m+1})
\to \Spec(\bA_{\inf,X}(X_{\infty})/(p,\pq)^{m+1})$ is an open immersion. 
\end{lemma}

\begin{proof}
(1) The right (resp.~left) morphism is a homeomorphism
by $\bA_{\inf,X}(U_{\infty})/(p,\pq)\cong A_{U_{\infty}}^+/p$
(resp.~because $A_{U_{\infty}}^+/p=\varinjlim_nA_{\fU,n}/p$
and the morphism $\Spec(\CO/p[T_i^{\pm 1/p^{\infty}}(i\in\Lambda)])
\to \Spec(\CO/p[T_i^{\pm 1}(i\in\Lambda)])$ 
is a universal homeomorphism.
(2) Since the reduction modulo $(p,\pq)$ of the morphism in question
is given by the base change of the open immersion $\Spec(A_{\fU}/p)\to \Spec(A/p)$
by $\Spec(\bA_{\inf,X}(X_{\infty})/(p,\pq))\cong\Spec(A_{\infty}/p)\to \Spec(A/p)$,
it suffices to prove that the morphism is flat by Lemma \ref{lem:NilpImmOpenImm} below.
Since the sequence $\pq$, $p$ is regular on $\bA_{\inf,X}(U_{\infty})$
and on $\bA_{\inf,X}(X_{\infty})$, this is reduced to the flatness of 
$A/p\to A_{\fU}/p$ by Remark \ref{rmk:IadicFlatRegSeq}.
\end{proof}

\begin{lemma}\label{lem:NilpImmOpenImm}
Let $f\colon X\to Y$ be a morphism of schemes, let $i\colon \oY\to Y$
be a nilpotent closed immersion, and let $\of\colon \oX\to \oY$ be the base change
of $f$ by $i$. If $\of$ is an open immersion and $f$ is flat, then $f$ is an open immersion.
\end{lemma}

\begin{proof} Since $\oX\to X$ and $\oY\to Y$ are homeomorphisms,
it suffices to prove that $f_x^*\colon \CO_{Y,y}\to \CO_{X,x}$ is an isomorphism
for every $x\in X$ and $y=f(x)$. It is injective since it is faithfully flat by assumption.
It is surjective because the kernel of the surjective homomorphism $\CO_{Y,y}\to \CO_{\oY,y}$
is nilpotent and the homomorphism 
$\of_x^*\colon \CO_{\oY,y}\to \CO_{\oX,x}\cong
\CO_{X,x}\otimes_{\CO_{Y,y}}\CO_{\oY,y}$ is surjective.
\end{proof}

Let $\fX_{\AffZar}$ be the 
full subcategory of $\fX_{\Zar}$ consisting of open affine formal subschemes of $\fX$
equipped with the induced topology. As every object of $\fX_{\Zar}$ is covered
by objects of $\fX_{\AffZar}$, the inclusion functor $\iota\colon \fX_{\AffZar}
\to \fX_{\Zar}$ is continuous and cocontinuous, the restriction 
functor $\iota^*\colon \fX_{\Zar}^{\sim}\to \fX_{\AffZar}^{\sim}$
is an equivalence (\cite[III Th\'eor\`eme 4.1 and its proof]{SGA4}, 
and the restriction along $\iota$ is compatible
with the sheafifying functors (\cite[III Proposition 2.3 2)]{SGA4}).

\begin{lemma}\label{lem:BKFGammaRepSheaf}
For $m\in \N$, let $\CM^{\Aff}_{\infty,m}$ be the presheaf of $A_{\inf}/(p,\pq)^{m+1}$-modules
on $\fX_{\AffZar}$ defined by $\fU\mapsto \bM(U_{\infty})/(p,\pq)^{m+1}$.\par
(1) The presheaf $\CM^{\Aff}_{\infty,m}$ is a sheaf on $\fX_{\AffZar}$.\par
Let $\CM_{\infty,m}$ be the sheaf on $\fX_{\Zar}$ whose restriction to $\fX_{\AffZar}$ is
$\CM^{\Aff}_{\infty,m}$.\par
(2) $H^r(\fU,\CM_{\infty,m})=0$ $(r>0)$ for any $\fU\in \Ob \fX_{\AffZar}$.\par
(3) $R^r\varprojlim_m \CM_{\infty,m}=0$ $(r>0)$. 
\end{lemma}

\begin{proof}
The claim (3) follows from (1) and (2). (See \cite[Lemma IV.4.2.3]{AGT} for example.)
By Definition \ref{BKFFunctor} and $X_{\infty}, U_{\infty}\in \Ob\fB_{\fX}$, $\bM(X_{\infty})$
is a finite projective $\bA_{\inf,X}(X_{\infty})$-module and the
$\bA_{\inf,X}(U_{\infty})$-linear
homomorphism 
$\bM(X_{\infty})\otimes_{\bA_{\inf,X}(X_{\infty})}\bA_{\inf,X}(U_{\infty})
\to \bM(U_{\infty})$ is an isomorphism. Therefore
Lemma \ref{lem:AinfOpenSub} implies that, via the homeomorphisms
in Lemma \ref{lem:AinfOpenSub} (1) for $\fU=\fX$, $\CM^{\Aff}_{\infty,m}$ may be identified with
the restriction to $\fX_{\AffZar}$ of the quasi-coherent module on $\Spec(\bA_{\inf,X}(X_{\infty})/(p,\pq)^{m+1})$
associated to the $\bA_{\inf,X}(X_{\infty})/(p,\pq)^{m+1}$-module
$\bM(X_{\infty})/(p,\pq)^{m+1}$. 
This shows (1) and (2).
\end{proof}

\begin{lemma}\label{lem:ProetToGammaSheafAVanishing}
The sheaf of $A_{\inf}$-modules $R^r\nu_{\fX,\ut*}\bM_m$
is almost zero (Definition \ref{def:AlmostIsom}) for  $m\geq 0$ and $r>0$.
\end{lemma}

\begin{proof}
Let $\bM_m\to \CI^{\bullet}$ be an injective resolution of $\bM_m$
by sheaves of $\bA_{\inf,X}$-modules. Then we have
$(\nu_{\fX,\ut*}\CI^{\bullet})(\fU)\cong\varinjlim_n\CI^{\bullet}(U_n)$
as \eqref{eq:ProetGammaShvProjDescrip} for $\fU\in \Ob \fX_{\AffZar}$. 
Hence $R^r\nu_{\fX,\ut*}\bM_m$
is the sheaf associated to $\fU\to \varinjlim_n H^r(U_n,\bM_m)$, which is
almost zero by Proposition \ref{prop:BKFProperties} (4).
\end{proof}

\begin{proof}[Proof of Proposition \ref{prop:KoszulProetProjRlim}]
Let $\CM_{\infty,m}$ $(m\in \N)$ be the sheaves on $\fX_{\Zar}$ defined in 
Lemma \ref{lem:BKFGammaRepSheaf}, and let 
$\uCM_{\infty}$ be the sheaf of $\uA_{\inf}$-modules $(\CM_{\infty,m})_{m\in \N}$
on $\fX_{\Zar}$. The right action of 
$\Gamma_{\Lambda}^{\disc}$ on $U_{\infty}$ for each $\fU\in \Ob \fX_{\AffZar}$
is functorial in $\fU$ and induces its left action on $\uCM_{\infty}$.
Via \eqref{eq:ProetGammaModIm}, we obtain a $\Gamma_{\Lambda}^{\disc}$-equivariant
morphism $\uCM_{\infty}\to \iota^*\nu_{\infty*}\ubM$, which is
an almost isomorphism and becomes an isomorphism after taking
$\varprojlim_{\N}$ by Proposition \ref{prop:BKFProperties} (1) and (2). 
By taking $K_{\Lambda}^{\bullet}(-)$, composing with
\eqref{eq:GammaGalCohIsomMap}, \eqref{eq:BKFGammaCohKoszulResol}, and \eqref{eq:BKFProetGammaCoh}, 
and using Lemma \ref{lem:ProetToGammaSheafAVanishing}, 
we obtain 
\begin{equation*}
K^{\bullet}_{\Lambda}(\uCM_{\infty})
\xrightarrow{\approx}
K_{\Lambda}^{\bullet}(\iota^*\nu_{\infty*}\ubM)
\xrightarrow[\alpha_{\mti,\CF}\;\eqref{eq:GammaGalCohIsomMap}]{\cong}
\pi_*K^{\bullet}_{\Lambda}(\iota_*\iota^*\nu_{\infty*}\ubM)
\xleftarrow[\eqref{eq:BKFGammaCohKoszulResol}]{\cong}
R\pi_*(\nu_{\infty*}\ubM)
\xrightarrow[\eqref{eq:BKFProetGammaCoh}]{\approx}
R\nu_{*}\ubM,
\end{equation*}
where $\approx$ means $\CH^r$ are almost isomorphisms
for every $r\in \Z$. We see that the proposition holds
by applying Lemma \ref{lem:DerLimAzero}
to $R\varprojlim_{\N}$ of the cone of the composition 
of the above morphisms, which is isomorphic to the
cone of $R\varprojlim_{\N}$ of the composition, 
and using Lemma \ref{lem:BKFGammaRepSheaf} (3) and
$\varprojlim_{\N}\uCM_{\infty}\cong \varprojlim_{\N}\iota^*\nu_{\infty*}\ubM$
mentioned above. 
\end{proof}

Next we derive the following proposition from Proposition 
\ref{prop:KoszulProetProjRlim}.

\begin{proposition}\label{prop:KoszulAOmega}
The morphism $b_{\mti,\CF}$ \eqref{eq:AOmegaKoszul} in $D(\fX_{\Zar},A_{\inf})$
is an isomorphism.
\begin{equation}\label{eq:KoszulAOmegaMap}
b_{\mti,\CF}\colon \eta_{\mu}\varprojlim_{\N}K^{\bullet}_{\Lambda}
(\iota^*\nu_{\infty*}\ubM)
\xrightarrow{\;\cong\;}L\eta_{\mu}
R\varprojlim_{\N}R\nu_*\ubM
\cong A\Omega_{\fX}(\bM)
\end{equation}
\end{proposition}

We use the following analogue of 
\cite[Lemma 111 (2)]{TsujiSimons} for sheaves.
(See also \cite[Lemma 5.14]{BhattSaltLake}.)

\begin{lemma}\label{lem:LetaSheafQI}
Let $C$ be a site and suppose that the topos
$C^{\sim}$ associated to $C$ has a conservative
family of points. Let $R$ be a commutative ring,
let $a$ be a regular element of $R$, and let $J$ 
be an ideal of $R$ containing $a$. 
Let $\CKb_1\to \CKb_2$ be a morphism of complexes
of $a$-torsion free sheaves of $R$-modules on $C$,
and let $\CKb_3$ be the cone of $f$.
Suppose that (i) $J\cdot \CH^r(\CKb_3)=0$ for all
$r\in \Z$ and (ii) $(\CH^r(\CKb_1/a\CKb_1)(U))[J^2]=0$
for all $r\in \Z$ and $U\in \Ob C$. 
Then the morphism $\eta_a\CKb_1\to \eta_a\CKb_2$
is a quasi-isomorphism.
\end{lemma}

\begin{proof}
Put $\oCKb_i=\CKb_i/a\CKb_i$ $(i\in \{1,2,3\})$. For a point $s$ of $C^{\sim}$,
we have $s^{-1}(\eta_a\CKb_i)=\eta_a(s^{-1}(\CKb_i))$,
$s^{-1}(\oCKb_i)=s^{-1}(\CKb_i)/as^{-1}(\CKb_i)$, 
$s^{-1}(\CKb_3)=$ Cone$(s^{-1}(f))$, and 
$a\cdot H^r(s^{-1}(\CKb_3))=a\cdot s^{-1}(\CH^r(\CKb_3))=0$.
Hence, by applying the argument in the proof of 
\cite[Lemma 111 (2)]{TsujiSimons} to $s^{-1}(f)$ for every point
$s$ of $C^{\sim}$, we are reduced to showing that
the boundary map $\CH^r(\oCKb_3)\to \CH^{r+1}(\oCKb_1)$
is $0$ for all $r\in \Z$. Since $\CKb_3$ is $a$-torsion free,
the assumption (i) implies $J^2\cdot \CH^r(\oCKb_3)=0$. 
Hence the image of the boundary map is trivial by
the assumption (ii). 
\end{proof}

\begin{lemma}\label{lem:FlatAinfAlgProperties}
Let $R$ be an $A_{\inf}$-algebra $(p,\pq)$-adically complete and separated
and $(p,\pq)$-adically flat (Definition \ref{def:IadicProperties} (1)). 
Let $a$ be an element of $A_{\inf}$ whose image in $A_{\inf}/pA_{\inf}=\CO^{\flat}$
is non-zero and non-invertible.\par
(1) $R$ is $p$-torsion free, and $p$-adically complete and separated.\par
(2) $R$ is $a$-torsion free, and $a$-adically complete and separated. \par
(3) $R/a^nR$ is $p$-torsion free, and $p$-adically complete and separated.\par
(4) $R/p^nR$ is $a$-torsion free, and $a$-adically complete and separated.\par
\end{lemma}

\begin{proof}
Note that $(p^m,a^l)$-adic topology on $A_{\inf}$ for positive integers $m$ and $l$
coincides with the $(p,\pq)$-adic topology since the $a$-adic topology 
on $\CO^{\flat}$ is the same as the $\pq$-adic topology.
Since the sequence $p$, $a$ is $A_{\inf}$-regular, the sequence
$p^n$, $a$ is $A_{\inf}$-regular. By the remark above, $A_{\inf}$ and $R$ are
$(p^n,a)$-adically complete and separated. Hence 
\cite[4.4]{TsujiPrismQHiggs}
for $N=0$ implies the first claim of (1), and (4). Then the second claim of
(4) implies the second claim of (1) as 
$\varprojlim_nR/p^n\cong \varprojlim_{n,m}R/(p^n,a^m)\cong R$.
We see that $A_{\inf}=\varprojlim_nA_{\inf}/p^n$ is $a$-torsion free, and 
$A_{\inf}/a$ is $p$-torsion free by \cite[3.2 (1)]{TsujiPrismQHiggs}.
Hence we obtain 
(2) and (3) similarly from 
\cite[4.4]{TsujiPrismQHiggs} for $N=0$ by exchanging the role 
of $p$ and $a$. 
\end{proof}

\begin{remark}\label{rmk:AinfPerfFlatness}
For an affinoid perfectoid object $V$ of $X_{\proet}$,
$[p]_q$ is regular on $A_{\inf}$ and $\bA_{\inf,X}(V)$,
and $\bA_{\inf,X}(V)/\pq\cong \hCO_{X}^+(V)$ is $p$-torsion free,
whence flat over $A_{\inf}/\pq\cong \CO$. Therefore 
$\bA_{\inf,X}(V)$ is $\pq$-adically flat over $A_{\inf}$
by Remark \ref{rmk:IadicFlatRegSeq}, and we can apply 
Lemma \ref{lem:FlatAinfAlgProperties} to $\bA_{\inf,X}(V)$.
\end{remark}

We define $t_{A,i,n}\in A_n$ and $t_{A,i}^{\flat}\in A_{\infty}^{\flat}$
$(i\in\Lambda)$
as before \eqref{eq:BToAinfDiag}, and write their images in $A_{\fU,n}$
and in $(A_{U_{\infty}}^+)^{\flat}$ for $\fU\in \Ob\fX_{\AffZar}$
by the same symbols. 
For $\fU\in \Ob \fX_{\AffZar}$,  we define $\fY_{\fU}=\Spf(B_{\fU})$
to be the open formal subscheme of $\fY$ whose
underlying space is the image of $\fU$ under $\mathtt i\colon \fX\to \fY$.
Since we have assumed that $\fY$ is a lifting of $\fX$, we have 
$D_{\fU}=B_{\fU}$.
Therefore the morphism
\begin{equation}\label{eq:SmLiftToAinfInfty}
B_{\fU}=D_{\fU}\longrightarrow \bA_{\inf,X}(U_{\infty})
\end{equation}
constructed before  \eqref{eq:CrysBKFMap} is a lifting of $A_{\fU}\to A_{U_{\infty}}^+$.
As in Construction \ref{constr:qHiggsModCpxShv} (1),
the $t_i\mu$-derivations
$\theta_{B,i}$ $(i\in\Lambda)$ of $B$ over $A_{\inf}$
extend uniquely to $t_i\mu$-derivations $\theta_{B_{\fU},i}$
over $A_{\inf}$ commuting with each other
by Proposition \ref{prop:TwDerivEtaleExt}. 
Let $\gamma_{B_{\fU},i}$
be the automorphism $1+t_i\mu\theta_{B_{\fU},i}$ of the
$A_{\inf}$-algebra $B_{\fU}$ associated to $\theta_{B_{\fU,}i}$
(Lemma \ref{lem:TwDerivEndomDeltaComp} (1)), 
which coincides with the automorphism induced by 
$\gamma_{B,i}=1+t_i\mu\theta_{B,i}$ via 
$\Spf (B_{\fU})\subset \Spf (B)$. Note that the action of 
$\gamma_{B,i}$ on the underlying space of $\Spf(B)$ is trivial
since $\gamma_{B,i}$ is trivial modulo $\mu$. 
We define the action of $\Gamma_{\Lambda}^{\disc}$ on 
$B_{\fU}$ by letting $\gamma_i\in\Gamma_{\Lambda}^{\disc}$
act by $\gamma_{B_{\fU},i}$. Then the morphism 
\eqref {eq:SmLiftToAinfInfty} is $\Gamma_{\Lambda}^{\disc}$-equivariant.

\begin{lemma}\label{lem:AinfUinftyBasis}
For $\fU\in \Ob\fX_{\AffZar}$ and an ideal $J$ of $A_{\inf}$ containing
$(p,\pq)^{m+1}$ for some $m\in \N$, the $B_{\fU}/J$-module
$\bA_{\inf,X}(U_{\infty})/J$ is  free 
with basis $\prod_{i\in\Lambda}[(t_{A,i}^{\flat})^{r_i}]$,
$(r_i)_{i\in\Lambda}\in (\Z[\frac{1}{p}]\cap [0,p[)^{\Lambda}$.
\end{lemma}

\begin{proof}
Since $\bA_{\inf,X}(U_{\infty})$ and $B_{\fU}$ are
$(p,\pq)$-adically flat over $A_{\inf}$ (Remark \ref{rmk:AinfPerfFlatness}),
the claim is reduced to the case $J=(p,\pq)$. In this case, 
we have $B_{\fU}/J\cong A_{\fU}/p$
and $\bA_{\inf,X}(U_{\infty})/J\cong A_{U_{\infty}}^+/p;
[(t_{A,i}^{\flat})^{p^{-n+1}}]\mapsto t_{A,i,n}$ $(n\in \N)$.
Hence the claim follows from $A_{U_{\infty}}^+/p
\cong \varinjlim_nA_{\fU,n}/p$ and the definition 
of $A_{\fU,n}$ given before \eqref{eq:ProetGammaModIm};
for $n\in \N$, since $t_{A,i}$ $(i\in \Lambda)$ are $p$-adic  coordinates
of $A_{\fU}$ over $\CO$ by the choice of $\mti$, 
the finite free $A_{\fU}$-algebra
$A_{\fU}\otimes_{\CO[T_i^{\pm 1}\,(i\in \Lambda)]}
\CO[T_i^{\pm 1/p^n}\,(i\in\Lambda)]$ is
$p$-adically smooth over $\CO$ with $p$-adic coordinates
given by the images of $T_i^{\pm 1/p^n}$ $(i\in \Lambda)$,
and therefore coincides with $A_{\fU,n}$.
\end{proof}

\begin{lemma}\label{lem:KoszulCohDescription}
Let $\fU=\Spf(A_{\fU})\in \Ob \fX_{\AffZar}$, and 
let $K_{\Lambda}^{\bullet}(\bA_{\inf,X}(U_{\infty})/\mu)$ be
the Koszul complex with respect to the action of $\Gamma_{\Lambda}^{\disc}$
on $\bA_{\inf,X}(U_{\infty})/\mu$
induced by the right action of $\Gamma_{\Lambda}^{\disc}$ on $U_{\infty}$.
Then, for $r\in \N$, the homomorphism 
$H^r(K_{\Lambda}^{\bullet}(\bA_{\inf,X}(U_{\infty}))/\mu)\to
\varprojlim_n H^r(K_{\Lambda}^{\bullet}(\bA_{\inf,X}(U_{\infty}))/(\mu,p^n))$
is an isomorphism, and 
$H^r(K_{\Lambda}^{\bullet}(\bA_{\inf,X}(U_{\infty})/(\mu,p^n)))$ 
for each positive integer $n$ is isomorphic to the direct sum 
of certain numbers of copies of 
$(B_{\fU}/\varphi^{-\nu}(\mu))/p^n$ $(\nu\in \N)$
in a manner compatible with $n$ and functorial in $\fU$. 
\end{lemma}

\begin{proof}
The latter claim with the compatibility with $n$ implies
$R^1\varprojlim_n H^r(K^{\bullet}_{\Lambda}(\bA_{\inf,X}(U_{\infty})/(\mu,p^n)))=0$,
and hence the first claim since $\bA_{\inf,X}(U_{\infty})/\mu$
is $p$-adically complete and separated (Lemma \ref{lem:FlatAinfAlgProperties} (3), 
Remark \ref{rmk:AinfPerfFlatness}). We will prove the latter claim.
The action of $\Gamma_{\Lambda}^{\disc}$ on $B_{\fU}/\mu B_{\fU}$
is trivial because $\gamma_{B_{\fU},i}=1+t_i\mu\theta_{B_{\fU},i}$ $(i\in\Lambda)$.
Hence as in the proof of \cite[Lemma 115]{TsujiSimons}
and that of \cite[Lemma 1.7]{MT}, we see,
 using Lemma \ref{lem:AinfUinftyBasis} and Lemma \ref{lem:KoszulExpFormula} below
  that $H^r(K_{\Lambda}^{\bullet}(\bA_{\inf,X}(U_{\infty})/(\mu,p^n)))$
is isomorphic to the direct sum of certain
numbers of  copies of 
$B_{\fU}/(\varphi^{-\nu}(\mu),p^n)$ and
$B_{\fU}/(\mu,p^n)[\varphi^{-\nu}(\mu)]$ 
for $\nu\in \N$ in a manner compatible with $n$ and
functorial in $\fU$. 
Since $B_{\fU}/p^n$ is $\mu$-torsion free by 
Lemma \ref{lem:FlatAinfAlgProperties} (4) and $\varphi^{-\nu}(\mu)\vert \mu$
in $A_{\inf}$, the multiplication by $\mu\varphi^{-\nu}(\mu)^{-1}$
induces an isomorphism 
$B_{\fU}/(\varphi^{-\nu}(\mu),p^n)\xrightarrow{\cong}
(B_{\fU}/(\mu,p^n))[\varphi^{-\nu}(\mu)]$.
This completes the proof of the latter claim of the lemma.
\end{proof}

\begin{lemma}[cf.~{\cite[Lemma 7.10]{BMS}}]\label{lem:KoszulExpFormula}
Let $R$ be a ring, let $d$ be a positive integer, let $\ug=(g_1,\ldots,g_d)\in R^d$,
and suppose that we are given $h_2,\ldots, h_d\in R$ such that
$g_i=g_1h_i$ $(i=2,\dots, d)$. Then, for an $R$-module $M$ and $n\in \N$,
we have a canonical isomorphism functorial in $M$
$$H^n(K^{\bullet}(\ug;M))\cong M[g_1]^{\binom {d-1}{n}}\oplus
(M/g_1M)^{\binom {d-1}{n-1}},$$
where $K^{\bullet}(\ug;M)$ denotes the Koszul
complex of $M$ with respect to $g_i\cdot \id_M$ $(i=1,\ldots, d)$.
\end{lemma}

\begin{proof} We use the following two facts:\par
(i) Let $h^{\bullet}\colon C^{\bullet}\to D^{\bullet}$ be a 
morphism of complexes of $R$-modules and suppose that we are given 
$R$-linear homomorphisms $k^n\colon C^n\to D^{n-1}$ $(n\in \Z)$
satisfying $h^n=k^{n+1}d^n+d^{n-1}k^n$ $(n\in \Z)$
(i.e., a homotopy between $h^{\bullet}$ and $0$). Then we have an isomorphism 
$\Phi^{\bullet}\colon C^{\bullet}[1]\oplus D^{\bullet}
\xrightarrow{\cong}\Cone(h^{\bullet})$ defined by 
$\Phi^n(x,y)=(x,-k^{n+1}(x)+y)$ $(n\in \Z, x\in C^{n+1}, y\in D^n)$. \par
(ii) Suppose that we are given a commutative diagram of complexes
of $R$-modules
\begin{equation*}
\xymatrix@R=15pt@C=50pt{
C_1^{\bullet}\ar[r]^{f^{\bullet}}\ar[d]_{h_1^{\bullet}}&
C_2^{\bullet}\ar[d]^{h_2^{\bullet}}\ar[dl]_{\ell^{\bullet}}\\
D_1^{\bullet}\ar[r]^{g^{\bullet}}&D_2^{\bullet},
}
\end{equation*}
and let $h^{\bullet}$ be the morphism $\Cone(f^{\bullet})\to \Cone(g^{\bullet})$
induced by $h_i^{\bullet}$ $(i=1,2)$. Then the $R$-linear homomorphisms
$k^n\colon C_1^{n+1}\oplus C_2^n\to D_1^{n}\oplus D_2^{n-1};(x,y)\mapsto
(\ell^{n}(y),0)$ $(n\in\Z)$ satisfy $h^n=k^{n+1}d^n+d^{n-1}k^n$ $(n\in\Z)$,
i.e., give a homotopy between $h^{\bullet}$ and $0$.\par
We can prove the lemma by induction on $d$ by applying (ii) to the diagram
\begin{equation*}
\xymatrix@C=40pt{
K^{\bullet}(g_2,\ldots, g_{d-1};M)\ar[r]^{g_1\cdot \id}\ar[d]_{g_d\cdot\id}&
K^{\bullet}(g_2,\ldots, g_{d-1};M)\ar[d]^{g_d\cdot \id}
\ar[dl]_{h_d\cdot \id}\\
K^{\bullet}(g_2,\ldots,g_{d-1};M)\ar[r]^{g_1\cdot \id}&
K^{\bullet}(g_2,\ldots, g_{d-1};M)
}
\end{equation*}
and then (i) to the homotopy obtained.
\end{proof}

\begin{lemma}\label{lem:SmoothAlgAzero}
For a $(p,\pq)$-adically smooth $A_{\inf}$-algebra
$R$ (Definition \ref{def:IadicProperties} (1)), $R/(\varphi^{-\nu}(\mu),p^{m+1})$ $(\nu,m\in \N)$
has no non-zero element annihilated by all $[\varphi^{-l}(\varepsilon-1)]$ $(l\in\N)$.
\end{lemma}

\begin{proof}
For $\nu\in \N$,  $R/\varphi^{-\nu}(\mu)$
is $p$-torsion free by Proposition \ref{lem:FlatAinfAlgProperties} (3).
Hence the claim is reduced to the case $m=0$ by induction. 
When $m=0$, the smooth algebra $R/(\varphi^{-\nu}(\mu),p)$
over $A_{\inf}/(\varphi^{-\nu}(\mu),p)
\cong A_{\inf}/(\varphi^{-\nu}(\mu),\pq)
\cong \CO/(\zeta_{\nu+1}-1)$ is free as a module
(\cite[the proof of Lemma 8.10]{BMS}). Thus the claim is reduced to the corresponding one for
$\CO/(\zeta_{\nu+1}-1)$, which is verified by using the valuation 
of $\CO$. 
\end{proof}

\begin{proof}[Proof of Proposition \ref{prop:KoszulAOmega}]
 It suffices to show that the restriction 
of the morphism \eqref{eq:KoszulBKFLocCohMap} to 
$\fX_{\AffZar}$ satisfies the conditions (i) and (ii) 
in Lemma \ref{lem:LetaSheafQI} for
$R=A_{\inf}$, $a=\mu$, and $J=\Ker(A_{\inf}\to W(k))$.
The condition (i) holds by Proposition 
\ref{prop:KoszulProetProjRlim}. 
Put $J'=\sum_{l\in \N}[\varphi^{-l}(\varepsilon-1)]A_{\inf}$,
which satisfies $J'=(J')^2\subset J^2$, and
$\CKb=\varprojlim_{\N}K_{\Lambda}^{\bullet}(\iota^*\nu_{\infty*}\ubM)$. For the second claim, we show 
$(\CH^r(\CKb/\mu))(\fU)[J']=0$ for $r\in \N$ and $\fU\in \Ob \fX_{\AffZar}$.

By \eqref{eq:ProetGammaModIm},  Proposition \ref{prop:BKFProperties} (1), 
and Lemma \ref{lem:BKFModLocalStr}, we have
$$(\varprojlim_{\N}\iota^*\nu_{\infty*}\ubM)(\fU)
\cong \varprojlim_{\N}\ubM(U_{\infty})\cong\bM(U_{\infty})
\cong \bM(X_{\infty})\otimes_{\bA_{\inf,X}(X_{\infty})}\bA_{\inf,X}(U_{\infty})$$
for $\fU\in \Ob \fX_{\AffZar}$, and 
$\bM(X_{\infty})$ is a finite projective $\bA_{\inf,X}(X_{\infty})$-module.
Moreover the $\Gamma_{\Lambda}^{\disc}$-equivariant homomorphism 
$\CF(D)/\mu\otimes_{D/\mu}\bA_{\inf,X}(X_{\infty})/\mu
\to \CF(\bA_{\inf,X}(X_{\infty}))/\mu=\bM(X_{\infty})/\mu$ induced by 
$p_{D,\ut}$ \eqref{eq:MorphAinfToD} is an isomorphism as $\CF$ is a crystal,
$\CF(D)/\mu$ is a finite projective $D/\mu$-module,
and the action of $\Gamma_{\Lambda}^{\disc}$
on $\CF(D)/\mu$ is trivial by Proposition \ref{thm:PrismCrysqHiggsEquiv} (1)
applied to $\CF_m$ and $M_m$ $(m\in \N)$.
Since $\CF(D)/\mu$ is a direct factor of a finite free
$D/\mu$-module, the claim is reduced to the case $\bM=\bA_{\inf,X}$, i.e.,
$\CF=\CO_{\fX/A_{\inf}}$. 

Since $\bA_{\inf,X}(U_{\infty})/\mu=
\varprojlim_n\bA_{\inf,X}(U_{\infty})/(\mu,p^n)$ for $\fU\in \Ob\fX_{\AffZar}$
by Lemma \ref{lem:FlatAinfAlgProperties} (3) and Remark \ref{rmk:AinfPerfFlatness},
the presheaf $\fU\to \bA_{\inf,X}(U_{\infty})/\mu$
on $\fX_{\AffZar}$ is a sheaf by Lemma \ref{lem:AinfOpenSub}.
This implies that $\CKb(\fU)/\mu\to (\CKb/\mu)(\fU)$
is an isomorphism for $\fU\in \Ob\fX_{\AffZar}$.
Since $\mfi\colon \fX\to\fY$ is a homeomorphism,
Lemma \ref{lem:KoszulCohDescription} implies that $\fU\mapsto 
H^r((\CKb/\mu)(\fU))\cong H^r(\CKb(\fU)/\mu)$ is a 
sheaf on $\fX_{\AffZar}$, whence $(\CH^r(\CKb/\mu))(\fU)\cong
H^r(\CKb(\fU)/\mu)$ for $\fU\in \Ob\fX_{\AffZar}$. 
Lemma \ref{lem:KoszulCohDescription} combined with \eqref{lem:SmoothAlgAzero}
further shows $H^r(\CKb(\fU)/\mu)[J']=0$. This completes the proof. 
\end{proof}

To finish the proof of Theorem \ref{thm:PrismCohAinfCohLocComp}, 
it remains to prove the following.
Recall that we have assumed that $\mathtt i\colon \fX\to \fY$ is a lifting
of $\fX$ before Proposition \ref{prop:KoszulProetProjRlim}.

\begin{proposition}\label{prop:qDolbeaultGalCoh}
The morphism $c_{\mti,\CF}$ \eqref{eq:qdRKoszulMorLoc} is a quasi-isomorphism.
\end{proposition}

For $\fU\in \Ob \fX_{\AffZar}$, let $M_{\fU}$
denote $(\varprojlim_mv_{D*}\CM_m)(\fU)$, which is a finite projective
$B_{\fU}$-module equipped with the action of $\Gamma_{\Lambda}^{\disc}$
induced by that on $v_{D*}\CM_m$ defined after \eqref{eq:CrysBKFMap}.
Let $\theta_{M_{\fU},i}$ be the $(t_i\mu,\theta_{B_{\fU},i})$-connection
(Definition \ref{def:TwConnectionOneVar}) 
$(\varprojlim_mv_{D*}(\theta_{\CM_m,i}))(\fU)$ on $M_{\fU}$ (Construction \ref{constr:qHiggsModCpxShv} (1)).
Then the action of $\gamma_i\in\Gamma_{\Lambda}^{\disc}$ on 
$M_{\fU}$ is given by $1+t_i\mu\theta_{M_{\fU},i}$ and
it is $\gamma_{B_{\fU},i}$-semilinear by Lemma \ref{lem:TwConnSemilinEnd}.

\begin{lemma}\label{lem:HiggsHNilp}
There exists $N\in \N$ such that 
$(\theta_{M_{\fU,i}})^N(M_{\fU})\subset (p,\pq)M_{\fU}$
and $(t_i\theta_{M_{\fU,i}})^N(M_{\fU})\subset (p,\pq)M_{\fU}$
for every $i\in \Lambda$ and $\fU\in \Ob \fX_{\AffZar}$.
\end{lemma}

\begin{proof}
By the same argument as the proof of \cite[10.3 (2)]{TsujiPrismQHiggs} 
(cf.~Proposition \ref{prop:qPrismEnvTheta}), we see
$\theta_{B_{\fU},i}(B_{\fU})\subset \pq B_{\fU}$, which implies that
$\otheta_{M_{\fU},i}:=(\theta_{M_{\fU},i}\bmod (p,\pq))$ is
$B_{\fU}/(p,\pq)=A_{\fU}/p$-linear.
Since the Higgs field $\utheta_{M_1}$ on $M_1=(v_{D*}\CM_1)(\fX)$
is quasi-nilpotent (Definition \ref{def:qNilpQHiggs}), $M_{\fX}/(p,\pq)\cong M_1$,
and $M_1$ is a finite $A/p$-module, there exists $N$ such that
$(\otheta_{M_{\fX},i})^N=0$ for all $i\in \Lambda$.
Since the homomorphism $M_{\fX}\to M_{\fU}$
is compatible with $\theta_{M_{\fX},i}$ and $\theta_{M_{\fU},i}$
and it induces an isomorphism $M_{\fX}\otimes_BB_{\fU}
\xrightarrow{\cong} M_{\fU}$, we have 
$(\otheta_{M_{\fU},i})^N=0$ for all $i\in \Lambda$
and $\fU\in \Ob \fX_{\AffZar}$. Since
$\otheta_{M_{\fU},i}$ is $A_{\fU}/p$-linear,
we have  $(t_i\otheta_{M_{\fU},i})^N=t_i^N(\otheta_{M_{\fU},i})^N=0$.
\end{proof}

\begin{proof}[Proof of Proposition \ref{prop:qDolbeaultGalCoh}]
For $\fU\in \Ob\fX_{\AffZar}$, we define $M_{\fU,\infty}$
to be the $\bA_{\inf,X}(U_{\infty})$-module
$(\varprojlim_{\N}\iota^*\nu_{\infty*}\ubM)(\fU)$
equipped with a semilinear $\Gamma_{\Lambda}^{\disc}$-action,
and let $c_{\fU}\colon M_{\fU}\to M_{\fU,\infty}$
be the section over $\fU$ of the inverse limit 
$\varprojlim_{\N}\delta_{\mti,\CF}
\colon \varprojlim_{\N}v_*\uCM\to
\varprojlim_{\N}\iota^*\nu_{\infty*}\ubM$ of 
$\delta_{\mti,\CF}$ \eqref{eq:qHigBKFmap}. 
By Proposition \ref{prop:BKFProperties} (1) and \eqref{eq:ProetGammaModIm}, we have
a $\Gamma_{\Lambda}^{\disc}$-equivariant $\bA_{\inf,X}(U_{\infty})$-linear
isomorphism $M_{\fU,\infty}\cong \bM(U_{\infty})$.
Hence, by the construction of $\delta_{\mti, \CF}$
\eqref{eq:qHigBKFmap}, $c_{\fU}$ can be identified with the morphism 
$\CF(B_{\fU})\to \CF(\bA_{\inf,X}(U_{\infty}))$
induced by the morphism 
$(A\to A_{U_{\infty}^+}\leftarrow \bA_{\inf,X}(U_{\infty}))
\to (A\to A_{\fU}\leftarrow B_{\fU})$ in $(\fX/A_{\inf})_{\prism}$ defined
by the $\delta$-homomorphism \eqref{eq:SmLiftToAinfInfty}. Hence $c_{\fU}$
induces a $\Gamma_{\Lambda}^{\disc}$-equivariant 
$\bA_{\inf,X}(U_{\infty})$-linear isomorphism 
$\bA_{\inf,X}(U_{\infty})\otimes_{B_{\fU}}M_{\fU}
\xrightarrow{\cong} M_{\fU,\infty}$. 

For $\ur=(r_i)_{i\in\Lambda}\in (\Z[\frac{1}{p}]\cap [0,p[)^{\Lambda}$,
let $\ut^{\flat\ur}$ denote $\prod_{i\in\Lambda}(t_{A,i}^{\flat})^{r_i}$,
let $M_{\ur}$ be the image of $[\ut^{\flat\ur}]\otimes M_{\fU}$
in $M_{\fU,\infty}$, which is $\Gamma_{\Lambda}^{\disc}$-stable,
and write $[\ut^{\flat\ur}]x$ for the image of $[\ut^{\flat\ur}]\otimes x$
$(x\in M_{\fU})$ in $M_{\fU,\infty}$. Let $M'_{\fU,\infty}$ be the
$(p,\pq)$-adic completion of the direct sum of $M_{\ur}$ $(\ur\neq \uzero)$.
Then we have a $\Gamma^{\disc}_{\Lambda}$-equivariant
$B_{\fU}$-linear isomorphism $M_{\fU}\oplus M'_{\fU,\infty}
\xrightarrow{\cong} M_{\fU,\infty}$ by Lemma \ref{lem:AinfUinftyBasis}. 
By Lemma \ref{lem:qdRKoszulComp}, it 
suffices to prove $H^r(\eta_{\mu}K_{\Lambda}^{\bullet}(M'_{\fU,\infty}))
\cong H^r(K_{\Lambda}^{\bullet}(M'_{\fU,\infty}))/(H^r(K_{\Lambda}^{\bullet}(M'_{\fU,\infty}))[\mu])$
vanishes, i.e., $\mu \cdot H^r(K_{\Lambda}^{\bullet}(M'_{\fU,\infty}))=0$ for all $r\in \N$.

Put $\eta_r=\mu([\varepsilon^r]-1)^{-1}\in A_{\inf}$ for $r\in \Z[\frac{1}{p}]\cap ]0,p[$.
Then, for $\ur\in (\Z[\frac{1}{p}]\cap [0,p[)^{\Lambda}$ and $i\in \Lambda$
such that $r_i\neq 0$, and $x\in M_{\fU}$, we have 
$(\gamma_i-1)([\ut^{\flat\ur}]x)
=[\varepsilon^{r_i}](x+\mu t_i\theta_{M_{\fU},i}(x))-x
=([\varepsilon^{r_i}]-1)(x+[\varepsilon^{r_i}]\eta_{r_i}t_i\theta_{M_{\fU},i}(x))$.
By Lemma \ref{lem:HiggsHNilp}, the endomorphism 
$g_{r,i}=1+[\varepsilon^r]\eta_rt_i\theta_{M_{\fU},i}$ 
of $M_{\fU}$ is an automorphism for $r\in \Z[\frac{1}{p}]\cap ]0,p[$
and $i\in\Lambda$.
Hence we can define an $A_{\inf}$-linear endomorphism $h_{\ur,i}$
of $M_{\ur,i}$ for $\ur$ and $i$ with $r_i\neq 0$ by 
$h_{\ur,i}([\ut^{\flat\ur}]x)=[\ut^{\flat\ur}]\eta_{r_i}g_{r_i,i}^{-1}(x)$
$(x\in M_{\fU})$, which satisfies
$h_{\ur,i}\circ(\gamma_i-1)=(\gamma_i-1)\circ h_{\ur,i}=\mu$.

Choose a decomposition $\sqcup_{i\in \Lambda}\CS_i$
of $(\Z[\frac{1}{p}]\cap [0,p[)^{\Lambda}\backslash\{\uzero\}$
such that $r_i\neq 0$ for every $i\in \Lambda$ and $\ur=(r_j)_{j\in \Lambda}\in \CS_i$,
let $M_i$ be the $(p,\pq)$-adic completion of $\oplus_{\ur\in \CS_i}M_{\ur}$,
and let $h_i$ be the endomorphism of $M_i$ induced by the $A_{\inf}$-linear
endomorphism $\oplus_{\ur\in \CS_i}h_{\ur,i}$ of the direct sum.
We have $M'_{\fU,\infty}=\oplus_{i\in\Lambda}M_i$ and 
$(\gamma_i-1)\circ h_i=h_i\circ(\gamma_i-1)=\mu$, which
implies that $\gamma_i-1$ on $M_i$ is injective and 
$\mu M_i\subset (\gamma_i-1)(M_i)$. Let $K^{\bullet}_i$ be the Koszul 
complex of $M_i$ with respect to $\gamma_j-1$ $(j\in\Lambda\backslash\{i\})$.
Then we have a quasi-isomorphism $K_{\Lambda}^{\bullet}(M_i)
\cong \text{Cone}(-(\gamma_i-1)\colon K^{\bullet}_i\to K_i^{\bullet})[-1]
\to K_i^{\bullet}/(\gamma_i-1)(K_i^{\bullet})[-1]$
and $\mu\cdot K^{\bullet}_i/(\gamma_i-1)(K_i^{\bullet})=0$.
Hence $\mu\cdot H^r(K_{\Lambda}^{\bullet}(M'_{\fU,\infty}))
=\mu\cdot\oplus_{i\in\Lambda} H^r(K_{\Lambda}^{\bullet}(M_i))=0$.
\end{proof}

\section{Comparison with $A_{\inf}$-cohomology with coefficients:
the global case}
\label{sec:compAinfcohGlobal}
In this section, we will derive the following global comparison theorem 
from the local one: 
Theorem \ref{thm:PrismCohAinfCohLocComp} by using cohomological descent.

\begin{theorem}\label{thm:PrismCohAinfCohGlobComp}
Let $\fX$ be a quasi-compact, separated, smooth, $p$-adic formal scheme over $\CO$.
Let $\CF\in \Ob(\Crystal^{\fproj}_{\prism}(\fX/A_{\inf}))$
(Definition \ref{def:PrismaticSiteCrystal} (2)) and put $\bM=\bM_{\BKF,\fX}(\CF)$ 
(Definition \ref{BKFFunctor}). Then we have the following canonical isomorphism 
in $D(\fX_{\Zar},A_{\inf})$ functorial in $\CF$.
\begin{equation}\label{eq:PrismCohAinfCohGlobComp}
Ru_{\fX/A_{\inf}*}\CF\xrightarrow{\;\sim\;} A\Omega_{\fX}(\bM)
\end{equation}
\end{theorem}

The morphisms of functors \eqref{eq:XproetToGammaShfFunct}
is not an isomorphism in general. Therefore to construct and study the functor
$\nu_{\fX_{\bcdot},\ut_{\bcdot}*}$ for  a simplicial small framed embedding over $A_{\inf}$
$(\mti_{\bcdot}\colon \fX_{\bcdot}\to \fY_{\bcdot},\ut_{\bcdot})$ (Definition \ref{def:AdmFramedSmEmbed}) 
of a Zariski hypercovering $\fX_{\bcdot}$ of $\fX$ by affine formal schemes,
we need some preliminaries on a family of topos over a category,
which are postponed until the next section 
\ref{sec:DTopos}.

As usual, let $\Delta$
denote the category whose objects are $[r]=\{0,1,\ldots,r\}$
$(r\in\N)$ and whose morphisms are non-decreasing maps, and
let $\Delta^{\circ}$ be its opposite category.
Recall that a simplicial (resp.~cosimplicial) object in a category $\CC$ means
a functor from $\Delta^{\circ}$ (resp.~$\Delta$) to $\CC$.
Let $\fX$ be a quasi-compact,
separated, smooth, $p$-adic formal scheme over $\CO$. Choose
a Zariski hypercovering $\fX_{\bcdot}=([r]\mapsto \fX_{[r]}=\Spf(A_{[r]}))_{r\in\N}$
of $\fX$ by affine formal schemes and a simplicial small
framed embedding of $\fX_{\bcdot}$ over $A_{\inf}$
$\mti_{\bcdot}:=
([r]\mapsto (\mfi_{[r]}\colon \fX_{[r]}\to \fY_{[r]}=\Spf(B_{[r]}),
\ut_{[r]}=(t_{[r],i})_{i\in\Lambda_{[r]}}))_{r\in\N}$.

\begin{remark} 
We can construct $\fX_{\bcdot}$ and $\mti_{\bcdot}$ above
as follows. Since $\fX$ is quasi-compact, there exists 
an affine open covering $\fU_{\nu}$ $(\nu=1,\ldots, N)$ of $\fX$,
and small framed embeddings 
$(\mfi_{\nu}\colon \fU_{\nu}\to \fV_{\nu}, \ut_{\nu}=(t_{\nu,i})_{i\in\Lambda})$ over $A_{\inf}$
$(\nu=1,\ldots, N)$ with a common totally ordered set $\Lambda$. 
Then we obtain an affine Zariski hypercovering $\fX_{\bcdot}$ of $\fX$
and its simplicial small framed embedding $\mti_{\bcdot}$ 
over  $A_{\inf}$ by applying the construction in \cite[15.1]{TsujiPrismQHiggs} to
$\mfi_{\nu}$ $(\nu\in \N\cap [1,N])$. 
\end{remark}

For a nondecreasing map $a\colon[r]\to [s]$, we write
$\mtg_a=(g_a,h_a,\psi_a)$ for the
morphism $\mathtt{i}_{[s]}\to \mathtt{i}_{[r]}$ of 
small framed  embeddings over $A_{\inf}$
(Definition \ref{def:AdmFramedSmEmbed} (2)) corresponding to $a$. 
Associated to each $\mathtt{i}_{[r]}$ $(r\in \N$),
we have morphisms of ringed topos \eqref{eq:ProetZarMorph1}
and \eqref{eq:ProetZarMorph2}, which will be denoted
with subscript $[r]$ in the following, and their relations 
\eqref{eq:ProetZarMorph3}. 
By applying the functoriality \eqref{eq:ProetZarMorphFunct1} 
and \eqref{eq:ProetZarMorphFunct2} of
\eqref{eq:ProetZarMorph1} and \eqref{eq:ProetZarMorph2}, respectively, to
$\mtg_a$ for each $a\in \Mor\Delta$, we
obtain corresponding direct and inverse image functors of ringed $\Delta$-topos 
and their relations except $\nu_{\infty,[r]}$ $(r\in\N)$, for which we have 
only a direct
image functor. See Remark \ref{rmk:DtoposConstruction} and 
the construction of direct and inverse image functors over $D$
from \eqref{eq:DtoposDirectImageFib} being assumed to be an isomorphism
for the inverse image functor. By taking topos of sections over $\Delta$, we
obtain  morphisms of ringed topos associated to 
$\mti_{\bcdot}=(\mfi_{\bcdot}\colon \fX_{\bcdot}\to\fY_{\bcdot},\ut_{\bcdot})$
\begin{gather}
((X_{\bcdot,\proet}^{\sim})^{\N^{\circ}},\ubA_{\inf,X_{\bcdot}})
\xrightarrow{\;\nu_{\bcdot}\;}
((\fX_{\bcdot,\Zar}^{\sim})^{\N^{\circ}},\uA_{\inf})
\xleftarrow{\:v_{\bcdot}\;}
((\fD_{\bcdot,\Zar}^{\sim})^{\N^{\circ}},\uA_{\inf})\label{eq:ProetZarProjHP0}\\
((\Gamma_{\Lambda_{\bcdot}}\hy\fX_{\bcdot,\Zar}^{\sim})^{\N^{\circ}},\uA_{\inf})
\overset{\pi_{\bcdot}}{\underset{\iota_{\bcdot}}{\rightleftarrows}}
((\fX_{\bcdot,\Zar}^{\sim})^{\N^{\circ}},\uA_{\inf})\label{eq:ProetZarProjHP}
\end{gather}
and a direct image functor
\begin{equation}\label{eq:ProetGammaModProjHP}
((X_{\bcdot,\proet}^{\sim})^{\N^{\circ}},\ubA_{\inf, X_{\bcdot}})
\xrightarrow{\;\nu_{\infty,\bcdot*}\;}
((\Gamma_{\Lambda_{\bcdot}}\hy\fX_{\bcdot,\Zar}^{\sim})^{\N^{\circ}},\uA_{\inf})
\end{equation}
preserving inverse limits and satisfying
\begin{equation}\label{eq:SimplcialToposMorphComp}
\pi_{\bcdot*}\circ \nu_{\infty,\bcdot*}\cong \nu_{\bcdot*},\qquad\pi_{\bcdot}\circ\iota_{\bcdot}\cong\id.
\end{equation}
Since $\fX_{\bcdot}$ is a simplicial $p$-adic formal scheme over $\fX$,
we have morphisms of topos $\bm{\theta}\colon
X_{\bcdot,\proet}^{\sim}\to X_{\proet}^{\sim}$,
$\theta\colon \fX_{\bcdot,\Zar}^{\sim}\to \fX_{\Zar}$,
and $\theta_{\prism}\colon(\fX_{\bcdot}/A_{\inf})_{\prism}^{\sim}
\to (\fX/A_{\inf})_{\prism}^{\sim}$, whose inverse image 
functors are simply given by taking the pullback under the morphisms of 
topos associated to the morphisms $g_{[r]}\colon \fX_{[r]}
\to \fX$ $(r\in \N)$ defining the hypercovering
$\fX_{\bcdot}\to \fX$ (\cite[V\bis (2.2.1)]{SGA4}). 
The inverse images of $\bA_{X,\inf}$, $\CO_{\fX/A_{\inf}}$,
and the constant sheaf $A_{\inf}$ on $\fX_{\Zar}$
under these inverse image functors coincide with 
$\bA_{X_{\bcdot},\inf}$, $\CO_{\fX_{\bcdot}/A_{\inf}}$
and $A_{\inf}$, respectively. The same obviously holds for their 
reduction modulo $(p,\pq)^{m+1}$, i.e., the sheaves
denoted with the subscript $m$.

Let $\CF\in \Ob \Crystal^{\fproj}_{\prism}(\fX/A_{\inf})$,
put $\bM=\bM_{\BKF,\fX}(\CF)$ (Definition \ref{BKFFunctor}),
and $\bM_m=\bM/(p,\pq)^{m+1}\bM$, and let 
$\ubM$ be the $\ubA_{\inf,X}$-module
$(\bM_m)_{m\in \N}$. Let $\CF_{\bcdot}=([r]\mapsto
\CF_{[r]})_{r\in\N}$ be $\theta_{\prism}^{-1}(\CF)$.
By applying the construction of $\ubM$ from $\CF$
also to $\CF_{[r]}$ for each $r\in \N$ and
then its functoriality satisfying the cocycle condition
(Proposition \ref{prop:BKFFunctComp}) to $g_{[r]}$ $(r\in \N)$ and $g_a$ $(a\in \Mor\Delta)$,
we obtain $\ubM_{\bcdot}\in 
\Mod((X_{\bcdot,\proet}^{\sim})^{\N^{\circ}},\ubA_{\inf,X_{\bcdot}})$
and an isomorphism 
\begin{equation}
\bm{\theta}^*(\ubM)\xrightarrow{\;\cong\;}\ubM_{\bcdot}.
\end{equation}

We have isomorphisms $g_{a\prism}^*(\CF_{[r]})\cong 
\CF_{[s]}$ for $a\colon [r]\to[s]\in\Mor\Delta$
satisfying the cocycle condition for composition of $a$'s.
Therefore, thanks to \eqref{eq:BKFGammaCohResolFunct}, 
the resolution $\beta_{\mti_{[r]},\CF_{[r]}}$
\eqref{eq:BKFGammaRepResol} of $\nu_{\infty,[r]*}\ubM_{[r]}$  for
each $r\in \N$, and the morphism $K^{\bullet}_{\psi_a}(\ootau_{\mtg_a,\CF_{[r]}})$
\eqref{eq:BKFKoszulFunctMap2} for each $a\colon [r]\to [s]\in \Mor\Delta$, which
satisfies the cocycle condition for two composable morphisms
in $\Mor\Delta$ by Remark \ref{rmk:BKFGammaKoszulFuncCocyc}, define a resolution 
\begin{equation}\label{eq:BKFKoszulResolHC}
\beta_{\mti_{\bcdot},\CF_{\bcdot}}\colon
\nu_{\infty,\bcdot*}\ubM_{\bcdot}
\longrightarrow K_{\Lambda_{\bcdot}}^{\bullet}
(\iota_{\bcdot*}\iota_{\bcdot}^*\nu_{\infty,\bcdot*}\ubM_{\bcdot})
\quad{\text{in}}\;\;
C^+((\Gamma_{\Lambda_{\bcdot}}\hy\fX_{\bcdot,\Zar}^{\sim})^{\N^{\circ}},
\uA_{\inf}).
\end{equation}
Then, by \eqref{eq:BKFGammaCohFunct}, the isomorphism 
$\alpha_{\mti_{[r]},\CF_{[r]}}$ \eqref{eq:GammaGalCohIsomMap} for each $r\in \N$,
and the morphism $K^{\bullet}_{\psi_a}(\otau_{\mtg_a,\CF_{[r]}})$
\eqref{eq:BKFKoszulFunctMap}
for each $a\colon [r]\to[s]\in \Mor\Delta$, which satisfies
the cocycle condition for composition of $a$'s by Remark \ref{rmk:BKFGammaKoszulFuncCocyc}, 
yield an isomorphism
\begin{equation}\label{eq:BKFKoszulReolGInvHP}
\alpha_{\mti_{\bcdot},\CF_{\bcdot}}\colon
K_{\Lambda_{\bcdot}}^{\bullet}(\iota_{\bcdot}^*\nu_{\infty,\bcdot*}\ubM_{\bcdot})
\xrightarrow{\;\cong\;}
\pi_{\bcdot*}K_{\Lambda_{\bcdot}}^{\bullet}(\iota_{\bcdot*}
\iota_{\bcdot}^*\nu_{\infty,\bcdot*}\ubM_{\bcdot})
\quad\text{in}\;\;
C^+((\fX_{\bcdot,\Zar}^{\sim})^{\N^{\circ}},\uA_{\inf}).
\end{equation}

Similarly, one can define a complex
\begin{equation}
q\Omegab(\uCM_{\bcdot}, \utheta_{\uCM_{\bcdot}})
\quad \text{in}\;\; C^+((\fD_{\bcdot,\Zar}^{\sim})^{\N^{\circ}},\uA_{\inf})
\end{equation}
by applying the construction of $q\Omegab(\uCM,\utheta_{\uCM})$
from $\mti$ and $\CF$ given before \eqref{eq:AinfCrysProj} 
to $\mti_{[r]}$ and $\CF_{[r]}$ for each $r\in \N$
and using $\sigma_{\mtg_a,\CF_{[r]}}$ \eqref{eq:qdRInvSysFunct} for
$a\colon[r]\to [s]\in \Mor\Delta$ which satisfy
the cocycle condition for composition of $a$'s by Remark 
\ref{rmk:qdRInvSysFunctCocyc}.
By \eqref{eq:qdRKoszulGammaCohFunct}, we see that the morphisms
$c_{\mti_{[r]},\CF_{[r]}}$ \eqref{eq:qdRKoszulMorLoc} for $r\in \N$ define
a morphism
\begin{equation}\label{eq:qdRKoszulMapHC}
c_{\mti_{\bcdot},\CF_{\bcdot}}\colon
\varprojlim_{\N} v_{\bcdot*}(q\Omegab(\uCM_{\bcdot},\utheta_{\uCM_{\bcdot}}))
\longrightarrow
\eta_{\mu}\varprojlim_{\N}K^{\bullet}_{\Lambda_{\bcdot}}(\iota_{\bcdot}^*\nu_{\infty,\bcdot*}
\ubM_{\bcdot})
\quad \text{in}\;\; C^+(\fX_{\bcdot,\Zar}^{\sim},\uA_{\inf}).
\end{equation}
Note that $\eta_{\mu}$ of $\mu$-torsion free complexes of $A_{\inf}$-modules
commutes with the inverse image functor of the flat morphism
of ringed topos $e_{[r]}\colon(\fX_{[r],\Zar}^{\sim},A_{\inf})
\to (\fX_{\bcdot,\Zar}^{\sim},A_{\inf})$ $(r\in\N)$.

From $\alpha_{\mti_{\bcdot},\CF_{\bcdot}}$ and $\beta_{\mti_{\bcdot},\CF_{\bcdot}}$,
we obtain a sequence of morphisms in $D^+(\fX^{\sim}_{\bcdot,\Zar},A_{\inf})$
\begin{multline}\label{eq:SimplicialKoszulBKFMap}
\varprojlim_{\N}K^{\bullet}_{\Lambda_{\bcdot}}(\iota_{\bcdot}^*\nu_{\infty,\bcdot*}\ubM_{\bcdot})
\xrightarrow[\varprojlim_{\N}\alpha_{\mti_{\bcdot},\CF_{\bcdot}}]{}
R\varprojlim_{\N}R\pi_{\bcdot*}
K^{\bullet}_{\Lambda_{\bcdot}}(\iota_{\bcdot*}\iota_{\bcdot}^*\nu_{\infty,\bcdot*}\ubM_{\bcdot})\\
\xleftarrow[R\varprojlim_{\N}R\pi_{\bcdot*}(\beta_{\mti_{\bcdot},\CF_{\bcdot}})]{\cong}
R\varprojlim_{\N}R\pi_{\bcdot*}(\nu_{\infty,\bcdot*}\ubM_{\bcdot})
\to
R\varprojlim_{\N}R\pi_{\bcdot*}R\nu_{\infty,\bcdot*}\ubM_{\bcdot}
\xleftarrow{\;\cong\;}
R\varprojlim_{\N}R\nu_{\bcdot*}\ubM_{\bcdot}.
\end{multline}
We see that the last morphism is an isomorphism by taking the inverse image
under the morphism of ringed topos $e_{[r]}\colon (\fX^{\sim}_{[r],\Zar},A_{\inf})
\to (\fX^{\sim}_{\bcdot,\Zar},A_{\inf})$ for each $r\in\N$ (see the two paragraphs
after Remark \ref{rmk:DtoposConstruction}) and using 
\eqref{eq:SimplcialToposMorphComp} and \eqref{eq:InvSysDtoposFBF}.
By taking $L\eta_{\mu}$ of \eqref{eq:SimplicialKoszulBKFMap} and composing it with 
$c_{\mti_{\bcdot},\CF_{\bcdot}}$ \eqref{eq:qdRKoszulMapHC} and \eqref{eq:PrismCohQHiggsCpxSimp}, we obtain
\begin{multline}\label{eq:SimplicialCompMap}
\kappa_{\mti_{\bcdot},\CF_{\bcdot}}\colon
Ru_{\fX_{\bcdot}/A_{\inf}*}\CF_{\bcdot}
\xrightarrow[\eqref{eq:PrismCohQHiggsCpxSimp}]{\cong}
\varprojlim_{\N} v_{\bcdot*}(q\Omegab(\uCM_{\bcdot}, \utheta_{\uCM_{\bcdot}}))\\
\xrightarrow[\eqref{eq:qdRKoszulMapHC}]{c_{\mti_{\bcdot},\CF_{\bcdot}}}
\eta_{\mu}\varprojlim_{\N}K^{\bullet}_{\Lambda_{\bcdot}}(\iota_{\bcdot}^*\nu_{\infty,\bcdot*}
\ubM_{\bcdot})
\xrightarrow{L\eta_{\mu}\;\;\eqref{eq:SimplicialKoszulBKFMap}}
L\eta_{\mu}R\varprojlim_{\N}R\nu_{\bcdot*}\ubM_{\bcdot}
\end{multline}
in $D(\fX_{\bcdot,\Zar},A_{\inf})$. 
By \eqref{eq:InvSysDtoposFBF} and Proposition \ref{prop:DtoposInvLimitFbF}, 
we see that the inverse image of 
$\kappa_{\mti_{\bcdot},\CF_{\bcdot}}$ under
$e_{[r]}\colon (\fX_{[r],\Zar},A_{\inf})\to (\fX_{\bcdot,\Zar},A_{\inf})$
may be identified with $\kappa_{\mti_{[r]},\CF_{[r]}}$ for
each $r\in \N$. Therefore $\kappa_{i_{\bcdot},\CF_{\bcdot}}$
is an isomorphism by Theorem \ref{thm:PrismCohAinfCohLocComp}.

Since $\fX_{\bcdot}$ is a Zariski hypercovering of $\fX$, and
$Ru_{\fX_{\bcdot}/A_{\inf}*}$, $R\nu_{\bcdot*}$, and
$R\varprojlim_{\N}$ on $((\fX^{\sim}_{\Zar,\bcdot})^{\N^{\circ}},\uA_{\inf})$
can be computed on each $\fX_{[r]}$ $(r\in\N)$ by
\eqref{eq:InvSysDtoposFBF} and
Proposition \ref{prop:DtoposInvLimitFbF}, 
we have
\begin{align}
L\theta^*L\eta_{\mu}R\varprojlim_{\N}R\nu_{*}\ubM
&\xrightarrow{\;\cong\;} L\eta_{\mu}R\varprojlim_{\N}R\nu_{\bcdot*}\ubM_{\bcdot},
\label{eq:AOmegaCohSimpPB}\\
L\theta^*Ru_{\fX/A_{\inf}*}\CF&\xrightarrow{\;\cong\;}
Ru_{\fX_{\bcdot}/A_{\inf}*}\CF_{\bcdot}.
\label{eq:PrismZarProjCohSimpPBAinf}
\end{align}
By cohomological descent \cite[V\bis Proposition (3.2.4), Proposition (3.3.1) a), Th\'eor\`eme (3.3.3)]{SGA4}, we obtain 
\begin{align}
L\eta_{\mu}R\varprojlim_{\N}R\nu_{*}\ubM
&\xrightarrow{\;\cong\;} R\theta_*L\eta_{\mu}R\varprojlim_{\N}R\nu_{\bcdot*}\ubM_{\bcdot},
\label{eq:AOmegaCohDesc}
\\
Ru_{\fX/A_{\inf}*}\CF&\xrightarrow{\;\cong\;}
R\theta_*Ru_{\fX_{\bcdot}/A_{\inf}*}\CF_{\bcdot}.
\label{eq:PrismZarProjCohDescAinf}
\end{align}
By taking $R\theta_*$ of the isomorphism $\kappa_{\mti_{\bcdot},\CF_{\bcdot}}$ 
\eqref{eq:SimplicialCompMap} and 
composing it with \eqref{eq:AOmegaCohDesc} and \eqref{eq:PrismZarProjCohDescAinf}, 
we obtain an isomorphism in $D(\fX_{\Zar},A_{\inf})$
\begin{equation}\label{eq:BKFPrismCohLocCompMapSimp}
\kappa_{i_{\bcdot},\CF}\colon Ru_{\fX/A_{\inf}*}\CF
\xrightarrow{\;\cong\:}L\eta_{\mu}R\varprojlim_{\N}R\nu_{*}\ubM
\cong A\Omega_{\fX}(\bM).
\end{equation}

\par\bigskip
We next prove that $\kappa_{\mti_{\bcdot},\CF}$ does not
depend on the choice of $\mti_{\bcdot}=(\mfi_{\bcdot}\colon \fX_{\bcdot}\to\fY_{\bcdot},\ut_{\bcdot})$,
and that it is functorial in $\fX$.\par

Let $\fX'$ be another quasi-compact, separated, smooth $p$-adic formal scheme,
and suppose that we are given a morphism $g\colon \fX'\to \fX$ over $\CO$.
Let $\fX'_{\bcdot}$ be a Zariski hypercovering of $\fX'$ by affine formal schemes, and
let $\mti'_{\bcdot}=(\mfi'_{\bcdot}\colon \fX'_{\bcdot}\to\fY'_{\bcdot},\ut'_{\bcdot},\Lambda'_{\bcdot})$ be
a simplicial small framed embedding of $\fX'_{\bcdot}$ over $A_{\inf}$.

We define $\fX''_{\bcdot}$ to be the product of simplicial formal schemes
$\fX'_{\bcdot}$  and $\fX_{\bcdot}\times_{\fX}\fX'$ over $\fX'$,
$\fY''_{\bcdot}$ to be the product of simplicial formal schemes $\fY_{\bcdot}$ 
and $\fY'_{\bcdot}$ over $A_{\inf}$, and $\Lambda''_{\bcdot}$ to be
the disjoint union of the cosimplicial sets $\Lambda_{\bcdot}$
and $\Lambda'_{\bcdot}$.  For $r\in \N$, we equip $\Lambda''_{[r]}$
with the unique total order compatible with those on $\Lambda_{[r]}$
and $\Lambda'_{[r]}$ and satisfying $i\leq i'$ for every 
$(i,i')\in \Lambda_{[r]}\times\Lambda'_{[r]}$. We define 
$t''_{[r],i''}$ $(i''\in \Lambda_{[r]}'')$ to be the inverse image
of $t_{[r],i''}$ (resp.~$t'_{[r],i''}$) to $\fY''_{[r]}$ if 
$i''\in \Lambda_{[r]}$ (resp.~$i''\in \Lambda'_{[r]}$), 
and put $\ut''_{[r]}=(t''_{[r],i''})_{i''\in\Lambda''_{[r]}}$.
Then $\fX''_{\bcdot}$ is a Zariski hypercovering of $\fX'$ by affine
formal schemes since $\fX$ and $\fX'$ are assumed to be separated;
the morphism $\mfi''_{\bcdot}\colon \fX''_{\bcdot}\to \fY''_{\bcdot}$
induced by $\mfi_{\bcdot}$ and $\mfi'_{\bcdot}$, 
$\ut''_{\bcdot}$, and $\Lambda''_{\bcdot}$ form a simplicial
small framed embedding of $\fX''_{\bcdot}$ over $A_{\inf}$;
the projections $\fX_{\bcdot}''\to \fX_{\bcdot}$
and $\fY_{\bcdot}''\to \fY_{\bcdot}$
(resp.~$\fX_{\bcdot}''\to \fX'_{\bcdot}$
and $\fY_{\bcdot}''\to \fY'_{\bcdot}$) and the
inclusion $\Lambda_{\bcdot}\to \Lambda''_{\bcdot}$
(resp.~$\Lambda'_{\bcdot}\to \Lambda''_{\bcdot}$)
define a morphism of simplicial small framed
embeddings 
$\mti''_{\bcdot}=(\mfi''_{\bcdot},\ut''_{\bcdot},\Lambda''_{\bcdot})
\to \mti_{\bcdot}=(\mfi_{\bcdot},\ut_{\bcdot},\Lambda_{\bcdot})$
(resp.~$\mti'_{\bcdot}=(\mfi'_{\bcdot},\ut'_{\bcdot},\Lambda'_{\bcdot})$)
over $A_{\inf}$.
Therefore it suffices to prove the following
functoriality of $\kappa_{\mti_{\bcdot},\CF}$.

\begin{proposition}\label{eq:BKFPrismLocCohCompSimpFunct}
Let $\fX_{\bcdot}\to \fX$, 
$\mti_{\bcdot}=(\mfi_{\bcdot}\colon \fX_{\bcdot}\to \fY_{\bcdot}, \ut_{\bcdot},\Lambda_{\bcdot})$,
$\fX'_{\bcdot}\to \fX'$, 
$\mti'_{\bcdot}=(\mfi'_{\bcdot}\colon \fX'_{\bcdot}\to \fY'_{\bcdot}, \ut'_{\bcdot},\Lambda'_{\bcdot})$,
and $g\colon \fX'\to \fX$ be as above, and suppose
that we are given a morphism of simplicial small framed 
embeddings $\mathtt{g}_{\bcdot}=(g_{\bcdot},h_{\bcdot},\psi_{\bcdot})
\colon \mti'_{\bcdot}=(\mfi_{\bcdot}',\ut'_{\bcdot},\Lambda'_{\bcdot})
\to \mti_{\bcdot}=(\mfi_{\bcdot},\ut_{\bcdot},\Lambda_{\bcdot})$ over $A_{\inf}$
such that $g_{\bcdot}$ is a morphism over $g$. Let 
$\CF\in \Ob \Crystal^{\fproj}_{\prism}(\fX/A_{\inf})$,
and put $\CF'=g_{\prism}^{-1}\CF\in \Ob\Crystal_{\prism}^{\fproj}(\fX'/A_{\inf})$.
Then the following diagram is commutative.
\begin{equation*}
\xymatrix@R=20pt{
Ru_{\fX/A_{\inf}*}\CF\ar[d]^{\cong}_{\kappa_{\mti_{\bcdot},\CF}}\ar[r]&
Ru_{\fX/A_{\inf}*}Rg_{\prism*}\CF'
\ar[r]^{\cong}&
Rg_*Ru_{\fX'/A_{\inf}*}\CF'
\ar[d]^{\cong}_{Rg_*(\kappa_{\mti'_{\bcdot},\CF'})}\\
A\Omega_{\fX}(\bM_{\BKF,\fX}(\CF))\ar[rr]^{\eqref{eq:AOmegaFunct}}&&
Rg_*A\Omega_{\fX'}(\bM_{\BKF,\fX'}(\CF'))
}
\end{equation*}
\end{proposition}

By \eqref{eq:ProetZarMorphFunct1}, \eqref{eq:ProetZarMorphFunct2}, 
Lemma \ref{lem:DToposDIComp}, and Lemma \ref{lem:DToposDIMorph}
(with $c_i$ isomorphisms), we see that the morphisms of ringed topos \eqref{eq:ProetZarProjHP0} and 
\eqref{eq:ProetZarProjHP}
for $\fX_{\bcdot}$ and $\fX'_{\bcdot}$ are functorial with respect to
$\mathtt{g}$, i.e.,  we have the commutative diagram 
\eqref{eq:ProetZarMorphFunct1} and the right commutative square of
\eqref{eq:ProetZarMorphFunct2} with $\fX$, $\fX'$, $X$, $X'$, $\fD$, $\fD'$, $\Lambda$, and
$\Lambda'$ replaced by $\fX_{\bcdot}$, $\fX'_{\bcdot}$, $X_{\bcdot}$, $X'_{\bcdot}$, $\fD_{\bcdot}$, 
$\fD'_{\bcdot}$, $\Lambda_{\bcdot}$, and $\Lambda'_{\bcdot}$. The morphisms of ringed topos in the
diagrams are denoted by the same symbols with the subscript $\bcdot$ added
as $\nu_{\bcdot}$, $v_{\bcdot}$, $\bmg_{\bcdot}$, $h_{\fD_{\bcdot}}$.
For the direct image functor \eqref{eq:ProetGammaModProjHP}, 
we have a morphism of functors
\begin{equation}
\label{eq:ProetZarMorphFunctHP3}
\xymatrix@R=15pt{
((X_{\bcdot\proet}^{\prime\sim})^{\N^{\circ}},\ubA_{\inf,X'_{\bcdot}})
\ar[r]^{\nu'_{\infty\bcdot*}}\ar[d]^{\bmg_{\bcdot*}}&
((\Gamma_{\Lambda'_{\bcdot}}\hy\fX_{\bcdot\Zar}^{\prime\sim})^{\N^{\circ}},\uA_{\inf})
\ar[d]^{g_{\bcdot\psi_{\bcdot}*}}
\\
((X_{\bcdot\proet}^{\sim})^{\N^{\circ}},\ubA_{\inf,X_{\bcdot}})
\ar[r]^{\nu_{\infty\bcdot*}}\ar@{=>}[ur]^{\Xi_{\mathtt{g}_{\bcdot}}}&
((\Gamma_{\Lambda_{\bcdot}}\hy\fX_{\bcdot\Zar}^{\sim})^{\N^{\circ}},\uA_{\inf})
}
\end{equation}
by the remark after \eqref{eq:ProetZarMorphFunct2}, 
Remark \ref{rmk:GammaCohFunctCocyc}, 
Lemma \ref{lem:DToposDIComp}, and Lemma \ref{lem:DToposDIMorph}. 

Since the morphisms \eqref{eq:BKFIndsysFunctMap}, \eqref{eq:BKFGammaRepFunct}, 
\eqref{eq:BKFKoszulFunctMap}, \eqref{eq:BKFKoszulFunctMap2}, and \eqref{eq:qdRInvSysFunct} satisfy
the cocycle conditions for composition of $\mathtt{g}$'s as mentioned 
in Remarks \ref{rmk:BKFGammaKoszulFuncCocyc} and \ref{rmk:qdRInvSysFunctCocyc}, 
these morphisms for $\mathtt{g}_{[r]}$ and $\CF_{[r]}$ $(r\in \N)$ 
define the following morphisms.
\begin{gather}
\uvarepsilon_{g_{\bcdot},\CF_{\bcdot}}\colon \ubM_{\bcdot}\longrightarrow 
\bmg_{\bcdot*}\ubM'_{\bcdot}
\quad\text{in}\;\;\Mod((X_{\bcdot\proet}^{\sim})^{\N^{\circ}}, \ubA_{\inf,X_{\bcdot}}),
\label{eq:BKFIndsysFunctMapHC}\\
\tau_{\mathtt g_{\bcdot},\CF_{\bcdot}}\colon \nu_{\infty\bcdot*}\ubM_{\bcdot}
\to \nu_{\infty\bcdot*}\bmg_{\bcdot*}\ubM'_{\bcdot}
\xrightarrow{\Xi_{\mathtt g_{\bcdot}}(\ubM'_{\bcdot})}
g_{\bcdot\psi_{\bcdot}*}\nu'_{\infty\bcdot*}\ubM'_{\bcdot}
\quad \text{in}\;
\Mod((\Gamma_{\Lambda_{\bcdot}}\hy\fX_{\bcdot\Zar}^{\sim})^{\N^{\circ}},
\uA_{\inf}),
\label{eq:BKFIndGammaFunctMapHC}\\
K^{\bullet}_{\psi_{\bcdot}}(\otau_{\mtg_{\bcdot},\CF_{\bcdot}})\colon 
K_{\Lambda_{\bcdot}}^{\bullet}(\iota_{\bcdot}^*\nu_{\infty\bcdot*}\ubM_{\bcdot})
\to g_{\bcdot*}K^{\bullet}_{\Lambda'_{\bcdot}}(\iota_{\bcdot}^{\prime*}\nu_{\infty\bcdot*}'\ubM'_{\bcdot})
\quad\text{in}\;\;C^+((\fX_{\bcdot\Zar}^{\sim})^{\N^{\circ}},\uA_{\inf}),
\label{eq:BKFKoszulFunctMapHP}\\
K^{\bullet}_{\psi_{\bcdot}}(\ootau_{\mtg_{\bcdot},\CF_{\bcdot}})
\colon K_{\Lambda_{\bcdot}}^{\bullet}(\iota_{\bcdot*}\iota_{\bcdot}^*
\nu_{\infty\bcdot*}\ubM_{\bcdot})
\to g_{\bcdot\psi_{\bcdot}*}K^{\bullet}_{\Lambda'_{\bcdot}}(\iota'_{\bcdot*}\iota_{\bcdot}^{\prime*}
\nu_{\infty\bcdot*}'\ubM'_{\bcdot})
\;\;\text{in}\;C^+((\Gamma_{\Lambda_{\bcdot}}\hy\fX_{\bcdot\Zar}^{\sim})^{\N^{\circ}},\uA_{\inf}),
\label{eq:BKFKoszulFunctMapHP2}\\
\sigma_{\mtg_{\bcdot},\CF_{\bcdot}}\colon v_{\bcdot*}q\Omegab(\uCM_{\bcdot}, \utheta_{\uCM_{\bcdot}})
\longrightarrow
g_{\bcdot*}v'_{\bcdot*}(q\Omegab(\uCM'_{\bcdot}, \utheta_{\uCM'_{\bcdot}}))
\quad\text{in}\;\;C^+((\fX_{\bcdot\Zar}^{\sim})^{\N^{\circ}},\uA_{\inf}).
\label{eq:qdRCrystalFunctMapHP}
\end{gather}

By the constructions of \eqref{eq:ProetZarMorphFunctHP3} and 
the analogues of \eqref{eq:ProetZarMorphFunct1}
and the right square of \eqref{eq:ProetZarMorphFunct2} for $\fX_{\bcdot}$ and 
$\fX'_{\bcdot}$ mentioned above
with the help of Lemmas \ref{lem:DToposDIComp} and \ref{lem:DToposDIMorph}, 
the commutative diagram \eqref{eq:BKFGammaCohResolFunct} 
(resp.~\eqref{eq:BKFGammaCohFunct}, resp.~\eqref{eq:qdRKoszulGammaCohFunct})
associated to $\mtg_{[r]}=(g_{[r]},h_{[r]},\psi_{[r]})$ and 
$\CF_{[r]}$ for each $r\in \N$ implies that the pair
$(\beta_{\mti_{\bcdot},\CF_{\bcdot}},\beta_{\mti'_{\bcdot},\CF'_{\bcdot}})$
\eqref{eq:BKFKoszulResolHC}
(resp.~$(\alpha_{\mti_{\bcdot},\CF_{\bcdot}},\alpha_{\mti'_{\bcdot},\CF'_{\bcdot}})$
\eqref{eq:BKFKoszulReolGInvHP}, resp.~$(c_{\mti_{\bcdot},\CF_{\bcdot}},c_{\mti'_{\bcdot},\CF'_{\bcdot}})$
\eqref{eq:qdRKoszulMapHC})
is compatible with the pair 
$(\tau_{\mtg_{\bcdot},\CF_{\bcdot}},K^{\bullet}_{\psi_{\bcdot}}(\ootau_{\mtg_{\bcdot},\CF_{\bcdot}}))$
(resp.~$(K^{\bullet}_{\psi_{\bcdot}}(\otau_{\mtg_{\bcdot},\CF_{\bcdot}}),K^{\bullet}_{\psi_{\bcdot}}(\ootau_{\mtg_{\bcdot},\CF_{\bcdot}}))$, \linebreak
resp.~$(\sigma_{\mtg_{\bcdot},\CF_{\bcdot}},K^{\bullet}_{\psi_{\bcdot}}(\otau_{\mtg_{\bcdot},\CF_{\bcdot}})$)
similarly to \eqref{eq:BKFGammaCohResolFunct} (resp.~\eqref{eq:BKFGammaCohFunct}, 
resp.~\eqref{eq:qdRKoszulGammaCohFunct}). 

By using the compatibility for the pairs $(\beta_{\mti_{\bcdot},\CF_{\bcdot}}, \beta_{\mti'_{\bcdot},\CF'_{\bcdot}})$
and $(\alpha_{\mti_{\bcdot},\CF_{\bcdot}},\alpha_{\mti'_{\bcdot},\CF'_{\bcdot}})$ mentioned above
and the commutative diagram \eqref{eq:ProetGammaSheafProjFunct}, we see 
that the compositions \eqref{eq:SimplicialKoszulBKFMap} 
for $(\fX_{\bcdot},\bM_{\bcdot})$ and $(\fX'_{\bcdot},\bM_{\bcdot}')$ are compatible with 
$K^{\bullet}_{\psi_{\bcdot}}(\otau_{\mtg_{\bcdot},\CF_{\bcdot}})$ \eqref{eq:BKFKoszulFunctMapHP}
and $\uvarepsilon_{g_{\bcdot},\CF_{\bcdot}}$ \eqref{eq:BKFIndsysFunctMapHC}, i.e., the following 
diagram is commutative similarly to \eqref{eq:GammaKoszulBKFZarProjFunct}.
\begin{equation}
\xymatrix@R=20pt@C=15pt{
\varprojlim_{\N}K_{\Lambda_{\bcdot}}^{\bullet}(\iota_{\bcdot}^*\nu_{\infty\bcdot*}\ubM_{\bcdot})
\ar[d]_{\eqref{eq:SimplicialKoszulBKFMap}}
\ar[rr]^{\varprojliml_{\N}K^{\bullet}_{\psi_{\bcdot}}(\otau_{\mtg_{\bcdot},\CF_{\bcdot}})}&&
\varprojlim_{\N}g_{\bcdot*}K_{\Lambda'_{\bcdot}}^{\bullet}(\iota_{\bcdot}^{\prime*}
\nu'_{\infty\bcdot*}\ubM'_{\bcdot})\ar[r]^{\cong}
&
g_{\bcdot\Zar*}\varprojlim_{\N}K_{\Lambda'_{\bcdot}}^{\bullet}(\iota_{\bcdot}^{\prime*}
\nu'_{\infty\bcdot*}\ubM'_{\bcdot})\ar[d]^{\eqref{eq:SimplicialKoszulBKFMap}}\\
R\varprojlim_{\N}R\nu_{\bcdot*}\ubM_{\bcdot}
\ar[rr]^(.45){R\varprojliml_{\N}R\nu_{\bcdot*}(\uvarepsilon_{g_{\bcdot},\CF_{\bcdot}})}&&
R\varprojlim_{\N}R\nu_{\bcdot*}R\bmg_{\bcdot*}\ubM'_{\bcdot}
\ar[r]^{\cong}&
Rg_{\bcdot\Zar*}R\varprojlim_{\N}R\nu'_{\bcdot*}\ubM'_{\bcdot}
}
\end{equation}
Note that the relevant derived functors can be computed component-wise
by \eqref{eq:InvSysDtoposFBF} and
Proposition \ref{prop:DtoposInvLimitFbF}.
By Theorem \ref{thm:PrismCohQHiggsCpxGlb} (3) and the compatibility for the pair 
$(c_{\mti_{\bcdot},\CF_{\bcdot}},c_{\mti'_{\bcdot},\CF'_{\bcdot}})$ 
mentioned above, we obtain a commutative diagram
\begin{equation}\label{eq:CompMorphFunctHP}
\xymatrix@R=20pt@C=20pt{
Ru_{\fX_{\bcdot}/A_{\inf}*}\CF_{\bcdot}\ar[r]
\ar[d]^{\cong}_{\kappa_{\mti_{\bcdot},\CF_{\bcdot}}}&
Ru_{\fX_{\bcdot}/A_{\inf}*}Rg_{\bcdot\prism*}\CF'_{\bcdot}
\ar[r]^{\cong}&
Rg_{\bcdot\Zar*}Ru_{\fX'_{\bcdot}/A_{\inf}*}\CF'_{\bcdot}
\ar[d]^{\cong}_{Rg_{\bcdot\Zar*}(\kappa_{\mti'_{\bcdot},\CF'_{\bcdot}})}\\
L\eta_{\mu}R\varprojlim_{\N}R\nu_{\bcdot*}\ubM_{\bcdot}
\ar[r]&
L\eta_{\mu}R\varprojlim_{\N}R\nu_{\bcdot*}R\bmg_{\bcdot*}\ubM'_{\bcdot}
\ar[r]&
Rg_{\bcdot\Zar*}L\eta_{\mu}R\varprojlim_{\N}R\nu'_{\bcdot*}\ubM'_{\bcdot},
}
\end{equation}
where the bottom right morphism is given by 
$R\nu_{\bcdot*}R\bmg_{\bcdot*}\cong Rg_{\bcdot*}R\nu'_{\bcdot*}$,
$R\varprojlim_{\N}Rg_{\bcdot*}\cong Rg_{\bcdot\Zar*}R\varprojlim_{\N}$,
and $L\eta_{\mu}Rg_{\bcdot\Zar*}\to Rg_{\bcdot\Zar*}L\eta_{\mu}$
\eqref{eq:DeclageDirectImage}.

It remains to prove the following lemma.

\begin{lemma}\label{lem:CohDescentFunctoriality}
The following diagrams are commutative, where
the right vertical morphism in the second diagram is given by 
the composition of the bottom morphism in \eqref{eq:CompMorphFunctHP}.
\begin{equation}\label{eq:CohDescentPrismZarFunct}
\xymatrix@R=15pt@C=40pt{
Ru_{\fX/A_{\inf}*}\CF\ar[rr]^{\cong}_{\eqref{eq:PrismZarProjCohDescAinf}}
\ar[d]&&
R\theta_*Ru_{\fX_{\bcdot}/A_{\inf}*}\CF_{\bcdot}\ar[d]\\
Rg_{\Zar*}Ru_{\fX'/A_{\inf}*}\CF'\ar[r]^(.45){\cong}_(.45){\eqref{eq:PrismZarProjCohDescAinf}}&
Rg_{\Zar*}R\theta'_*Ru_{\fX'_{\bcdot}/A_{\inf}*}\CF'_{\bcdot}
\ar[r]^{\cong}&
R\theta_*Rg_{\bcdot\Zar*}Ru_{\fX'_{\bcdot}/A_{\inf}*}\CF'_{\bcdot}}
\end{equation}
\begin{equation}\label{eq:CohDescentAOmegaFunct}
\xymatrix@R=15pt{
L\eta_{\mu}R\varprojlim_{\N}R\nu_*\ubM
\ar[rr]^{\cong}_{\eqref{eq:AOmegaCohDesc}}\ar[d]_{\eqref{eq:AOmegaFunct}}
&&
R\theta_*L\eta_{\mu}R\varprojlim_{\N}R\nu_{\bcdot*}\ubM_{\bcdot}\ar[d]\\
Rg_{\Zar*}(L\eta_{\mu}R\varprojlim_{\N}R\nu'_*\ubM')
\ar[r]^(.47){\cong}_(.47){\eqref{eq:AOmegaCohDesc}}&
Rg_{\Zar*}R\theta'_*L\eta_{\mu}R\varprojlim_{\N}R\nu'_{\bcdot*}\ubM'_{\bcdot}
\ar[r]^{\cong}&
R\theta_*Rg_{\bcdot\Zar*}
L\eta_{\mu}R\varprojlim_{\N}R\nu'_{\bcdot*}\ubM'_{\bcdot}
}
\end{equation}
\end{lemma}

\begin{sublemma}\label{sublem:bcformalism}
We consider the following diagrams of categories such that,
for each pair of vertical functors, the left one is a left adjoint of the right one.
\begin{equation*}
\xymatrix@R=10pt@C=50pt{
C_1\ar@<-0.5ex>[d]_{\alpha^*}\ar[r]^{F_1}&
C_1'\ar@<-.5ex>[d]_{\alpha^{\prime*}}\ar[r]^{F_1'}&
C_1''\ar@<-.5ex>[d]_{\alpha^{\prime\prime*}}\\
C_2\ar@<-.5ex>[d]_{\beta^*}\ar[r]^{F_2}
\ar@<-.5ex>[u]_{\alpha_*}&
C_2'\ar@<-.5ex>[d]_{\beta^{\prime*}}\ar[r]^{F_2'}
\ar@<-.5ex>[u]_{\alpha'_*}&
C_2''
\ar@<-.5ex>[u]_{\alpha''_*}\\
C_3\ar[r]^{F_3}
\ar@<-.5ex>[u]_{\beta_*}&
C_3'
\ar@<-.5ex>[u]_{\beta'_*}&
}
\end{equation*}

(1) Giving a morphism $a^*\colon \alpha'^*F_1\to F_2\alpha^*$ is equivalent
to giving a morphism $a_*\colon F_1\alpha_*\to \alpha_*'F_2$. The correspondence
is given by $a_*\colon F_1\alpha_*\to\alpha'_*\alpha^{\prime*}F_1\alpha_*
\xrightarrow{\alpha'_*(a^*)\alpha_*}\alpha'_*F_2\alpha^*\alpha_*\to \alpha'_*F_2$
and $a^*\colon \alpha^{\prime*}F_1\to\alpha^{\prime*}F_1\alpha_*\alpha^*
\xrightarrow{\alpha^{\prime*}(a_*)\alpha^*}\alpha^{\prime*}\alpha'_*F_2\alpha^*
\to F_2\alpha^*$. 
When $a^*$ and $a_*$ correspond to each other as above, we say 
$a_*$ (resp.~$a^*$) is the right (resp.~left) adjoint of $a^*$ (resp.~$a_*$).
We can apply this correspondence also to the other two squares and
the two outer rectangles in the above diagram.\par
(2) Let $a^*\colon \alpha^{\prime*}F_1\to F_2\alpha^*$
and $b^*\colon \beta^{\prime*}F_2\to F_3\beta^*$ be
morphisms of functors, and let 
$a_*\colon F_1\alpha_*\to \alpha_*'F_2$
and $b_*\colon F_2\beta_*\to \beta'_*F_3$ be the
right adjoints of $a^*$ and $b^*$, respectively, in the sense of (1). 
We define the composition $c^*$ (resp.~$c_*$) of
$a^*$ and $b^*$ (resp.~$a_*$ and $b_*)$ by
$c^*\colon\beta^{\prime*}\alpha^{\prime*}F_1\xrightarrow{\beta^{\prime*}(a^*)}
\beta^{\prime*}F_2\alpha^*\xrightarrow{(b^*)\alpha^*}
F_3\beta^*\alpha^*$ (resp.~$c_*\colon F_1\alpha_*\beta_*\xrightarrow{(a_*)\beta_*}
\alpha'_*F_2\beta_*\xrightarrow{\alpha'_*(b_*)}\alpha'_*\beta'_*F_3$).
Then $c_*$ is the right adjoint of $c^*$ in the sense of (1) with respect to the vertical outer rectangle
in the diagram.\par
(3) Let $a^*\colon \alpha^{\prime*}F_1\to F_2\alpha^*$ and
$a^{\prime*}\colon \alpha^{\prime\prime*}F_1'\to F'_2\alpha^{\prime*}$
be morphisms of functors, and let 
$a_*\colon F_1\alpha_*\to \alpha_*'F_2$
and $a'_*\colon F'_1\alpha'_*\to \alpha''_*F'_2$ be
the right adjoints of $a^*$ and $a^{\prime*}$, respectively, in the sense of (1).
We define the composition $a^{\prime\prime*}$ (resp.~$a''_*$) of
$a^*$ and $a^{\prime*}$ (resp.~$a_*$ and $a'_*$) by 
$a^{\prime\prime*}\colon \alpha^{\prime\prime*}F'_1F_1\xrightarrow{(a^{\prime*})F_1}
F'_2\alpha^{\prime*}F_1\xrightarrow{F'_2(a^*)}
F'_2F_2\alpha^*$ 
(resp.~$a^{\prime\prime}_*\colon F'_1F_1\alpha_*
\xrightarrow{F'_1(a_*)}F'_1\alpha'_*F_2
\xrightarrow{(a'_*)F_2}\alpha''_*F_2'F_2$).
Then $a''_*$ is the right adjoint of $a^{\prime\prime*}$ in the sense of (1)
with respect to the horizontal outer rectangle in the diagram.\par
(4) Let $X_1$ and $X_2$ be objects of $C_1$ and $C_2$,  let
$f\colon X_1\to \alpha_*(X_2)$ be a morphism in $C_1$,
and let $g\colon \alpha^*(X_1)\to X_2$ be the left adjoint of $f$.
Let $a^*\colon \alpha'^*F_1\to F_2\alpha^*$ be a morphism, and
let $a_*\colon F_1\alpha_*\to \alpha_*'F_2$ be its right adjoint
in the sense of (1). We define the image of $f$ by $(F_1,a_*)$
(resp.~$g$ by $(a^*,F_2)$) to be the composition 
$f'\colon F_1(X_1)\xrightarrow{F_1(f)}F_1\alpha_*(X_2)
\xrightarrow{a_*(X_2)}\alpha'_*F_2(X_2)$
(resp.~$g'\colon \alpha^{\prime*}F_1(X_1)\xrightarrow{a^*(X_1)}
F_2\alpha^*(X_1)\xrightarrow{F_2(g)}F_2(X_2)$).
Then $g'$ is the left adjoint of $f'$. \par 
(5) Let $a_*$, $b_*$, and $c_*$ be the same as in (2).
Let $X_1$, $X_2$, and $X_3$ be objects of $C_1$, $C_2$, and $C_3$, respectively.
Let $f_1\colon X_1\to \alpha_*(X_2)$
and $f_2\colon X_2\to \beta_*(X_3)$ be morphisms in $C_1$ and $C_2$.
We define the composition $f_3$
of $f_1$ and $f_2$ to be 
$X_1\xrightarrow{f_1}\alpha_*(X_2)\xrightarrow{\alpha_*(f_2)}\alpha_*\beta_*(X_3)$.
Let $f'_1\colon F_1(X_1)\to \alpha'_*F_2(X_2)$
(resp.~$f'_2\colon F_2(X_2)\to \beta'_*F_3(X_3)$,
resp.~$f'_3\colon F_1(X_1)\to\alpha'_*\beta'_*F_3(X_3)$) be the 
image of $f_1$ (resp.~$f_2$, resp.~$f_3$) 
by $(F_1,a_*)$ (resp.~$(F_2,b_*)$, resp.~$(F_1,c_*)$) in the sense of (4).
Then the morphism $f_3'$ coincides with the composition
$F_1(X_1)\xrightarrow{f_1'}\alpha'_*F_2(X_2)
\xrightarrow{\alpha'_*(f'_2)}\alpha'_*\beta'_*F_3(X_3)$
of $f_1'$ and $f_2'$.
\end{sublemma}

\begin{proof} Straightforward. \end{proof}

\begin{proof}[Proof of Lemma \ref{lem:CohDescentFunctoriality}]
We use the terminology introduced in Sublemma \ref{sublem:bcformalism}.\par
Proof of \eqref{eq:CohDescentPrismZarFunct}: 
We abbreviate $u_{\fX/A_{\inf}}$, $u_{\fX_{\bcdot}/A_{\inf}}$,
$u_{\fX'/A_{\inf}}$, $u_{\fX'_{\bcdot}/A_{\inf}}$, $g_{\Zar}$, and $g_{\bcdot\Zar}$ to 
$u$, $u_{\bcdot}$, $u'$, $u'_{\bcdot}$, $g$, and $g_{\bcdot}$. 
We apply Sublemma \ref{sublem:bcformalism} to the 
following diagrams of derived categories, where we
omit the structure rings of ringed topos.
\begin{equation*}
\xymatrix@R=15pt@C=50pt{
D^+((\fX/A_{\inf})_{\prism})\ar[r]^(.55){Ru_*}
\ar@<-0.5ex>[d]_{L\theta_{\prism}^*}
&
D^+(\fX_{\Zar})
\ar@<-0.5ex>[d]_{L\theta^*}
\\
D^+((\fX_{\bcdot}/A_{\inf})_{\prism})\ar[r]^(.55){Ru_{\bcdot*}}
\ar@<-0.5ex>[u]_{R\theta_{\prism*}}
\ar@<-0.5ex>[d]_{Lg_{\bcdot\prism}^*}
&
D^+(\fX_{\bcdot\Zar})
\ar@<-0.5ex>[u]_{R\theta_*}
\ar@<-0.5ex>[d]_{Lg_{\bcdot}^*}
\\
D^+((\fX'_{\bcdot}/A_{\inf})_{\prism})\ar[r]^(.55){Ru'_{\bcdot*}}
\ar@<-0.5ex>[u]_{Rg_{\bcdot\prism*}}
&
D^+(\fX'_{\bcdot\Zar})
\ar@<-0.5ex>[u]_{Rg_{\bcdot*}}
}
\qquad
\xymatrix@R=15pt@C=50pt{
D^+((\fX/A_{\inf})_{\prism})\ar[r]^(.55){Ru_*}
\ar@<-0.5ex>[d]_{Lg_{\prism}^*}
&
D^+(\fX_{\Zar})
\ar@<-0.5ex>[d]_{Lg^*}
\\
D^+((\fX'/A_{\inf})_{\prism})\ar[r]^(.55){Ru'_*}
\ar@<-0.5ex>[u]_{Rg_{\prism*}}
\ar@<-0.5ex>[d]_{L\theta^{\prime*}_{\prism}}
&
D^+(\fX'_{\Zar})
\ar@<-0.5ex>[u]_{Rg_*}
\ar@<-0.5ex>[d]_{L\theta^{\prime*}}
\\
D^+((\fX'_{\bcdot}/A_{\inf})_{\prism})\ar[r]^(.55){Ru'_{\bcdot*}}
\ar@<-0.5ex>[u]_{R\theta'_{\prism*}}
&
D^+(\fX'_{\bcdot\Zar})
\ar@<-0.5ex>[u]_{R\theta'_*}
}
\end{equation*}
The base change morphism $a^*\colon L\theta^*Ru_*
\to Ru_{\bcdot*}L\theta_{\prism}^*$ of the upper left square
is the left adjoint of the canonical isomorphism 
$a_*\colon Ru_*R\theta_{\prism*}
\xrightarrow{\cong}R\theta_*Ru_{\bcdot*}$. 
Therefore,
by applying Sublemma \ref{sublem:bcformalism} (4) to 
the isomorphism $L\theta_{\prism}^*\CF\xrightarrow{\cong}\CF_{\bcdot}$
and its right adjoint $k\colon \CF\to R\theta_{\prism*}\CF_{\bcdot}$, 
we see that the isomorphism \eqref{eq:PrismZarProjCohDescAinf} coincides with the image of
$k$ by $(Ru_*,a_*)$: $Ru_*\CF\xrightarrow{Ru_*(k)} Ru_*
R\theta_{\prism*}\CF_{\bcdot}\xrightarrow[a_*(\CF_{\bcdot})]{\cong} R\theta_*Ru_{\bcdot*}\CF_{\bcdot}$.
Via the isomorphism $R\theta_{\prism*}Rg_{\bcdot\prism*}\cong 
Rg_{\prism*}R\theta'_{\prism*}$ and $R\theta_*Rg_{\bcdot*}\cong Rg_*R\theta'_*$,
the composition  of $a_*\colon Ru_*R\theta_{\prism*}
\xrightarrow{\cong}R\theta_*Ru_{\bcdot*}$
and $b_{\bcdot*}\colon Ru_{\bcdot*}Rg_{\bcdot\prism*}\xrightarrow{\cong}
Rg_{\bcdot*}Ru_{\bcdot*}'$ coincides with
that of $b_*\colon Ru_*Rg_{\prism*}\xrightarrow{\cong}Rg_*Ru'_*$
and $a'_*\colon Ru'_*R\theta'_{\prism*}\xrightarrow{\cong}R\theta'_*Ru'_{\bcdot*}$,
and the composition of 
$k\colon \CF\to R\theta_{\prism*}\CF_{\bcdot}$ and
$\ell_{\bcdot}\colon \CF_{\bcdot}\to Rg_{\bcdot\prism*}\CF'_{\bcdot}$
coincides with that of $\ell\colon \CF\to Rg_{\prism*}\CF'$
and $k'\colon \CF'\to R\theta_{\prism*}'\CF'_{\bcdot}$.
Therefore we see that the diagram \eqref{eq:CohDescentPrismZarFunct} is commutative 
by comparing the composition of the images
of $k$ and $\ell_{\bcdot}$ by $(Ru_*,a_*)$ and $(Ru_{\bcdot*},b_{\bcdot*})$
with the composition of the images of
$\ell$ and $k'$ by $(Ru_*,b_{*})$ and $(Ru'_*, a'_*)$
by using Sublemma \ref{sublem:bcformalism} (5).\par
Proof of \eqref{eq:CohDescentAOmegaFunct}: 
We apply Sublemma \ref{sublem:bcformalism} to the following diagrams of derived categories, 
where we omit the structure rings of ringed topos.
\begin{equation*}
\xymatrix@R=15pt@C=50pt{
D^+(X_{\proet}^{\N^{\circ}})\ar[r]^{R\nu_*}
\ar@<-0.5ex>[d]_{L\bm{\theta}^*}
&
D^+(\fX_{\Zar}^{\N^{\circ}})\ar[r]^{R\varprojlim_{\N}}
\ar@<-0.5ex>[d]_{L\theta^*}
&
D^+(\fX_{\Zar})\ar[r]^{L\eta_{\mu}}
\ar@<-0.5ex>[d]_{L\theta^*}
&
D(\fX_{\Zar})
\ar@<-0.5ex>[d]_{L\theta^*}\\
D^+(X_{\bcdot\proet}^{\N^{\circ}})\ar[r]^{R\nu_{\bcdot*}}
\ar@<-0.5ex>[u]_{R\bm{\theta}_*}
\ar@<-0.5ex>[d]_{L\bmg_{\bcdot}^*}
&
D^+(\fX_{\bcdot\Zar}^{\N^{\circ}})\ar[r]^{R\varprojlim_{\N}}
\ar@<-0.5ex>[u]_{R\theta_*}
\ar@<-0.5ex>[d]_{Lg_{\bcdot}^*}
&
D^+(\fX_{\bcdot\Zar})\ar[r]^{L\eta_{\mu}}
\ar@<-0.5ex>[u]_{R\theta_*}
\ar@<-0.5ex>[d]_{Lg_{\bcdot}^*}
&
D(\fX_{\bcdot\Zar})
\ar@<-0.5ex>[u]_{R\theta_*}
\ar@<-0.5ex>[d]_{Lg_{\bcdot}^*}\\
D^+(X_{\bcdot\proet}^{\prime\N^{\circ}})\ar[r]^{R\nu'_{\bcdot*}}
\ar@<-0.5ex>[u]_{R\bmg_{\bcdot*}}
&
D^+(\fX_{\bcdot\Zar}^{\prime\N^{\circ}})\ar[r]^{R\varprojlim_{\N}}
\ar@<-0.5ex>[u]_{Rg_{\bcdot*}}
&
D^+(\fX'_{\bcdot\Zar})\ar[r]^{L\eta_{\mu}}
\ar@<-0.5ex>[u]_{Rg_{\bcdot*}}
&
D(\fX'_{\bcdot\Zar})
\ar@<-0.5ex>[u]_{Rg_{\bcdot*}}\\
}
\end{equation*}
\begin{equation*}
\xymatrix@R=15pt@C=50pt{
D^+(X_{\proet}^{\N^{\circ}})\ar[r]^{R\nu_*}
\ar@<-0.5ex>[d]_{L\bmg^*}
&
D^+(\fX_{\Zar}^{\N^{\circ}})\ar[r]^{R\varprojlim_{\N}}
\ar@<-0.5ex>[d]_{Lg^*}
&
D^+(\fX_{\Zar})\ar[r]^{L\eta_{\mu}}
\ar@<-0.5ex>[d]_{Lg^*}
&
D(\fX_{\Zar})
\ar@<-0.5ex>[d]_{Lg^*}\\
D^+(X_{\proet}^{\prime\N^{\circ}})\ar[r]^{R\nu'_*}
\ar@<-0.5ex>[u]_{R\bmg_*}
\ar@<-0.5ex>[d]_{L\bm{\theta}^{\prime*}}
&
D^+(\fX_{\Zar}^{\prime\N^{\circ}})\ar[r]^{R\varprojlim_{\N}}
\ar@<-0.5ex>[u]_{Rg_*}
\ar@<-0.5ex>[d]_{L\theta^{\prime*}}
&
D^+(\fX'_{\Zar})\ar[r]^{L\eta_{\mu}}
\ar@<-0.5ex>[u]_{Rg_*}
\ar@<-0.5ex>[d]_{L\theta^{\prime*}}
&
D(\fX'_{\Zar})
\ar@<-0.5ex>[u]_{Rg_*}
\ar@<-0.5ex>[d]_{L\theta^{\prime*}}\\
D^+(X_{\bcdot\proet}^{\prime\N^{\circ}})\ar[r]^{R\nu'_{\bcdot*}}
\ar@<-0.5ex>[u]_{R\bm{\theta}'_*}
&
D^+(\fX_{\bcdot\Zar}^{\prime\N^{\circ}})\ar[r]^{R\varprojlim_{\N}}
\ar@<-0.5ex>[u]_{R\theta'_*}
&
D^+(\fX'_{\bcdot\Zar})\ar[r]^{L\eta_{\mu}}
\ar@<-0.5ex>[u]_{R\theta'_*}
&
D(\fX'_{\bcdot\Zar})
\ar@<-0.5ex>[u]_{R\theta'_*}\\
}
\end{equation*}
We define $A\Omega_{\fX}$, $A\Omega_{\fX_{\bcdot}}$, $A\Omega_{\fX'}$,
and $A\Omega_{\fX'_{\bcdot}}$ to be the compositions of the three
horizontal functors for $\fX$, $\fX_{\bcdot}$, $\fX'$, and $\fX'_{\bcdot}$, respectively.
The base change isomorphisms
$a^*\colon L\theta^*R\nu_*\xrightarrow{\cong}R\nu_{\bcdot*}L\bm{\theta}^*$
and $b^*\colon L\theta^*R\varprojlim_{\N}\xrightarrow{\cong} R\varprojlim_{\N} L\theta^*$
of the top left and middle squares are the left adjoints of 
the canonical isomorphisms
$a_*\colon R\nu_*R\bm{\theta}_*{\cong}R\theta_*R\nu_{\bcdot*}$
and $b_*\colon R\varprojlim_{\N}R\theta_*\xrightarrow{\cong}
R\theta_*R\varprojlim_{\N}$ (Sublemma \ref{sublem:bcformalism} (1)).
Let $c_*\colon L\eta_{\mu}R\theta_*
\to R\theta_*L\eta_{\mu}$ be the right adjoint of 
the isomorphism $c^*\colon L\theta^*L\eta_{\mu}
\xrightarrow{\cong}L\eta_{\mu}L\theta^*$. 
Then, by composing $(a_*,b_*,c_*)$ and $(a^*,b^*,c^*)$,
we obtain a morphism 
$d_*\colon A\Omega_{\fX}R\bm{\theta}_*
\to R\theta_*A\Omega_{\fX_{\bcdot}}$
and $d^*\colon L\theta^*A\Omega_{\fX}
\to A\Omega_{\fX_{\bcdot}}L\bm{\theta}^*$.
The latter is the left adjoint of the former
by Sublemma \ref{sublem:bcformalism} (3).
By applying Sublemma \ref{sublem:bcformalism} (4) to $L\bm{\theta}^*\ubM
\xrightarrow{\cong}\ubM_{\bcdot}$
and its adjoint $k\colon \ubM\to R\bm{\theta}_*\ubM_{\bcdot}$, we see that 
the isomorphism \eqref{eq:AOmegaCohDesc} coincides with the image of $k$ under
$(A\Omega_{\fX},d_*)$: $A\Omega_{\fX}(\ubM)
\xrightarrow{A\Omega_{\fX}(k)} A\Omega_{\fX}(R\bm{\theta}_*\ubM_{\bcdot})
\xrightarrow{d_*(\ubM_{\bcdot})} R\theta_*A\Omega_{\fX_{\bcdot}}(\ubM_{\bcdot})$.
By applying the same argument to the other three horizontal 
sequences of squares in the diagram above,
we obtain morphisms 
$e_{\bcdot*}\colon A\Omega_{\fX_{\bcdot}}R\bm{g}_{\bcdot*}\to Rg_{\bcdot*}A\Omega_{\fX'_{\bcdot}}$, 
$e_*\colon A\Omega_{\fX}R\bm{g}_*\to Rg_*A\Omega_{\fX'}$,
and $d'_*\colon A\Omega_{\fX'}R\bm{\theta}'_*
\to R\theta'_*A\Omega_{\fX'_{\bcdot}}$. 
The composition of $d_*$ and $e_{\bcdot*}$ coincides
with that of $e_*$ and $d'_*$ up to canonical isomorphisms of their
domains and codomains; 
one can verify it by showing the corresponding claims for
($R\nu_*$, $R\nu'_*$, $R\nu_{\bcdot*}$, $R\nu'_{\bcdot*}$), $R\varprojlim_{\N}$, and $L\eta_{\mu}$. 
Since the composition of $k$ and
$\ell_{\bcdot}\colon\ubM_{\bcdot}
\to R\bmg_{\bcdot*}\ubM'_{\bcdot}$
coincides with that of 
$\ell\colon \ubM\to R\bmg_*\ubM'$
and $k'\colon \ubM'\to R\bm{\theta}'_*\ubM'_{\bcdot}$
up to canonical isomorphism of their codomains,
we see that the diagram \eqref{eq:CohDescentAOmegaFunct} 
commutes by comparing 
the composition of the images of $k$ and $\ell_{\bcdot}$
by $(A\Omega_{\fX}, d_*)$ and $(A\Omega_{\fX_{\bcdot}},e_{\bcdot*})$
with that of the images of $\ell$ and $k'$ by 
$(A\Omega_{\fX}, e_*)$ and $(A\Omega_{\fX'},d_*')$
by using Sublemma \ref{sublem:bcformalism} (5).
\end{proof}

\begin{proposition}\label{prop:PrismAinfCohIsomFrobCompGlobal}
Let $\fX$ be a quasi-compact, separated, smooth $p$-adic formal scheme over $\CO$,
let $\CF$ be an object of $\Crystal^{\fproj}_{\prism}(\fX/A_{\inf})$, and
put $\CF_{\varphi}=\varphi^*\CF\in \Ob \Crystal_{\prism}^{\fproj}(\fX/A_{\inf})$
 (Remark \ref{rmk:PrismCrystalFrobTensor} (2)). We write $\bM$ and $\bM_{\varphi}$
 for $\bM_{\BKF,\fX}(\CF)$ and $\bM_{\BKF,\fX}(\CF_{\varphi})$, respectively.
Then the following diagram is commutative, where the
left vertical morphism is induced by $\CF\to \CF_{\varphi};x\mapsto x\otimes 1$.
\begin{equation}
\xymatrix@R=10pt@C=60pt{
Ru_{\fX/A_{\inf}*}\CF\ar[r]^{\sim}_{\eqref{eq:PrismCohAinfCohGlobComp}}\ar[d]&
A\Omega_{\fX}(\bM)\ar[d]^{\eqref{eq:AinfCohBKFFrobPB}}\\
\varphi_*Ru_{\fX/A_{\inf}*}\CF_{\varphi}
\ar[r]^{\sim}_{\eqref{eq:PrismCohAinfCohGlobComp}}&
\varphi_*A\Omega_{\fX}(\bM_{\varphi})
}
\end{equation}
\end{proposition}

\begin{proof}
Let $\fX_{\bcdot}$ be a Zariski hypercovering of $\fX$ by affine formal schemes, 
and let $\mti_{\bcdot}=(\mfi_{\bcdot}\colon \fX_{\bcdot}\to \fY_{\bcdot},\ut_{\bcdot},\Lambda_{\bcdot})$
be a simplicial small framed embedding of $\fX_{\bcdot}$ over $A_{\inf}$.
We follow the notation introduced in the construction of 
$\kappa_{\mti_{\bcdot},\CF}$ \eqref{eq:BKFPrismCohLocCompMapSimp}. 
We define $\ubM_{\varphi}$, $\ubM_{\varphi\bcdot}$,
$\uCM_{\varphi\bcdot}$, and $q\Omega^{\bullet}(\uCM_{\varphi\bcdot},\utheta_{\uCM_{\varphi\bcdot}})$
by using $\CF_{\varphi}$ instead of $\CF$.
By applying Remarks \ref{rmk:qdRInvSysFunctFrobComp} and 
\ref{rmk:BKFGammaKoszulFunctFrobPB} to $\mtg_{\alpha}$
$(\alpha\in \Mor \Delta)$, we see that the Frobenius pullback morphisms
\eqref{eq:BKFGammaShfFrob}, \eqref{eq:BKFGammaShfKosFrob2}, 
\eqref{eq:qHiggsInvSysFrobPB} for $\CF_{[r]}$ and $\mti_{[r]}$ for each $[r]\in \Ob \Delta$
define Frobenius pullback morphisms
\begin{gather*}
\nu_{\infty\bcdot*}\ubM_{\bcdot}\longrightarrow\varphi_*\nu_{\infty\bcdot*}\ubM_{\varphi\bcdot},
\qquad
K_{\Lambda_{\bcdot}}^{\bullet}(\iota_{\bcdot*}\iota_{\bcdot}^*\nu_{\infty\bcdot*}\ubM_{\bcdot})
\longrightarrow
\varphi_*K_{\Lambda_{\bcdot}}^{\bullet}(\iota_{\bcdot*}\iota_{\bcdot}^*\nu_{\infty\bcdot*}\ubM_{\varphi\bcdot})\\
v_{\bcdot*}(q\Omega^{\bullet}(\uCM_{\bcdot},\utheta_{\uCM_{\bcdot}}))
\longrightarrow 
\varphi_*v_{\bcdot*}(q\Omega^{\bullet}(\uCM_{\varphi\bcdot},\utheta_{\uCM_{\varphi\bcdot}})).
\end{gather*}
By \eqref{eq:qHiggsKoszFrobComp}, \eqref{eq:HiggsKoszulBKFFrobComp}, 
\eqref{eq:KoszGammaCohFrobComp}, and \eqref{eq:BKFGammaResolFrobComp}, 
we see that the morphisms 
\begin{align*}
&(\eta_{\mu}\varprojlim_{\N}\alpha_{\mti_{\bcdot},\CF_{\bcdot}})\circ c_{\mti_{\bcdot},\CF_{\bcdot}}
\colon \varprojlim_{\N} q\Omega^{\bullet}(\uCM_{\bcdot},\utheta_{\uCM_{\bcdot}})\longrightarrow 
\eta_{\mu}\varprojlim_{\N} \pi_{\bcdot*}
K^{\bullet}_{\Lambda_{\bcdot}}(\iota_{\bcdot*}\iota_{\bcdot}^*\nu_{\infty\bcdot*}\ubM_{\bcdot}),\\
&L\eta_{\mu}R\varprojlim_{\N}R\pi_{\bcdot*}(\beta_{\mti_{\bcdot},\CF_{\bcdot}})\colon 
L\eta_{\mu}R\varprojlim_{\N} R\pi_{\bcdot*}\nu_{\infty\bcdot*}\ubM_{\bcdot}
\xrightarrow{\;\;\sim\;\;}
L\eta_{\mu}R\varprojlim_{\N} R\pi_{\bcdot*}K^{\bullet}_{\Lambda_{\bcdot}}
(\iota_{\bcdot*}\iota_{\bcdot}^*\nu_{\infty\bcdot*}\ubM_{\bcdot})
\end{align*}
induced by $c_{\mti_{\bcdot},\CF_{\bcdot}}$ \eqref{eq:qdRKoszulMapHC}, 
$\alpha_{\mti_{\bcdot},\CF_{\bcdot}}$\eqref{eq:BKFKoszulReolGInvHP}, 
and $\beta_{\mti_{\bcdot},\CF_{\bcdot}}$ \eqref{eq:BKFKoszulResolHC}
and the corresponding ones for $\CF_{\varphi}$
are compatible with the Frobenius pullbacks above.
We obtain the claim by combining the two 
compatibility and Theorem \ref{thm:PrismCohQHiggsCpxGlb} (2), and by noting that the morphism
$L\eta_{\mu}R\varprojlim_{\N}R\pi_{\bcdot*}\nu_{\infty\bcdot*}
\ubM_{\bcdot}\to L\eta_{\mu}R\varprojlim_{\N}R\nu_{\bcdot*}
\ubM_{\bcdot}$ and the corresponding one for $\CF_{\nu}$ are compatible 
with the Frobenius pullbacks.
\end{proof}

\begin{proposition}\label{prop:PrismAinfCohIsomProdCompGlobal}
Let $\fX$ be a quasi-compact, separated, smooth, $p$-adic formal scheme over $\CO$,
 let $\CF_{\nu}$ $(\nu\in \{1,2\})$ be objects of
 $\Crystal_{\prism}^{\fproj}(\fX/A_{\inf})$, and put
 $\CF_3=\CF_1\otimes_{\CO_{\fX/A_{\inf}}}\CF_2
 \in\Ob(\Crystal_{\prism}^{\fproj}(\fX/A_{\inf}))$
 (Remark \ref{rmk:PrismCrystalFrobTensor} (3)).
Put $\bM_{\nu}=\bM_{\BKF,\fX}(\CF_{\nu})$  (Definition \ref{BKFFunctor})
for $\nu\in \{1,2,3\}$. Then the following
diagram in $D(\fX_{\Zar},A_{\inf})$ is commutative.
\begin{equation}
\xymatrix@R=10pt@C=60pt{
Ru_{\fX/A_{\inf}*}\CF_1\otimes^L_{A_{\inf}}Ru_{\fX/A_{\inf}*}\CF_2
\ar[r]^{\sim}_{\eqref{eq:PrismCohAinfCohGlobComp}}
\ar[d]
&
A\Omega_{\fX}(\bM_1)\otimes^L_{A_{\inf}}A\Omega_{\fX}(\bM_2)
\ar[d]^{\eqref{eq:BKFAinfCohProd}}\\
Ru_{\fX/A_{\inf}*}\CF_3
\ar[r]^{\sim}_{\eqref{eq:PrismCohAinfCohGlobComp}}
&
A\Omega_{\fX}(\bM_3)
}
\end{equation}
\end{proposition}

\begin{proof}
When $\fX$ admits a small framed embedding $(\mti\colon \fX\to \fY,\ut,\Lambda)$ over $A_{\inf}$
(Definition \ref{def:AdmFramedSmEmbed} (1)), the compatibility of the comparison morphism $\kappa_{\mti,\CF}$ 
\eqref{eq:PrismCohAinfCohCompMap2} with
the products \eqref{eq:PrismAinfCohLocCompFrProd} 
is  verified by using the product morphisms 
\eqref{eq:qHiggsInvSysProd}, 
\eqref{eq:qHiggsInvSysKoszProd}, 
\eqref{eq:BKFInvSySKoszProd1}, and 
\eqref{eq:BKFInvSySKoszProd2} for 
$q\Omega^{\bullet}(\uCM_{\nu},\utheta_{\uCM_{\nu}})$, 
$K_{\Lambda}^{\bullet}(v_*\uCM_{\nu})$,
$K_{\Lambda}^{\bullet}(\iota^*\nu_{\infty*}\ubM_{\nu})$,
and $K_{\Lambda}^{\bullet}(\iota_*\iota^*\nu_{\infty*}\ubM_{\nu})$.
These product morphisms are not compatible with the pullback morphisms
$\sigma_{\mtg,\CF_{\nu}}$ \eqref{eq:qdRInvSysFunct}, 
$K_{\psi}^{\bullet}(\sigma_{\mtg,\CF_{\nu}}^0)$ \eqref{eq:KoszulqdRFunctComp}, 
 $K_{\psi}^{\bullet}(\otau_{\mtg,\CF_{\nu}})$ \eqref{eq:BKFKoszulFunctMap}, 
and $K_{\psi}^{\bullet}(\ootau_{\mtg,\CF_{\nu}})$ \eqref{eq:BKFKoszulFunctMap2} 
with respect to  a morphism $\mtg$ of small
framed embeddings over $A_{\inf}$ (Definition \ref{def:AdmFramedSmEmbed} (2))
unless the map between the index sets of coordinates is injective. Therefore
we cannot extend the argument in the proof of Proposition 
\ref{prop:PrismAinfCohLocMapProdComp} 
to the simplicial setting. Similarly to Theorem \ref{thm:PrismCohQHiggsCpxGlb} (4), 
we can solve this problem, thanks to \cite[9.24]{TsujiPrismQHiggs} and 
Remark \ref{rmk:KoszulCpxProdFunct} (2), 
by taking products after pulling back to the product of two copies of the chosen pair of a
Zariski hypercovering of $\fX$ and its simplicial small framed embedding over 
$A_{\inf}$ as follows.\par
Let $\fX_{\bcdot}$ be a Zariski hypercovering of $\fX$ by affine formal schemes, 
and let $\mti_{\bcdot}=(\mfi_{\bcdot}\colon \fX_{\bcdot}\to \fY_{\bcdot},\ut_{\bcdot},\Lambda_{\bcdot})$
be a simplicial small framed embedding of $\fX_{\bcdot}$ over $A_{\inf}$.
Let $\fX_{\bcdot}^{\pone}$ be the product of two copies of the simplicial formal scheme
$\fX_{\bcdot}$ over $\fX$, and let $\mti^{\pone}_{\bcdot}=(\mfi^{\pone}_{\bcdot}\colon 
\fX^{\pone}_{\bcdot}\to \fY^{\pone}_{\bcdot},\ut^{\pone}_{\bcdot},\Lambda^{\pone}_{\bcdot})$
be the simplicial small framed embedding of $\fX_{\bcdot}^{\pone}$ over $A_{\inf}$
obtained by taking the product of two copies of the simplicial formal scheme
$\fY_{\bcdot}^{\pone}$ over $A_{\inf}$ as before Proposition 
\ref{eq:BKFPrismLocCohCompSimpFunct}. 
For a morphism $\alpha\colon [r]\to [s]$ in $\Delta$,
we write $\mtg_{\alpha}=(g_{\alpha},h_{\alpha},\psi_{\alpha})$ 
(resp.~$\mtg_{\alpha}^{\pone}=(g_{\alpha}^{\pone},h_{\alpha}^{\pone},\psi_{\alpha}^{\pone})$) for
the corresponding morphism 
$\mti_{[s]}\to \mti_{[r]}$ (resp.~$\mti_{[s]}^{\pone}\to \mti_{[r]}^{\pone}$)
of small framed embeddings over $A_{\inf}$.
For $\nu\in \{1,2\}$, let
$\mtp_{\nu\bcdot}=(p_{\nu\bcdot},p_{\fY,\nu\bcdot},\chi_{\nu\bcdot})$ denote the $\nu$th projection $\mti^{\pone}_{\bcdot}\to \mti_{\bcdot}$.
We have morphisms of ringed topos \eqref{eq:ProetZarProjHP0}
and \eqref{eq:ProetZarProjHP}, and a functor \eqref{eq:ProetGammaModProjHP};
we write the letters with superscript (1) for the corresponding
ones for $\mti^{\pone}_{\bcdot}$. We define $\ubM_{\nu}$, $\ubM_{\nu\bcdot}$,
$\uCM_{\nu\bcdot}$, and $q\Omega^{\bullet}(\uCM_{\nu\bcdot},\utheta_{\uCM_{\nu\bcdot}})$
as in the construction of $\kappa_{\mti_{\bcdot},\CF}$ 
\eqref{eq:BKFPrismCohLocCompMapSimp}
for $\CF=\CF_{\nu}$ $(\nu\in \{1,2,3\})$, abbreviate the last one to
$q\Omega^{\bullet}(\uCM_{\nu\bcdot})$, and write them with superscript (1) when we work with 
$\mti^{\pone}_{\bcdot}$ instead of $\mti_{\bcdot}$.

For each $r\in \N$, we obtain the following morphism by composing
the tensor product of $\sigma_{\mtp_{\nu[r]},\CF_{\nu[r]}}$ $(\nu\in \{1,2\})$
\eqref{eq:qdRInvSysFunct} and the product \eqref{eq:qHiggsInvSysProd} 
for $\CF_{\nu[r]}$ $(\nu\in \{1,2\})$ and $\mti_{[r]}^{\pone}$.
\begin{equation*}
v_{[r]*}q\Omega^{\bullet}(\uCM_{1[r]})\otimes_{\uA_{\inf}} v_{[r]*}q\Omega^{\bullet}(\uCM_{2[r]})\longrightarrow
v_{[r]*}^{\pone}q\Omega^{\bullet}(\uCM^{\pone}_{3[r]})
\end{equation*}
For a morphism $\alpha\colon [r]\to [s]$ in $\Delta$,
the products above 
for $r$ and $s$ are compatible
with the pullbacks $\sigma_{\mtg_{\alpha},\CF_{\nu[r]}}$ $(\nu\in \{1,2\})$
and $\sigma_{\mtg^{\pone}_{\alpha},\CF_{3[r]}}$ \eqref{eq:qdRInvSysFunct}
by \cite[14.21]{TsujiPrismQHiggs}. Thus we obtain a morphism
\begin{equation}\label{eq:qdRCompPBProdSimp}
v_{\bcdot*}q\Omega^{\bullet}(\uCM_{1\bcdot})\otimes_{\uA_{\inf}}
v_{\bcdot*}q\Omega^{\bullet}(\uCM_{2\bcdot})\longrightarrow
v_{\bcdot*}^{\pone}q\Omega^{\bullet}(\uCM^{\pone}_{3\bcdot})
\quad\text{in}\;\;
C^+((\fX_{\bcdot\Zar}^{\sim})^{\N^{\circ}},\uA_{\inf}).
\end{equation}

For each $r\in \N$, we obtain a morphism
\begin{equation*}
p_{\Gamma,1[r]}^*\nu_{\infty[r]*}\ubM_{1[r]}\otimes_{\uA_{\inf}}
p_{\Gamma,2[r]}^*\nu_{\infty[r]*}\ubM_{2[r]}
\longrightarrow \nu^{\pone}_{\infty[r]*}\ubM_{3[r]}
\end{equation*}
by composing the tensor product of the left adjoints of 
$\tau_{\mtp_{\nu[r]},\CF_{\nu[r]}}$ $(\nu\in \{1,2\})$
\eqref{eq:BKFGammaRepFunct} with the product 
\eqref{eq:BKFGammaInvSysTensor} for $\CF_{\nu[r]}$ and $\mti_{[r]}^{\pone}$.
Here $p_{\Gamma,\nu[r]}$ denotes the morphism of ringed
topos $((\Gamma_{\Lambda^{\pone}_{[r]}}$-$\fX_{[r]\Zar}^{\sim})^{\N^{\circ}}, \uA_{\inf})
\to ((\Gamma_{\Lambda_{[r]}}$-$\fX_{[r]\Zar}^{\sim})^{\N^{\circ}},\uA_{\inf})$ induced by 
$\mtp_{\nu[r]}$. For a morphism $\alpha\colon [r]\to [s]$ in $\Delta$,
the products above 
for $r$ and $s$ are
compatible with $\tau_{\mtg_{\alpha},\CF_{\nu[r]}}$ $(\nu\in \{1,2\})$
and $\tau_{\mtg^{\pone}_{\alpha},\CF_{3[r]}}$ \eqref{eq:BKFGammaRepFunct}; 
by Remark
\ref{rmk:BKFGammaKoszulFuncCocyc} for $\tau_{\mtg,\CF}$,
the claim is reduced to the compatibility of the product morphisms for 
$\nu^{\pone}_{\infty[\ell]*}\ubM_{\nu[\ell]}$ $(\nu\in \{1,2,3\}, \ell\in\{r,s\})$
with $\tau_{\mtg_{\alpha}^{\pone},\CF_{\nu[r]}}$ \eqref{eq:BKFGammaRepFunct}, which
follows from \eqref{eq:BKFFunctorTensorFunct} 
and the fact that 
the morphism of functors $\Xi_{\mtg^{\pone}_{\alpha}}$
used in the construction of $\tau_{\mtg^{\pone}_{\alpha},\CF_{\nu[r]}}$
is compatible with the lax monoidal structures.
Hence we obtain a morphism 
\begin{equation}
\label{eq:BKFGammaModPBProdSimp}
p_{\Gamma,1\bcdot}^*\nu_{\infty\bcdot*}\ubM_{1\bcdot}\otimes_{\uA_{\inf}}
p_{\Gamma,2\bcdot}^*\nu_{\infty\bcdot*}\ubM_{2\bcdot}
\longrightarrow \nu^{\pone}_{\infty\bcdot*}\ubM_{3\bcdot}
\quad\text{in}\quad 
C^+((\Gamma_{\Lambda_{\bcdot}^{\pone}}\hy\fX_{\bcdot\Zar}^{\sim})^{\N^{\circ}},\uA_{\inf})
\end{equation}

For each $r\in \N$, we obtain a morphism
\begin{equation*}
p_{\Gamma,1[r]}^*K_{\Lambda_{[r]}}^{\bullet}(\iota_{[r]*}\iota_{[r]}^*\nu_{\infty[r]*}\ubM_{1[r]})
\otimes_{\uA_{\inf}}
p_{\Gamma,2[r]}^*K_{\Lambda_{[r]}}^{\bullet}(\iota_{[r]*}\iota_{[r]}^*\nu_{\infty[r]*}\ubM_{2[r]})
\longrightarrow K_{\Lambda^{\pone}_{[r]}}^{\bullet}(\iota^{\pone}_{[r]*}\iota^{\pone*}_{[r]}\nu_{\infty[r]*}^{\pone}\ubM_{3[r]})
\end{equation*}
by composing the left adjoints of $K^{\bullet}_{\chi_{\nu[r]}}(\ootau_{\mtp_{\nu[r]},\CF_{\nu[r]}})$
$(\nu\in \{1,2\})$ \eqref{eq:BKFKoszulFunctMap2}
 and the product \eqref{eq:BKFInvSySKoszProd2} 
for $\CF_{\nu[r]}$ and $\mti_{[r]}^{\pone}$.
For a morphism $\alpha\colon [r]\to [s]$ in $\Delta$, we see that
the products above 
 for $r$ and $s$ are compatible  with
the pullbacks by $\mtg_{\alpha}$ and $\mtg^{\pone}_{\alpha}$ as follows. 
By Remark \ref{rmk:BKFGammaKoszulFuncCocyc} 
and $\mtg_{\alpha}\circ \mtp_{\nu[s]}=\mtp_{\nu[r]}\circ \mtg_{\alpha}^{\pone}$
$(\nu\in \{1,2\})$, we can apply Remark \ref{rmk:KoszulCpxProdFunct} (2) to
the left adjoints of $\ootau_{\mtp_{\nu[\ell]},\CF_{\nu[\ell]}}$, 
$\ootau_{\mtg_{\alpha},\CF_{\nu[r]}}$, and 
$\ootau_{\mtg^{\pone}_{\alpha},\CF_{\nu[r]}}$ $(\nu\in \{1,2\}, \ell\in \{r,s\})$.
By the compatibility of the product morphisms for
$\nu_{\infty[\ell]*}^{\pone}\ubM_{\nu[\ell]}$ $(\nu\in \{1,2,3\}, \ell\in \{r,s\})$ with 
$\tau_{\mtg_{\alpha}^{\pone}, \CF_{\nu[r]}}$ discussed in
the construction of \eqref{eq:BKFGammaModPBProdSimp}, the commutative diagram
\eqref{eq:GammaSheafProdFunct3} in Remark \ref{rmk:KoszulResolProd} 
implies the compatibility of $\ootau_{\mtg^{\pone}_{\alpha},\CF_{\nu[r]}}$ 
with the product morphisms 
for $\iota_{[\ell]*}^{\pone}\iota_{[\ell]}^{\pone*}\nu_{\infty[\ell]*}^{\pone}\ubM_{\nu[\ell]}$.
We obtain the desired claim by combining the two.  Thus we obtain a morphism
\begin{equation}\label{eq:KoszulResolProdSimp}
p_{\Gamma,1\bcdot}^*K_{\Lambda_{\bcdot}}^{\bullet}(\iota_{\bcdot*}\iota_{\bcdot}^*\nu_{\infty\bcdot*}\ubM_{1\bcdot})
\otimes_{\uA_{\inf}}
p_{\Gamma,2\bcdot}^*K_{\Lambda_{\bcdot}}^{\bullet}(\iota_{\bcdot*}\iota_{\bcdot}^*\nu_{\infty\bcdot*}\ubM_{2\bcdot})
\longrightarrow K_{\Lambda^{\pone}_{\bcdot}}^{\bullet}(\iota^{\pone}_{\bcdot*}\iota^{\pone*}_{\bcdot}\nu_{\infty\bcdot*}^{\pone}\ubM_{3\bcdot})
\end{equation}
in $C^+((\Gamma_{\Lambda_{\bcdot}^{\pone}}\hy\fX_{\bcdot\Zar}^{\sim})^{\N^{\circ}},\uA_{\inf})$.

By applying \eqref{eq:KoszulqdRFunctComp}, \eqref{eq:qdRBKFMorphKoszulFunct}, 
\eqref{eq:BKFGammaCohFunct}, and \eqref{eq:BKFGammaCohResolFunct} to $\mtg_{\alpha}$
and $\mtg_{\alpha}^{\pone}$ for $\alpha\in \Mor \Delta$, we obtain 
morphisms of complexes of $\uA_{\inf}$-modules
\begin{gather*}
v_{\bcdot*}q\Omega^{\bullet}(\uCM_{\nu\bcdot})
\xrightarrow{\gamma_{\mti_{\bcdot},\CF_{\nu\bcdot}}}
K_{\Lambda_{\bcdot}}^{\bullet}(v_{\bcdot*}\uCM_{\nu\bcdot})
\xrightarrow{K_{\Lambda_{\bcdot}}^{\bullet}(\delta_{\mti_{\bcdot},\CF_{\nu\bcdot}})}
K_{\Lambda_{\bcdot}}^{\bullet}(\iota_{\bcdot}^*\nu_{\infty\bcdot*}\ubM_{\nu\bcdot})
\xrightarrow{\alpha_{\mti_{\bcdot},\CF_{\nu\bcdot}}}
\pi_{\bcdot*}K_{\Lambda_{\bcdot}}^{\bullet}(\iota_{\bcdot*}\iota^*_{\bcdot}\nu_{\infty\bcdot*}\ubM_{\nu\bcdot}),\\
\nu_{\infty\bcdot*}\ubM_{\nu\bcdot}
\xrightarrow{\beta_{\mti_{\bcdot},\CF_{\nu\bcdot}}}
K_{\Lambda_{\bcdot}}^{\bullet}(\iota_{\bcdot*}\iota_{\bcdot}^*\nu_{\infty\bcdot*}\ubM_{\nu\bcdot})
\end{gather*}
for $\nu\in \{1,2\}$ on $(\fX_{\bcdot\Zar}^{\sim})^{\N^{\circ}}$
and $(\Gamma_{\Lambda_{\bcdot}}\hy\fX_{\bcdot\Zar}^{\sim})^{\N^{\circ}}$, 
and the corresponding ones for $\mti_{\bcdot}^{\pone}$ and $\CF_3$.
These are compatible with the products \eqref{eq:qdRCompPBProdSimp}, 
\eqref{eq:BKFGammaModPBProdSimp}, and \eqref{eq:KoszulResolProdSimp}
by \eqref{eq:KoszulqdRFunctComp} etc.~above applied to $\mtp_{\nu[r]}$ $(\nu\in \{1,2\}, r\in \N)$, and
\eqref{eq:qDRKoszProdComp}, \eqref{eq:qHigCpxKoszProdComp}, 
\eqref{eq:KoszulGammaCohProjFunct}, and \eqref{eq:BKFGammaModResolFunct};
for the codomain of the first composition, we compose the image of 
\eqref{eq:KoszulResolProdSimp}
under $\pi_{\bcdot*}^{\pone}$ with $\pi_{\bcdot*}\to \pi_{\bcdot*}p_{\Gamma,\nu\bcdot*}
p_{\Gamma,\nu\bcdot}^*\cong\pi_{\bcdot*}^{\pone}p_{\Gamma,\nu\bcdot}^*$.
By taking $\varprojlim_{\N}$ of the first compatibility, we see that 
the morphism $\eta_{\mu}\varprojlim_{\N}(\alpha_{\mti_{\bcdot},\CF_{\nu\bcdot}})\circ
c_{\mti_{\bcdot},\CF_{\nu\bcdot}}$ $(\nu\in\{1,2\})$ \eqref{eq:qdRKoszulMapHC} and the
corresponding one for $\mti^{\pone}_{\bcdot}$ and $\CF_{3\bcdot}$ are compatible
with \eqref{eq:qdRCompPBProdSimp} and 
\eqref{eq:KoszulResolProdSimp}. 
Combining this with the second compatibility,
Theorem \ref{thm:PrismCohQHiggsCpxGlb}  (4), and a commutative diagram
\begin{equation}\label{eq:BKFZarProjSimpProdComp}
\xymatrix@R=15pt{
R\pi_{\bcdot*}\nu_{\infty\bcdot*}\ubM_{1\bcdot}\otimes_{\uA_{\inf}}
R\pi_{\bcdot*}\nu_{\infty\bcdot*}\ubM_{2\bcdot}\ar[rr]\ar[d]
&&
R\nu_{\bcdot*}\ubM_{1\bcdot}\otimes_{\uA_{\inf}}
R\nu_{\bcdot*}\ubM_{2\bcdot}\ar[d]
\\
R\pi^{\pone}_{\bcdot*}p_{\Lambda_{\bcdot},1}^*\nu_{\infty\bcdot*}\ubM_{1\bcdot}\otimes_{\uA_{\inf}}
R\pi^{\pone}_{\bcdot*}p_{\Lambda_{\bcdot},2}^*\nu_{\infty\bcdot*}\ubM_{2\bcdot}
\ar[r]^(.7){\eqref{eq:BKFGammaModPBProdSimp}}
& R\pi_{\bcdot*}^{\pone}\nu_{\infty\bcdot*}^{\pone}\ubM_{3\bcdot}\ar[r]&
R\nu_{\bcdot*}\ubM_{3\bcdot}
}
\end{equation}
which we prove below, we see that 
$\kappa_{\mti_{\bcdot},\CF_{\nu\bcdot}}$ $(\nu\in \{1,2\})$ and
$\kappa_{\mti_{\bcdot}^{\pone},\CF_3}$ \eqref{eq:SimplicialCompMap} are
compatible with the products of 
$Ru_{\fX_{\bcdot}/A_{\inf}}\CF_{\nu\bcdot}$ and
$L\eta_{\mu}R\varprojlim_{\N}R\nu_{\bcdot*}\ubM_{\nu}$
(defined in the same way as Remark \ref{rmk:FrobProdAOmega} (2)). Since the pullback isomorphisms
\eqref{eq:AOmegaCohSimpPB} and \eqref{eq:PrismZarProjCohSimpPBAinf} by 
$\theta$ are compatible with the products,
we obtain the desired compatibility by taking $R\theta_*$.\par
It remains to prove \eqref{eq:BKFZarProjSimpProdComp}. 
We see that the morphisms $R\pi_{\bcdot*}^{\pone}\nu_{\infty\bcdot*}^{\pone}\ubM_{\nu}
\to R\nu_{\bcdot*}\ubM_{\nu}$ $(\nu\in \{1,2,3\})$ are compatible
with the products similarly to the argument before
\eqref{eq:BKFZarProjProdFunct}. Therefore, by the construction of 
\eqref{eq:BKFGammaModPBProdSimp}, the claim is reduced to showing that the morphism 
$R\pi_{\bcdot*}\nu_{\infty\bcdot*}\ubM_{\nu\bcdot}
\to R\pi^{\pone}_{\bcdot*}\nu_{\infty\bcdot*}^{\pone}\ubM_{\nu\bcdot}$
induced by $\tau_{\mtp_{\nu\bcdot},\CF_{\nu\bcdot}}$
\eqref{eq:BKFIndGammaFunctMapHC} is compatible with the morphisms from its
domain and codomain to
$R\pi_{\bcdot*}R\nu_{\infty\bcdot*}\ubM_{\nu\bcdot}
\cong R\nu_{\bcdot*}\ubM_{\nu\bcdot}
\cong R\pi_{\bcdot*}^{\pone}R\nu_{\infty\bcdot*}^{\pone}\ubM_{\nu\bcdot}$
\eqref{eq:SimplcialToposMorphComp}, \eqref{eq:InvSysDtoposFBF}.
This follows from the fact that 
$\pi_{\bcdot*}(\Xi_{\mtp_{\nu\bcdot}})\colon
\pi_{\bcdot*}\nu_{\infty\bcdot*}
\to \pi_{\bcdot*}p_{\Gamma,\nu\bcdot}\nu^{\pone}_{\infty\bcdot*}
\cong \pi_{\bcdot*}^{\pone}\nu_{\infty\bcdot*}^{\pone}$
\eqref{eq:ProetZarMorphFunctHP3} coincides with 
$\pi_{\bcdot*}\nu_{\infty\bcdot*}\cong\nu_{\bcdot*}\cong
\pi^{\pone}_{\bcdot*}\nu^{\pone}_{\infty\bcdot*}$
\eqref{eq:SimplcialToposMorphComp}
by \eqref{eq:ProetGammaSheafProjFunct}.
\end{proof}

\section{$D$-topos and topos of inverse systems}\label{sec:DTopos}
In this section, we study direct image functors not necessarily cartesian
for families of topos over a category $D$.
The cartesian case is studied in \cite[V$^{\text{bis}}$ \S1]{SGA4}, and we verify 
that the cartesian condition is not necessary for some claims, in particular, 
for the fiber by fiber computation of derived direct images
\eqref{eq:DtoposFBF}. We apply the last result to topos of inverse systems.
See \eqref{eq:InvSySFBF}, \eqref{eq:InvSysDtoposFBF}, 
and Proposition \ref{prop:DtoposInvLimitFbF}.

Let $D$ be a $\bU$-small category. A {\it $D$-topos} is a
fibered and cofibered category (\cite[VI]{SGA1}) $\pi\colon E\to D$ over $D$
satisfying the following two conditions (\cite[V\bis D\'efinition (1.2.1)]{SGA4}).\par
(a) For every $i\in \Ob D$, the fiber $E_i$ of $\pi$ over $i$ is a $\bU$-topos.\par
(b) For every morphism $m\colon i\to j$ in $D$, there exists a morphism
of topos $f_m=(f_{m}^*,f_{m*})\colon E_j\to E_i$ such that 
$f_{m*}=m^*$ and $f_m^*=m_*$. \par
We define $f_m^*$ and $f_{m*}$ to be the identity functor of $E_i$ if $m=\id_i$ for $i\in \Ob D$. 

\begin{remark}\label{rmk:DtoposConstruction}
Let $D$ be a $\bU$-small category. 
Suppose that we are given a topos $E_i$ for each $i\in \Ob D$,
a morphism of topos $f_m\colon E_j\to E_i$ for each 
$m\colon i\to j\in \Mor D$, and an isomorphism 
$c_{n,m}\colon f_{m*}\circ f_{n*}\cong f_{nm*}$ for
each $i\xrightarrow{m}j\xrightarrow{n} k$ in $D$
satisfying $c_{n,\id}=\id_{f_{n*}}$, $c_{\id,m}=\id_{f_{m*}}$,
and $c_{l,nm}\circ c_{n,m}f_{l*}
=c_{ln,m}\circ f_{m*}c_{l,n}$
for every $i\xrightarrow{m}j\xrightarrow{n}k
\xrightarrow{l}h$ in $D$.
Then we can define a $D$-topos $\pi\colon E\to D$
whose fiber over $i\in \Ob D$ is $E_i$
by setting $\Hom_{E,m}(\CF,\CG)=\Hom_{E_i}(\CF,f_{m*}\CG)$
for $m\colon i\to j\in \Mor D$,
$\CF\in \Ob E_i$ and $\CG\in \Ob E_j$, where $\Hom_{E,m}$
means the set of morphisms whose images under $\pi$ are $m$,
 and defining the
composition of $\alpha\colon \CF\to f_{m*}\CG$ and
$\beta\colon \CG\to f_{n*}\CH$ for 
$i\xrightarrow{m} j\xrightarrow{n} k$ in $D$ by 
$\CF\xrightarrow{\alpha}f_{m*}\CG
\xrightarrow{f_{m*}\beta}
f_{m*}f_{n*}\CH\xrightarrow[c_{n,m}(\CH)]
{\cong}f_{nm*}\CH$  (\cite[VI \S7]{SGA1}).
\end{remark}

Let $\uGamma(E)$ denote the category of functors from $D$
to $E$ over $D$, i.e., sections of $\pi$
(\cite[V\bis (1.2.8)]{SGA4}), which is known
to be a $\bU$-topos (\cite[V\bis Proposition (1.2.12)]{SGA4}).
For $i\in \Ob D$, the evaluation at $i$ defines a functor $e_i^*\colon 
\uGamma(E)\to E_i;\CF\mapsto \CF(i)$, which admits a right
adjoint $e_{i*}$ and a left adjoint $e_{i!}$
(\cite[V\bis Corollaire (1.2.11)]{SGA4});
the pair $e_i=(e_i^*,e_{i*})$ defines a morphism of topos 
$e_i\colon E_i\to \uGamma(E)$. 
Let $\CF$ be an object of $\uGamma(E)$.
We write $\CF_i$ for $e_i^*\CF$ for an object $i$ of $D$.
For a morphism $m\colon i\to j$ in $D$, we write
$\tau_{\CF,m}\colon\CF_i\to f_{m*}\CF_j=m^*\CF_j$ 
for the unique morphism in $E_i$ whose composition with
the cartesian morphism $m^*\CF_j\to \CF_j$ over $m$ is $\CF(m)$.
This gives an isomorphism between $\uGamma(E)$
and the category of data consisting of an object
$\CG_i$ of $E_i$ for each $i\in \Ob D$ and 
a morphism $\tau_{\CG,m}\colon \CG_i\to f_{m*}\CG_j$ in $E_i$
for each morphism $m\colon i\to j$ in $D$
satisfying $\tau_{\CG,\id_i}=\id_{\CG_i}$ and the obvious
cocycle condition for every pair of composable
morphisms in $D$. Under this interpretation of 
$\uGamma(E)$, the inverse limit of a $\bU$-small
inverse system $(\CF_{\lambda,i}, \tau_{\CF_{\lambda},m})_{\lambda\in\Lambda}$
is given by $\varprojlim_{\lambda}\CF_{\lambda,i}$
and $\varprojlim_{\lambda}\tau_{\CF_{\lambda},m}
\colon \varprojlim_{\lambda}\CF_{\lambda, i}
\to \varprojlim_{\lambda}(f_{m*}\CF_{\lambda,j})
\cong f_{m*}(\varprojlim_{\lambda}\CF_{\lambda,j})$. 

Let $\CA$ be a ring object of $\uGamma(E)$.
We call the pair $(E, \CA)$ a {\it ringed $D$-topos}. 
For each $i\in \Ob D$, $\CA_i=e_i^*\CA$ is a ring object of $E_i$.
For a morphism $m\colon i\to j$ in $D$, 
the morphism $\tau_{\CA,m}\colon \CA_i\to f_{m*}\CA_j$
is a ring homomorphism. By this construction, 
giving a ring $\CA$ in $\uGamma(E)$ is 
equivalent to giving rings $\CA_i$ in $E_i$
and  ring homomorphisms $\tau_{\CA,m}$ in $E_i$ 
satisfying $\tau_{\CA,\id_i}=\id_{\CA_i}$ and
the cocycle condition for composition of $m$'s.
(See the description of inverse limits in $\uGamma(E)$ in the previous paragraph.)
For another ring object $\CA'$ of $\uGamma(E)$, a morphism 
$\alpha\colon \CA\to \CA'$ in $\uGamma(E)$ is a ring homomorphism
if and only if $\alpha_i:=e_i^*(\alpha)\colon \CA_i\to \CA'_i$ is a ring homomorphism
for every $i\in \Ob D$. 
For $i\in \Ob D$, we have a flat morphism of ringed topos $e_i\colon (E_i,\CA_i)\to (E,\CA)$,
which induces an adjoint pair of functors
$(e_i^*,e_{i*})\colon \Mod(E_i,\CA_i)\to \Mod(\uGamma(E),\CA)$. 
For $m\colon i\to j\in \Mor D$, the ring homomorphism 
$\tau_{\CA,m}\colon \CA_i\to f_{m*}\CA_j$
defines a morphism of ringed topos $f_m\colon (E_j,\CA_j)\to (E_i,\CA_i)$,
which induces an adjoint pair of functors
$(f_m^*,f_{m*})\colon \Mod(E_j,\CA_j)\to \Mod(E_i,\CA_i)$.
These are compatible with compositions. 
The category $\Mod(\uGamma(E), \CA)$ is isomorphic to
the category of data consisting of an $\CA_i$-module
$\CM_i$ for each $i\in \Ob D$ and a morphism 
$\tau_{\CM,m}\colon \CM_i\to f_{m*}\CM_j$ of $\CA_i$-modules
for each
morphism $m\colon i\to j$ in $D$ satisfying 
$\tau_{\CM,\id_i}=\id_{\CM_i}$ and the cocycle
condition for composition of $m$'s.
This description allows us to show that a
sequence $\CM\to \CM'\to \CM''$ in 
$\Mod(\uGamma(E),\CA)$ is exact if and only if 
the sequence $\CM_i\to \CM'_i\to \CM''_i$
is exact for every $i\in \Ob D$.
We define $J_{(E,\CA)}$ to be the full subcategory
of $\Mod(\uGamma(E),\CA)$
consisting of $\CM$ with $\CM_i=e_i^*(\CM)
\in \Ob \Mod(E_i,\CA_i)$ flasque 
(\cite[V D\'efinition 4.1]{SGA4})
(i.e., $H^r(X,-)$ vanishes for every $r>0$ and $X\in \Ob E_i$)
for every $i\in \Ob D$. 
Then the subcategory $J_{(E,\CA)}$ is stable under
extensions and every object of $\Mod(\uGamma(E),\CA)$
admits a monomorphism to an object of $J_{(E,\CA)}$
(\cite[V\bis Proposition (1.3.10) (i), (ii)]{SGA4}). 

Let $\pi'\colon E'\to D$ be another $D$-topos, and let 
$\Phi_*\colon E\to E'$ be a functor over $D$ satisfying
the following condition.
\begin{equation}\label{cond:DtoposDIFunctor}
\text{\parbox[t]{.85\linewidth}{
For every $i\in \Ob D$, there exists a morphism of topos
$\Phi_i=(\Phi_i^*,\Phi_{i*})\colon E_i\to E_i'$ such that
the fiber of $\Phi_*$ over $i$ is $\Phi_{i*}$. 
}}
\end{equation}
For a morphism $m\colon i\to j$ in $D$ and $\CF\in \Ob E_j$,
taking $\Phi_*$ of the cartesian morphism $f_{m*}\CF=m^*\CF\to \CF$
lying over $m$ gives a morphism 
$\Phi_{i*}f_{m*}\CF\to \Phi_{j*}\CF$ over $m$, which is uniquely decomposed
into the composition of a morphism $\Phi_{i*}f_{m*}\CF\to f'_{m*}\Phi_{j*}\CF$
in $E'_i$ with the unique cartesian morphism $f'_{m*}\Phi_{j*}\CF
\to \Phi_{j*}\CF$ over $m$, where $f_m'$ denotes the morphism of topos 
$(m_*,m^*)\colon E'_j\to E_i'$ defined by $\pi'\colon E'\to D$. 
This defines  a morphism of functors
\begin{equation}\label{eq:DtoposDirectImageFib}
b_{m*}\colon \Phi_{i*}f_{m*}\longrightarrow f'_{m*}\Phi_{j*}
\end{equation}
satisfying $b_{\id_i*}=\id_{\Phi_{i*}}$ 
for all $i\in \Ob D$ and the obvious cocycle condition
for every pair of composable morphisms in $D$. 
Giving the functor $\Phi_*$ satisfying \eqref{cond:DtoposDIFunctor}
is equivalent to giving  a functor $\Phi_{i*}\colon E_i\to E_i'$
admitting an exact left adjoint for each $i$ and 
$b_{m*}$ for each $m\colon i\to j$  satisfying the
conditions above. The functor $\Phi_*$ is reconstructed
from $\Phi_{i*}$ and $b_{m*}$ as follows: We have
$\Phi_*\CF=\Phi_{i*}\CF$ for $i\in \Ob D$ and $\CF\in \Ob E_i$.
For $m\colon i\to j\in \Mor D$ and $\alpha\colon \CF\to \CG\in \Mor E$
lying over $m$, $\Phi_*(\alpha)$ is the composition 
$\Phi_{i*}\CF\xrightarrow{\Phi_*(\beta)}\Phi_{i*}f_{m*}\CG
\xrightarrow{b_{m*}(\CG)}
f'_{m*}\Phi_{j*}\CG\to \Phi_{j*}\CG$,
where the last morphism is the cartesian morphism in $E'$
lying over $m$, and $\beta$ is the unique morphism 
$\CF\to f_{m*}\CG$ in $E_i$ whose composition with
the cartesian morphism $f_{m*}\CG\to \CG$ in $E$ is $\alpha$.
The composition with $\Phi_*$ defines
a functor $\uGamma(\Phi_*)\colon 
\uGamma(E)\to \uGamma(E')$. In terms of $\Phi_{i*}$ and $b_{m*}$,
the functor $\uGamma(\Phi_*)$ is given by 
$(\uGamma(\Phi_*)\CF)_i=\Phi_{i*}\CF_i$ and 
\begin{equation}\label{eq:DtoposPFConst}
\tau_{\Phi_*\CF,m}\colon 
\Phi_{i*}\CF_i\xrightarrow{\Phi_{i*}(\tau_{\CF,m})}
\Phi_{i*}f_{m*}\CF_j\xrightarrow{b_{m*}(\CF_j)}
f'_{m*}\Phi_{j*}\CF_j.\end{equation}
We have
$e_i^*\uGamma(\Phi_*)=\Phi_{i*}e_i^*\colon 
\uGamma(E)\to E_i'$. Since $\Phi_{i*}$
preserves $\bU$-small inverse limits for every $i\in \Ob D$,
we see that $\uGamma(\Phi_*)$ preserves
$\bU$-small inverse limits. 

By taking left adjoints of $b_{m*}$, we obtain a morphism
of functors $b_m^*\colon \Phi_j^*f_m^{\prime*}\to f_m^*\Phi_i^*$
satisfying $b_{\id_i}^*=\id_{\Phi_i^*}$ and the cocycle
condition for every pair of composable morphisms in $D$.
If $b_{m*}$ are isomorphisms for all $m$, which
is equivalent to $\Phi_*$ being cartesian, then 
$b_m^*$ are isomorphisms for all $m$ and define
a functor $\Phi^*\colon E'\to E$ with fiber $\Phi_i^*$ over $i\in \Ob D$
as follows
(\cite[V\bis Lemme (1.2.16)]{SGA4}):
For a morphism $\alpha\colon \CF\to \CG$ in $E'$ lying over a morphism 
$m\colon i\to j$ in $D$, $\Phi^*(\alpha)$ is the
composition of the cocartesian morphism 
$\Phi_i^*\CF\to m_*\Phi_i^*\CF=f_m^*\Phi_i^*\CF$
over $m$ with $f_m^*\Phi_i^*\CF\xrightarrow{(b_m^*)^{-1}}
\Phi_j^*f_m^{\prime*}\CF\xrightarrow{\Phi_j^*(\beta)}\Phi_j^*\CG$,
where $\beta$ denotes the unique morphism $f_{m}^{\prime*}\CF=m_*\CF\to \CG$
in $E'_j$ whose composition with the cocartesian morphism 
$\CF\to m_*\CF$ over $m$ is $\alpha$. The composition 
with $\Phi^*$ defines an exact left adjoint 
$\uGamma(\Phi^*)\colon \uGamma(E')\to \uGamma(E)$
of $\uGamma(\Phi_*)$, and therefore the adjoint
pair $(\uGamma(\Phi^*),\uGamma(\Phi_*))$
defines a morphism of topos $\uGamma(\Phi)\colon
\uGamma(E)\to\uGamma(E')$ (\cite[V\bis Proposition (1.2.15)]{SGA4}.

\begin{lemma}\label{lem:DToposDIComp}
Let $E$, $E'$, and $E''$ be $D$-topos, and let $E\xrightarrow{\Phi_*}E'
\xrightarrow{\Phi'_*}E''$ be functors over $D$ satisfying \eqref{cond:DtoposDIFunctor}.
Let $E_i\xrightarrow{\Phi_{i*}} E_i'\xrightarrow{\Phi_{i*}'} E_i''$
$(i\in \Ob D)$, and $b_{m*}\colon \Phi_{i*}f_{m*}
\to f'_{m*}\Phi_{j*}$, $b'_{m*}\colon \Phi'_{i*}f'_{m*}
\to f''_{m*}\Phi'_{j*}$ $(m\colon i\to j\in \Mor D)$
denote the functors and the morphisms of functors 
corresponding to $\Phi_*$ and $\Phi'_*$. Then 
the composition $\Phi'_*\circ\Phi_*$ 
corresponds to $\Phi'_{i*}\circ\Phi_{i*}$ $(i\in \Ob D)$
and $b_{m*}''\colon
\Phi_{i*}'\Phi_{i*}f_{m*}
\xrightarrow{\Phi'_{i*}b_{m*}}
\Phi'_{i*}f'_{m*}\Phi_{j*}
\xrightarrow{b'_{m*}\Phi_{j*}}
f''_{m*}\Phi'_{j*}\Phi_{j*}$ $(m\colon i\to j\in \Mor D)$.
\end{lemma}
\begin{proof} Straightforward. \end{proof}

\begin{lemma}\label{lem:DToposDIMorph}
Let $E$ and $E'$ be $D$-topos, and let $\Phi_*$, $\Phi'_*\colon E\to E'$
be two functors over $D$ satisfying \eqref{cond:DtoposDIFunctor}.
Let $\Phi_{i*}$, $\Phi'_{i*}\colon E_i\to E_i'$ $(i\in \Ob D)$
and $b_{m*}\colon \Phi_{i*}f_{m*}\to f'_{m*}\Phi_{j*}$,
$b'_{m*}\colon \Phi'_{i*}f_{m*}\to f'_{m*}\Phi'_{j*}$ $(m\colon i\to j\in \Mor D)$
be the functors and the morphisms of functors corresponding
to $\Phi_*$ and $\Phi'_*$. Suppose that we are given
a morphism of functors $c_i\colon \Phi_{i*}\to \Phi_{i*}'$
for each $i\in \Ob D$. Then there exists
a morphism of functors $c\colon \Phi_*\to \Phi'_*$
over $D$ satisfying $c\vert_{E_i}=c_i$ for each $i\in \Ob D$
if and only if the following diagram is commutative for
every $m\colon i\to j\in \Mor D$.
\begin{equation*}
\xymatrix@R=15pt{
\Phi_{i*}f_{m*}\ar[d]_{c_if_{m*}}\ar[r]^{b_{m*}}&f'_{m*}\Phi_{j*}
\ar[d]^{f'_{m*}c_j}\\
\Phi_{i*}'f_{m*}\ar[r]^{b'_{m*}}&f'_{m*}\Phi'_{j*}
}
\end{equation*}
\end{lemma}

\begin{proof}
There exists $c$ if and only if, for any $\alpha\colon \CF\to \CG\in \Mor E$
lying over $m\colon i\to j\in\Mor D$, we have
$c_j(\CG)\circ\Phi_*(\alpha)=\Phi'_*(\alpha)\circ c_i(\CF)
\colon \Phi_*(\CF)=\Phi_{i*}(\CF)\to \Phi'_*(\CG)=\Phi'_{j*}(\CG)$.
Let $\beta\colon \CF\to f_{m*}\CG$ be the unique morphism in $E_i$ whose
composition with the unique cartesian morphism $f_{m*}\CG\to \CG$
over $m$ is $\alpha$. Then we have the following diagram, where
the right two horizontal morphisms are cartesian over $m$. 
\begin{equation*}
\xymatrix@C=40pt{
\Phi_{i*}(\CF)\ar[r]^(.47){\Phi_{i*}(\beta)}\ar[d]_{c_i(\CF)}&
\Phi_{i*}f_{m*}(\CG)\ar[r]^{b_{m*}(\CG)}\ar[d]_{c_i(f_{m*}\CG)}&
f'_{m*}\Phi_{j*}(\CG)\ar[r]\ar[d]^{f'_{m*}c_j(\CG)}&
\Phi_{j*}(\CG)\ar[d]^{c_j(\CG)}\\
\Phi'_{i*}(\CF)\ar[r]^(.47){\Phi'_{i*}(\beta)}&
\Phi'_{i*}f_{m*}(\CG)\ar[r]^{b'_{m*}(\CG)}&
f'_{m*}\Phi'_{j*}(\CG)\ar[r]&
\Phi'_{j*}(\CG)
}
\end{equation*}
The composition of the upper (resp.~lower) horizontal morphisms
are $\Phi_*(\alpha)$ (resp.~$\Phi'_*(\alpha)$), and the left and
the right squares are commutative. This implies the desired
equivalence; we see the necessity by considering the case
$\beta=\id$ and using the cartesian property of the bottom
right horizontal morphism.
\end{proof}

Let $\CA'$ be a ring object of $\uGamma(E')$ and
suppose that we are given a ring homomorphism 
$\theta\colon \CA'\to \uGamma(\Phi_*)\CA$, 
which induces a ring homomorphism $\theta_i\colon \CA_i'=
e_i^*\CA_i'\to e_i^*\uGamma(\Phi_*)\CA
=\Phi_{i*}e_i^*\CA=\Phi_{i*}\CA_i$. 
Giving a ring homomorphism $\theta\colon \CA'\to \uGamma(\Phi_*)\CA$
is equivalent to giving a ring homomorphism
$\theta_i\colon \CA'_i\to \Phi_{i*}\CA_i$ for each $i\in \Ob D$
compatible with $\tau_{\CA',m}\colon \CA'_i\to f'_{m*}\CA'_j$
and $\tau_{\Phi_*\CA,m}\colon \Phi_{i*}\CA_i\to f'_{m*}\Phi_{j*}\CA_j$
\eqref{eq:DtoposPFConst}. 
The pair of $\Phi_i$ and $\theta_i$
defines a morphism of ringed topos
$\varphi_i\colon (E_i,\CA_i)\to (E'_i,\CA'_i)$ for each $i\in \Ob D$. 
The pair $\varphi_*=(\Phi_*,\theta)$ induces a left exact functor
$\uGamma(\varphi_*)\colon \Mod(\uGamma(E),\CA)
\to \Mod(\uGamma(E'),\CA')$ since
$\uGamma(\Phi_*)$ preserves $\bU$-small inverse limits.
We define $J_{(E',\CA')}$ in the same way as $J_{(E,\CA)}$,
and $L_{(E,\CA)}$ (resp.~$L_{(E',\CA')}$) to be the
full subcategory of $K^+(\uGamma(E),\CA)$ 
(resp.~$K^+(\uGamma(E'),\CA')$) consisting of 
complexes of objects of $J_{(E,\CA)}$ (resp.~$J_{(E',\CA')}$),
which is a triangulated subcategory. 
Since $(\uGamma(\varphi)_*\CM)_i=\varphi_{i*}\CM_i$
for $\CM\in \Ob \Mod(\uGamma(E),\CA)$,
we have $\uGamma(\Phi_*)(L_{(E,\CA)})
\subset L_{(E',\CA')}$ 
(\cite[V Proposition 4.9 1)]{SGA4}),
and,  if $\CI^{\bullet}\in \Ob L_{(E,\CA)}$  is acyclic,
$\uGamma(\Phi_*)(\CI^{\bullet})$ is acyclic
(\cite[V Proposition 5.2]{SGA4}).
Therefore the right derived functor of $\uGamma(\varphi_*)
\colon K^+(\uGamma(E),\CA)\to K^+(\uGamma(E'),\CA')$
is given by the composition (\cite[I Lemma 4.6 1), Theorem 5.1 and its proof]{RD})
\begin{equation}
R^+\uGamma(\varphi_*)\colon D^+(\uGamma(E),\CA)
\xleftarrow{\sim} (L_{(E,\CA)})_{\Qis}
\xrightarrow{\uGamma(\varphi_*)_{\Qis}}
(L_{(E',\CA')})_{\Qis}
\xrightarrow{\sim}
D^+(\uGamma(E'),\CA'),
\end{equation}
where $(-)_{\Qis}$ denotes the localization by quasi-isomorphisms;
every $\CM^{\bullet}\in \Ob C^+(\uGamma(E),\CA)$
admits a quasi-isomorphism $\CM^{\bullet}\to \CI^{\bullet}$
with $\CI^{\bullet}\in \Ob L_{(E,\CA)}$ and it induces
$R^+\uGamma(\varphi_*)(\CM^{\bullet})
\cong R^+\uGamma(\varphi_*)(\CI^{\bullet})
\cong \uGamma(\varphi_*)(\CI^{\bullet})$
in $D^+(\uGamma(E'),\CA')$. 
Since $\CM_i^{\bullet}\to \CI^{\bullet}_i$ is a 
quasi-isomorphism to a complex of flasque $\CA_i$-modules
on $E_i$ for every $i\in \Ob D$,
we see that the isomorphism $e_i^*\circ\uGamma(\varphi_*)
\xrightarrow{\sim}\varphi_{i*}\circ e_i^*$ induces an isomorphism
\begin{equation}\label{eq:DtoposFBF}
e_i^*\circ R^+\uGamma(\varphi_*)\xrightarrow{\;\cong\;}
R^+\varphi_{i*}\circ e_i^*
\colon D^+(E,\CA)\to D^+(E'_i,\CA'_i).
\end{equation}

Next let us discuss topos of inverse systems and inverse
limit functors in connection with $D$-topos discussed above.
Let $I$ be a $\bU$-small category, and let $I^{\circ}$
be its opposite category. Let $T$ be a topos. 
We define $T^{I^{\circ}}$ to be the category of functors
$I^{\circ}\to T$, i.e., inverse systems in $T$ indexed by $I$.
For an object $\CF$ (resp.~a morphism $\alpha$) in $T^{I^{\circ}}$
and $r\in \Ob I$, we write $\CF_r$ (resp.~$\alpha_r$) for its value
at $r$. The functor taking the inverse limit $\varprojlim_I\colon 
T^{I^{\circ}}\to T$ is a right adjoint of the exact functor
$c_I\colon T\to T^{I^{\circ}}$ sending $\CF$ to the
constant inverse system consisting of $\CF$. We write
$\projl_I$ for the morphism of topos
$(c_I,\varprojlim_I)\colon T^{I^{\circ}}\to T$. 
The projection functor $T\times I^{\circ}\to I^{\circ}$ is
an $I^{\circ}$-topos, and we have $\uGamma(T\times I^{\circ})=
T^{I^{\circ}}$. Hence $T^{I^{\circ}}$ is a topos and we have a morphism 
of topos $e_r=(e_r^*,e_{r*})\colon T\to T^{I^{\circ}}$ with $e_r^*\CF=\CF_r$
for $r\in \Ob I$. 
Let $\CA$ be a ring object of $T^{I^{\circ}}$, which is an inverse
system of ring objects of $T$ indexed by $I$. Then an $\CA$-module
on $T^{I^{\circ}}$ is an inverse system of $\CA_r$-modules on $T$.
A ring homomorphism $\sigma\colon \CB\to \varprojlim_I\CA$
on $T$ induces a morphism of ringed topos $(\projl_I,\sigma)\colon
(T^{I^{\circ}},\CA)\to (T,\CB)$. 
Let $\Phi=(\Phi^*,\Phi_*)\colon T\to T'$ be a morphism of topos.
Then the composition with $\Phi$ defines a morphism of topos
$\Phi^{I^{\circ}}\colon T^{I^{\circ}}\to T^{\prime I^{\circ}}$, which is associated to the cartesian functor
$\Phi_*\times \id_{I^{\circ}}\colon T\times I^{\circ}\to T'\times I^{\circ}$
over $I^{\circ}$.
Let $\CA$ and $\CA'$ be ring objects of $T^{I^{\circ}}$ and
$T^{\prime I^{\circ}}$, respectively, and suppose that
we are given a ring homomorphism 
$\theta\colon \CA'\to \Phi^{I^{\circ}}_*\CA$.
Then, for morphisms of ringed topos
$\varphi=(\Phi^{I^{\circ}},\theta)\colon 
(T^{I^{\circ}},\CA)\to (T^{\prime I^{\circ}},\CA')$
and $\varphi_r=(\Phi,\theta_r)\colon 
(T,\CA_r)\to (T',\CA'_r)$ $(r\in \Ob I)$, we obtain the
following isomorphism from \eqref{eq:DtoposFBF}.
\begin{equation}\label{eq:InvSySFBF}
e_r^*\circ R^+\varphi_*\xrightarrow{\cong}
R^+\varphi_{r*}\circ e_r^*\colon D^+(T^{I^{\circ}},\CA)
\to D^+(T',\CA'_r)
\end{equation}

Let $D$ and $I$ be $\bU$-small categories, and let $\pi\colon E\to D$ be
a $D$-topos. We study the case $T=\uGamma(E)$. 
The functor $\pi\times \id_{I^{\circ}}\colon 
E\times I^{\circ}\to D\times I^{\circ}$ defines a
$D\times I^{\circ}$-topos: For a morphism 
$(m,u)\colon (i,r)\to (j,s)$ in $D\times I^{\circ}$ and
objects $\CF$ and $\CG$ above $(i,r)$ and $(j,s)$,
we have $\Hom_{E\times I^{\circ},(m,u)}(\CF,\CG)
=\Hom_{E_i}(\CF,m^*\CG)=\Hom_{E_j}(m_*\CF,\CG)$.
We have $\uGamma(E\times I^{\circ})=\uGamma(E)^{I^{\circ}}$
and the composition 
$\uGamma(E)^{I^{\circ}}\xrightarrow{e_r^*}
\uGamma(E)\xrightarrow{e_i^*}E_i$
coincides with $e_{(i,r)}^*\colon \uGamma(E\times I^{\circ})
\to E_i$ for $(i,r)\in \Ob (D\times I^{\circ})$. 
Let $\pi'\colon E'\to D$ be another $D$-topos, and let 
$\Phi_*\colon E\to E'$ be a functor over $D$
satisfying \eqref{cond:DtoposDIFunctor}. Then the functor
$\Phi_*\times\id_{I^{\circ}}\colon 
E\times I^{\circ}\to E'\times I^{\circ}$
over $D\times I^{\circ}$ also satisfies \eqref{cond:DtoposDIFunctor}
and $\uGamma(\Phi_*)^{I^{\circ}}$
coincides with $\uGamma(\Phi_*\times \id_{I^{\circ}})$. 
Let $\CA$ and $\CA'$ be ring objects of $\uGamma(E)^{I^{\circ}}$
and $\uGamma(E')^{I^{\circ}}$, respectively, 
and suppose that we are given a ring homomorphism 
$\theta\colon \CA'\to \uGamma(\Phi_*)^{I^{\circ}}\CA$.
Then the pair $(\uGamma(\Phi_*)^{I^{\circ}},\theta)
=(\uGamma(\Phi_*\times\id_{I^{\circ}}),\theta)$ induces
a left exact functor 
$\varphi_*\colon \Mod(\uGamma(E)^{I^{\circ}},\CA)
\to \Mod(\uGamma(E')^{I^{\circ}},\CA')$.
Let $\theta_{i,r}\colon \CA'_{i,r}\to \Phi_{i*}\CA_{i,r}$,
$(i,r)\in \Ob(D\times I^{\circ})$ be the pullback 
of $\theta$ under the morphism of topos
$e_{(i,r)}\colon E'_i\to \uGamma(E'\times I^{\circ})$.
Then the morphism of ringed topos
$\varphi_{i,r}=(\Phi_i,\theta_{i,r})\colon 
(E_i,\CA_{i,r})\to (E_i',\CA_{i,r}')$ defines a left exact
functor $\varphi_{i,r*}\colon \Mod(E_i,\CA_{i,r})
\to \Mod(E_i',\CA'_{i,r})$. By \eqref{eq:DtoposFBF}, we have an 
isomorphism 
\begin{equation}\label{eq:InvSysDtoposFBF}
e_{(i,r)}^*\circ R^+\varphi_*\xrightarrow{\;\sim\;}
R^+\varphi_{i,r*}\circ e_{(i,r)}^*\colon 
D^+(\uGamma(E)^{I^{\circ}},\CA)
\longrightarrow D^+(E'_i,\CA'_{i,r}).
\end{equation}

We show that the right derived functor of the inverse limit functor for $\uGamma(E)^{I^{\circ}}$
can be also computed fiber by fiber. By the construction of 
$\bU$-small inverse limits in $\uGamma(E)$, 
the composition 
$\uGamma(E)^{I^{\circ}}\xrightarrow{e_i^{*I^{\circ}}}
E_i^{I^{\circ}}\xrightarrow{\varprojlim_I} E_i$
is canonically isomorphic to the composition 
$\uGamma(E)^{I^\circ}
\xrightarrow{\varprojlim_I}\uGamma(E)
\xrightarrow{e_i^*}E_i$ for $i\in \Ob D$. 
For $\CF\in \Ob \uGamma(E)^{I^{\circ}}$,
we write $\CF_i$ for $(e_i^*)^{I^{\circ}}\CF
\in \Ob (E^{I^{\circ}}_i)$.
Let $\CA$ and $\CB$ be ring objects of
$\uGamma(E)^{I^{\circ}}$ and $\uGamma(E)$,
respectively,  suppose that we are given
a ring homomorphism $\theta\colon \CB\to 
\varprojlim_I\CA$, and let $\theta_i$
be $e_i^*(\theta)\colon 
\CB_i\to e_i^*\varprojlim_I\CA=
\varprojlim_I\CA_i$ for $i\in \Ob D$. We have a
diagram commutative up to a canonical isomorphism
for each $i\in \Ob D$
\begin{equation}\label{eq:DToposInvLimShFBF}
\xymatrix@C=50pt{
\Mod(\uGamma(E)^{I^{\circ}},\CA)
\ar[r]^{e_i^{*I^{\circ}}}\ar[d]_{\varprojlim_I}&
\Mod(E_i^{I^{\circ}},\CA_i)\ar[d]^{\varprojlim_I}\\
\Mod(\uGamma(E),\CB)\ar[r]^{e_i^*}&
\Mod(E_i,\CB_i).
}
\end{equation}
The two horizontal functors are exact.

\begin{proposition}\label{prop:DtoposInvLimitFbF}
The following morphism of functors induced by \eqref{eq:DToposInvLimShFBF}
is an isomorphism for $i\in \Ob D$.
$$
e_i^*\circ R\varprojlim_I
\to R\varprojlim_I\circ e_i^{*I^{\circ}}\colon 
D^+(\uGamma(E)^{I^{\circ}},\CA)
\longrightarrow D^+(E_i,\CB_i)$$
\end{proposition}

\begin{proof}
We define a category $E^{I^{\circ}/D}$ as follows. 
An object is a functor $\CF\colon I^{\circ}\to E$ whose
composition with $\pi\colon E\to D$ is constant, i.e.,
a functor to $E_i$ for an object $i\in \Ob D$. A morphism
$\alpha\colon \CF\to \CG$ is a morphism of functors
whose composition with $\pi$ is constant. 
We can define a functor $\pi^{I^{\circ}/D}\colon
E^{I^{\circ}/D}\to D$ by associating the compositions with
$\pi$ to $\CF$ and $\alpha$ above. The 
fiber over $i\in \Ob D$ is $E_i^{I^{\circ}}$. 
For a morphism $m\colon i\to j$ in $D$,
$\CF\in \Ob E_i^{I^{\circ}}$,
and $\CG\in \Ob E_j^{I^{\circ}}$,
we have $\Hom_{E^{I^{\circ}/D},m}(\CF,\CG)
=\Hom_{E_i^{I^{\circ}}}(\CF,f_{m*}^{I^{\circ}}\CG)
\cong \Hom_{E_j^{I^{\circ}}}(f_m^{*I^{\circ}}\CF,\CG)$.
Hence $E^{I^{\circ}/D}$ is a $D$-topos whose
pullback and pushforward functors with respect
to a morphism $m\colon i\to j$ in $D$ are given by 
the morphism of topos $f_m^{I^{\circ}}\colon 
E_j^{I^{\circ}}\to E_i^{I^{\circ}}$. We can define
a functor $\varprojlim_{I/D}\colon E^{I^{\circ}/D}\to E$ over $D$ by 
sending $\CF\in \Ob E_i^{I^{\circ}}$
$(i\in \Ob D)$ to $\varprojlim_I\CF$ in $E_i$ and $\alpha\colon \CF\to\CG$
over $m\colon i\to j\in \Mor D$ to 
the composition of 
$\varprojlim_I\CF\to\varprojlim_If_{m*}^{I^{\circ}}\CG
\cong f_{m*}\varprojlim_I\CG$ with 
the cartesian morphism 
$f_{m*}\varprojlim_I\CG\to \varprojlim_I\CG$
over $m$. 

An object of $\uGamma(E^{I^{\circ}/D})$
consists of an inverse system 
$((\CF_{r,i})_{r\in \Ob I^{\circ}}, (\sigma_{\CF,u,i}
\colon \CF_{r,i}\to \CF_{s,i})_{u\colon r\to s\in \Mor I^{\circ}})
$$\in E^{I^{\circ}}_i$
for each $i\in \Ob D$ and a morphism 
$\tau_{\CF,r,m}\colon \CF_{r,i}\to f_{m*}\CF_{r,j}$
in $E_i$ for each $m\colon i\to j\in \Mor D$ and
$r\in \Ob I^{\circ}$ compatible with $\sigma_{\CF,u,i}$
and $\sigma_{\CF,u,j}$, and satisfying
$\tau_{\CF,r,\id}=\id$ and the cocycle condition for
composable morphisms in $D$. 
An object of $\uGamma(E)^{I^{\circ}}$ consists of
an object $((\CG_{r,i})_{i\in \Ob D},
(\tau_{\CG,r,m}\colon \CG_{r,i}\to f_{m*}\CG_{r,j})_{m\colon i\to j\in \Mor D})$
of $\uGamma(E)$ for each $r\in \Ob I^{\circ}$
and a morphism 
$\sigma_{\CG,u,i}\colon \CG_{r,i}\to \CG_{s,i}$ in $E_i$
for each $u\colon r\to s\in \Mor I^{\circ}$ and
$i\in \Ob D$ compatible with $\tau_{\CG,r,m}$ and 
$\tau_{\CG,s,m}$, and satisfying 
$\sigma_{\CG,\id,i}=\id$ and the cocycle condition for 
composable morphisms in $I^{\circ}$. Morphisms
in $\uGamma(E^{I^{\circ}/D})$ and
in $\uGamma(E)^{I^{\circ}}$ are given by morphisms in 
$E_i$'s compatible with $\sigma$'s and $\tau$'s. This observation
leads us a natural identification $\uGamma(E^{I^{\circ}/D})
=\uGamma(E)^{I^{\circ}}$, under which we have 
$e_i^*=(e_i^*)^{I^{\circ}}$ for
each $i\in \Ob D$. 
By the construction of $\bU$-small inverse limits in $\uGamma(E)$
fiber by fiber,
$\varprojlim_I\colon \uGamma(E)^{I^{\circ}}\to \uGamma(E)$
coincides with $\uGamma(\varprojlim_{I/D})
\colon \uGamma(E^{I^{\circ}/D})\to \uGamma(E)$. 
Since the fiber $\varprojlim_{I}\colon 
E_i^{I^{\circ}}\to E_i$ of $\varprojlim_{I/D}$
over $i\in \Ob D$ is the
direct image functor of the morphism of topos
$\projl_I\colon E_i^{I^{\circ}}\to E_i$,
we can apply \eqref{eq:DtoposFBF} to $\varprojlim_{I/D}$
and $\theta\colon \CB\to \uGamma(\varprojlim_{I/D})_*\CA$,
obtaining the desired claim.
\end{proof}

\bibliographystyle{amsplain}
\bibliography{qHiggsRefData}
\end{document}